\documentclass[12pt,reqno]{amsart}
\usepackage[dvipdfmx]{hyperref}
\usepackage{latexsym,amssymb,braket}
\usepackage{mathrsfs}
\usepackage{enumerate}
\usepackage{bm}
\usepackage[all]{xy}
\allowdisplaybreaks[1]
\newtheorem{thm}{Theorem}[section]
\newtheorem{lem}[thm]{Lemma}
\newtheorem{cor}[thm]{Corollary}
\newtheorem{prop}[thm]{Proposition}
\newtheorem{prob}[thm]{Problem}

\newtheorem{clam}[]{Claim}
\newtheorem*{clm}{Claim}

\newtheorem{thme}[]{Theorem}

\theoremstyle{definition}

\newtheorem{defn}[thm]{Definition}
\newtheorem{rem}[thm]{Remark}

\numberwithin{equation}{section}

\newcommand{\N}{\mathbb{N}}
\newcommand{\Z}{\mathbb{Z}}

\newcommand{\R}{\mathbb{R}}
\newcommand{\C}{\mathbb{C}}

\def\cF{\mathcal{F}}
\def\cG{\mathcal{G}}
\def\cH{\mathcal{H}}

\def\cK{\mathcal{K}}
\def\cL{\mathcal{L}}

\def\sB{\mathscr{B}}
\def\sC{\mathscr{C}}
\def\sD{\mathscr{D}}

\def\tM{\widetilde{M}}

\DeclareMathOperator{\Ad}{Ad}
\DeclareMathOperator{\Aut}{Aut}

\DeclareMathOperator{\End}{End}

\DeclareMathOperator{\id}{id}

\DeclareMathOperator{\Ind}{Ind}
\DeclareMathOperator{\Int}{Int}

\DeclareMathOperator{\Irr}{Irr}
\DeclareMathOperator{\mo}{mod}
\DeclareMathOperator{\Mor}{Mor}
\DeclareMathOperator{\ONB}{ONB}
\DeclareMathOperator{\Out}{Out}

\DeclareMathOperator{\tr}{tr}
\DeclareMathOperator{\Tr}{Tr}

\def\al{\alpha}
\def\hal{{\widehat{\al}}}
\def\tal{\widetilde{\alpha}}

\def\de{\delta}

\def\la{\lambda}

\def\orho{{\overline{\rho}}}

\def\osi{{\overline{\sigma}}}

\def\btr{\mathbf{1}}

\def\nin{\notin}

\def\ovl{\overline}

\def\hG{\widehat{G}}

\begin{document}
\title{Centrally free
actions of amenable C$^*$-tensor categories on von Neumann algebras}

\author[R. Tomatsu]{Reiji Tomatsu}
\address{
Department of Mathematics, Hokkaido University,
Hokkaido \mbox{060-0810},
JAPAN}
\email{tomatsu@math.sci.hokudai.ac.jp}

\thanks{The author was partially
supported by JSPS KAKENHI Grant Number 18K03317.}

\maketitle

\begin{abstract}
We will show a centrally free action of an amenable rigid
C$^*$-tensor category on a properly infinite von Neumann algebra
has the Rohlin property.
This enables us to prove
the fullness of the crossed product of a full factor by minimal
action of a compact group.
As another application of the Rohlin property,
we will show the uniqueness of centrally free cocycle actions
of an amenable rigid C$^*$-tensor category
on properly infinite von Neumann algebras.
\end{abstract}

\section{Introduction}
In this paper,
we will formulate the Rohlin property for a cocycle action
of an amenable rigid C$^*$-tensor category on a properly infinite
von Neumann algebra with separable predual
and discuss its applications to two subjects:
the fullness and classification of actions.

A norm bounded sequence $(x_n)$ of a von Neumann algebra $M$
is said to be central
when
$x_n\varphi-\varphi x_n\to0$ in the strong$*$ topology
as $n\to\infty$ for all $\varphi\in M_*$.
When for a central sequence $(x_n)$ there exists $\lambda\in\C$
such that $x_n\to\lambda$ in the strong$*$ topology,
we will say $(x_n)$ is trivial.
A von Neumann factor $M$ with separable predual
is said to be \emph{full}
if any central sequence is trivial.
One of equivalent conditions to the fullness
is the closedness of the inner automorphism group
$\Int(M)$ inside the automorphism group $\Aut(M)$
with respect to the \textit{u}-topology \cite{Co},
and the quotient group $\Out(M)=\Aut(M)/\Int(M)$ is Polish.
Our main concern on fullness is to study the fullness
of the crossed product
$M\rtimes_\alpha G$ for an action
of a compact (quantum) group $G$.

In \cite{TU},
it is shown that if $G$ is compact abelian
and $\alpha$ is a pointwise outer action on a full factor $M$
such that the crossed product $M\rtimes_\alpha G$ is a factor,
then $M\rtimes_\alpha G$ is full.
This result enables us to solve the problem
of a characterization
of the fullness of the continuous core factor
of a type III$_1$ free product factor
in terms of Connes' $\tau$-invariant.
This is generalized to any type III$_1$ full factor
by A. Marrakchi \cite{Mar}.
These results lead us to a more general question
(Problem \ref{prob:outerfull})
raised by Marrakchi which asks whether
the crossed product $M\rtimes_\alpha G$ is a full factor
for
$\alpha$ being a pointwise outer
action of a locally compact group $G$ on a full factor $M$
such that $G$ is homeomorphic to its image in $\Out(M)$
through the quotient map.
In this paper, we will tackle this problem
assuming $G$ is compact and $\alpha$ is minimal.
Our main result of Section \ref{sect:discrete}
is the following
(Theorem \ref{thm:minimal}):

\begin{thme}
\label{thme:minimal}
Let $\alpha$ be a minimal action of a compact group
$G$ on a full factor $M$.
Then $M\rtimes_\alpha G$ is a full factor.
\end{thme}

Our strategy is to analyze the dual action.
We first study a free action $\beta$ of $\widehat{G}$
on a factor $N$ and show the following result
(Theorem \ref{thm:enhance}).
Then we obtain Theorem \ref{thme:minimal}
as its immediate corollary.

\begin{thme}
Let $G$ be a compact group
and $\beta$ a free action of $\hG$
on a von Neumann algebra $N$.
Suppose that
the crossed product $N\rtimes_\beta\hG$ is a full factor
and the dual action $\widehat{\beta}$ of $G$ is pointwise outer.
Then $N$ is a full factor.
\end{thme}

Our key ingredient is a Rohlin tower construction
studied in \cite{MT-minimal,MT-III,MT-Roh,MT-discrete}
for amenable discrete Kac algebras.
Although we are actually interested in a compact group action,
we will construct a Rohlin tower
for a centrally free cocycle action of an amenable rigid C$^*$-tensor category
on a von Neumann algebra (Theorem \ref{thm:Rohlin-tensor})
since this enables us to omit many tensor notations.
An amenable rigid C$^*$-tensor category $\sC$
is a C$^*$-tensor category $\sC$ such that
each object has its conjugate object
and the simple objects of $\sC$
satisfy the F\o lner type condition.
Throughout this paper,
$\Irr(\sC)$,
the complete set of simple objects of $\sC$, is assumed
to be at most countable,
but $\Irr(\sC)$ is not necessarily finitely generated.
A \emph{cocycle action} of $\sC$ on a properly infinite von Neumann algebra
$M$ means a unitary tensor functor $\alpha$ from $\sC$
into the W$^*$-multitensor category consisting of endomorphisms
with conjugate objects on $M$.
Then we show the following result on a discrete subfactor
(Theorem \ref{thm:discretefull}):

\begin{thme}
\label{thme:discretefull}
Let $N\subset M$ be a discrete inclusion of infinite factors
with separable preduals
such that the associated C$^*$-tensor category $\sC$
generated by
${}_N L^2(M)_N$
is amenable and gives a centrally free cocycle action on $N$.
If $M$ is a full factor,
then $\sC$ is in fact trivial,
that is, $N=M$.
\end{thme}

We will discuss another application of the Rohlin property,
that is, the classification of centrally free actions
of an amenable rigid C$^*$-tensor category.
Two cocycle actions $\alpha$ and $\beta$ of $\sC$ on $M$
are said to be \emph{cocycle conjugate}
if they are naturally isomorphic unitary tensor functors.
Our main result in Section \ref{sect:classification}
is the following one (Theorem \ref{thm:classification}):

\begin{thme}
\label{thme:classification}
Let $\sC$ be an amenable rigid C$^*$-tensor category.
Let $(\alpha,c^\alpha)$ and $(\beta,c^\beta)$
be centrally free cocycle actions
of $\sC$ on a properly infinite von Neumann algebra $M$
with separable predual.
Suppose that $\alpha_X$ and $\beta_X$ are approximately
unitarily equivalent for all $X\in\sC$
and also $\alpha$ has an invariant
faithful normal state on $Z(M)$.
Then $(\alpha,c^\alpha)$ and $(\beta,c^\beta)$ are strongly cocycle conjugate.
\end{thme}

The ``strong" means we can take an approximately inner automorphism
as a conjugacy.
This generalizes the preceding works by M. Izumi and T. Masuda
\cite[Theorem 2.2]{Iz-near}
and \cite[Theorem 3.4]{Mas-Rob},
where the injective type III$_1$ factor with separable predual
is considered.
It should be mentioned that
S. Popa's celebrated classification result
of strongly amenable subfactors \cite[Theorem 5.1]{Popa-endo}
is crucially used in their proofs.
We, however, do not need the classification result
nor the factoriality to prove Theorem \ref{thme:classification}.

We can also apply our discussion for C$^*$-tensor categories
to cocycle actions of amenable rigid C$^*$-2-categories as follows
(Theorem \ref{thm:classification2}):

\begin{thme}
\label{thme:classification2}
Let $\sC=(\sC_{rs})_{rs\in\Lambda}$
be an amenable rigid C$^*$-2-category
with $\Lambda=\{0,1\}$.
Let $(\alpha,c^\alpha)$ and $(\beta,c^\beta)$
be cocycle actions
of $\sC$ on a system of properly infinite von Neumann algebras
$M=(M_r)_{r\in\Lambda}$
with separable preduals.
Suppose that the following four conditions hold:
\begin{itemize}
\item 
$\alpha_X(Z(M_s))=Z(M_r)$
for $X\in\Irr(\sC_{rs})$ with $r,s\in\Lambda$.

\item
There exist faithful normal states $\varphi_r$ on $M_r$
with $r\in\Lambda$
such that $\alpha_X(\varphi_s)=\varphi_r$ on $Z(M_r)$
for all $X\in\sC_{rs}$ and $r,s\in\Lambda$.

\item
$(\alpha^{00},c)$ is centrally free.

\item
$\alpha_X$ and $\beta_X$ are approximately unitarily equivalent
for all $X\in\sC$.
\end{itemize}
Then $(\alpha,c^\alpha)$ and $(\beta,c^\beta)$ are strongly cocycle conjugate.
\end{thme}

As a corollary of Theorem \ref{thme:classification},
we will show Theorem \ref{thme:subfactor2}
(Theorem \ref{thm:subfactor2}),
which recovers the uniqueness part of Popa's classification result
\cite[Theorem 5.1]{Popa-endo}.
Our proof of Theorem \ref{thme:classification} uses
the Bratteli--Elliott--Evans--Kishimoto type intertwining argument
\cite{EK}.
This enables us to directly compare two subfactors
and to remove the assumption of the ergodicity of standard invariants.
This fact has been
announced by Popa \cite[Problem 5.4.7]{Popa-acta}.

\begin{thme}
\label{thme:subfactor2}
Let $N\stackrel{E}{\subset} M$
and $Q\stackrel{F}{\subset} P$ be an inclusion of factors
with separable preduals and finite indices.
Suppose the following conditions hold:
\begin{itemize}
\item
$N$ and $Q$ are isomorphic.

\item
The standard invariants
of $N\stackrel{E}{\subset} M$ and $Q\stackrel{F}{\subset} P$
are amenable and isomorphic.

\item
$N\subset M$ and $Q\subset P$ are centrally free inclusions.

\item
$N\stackrel{E}{\subset} M$ and $Q\stackrel{F}{\subset} P$
are approximately inner inclusions.
\end{itemize}
Then $N\stackrel{E}{\subset} M$
is isomorphic to $Q\stackrel{F}{\subset} P$.
\end{thme}

This paper is organized as follows.
In Section 2, we will quickly review the concepts of
an ultraproduct von Neumann algebra,
a C$^*$-tensor category and its amenability
and the (central) freeness of cocycle actions.
In Section 3,
we will construct a Rohlin tower for a centrally free action
of an amenable rigid C$^*$-tensor category on a properly infinite von Neumann algebra with separable predual.
In Section 4,
we will study
actions of compact groups on full factors.
In Section 5,
we will classify centrally free actions of an amenable
C$^*$-tensor categories up to strong cocycle conjugacy.
In Section 6, we will extend the classification results
obtained in Section 5 to an amenable rigid C$^*$-2-categories.
In section 7,
we will discuss classification of subfactors such that
their standard invariants are amenable.
Although our proof of Theorem \ref{thme:subfactor2}
is presented in Section 7,
it is worth mentioning that
Theorem \ref{thme:classification} is actually enough
(see Theorem \ref{thm:M1P1isom}).
Hence it is logically possible to
put Theorem \ref{thme:subfactor2} into Section 5.

\vspace{10pt}
\noindent
{\bf Acknowledgements.}
The author is grateful
to Tomohiro Hayashi, Masaki Izumi, Yasuyuki Kawahigashi,
Toshihiko Masuda and Yoshimichi Ueda
for various advice and valuable comments on this paper
and
to Yuki Arano and Makoto Yamashita
for a fruitful discussion on the Rohlin property
and subfactors.
He also thanks Amine Marrakchi
for informing him of the interesting problem
(Problem \ref{prob:outerfull})
and related topics.

\tableofcontents

\section{Notations and Terminology}

\subsection{General Notations}
\label{subsect:notation}
We will mainly treat a von Neumann algebra $M$
(except ultraproducts)
with separable predual $M_*$.
By $Z(M)$ we denote the center of $M$.
We equip $M_*$ with the $M$-$M$ bimodule structure
as usual: $a\varphi b(x):=\varphi(bxa)$
for $a,b,x\in M$ and $\varphi\in M_*$.
We denote by $M_\varphi$ the collection of all elements
in $M$
which commute with a functional $\varphi\in M_*$.
For a positive $\varphi\in M_*$,
we set $|x|_\varphi:=\varphi(|x|)$
and $\|x\|_\varphi:=\varphi(x^*x)^{1/2}$
for $x\in M$.
Note that $|\cdot|_\varphi$ satisfies the triangle inequality
on $M_\varphi$
and also
the inequalities
$|\varphi(ax)|\leq\|a\||x|_\varphi$
and
$|\varphi(xa)|\leq\|a\||x|_\varphi$
for $a\in M$ and $x\in M_\varphi$.

By $\Mor(N,M)$,
we denote the set of unital faithful normal $*$-homomorphisms
from a von Neumann algebra $N$ into $M$.
For $\rho,\sigma\in\Mor(N,M)$,
we denote by $(\rho,\sigma)$
the collection of all operators $a\in M$
such that $a\rho(x)=\sigma(x)a$ for all $x\in N$.
When a $\rho\in\Mor(N,M)$
has a $\sigma\in\Mor(M,N)$
such that
the both intertwiner spaces
$(\id_M,\rho\circ\sigma)$
and $(\id_N,\sigma\circ\rho)$
contain elements $a$ and $b$, respectively,
with identities
$\rho(b^*)a=c1_M$
and $\sigma(a^*)b=c1_N$ for some non-zero $c\in\C$,
we call $\sigma$ a \emph{conjugate endomorphism} of $\rho$.
We will denote by $\Mor(N,M)_0$
the collection of all $\rho$ with conjugates.
If $N=M$,
we will write $\End(M)$ and $\End(M)_0$
for $\Mor(M,M)$ and $\Mor(M,M)_0$ as usual.

\subsection{Ultraproduct von Neumann algebras}
Readers are referred to \cite{AH,Oc} for ultraproduct von Neumann algebras.
Let $\omega$ be a free ultrafilter on $\N$ which we also regard
as a point in the boundary of $\N$ inside its Stone--C\v{e}ch compactification
$\beta\N$.
Let $M$ be a von Neumann algebra and $\ell^\infty(M)$ the collection of
norm bounded sequences in $M$.
Let
$\mathcal{T}^\omega(M)$ be the collection of all norm bounded sequences
$(x_n)$ being \emph{trivial},
that is, $x_n$ converges to $0$ as $n\to\omega$ in the strong$*$ topology.
Let $\mathcal{M}^\omega(M)$ be the multiplier algebra
of $\mathcal{T}^\omega(M)$ in $\ell^\infty(M)$.
An element $(x_n)\in\ell^\infty(M)$ is said to be \emph{central}
when $\|x_n\varphi-\varphi x_n\|_{M_*}\to0$ as $n\to\omega$
for all $\varphi\in M_*$.
By $\mathcal{C}^\omega(M)$ we denote the set of all central sequences.
Then we introduce the ultraproduct von Neumann algebras
$M^\omega:=\mathcal{M}^\omega(M)/\mathcal{T}^\omega(M)$
and
$M_\omega:=\mathcal{C}^\omega(M)/\mathcal{T}^\omega(M)$.
The equivalence class of $(x_n)$ is denoted by $(x_n)^\omega$.
By constant embedding $M\ni x\mapsto (x)^\omega\in M^\omega$,
we always regard $M$ as a von Neumann subalgebra of $M^\omega$.
Note that $M_\omega\subset M'\cap M^\omega$.
We denote by $\tau^\omega$
the canonical faithful normal conditional expectation from $M^\omega$
onto $M$.
Namely, we have
$\displaystyle\tau^\omega((x_n)^\omega):=\lim_{n\to\omega}x_n$
for $(x_n)^\omega\in M^\omega$,
where the limit is taken in the $\sigma$-weak topology of $M$.
Then a functional $\phi$ on $M$ extends to $M^\omega$
by $\phi^\omega:=\phi\circ\tau^\omega$.
The following result is probably well-known for experts,
but we will present a proof for readers' convenience.

\begin{lem}
\label{lem:ultra-functional}
For all $x=(x_n)^\omega,y=(y_n)^\omega\in M^\omega$
and $\varphi,\psi\in M_*$,
one has
\begin{equation}
\label{eq:ultra-functional}
\|x\varphi^\omega-\psi^\omega y\|
=
\lim_{n\to\omega}
\|x_n\varphi-\psi y_n\|.
\end{equation}
\end{lem}
\begin{proof}
When $\varphi=\psi$ is a faithful normal positive functional,
the equality is proved in \cite[Lemma 4.36]{AH}.
Next we suppose that
$\varphi=\psi$ is a normal positive functional.
Let $e$ be the support projection of $\varphi$
and $\hat{\varphi}$ a faithful normal positive functional
on $M$ such that $e\hat{\varphi}=\varphi=\hat{\varphi}e$.
Then we have
$e \hat{\varphi}^\omega=\varphi^\omega=\hat{\varphi}^\omega e$,
and
\begin{align*}
\|x\varphi^\omega-\varphi^\omega y\|
&=
\|xe\hat{\varphi}^\omega-\hat{\varphi}^\omega ey\|
=
\lim_{n\to\omega}
\|x_ne\hat{\varphi}-\hat{\varphi}ey_n\|
\\
&=
\lim_{n\to\omega}
\|x_n\varphi^\omega-\varphi^\omega y_n\|
\end{align*}
Thus (\ref{eq:ultra-functional})
holds for all positive $\varphi=\psi\in M_*$.

Next let $\varphi,\psi$ be positive functionals on $M$.
Consider $N=M\otimes M_2(\C)$.
Let $\{e_{ij}\}_{i,j=1}^2$ be a system of matrix units of $M_2(\C)$
and $\Tr$ the non-normalized trace of $M_2(\C)$.
Set $\chi\in N_*$ by
$\chi:=\varphi\otimes (\Tr e_{11})+\psi\otimes(\Tr e_{22})$.
Let $x=(x_n)^\omega, y=(y_n)^\omega\in M^\omega$.
Using the natural identification $N^\omega=M^\omega\otimes M_2(\C)$,
we have
\[
(x\otimes e_{21})\cdot\chi^\omega-\chi^\omega\cdot (y\otimes e_{21})
=
(x\varphi^\omega-\psi^\omega y)\otimes \Tr e_{21}.
\]
Then (\ref{eq:ultra-functional}) for positive $\varphi,\psi\in M_*$
follows from
\begin{align*}
\|x\varphi^\omega-\psi^\omega y\|
&=
\|(x\varphi^\omega-\psi^\omega y)\otimes \Tr e_{21}\|
\\
&=
\|(x\otimes e_{21})\cdot\chi^\omega
-\chi^\omega\cdot (y\otimes e_{21})\|
\\
&=
\lim_{n\to\omega}
\|(x_n\otimes e_{21})\cdot\chi
-\chi\cdot (y_n\otimes e_{21})\|
\\
&=
\lim_{n\to\omega}
\|x_n\varphi-\psi y_n\|.
\end{align*}

Finally, we let $\varphi,\psi\in M_*$.
Let $\varphi=v|\varphi|$ and $\psi=|\psi^*|w$ be the polar decompositions
with partial isometries $v,w\in M$.
Note that $\varphi^\omega=v|\varphi|^\omega$
and $\psi^\omega=|\psi^*|^\omega w$.
Then we have
\begin{align*}
\|x\varphi^\omega-\psi^\omega y\|
&=
\|xv|\varphi|^\omega-|\psi^*|^\omega wy\|
=
\lim_{n\to\omega}
\|x_nv|\varphi|-|\psi^*|wy_n\|
\\
&=
\lim_{n\to\omega}
\|x_n\varphi-\psi y_n\|.
\end{align*}
\end{proof}

\subsection{C$^*$-tensor category}
\label{subsect:tensorcat}
Our references are \cite{EGNO,HY,Mueger-guided,NeTu}
on the notion of a C$^*$-tensor category.
We will freely use the notations and the terminology
introduced in \cite[Chapter 2]{NeTu}.
A C$^*$-tensor category $\sC$ is a C$^*$-category consisting
of objects denoted by $X,Y,Z,\dots\in\sC$,
morphism spaces $\sC(X,Y)$
which consist of morphisms
from $X$ to $Y$
and the bilinear bifunctor $\otimes$ from $\sC\times\sC$
to $\sC$ with several identities.
The unit object is denoted by $\btr$.
In this paper, we always assume that
$\sC$ is \emph{strict}, that is,
the equalities
$(X\otimes Y)\otimes Z=X\otimes (Y\otimes Z)$
and
$\btr\otimes X=X=X\otimes \btr$
hold for all $X,Y,Z\in\sC$.
By definition of a tensor category,
we have $\sC(\btr,\btr)=\C$.
If we treat $\sC$ with non-trivial commutative
C$^*$-(or W$^*$-)algebra $\sC(\btr,\btr)$,
we call $\sC$ a C$^*$-(resp. W$^*$-)multitensor category.
If there exists an isometry in $\sC(X,Y)$,
then we will write $X\prec Y$.

We will say that an object $X\in\sC$ is \emph{simple}
or \emph{irreducible}
if $\sC(X,X)=\C1_X$.
By $\Irr(\sC)$, a complete list of
simple objects such that
any simple object $X\in\sC$
is isomorphic to
a unique object $Y\in\Irr(\sC)$.
The unit object $\btr$ is assumed to be contained in $\Irr(\sC)$.
If $X\in\Irr(\sC)$ and $Y\in\sC$,
the Banach space $\sC(X,Y)$ is a Hilbert space
with the inner product $\langle S,T \rangle1_X:=T^*S$
for $S,T\in\sC(X,Y)$.
We denote by $\ONB(X,Y)$ a fixed orthonormal base.
When $\sC(X,Y)=\{0\}$,
then $\ONB(X,Y)$ is regarded as the empty set.

In this paper,
we always assume that $\sC$ is \emph{rigid},
that is,
each object has a conjugate object.
For each $X\in\sC$,
we fix a conjugate object $\ovl{X}$
and a standard solution $(R_X,\ovl{R}_X)$
of the conjugate equation
as in \cite[Chapter 2.2]{NeTu}.
We will, however, use the \emph{normalized} solution.
Namely, $R_X\in\sC(\btr,\ovl{X}\otimes X)$
and $\ovl{R}_X\in\sC(\btr,X\otimes\ovl{X})$
satisfy $R_X^*R_X=1=\ovl{R}_X^*\ovl{R}_X$
and the following \emph{conjugate equations}:
\[
(\ovl{R}_X^*\otimes1_X)\circ(1_X\otimes R_X)
=\frac{1}{d(X)}1_X,
\quad
(R_X^*\otimes1_{\ovl{X}})\circ(1_{\ovl{X}}\otimes \ovl{R}_X)
=\frac{1}{d(X)}1_{\ovl{X}},
\]
where $d(X)$ denotes the \emph{intrinsic dimension} of $X$
\cite[Definition 2.2.10]{NeTu}.
For $X\in\Irr(\sC)$,
we find its conjugate object $\ovl{X}$ in $\Irr(\sC)$.

We will introduce the left and right Frobenius maps
$F_\ell\colon
\sC(X,Y\otimes Z)\to \sC(Z,\ovl{Y}\otimes X)$
and
$F_r\colon
\sC(X,Y\otimes Z)\to \sC(Y,X\otimes \ovl{Z})$
defined by
\begin{align}
\frac{d(X)^{1/2}}{d(Y)^{1/2}d(Z)^{1/2}}
F_\ell(S)
&:=
(1_{\ovl{Y}}\otimes S^*)(R_Y\otimes1_{Z}),
\label{eq:leftFrob}
\\
\frac{d(X)^{1/2}}{d(Y)^{1/2}d(Z)^{1/2}}
F_r(S)
&:=(S^*\otimes1_{\ovl{Z}})(1_Y\otimes\ovl{R}_Z).
\label{eq:rightFrob}
\end{align}
Then they are actually bijective maps.
This fact is called the Frobenius reciprocity.
For $X,Y,Z\in\sC$,
the dimension of $\sC(Z,X\otimes Y)$ as a linear space
is denoted by $N_{X,Y}^Z$.
The Frobenius reciprocity
implies $N_{X,Y}^Z=N_{\ovl{X},Z}^Y=N_{Z,\ovl{Y}}^X$.

When $X,Y,Z\in\Irr(\sC)$,
both the left and right Frobenius maps
preserve the inner products $\langle\cdot,\cdot\rangle$.
Hence the following equality holds
for $X,Y\in\Irr(\sC)$:
\begin{align}
1_{\ovl{X}}\otimes 1_{Y}
&=
\sum_Z\sum_S
F_\ell(F_r(S))F_\ell(F_r(S))^*
\notag
\\
&=
\sum_Z\sum_S
d(Z)^2(1_{\ovl{X}}\otimes 1_{Y}\otimes \ovl{R}_Z^*)
(1_{\ovl{X}}\otimes S\otimes 1_{\ovl{Z}})
(R_X\otimes 1_{\ovl{Z}})
\notag
\\
&
\hspace{80pt}\cdot
(R_X^*\otimes 1_{\ovl{Z}})
(1_{\ovl{X}}\otimes S^*\otimes 1_{\ovl{Z}})
(1_{\ovl{X}}\otimes 1_{Y}\otimes \ovl{R}_Z),
\label{eq:leftinv}
\end{align}
where the summation is taken for
$Z\in\Irr(\sC)$ and $S\in\ONB(X,Y\otimes Z)$.

Let $X\in \sC$ and $Z\in\Irr(\sC)$.
We let $P_X^Z\in \sC(X,X)$
be the projection to the $Z$-component of $X$.
Namely, $P_X^Z$ is the sum of $SS^*$
with $S\in\ONB(Z,X)$. 
Then $P_X^Z$ is natural in $X$,
that is,
we have $TP_X^Z=P_Y^Z T$
for $X,Y\in\sC$ and $T\in\sC(X,Y)$.
We also have
$R_X^*(1_{\ovl{X}}\otimes P_X^Z)R_X
=d(Z)d(X)^{-1}N_X^Z$,
where $N_X^Z$ denotes the dimension of
the linear space $\sC(Z,X)$.

For a finite subset $\cK\subset \Irr(\sC)$,
we set a projection $P_X^\cK:=\sum_{Z\in\cK}P_X^Z$.
Then for $X\in \Irr(\sC)$,
we have $P_X^\cK=1_X$ or $0$
if $X\in\cK$ or $X\nin\cK$, respectively.

In this paper, $\sC$ always denotes
a C$^*$-tensor category
such that $\Irr(\sC)$ is at most countable.
Hence $\sC$ is essentially small.
We will need a set of at most countable objects
$\sC_0$ being dense in $\sC$
for technical reason such as index selection trick
in an ultraproduct von Neumann algebra,
an induction and so on.
We will fix such $\sC_0$ throughout this paper.
Namely,
let $\sC_0$ be at most countable collection of objects in $\sC$
such that the following conditions hold:
\begin{itemize}
\item
Any object of $\sC$
has an isomorphic
object in $\sC_0$.

\item 
$\Irr(\sC)\subset\sC_0$.

\item
$X\otimes Y\in\sC_0$ for all $X,Y\in\sC_0$.

\item
$\ovl{X}\in\sC_0$ for all $X\in\sC_0$.
\end{itemize}

The set of endomorphisms on infinite factors
is a typical example of a C$^*$-tensor category
(such that the complete set of simple objects could be uncountable).
Our references are \cite{BKLR,Iz-fusion,LongoI,LongoII}.
Let $M$ be an infinite factor
and $\End(M)_0$ the set of endomorphisms
which have conjugate endomorphisms as introduced
in Section \ref{subsect:notation}.
The objects are endomorphisms $\rho\in\End(M)_0$
and the set of the morphisms from $\rho$ to $\sigma$
is $(\rho,\sigma)$.
The tensor product operation is nothing but
the composition, that is,
$\rho\otimes\sigma:=\rho\sigma$.
For intertwiners $S\colon\rho\to\sigma$
and $T\colon\lambda\to\mu$ in $\End(M)_0$,
we set their tensor product
$S\otimes T:=S\rho(T)=\sigma(T)S$.

Let $\rho\in\End(M)_0$ and take its conjugate $\orho\in\End(M)_0$.
Let $S_\rho\in(\btr,\orho\rho)$
and $\ovl{S}_\rho\in(\btr,\rho\orho)$ be a solution
of the conjugate equations for $\rho$ and $\orho$.
Recall that we always treat $S_\rho$ and $\ovl{S}_\rho$ being isometries.
Let us denote by $d^S(\rho)\geq1$
the positive scalar satisfying
$S_\rho^*\orho(\ovl{S}_\rho)=d^S(\rho)^{-1}1_\orho$
and
$\ovl{S}_\rho^*\rho(S_\rho)=d^S(\rho)^{-1}1_\rho$.
The intrinsic dimension of $\rho$, which is denoted by $d(\rho)$ as usual,
is nothing but the infimum of
$d^S(\rho)$ among all solutions $(S_\rho,\ovl{S}_\rho)$.

Let $\phi_\rho^S(x):=S_\rho^*\orho(x)S_\rho$ for $x\in M$.
Then $\phi_\rho^S$ is a faithful normal unital completely positive map
satisfying the left inverse property, $\phi_\rho^S\circ\rho=\id$.
Put $E_\rho^S:=\rho\circ\phi_\rho^S$ that is a faithful normal conditional
expectation from $M$ onto $\rho(M)$.
Then the singleton $\{d^S(\rho)\ovl{S}_\rho^*\}$
gives a Pimsner--Popa base
for $E_\rho^S$,
and the Jones index
(generalized by Kosaki to subfactors
of arbitrary type \cite{Jo,Ko}) $\Ind E_\rho^S$
equals $d^S(\rho)^2$
(see \cite{PP,Wata} for the definition of a base).
Conversely, we can prove any faithful normal
conditional expectation from $M$ onto $\rho(M)$
is of the form $E^S$ for some $(S_\rho,\ovl{S}_\rho)$.
Hence the minimal index $[M:\rho(M)]_0$ equals $d(\rho)^2$
(see \cite{Hi}).

We let $(R_\rho,\ovl{R}_\rho)$ be a standard solution
of the conjugate equations of $\rho$ and $\orho$.
We denote by $\phi_\rho$ and $d(\rho)$
the left inverse and the dimension with respect to
$(R_\rho,\ovl{R}_\rho)$.
Note that $\phi_\rho$ does not depend on a choice of
a conjugate $\orho$
nor a standard solution $(R_\rho,\ovl{R}_\rho)$.
The map $\phi_\rho$ is called the
\emph{standard left inverse} of $\rho$.

\subsection{Approximate unitary equivalence of endomorphisms}
The results in this subsection are not necessary materials
to read Sections 1--5.
In this subsection,
$M$ denotes an infinite factor with separable predual.
Each $\rho\in\End(M)_0$ has a faithful normal left inverse,
and $\rho^\omega\in\End(M^\omega)_0$ is naturally defined.
We will use the ultraproduct $\phi_\rho^\omega$
that is a left inverse of $\rho^\omega$.
In the following,
we simply denote by $\rho$ and $\phi_\rho$
the ultraproducts $\rho^\omega$ and $\phi_\rho^\omega$,
respectively.
For $\psi\in M_*$,
we set
$\rho(\psi^\omega):=\psi^\omega\circ\phi_\rho=(\rho(\psi))^\omega$.

We will recall the notion of the approximate innerness
of endomorphisms \cite[Definition 2.4]{MT-app}.
Let us denote by $\lceil \cdot\rceil$ the ceiling function
on $\R$.

\begin{defn}
An endomorphism $\rho$ on $M$ is said to be
an \emph{approximately inner endomorphism of rank $r>0$}
if there exists a family of partial isometries
$\{v_i\}_{i=1}^{\lceil r\rceil}$ in $M^\omega$
such that
\begin{itemize}
\item
$v_i^{*}v_i=1$ for $1\leq i<\lceil r\rceil$.

\item
$\sum_{i=1}^{\lceil r\rceil}v_i v_i^{*}=1$.

\item
$v_i\psi^\omega=r \rho(\psi^\omega)v_i$
for all $i$ and $\psi\in M_*$.
\end{itemize}
\end{defn}

It follows from the third equality
that
$v_i x=\rho(x)v_i$ for $x\in M$
since for $\psi\in M_*$
we have
$v_i x\psi^\omega
=
v_i (x\psi)^\omega
=
r\rho(x\psi)^\omega v_i
=
\rho(x)v_i\psi^\omega$.
Let $\rho_1,\rho_2\in\End(M)_0$.
Let $Q$ be the von Neumann subalgebra
of $M^\omega\otimes M_2(\C)$ defined by:
\[
Q:=\set{(x_{ij})_{ij}|
x_{ij}\rho_j(\psi^\omega)=\rho_i(\psi^\omega)x_{ij}
\mbox{ for all }i,j=1,2,
\ \psi\in M_*}.
\]
Let $\varphi\in M_*$ be a faithful state.
Then $Q$ has a faithful normal tracial state
$\tau\colon (x_{ij})_{ij}\mapsto
\rho_1(\varphi^\omega)(x_{11})/2+\rho_2(\varphi^\omega)(x_{22})/2$,
and $Q$ is a finite von Neumann algebra.
Let $\{e_{ij}\}_{ij}$ be a system of matrix units
of $M_2(\C)$.
Then $1\otimes e_{11}$
and $1\otimes e_{22}$ are contained in $Q$,
and any central element of $Q$ is of the diagonal form.

\begin{defn}
Let $\rho_1,\rho_2\in\End(M)_0$.
We will say $\rho_1$ and $\rho_2$ are \emph{approximately unitarily
equivalent}
if there exists a sequence of unitaries $u_n\in M$
such that
$\displaystyle\lim_{n\to\infty}\|u_n\rho_1(\psi)u_n^*-\rho_2(\psi)\|_{M_*}=0$
for all $\psi\in M_*$.
\end{defn}

From Lemma \ref{lem:ultra-functional},
$\rho_1$ and $\rho_2$ are approximately unitarily equivalent
if and only if
there exists a unitary $u\in M^\omega$
such that
$u\rho_1(\psi^\omega) u^*=\rho_2(\psi^\omega)$
for all $\psi\in M_*$.

\begin{rem}
We have introduced the approximate unitary equivalence
for endomorphisms on an infinite factor.
When we treat a von Neumann algebra with the non-trivial center,
we should designate left inverses of endomorphisms.
For instance,
in Definition \ref{defn:approx}
we will introduce the approximate unitary equivalence
for two cocycle actions,
where we will use the left inverses coming from
standard solutions of conjugate equations
in a C$^*$-tensor category
(see the equality (\ref{eq:phileftinv})).
\end{rem}

\begin{lem}
\label{lem:MvNequiv}
Let $\rho_1,\rho_2\in\End(M)_0$.
Let $\varphi\in M_*$ be a faithful state.
Then the following statements are equivalent:
\begin{enumerate}
\item
$\rho_1$ and $\rho_2$ are approximately unitarily
equivalent.

\item
$1\otimes e_{11}$
and $1\otimes e_{22}$ are Murray--von Neumann equivalent in $Q$.

\item
For any $z=\sum_{i=1}^2 z_{ii}\otimes e_{ii}\in Z(Q)$,
one has $\rho_1(\varphi^\omega)(z_{11})=\rho_2(\varphi^\omega)(z_{22})$.
\end{enumerate}
\end{lem}
\begin{proof}
(1) $\Rightarrow$ (2).
Take a unitary $u\in M^\omega$ with
$u\rho_1(\psi^\omega)u^*=\rho_2(\psi^\omega)$
for all $\psi\in M_*$.
Then $u\otimes e_{21}$ gives the equivalence of $e_{11}$ and $e_{22}$.
The implications
(2) $\Rightarrow$ (1) and (2) $\Rightarrow$ (3) are trivial.

(3) $\Rightarrow$ (1).
Assuming (3),
we have $\tau((1\otimes e_{11})z)=\tau((1\otimes e_{22})z)$
for all central $z\in Q$.
Hence $e_{11}$ and $e_{22}$ are equivalent.
\end{proof}

In the following two lemmas,
we will use the following well-known equality
(see the proofs of \cite[Lemma A.2]{M-CH}
or Lemma \ref{lem:linv} and \ref{lem:linv2}):
\[
d(\rho)S\rho(\psi)=d(\sigma)\sigma(\psi)S
\quad
\mbox{for all }
S\in(\rho,\sigma),
\
\psi\in M_*.
\]

\begin{prop}
\label{prop:rhosigma}
Let $\rho,\sigma\in\End(M)_0$.
Suppose that $\orho\sigma$
is an approximately inner endomorphism of rank $d(\rho)^2$
and $\rho(M)'\cap M=\C$.
Then $\rho$ and $\sigma$ are approximately unitarily equivalent.
\end{prop}
\begin{proof}
Put $\rho_1:=\rho$ and $\rho_2:=\sigma$.
We will check the condition (3) of Lemma \ref{lem:MvNequiv}.
Let $z=z_{11}\otimes e_{11}+z_{22}\otimes e_{22}\in Z(Q)$.
Take a family of partial isometry $\{v_i\}_{i\in I}$ in $M^\omega$
such that
$v_i\psi^\omega=d(\rho_1)^2\ovl{\rho_1}
\rho_2(\psi^\omega)v_i$ for $\psi\in M_*$
and
$\sum_i v_iv_i^*=1$.
Set $w_i:=\ovl{R}_{\rho_1}^* \rho_1(v_i)$ for $i\in I$.
Then we have $\sum_i w_i w_i^*=1$
and
\[
w_i \rho_1(\psi^\omega)
=
d(\rho_1)^2
\ovl{R}_{\rho_1}^* \rho_1\ovl{\rho_1}\rho_2(\psi^\omega)\rho_1(v_i)
=
\rho_2(\psi^\omega)w_i
\quad
\mbox{for all }i\in I.
\]
Thus $w_i\otimes e_{21}\in Q$ for $i\in I$.
By fast reindexation of $w_i$ (see \cite[Lemma 5.3]{Oc}),
we may and do assume that
$\tau^\omega(z_{11}w_i^*w_i)=\tau^\omega(z_{11})\tau^\omega(w_i^*w_i)$.
Note $\tau^\omega(z_{11})$ and $\tau^\omega(w_i^*w_i)$
are contained in $\rho_1(M)'\cap M=\C$.
Since $w_i z_{11}=z_{22}w_i$,
we have $\sum_i w_i z_{11}w_i^*=z_{22}$.
Thus
\begin{align*}
\rho_2(\varphi^\omega)(z_{22})
&=
\sum_i \rho_2(\varphi^\omega)(w_i z_{11}w_i^*)
=
\sum_i \rho_1(\varphi^\omega)(z_{11}w_i^*w_i)
\\
&=
\sum_i\rho_1(\varphi)(\tau^\omega(z_{11})\tau^\omega(w_i^*w_i))
\\
&=
\rho_1(\varphi^\omega)(z_{11})
\sum_i\rho_1(\varphi^\omega)(w_i^*w_i)
\\
&=
\rho_1(\varphi^\omega)(z_{11})
\sum_i\rho_2(\varphi^\omega)(w_iw_i^*)
=
\rho_1(\varphi^\omega)(z_{11}).
\end{align*}
\end{proof}

\begin{prop}
\label{prop:appuer}
Let $\rho,\sigma$ be an approximately
inner endomorphism on $M$ of rank $r>0$.
Then $\rho$ and $\sigma$ are approximately unitarily equivalent.
\end{prop}
\begin{proof}
Put $\rho_1:=\rho$ and $\rho_2:=\sigma$.
Let $z=z_{11}\otimes e_{11}+z_{22}\in Z(Q)$.
We will check the condition (3) of Lemma \ref{lem:MvNequiv}.
We first take a family of partial isometries
$\{w_i\}_{i=1}^{\lceil r\rceil}$ in $M^\omega$
such that
$w_i^*w_i=1$ for $i<\lceil r\rceil$,
$\sum_i w_iw_i^*=1$
and
$w_i \psi^\omega=r\rho_2(\psi^\omega)w_i$
for all $i$ and $\psi\in M_*$.
We may and do assume that
$w_i$'s also satisfy
$w_i\phi_{\rho_1}(z_{11})=\rho_2(\phi_{\rho_1}(z_{11}))w_i$
by fast reindexation of $w_i$.
We next take a family of partial isometries
$\{v_i\}_{i=1}^{\lceil r\rceil}$
such that
$v_i^*v_i=1$ for $i<\lceil r\rceil$,
$\sum_i v_iv_i^*=1$,
$v_i \psi^\omega=r\rho_1(\psi^\omega)v_i$
and
$v_i x=\rho_1(x)v_i$
for all $i$, $\psi\in M_*$ and
$x\in\{z_{22}\}\cup\{w_j,w_j^*\}_j$.

Since $w_kv_\ell^*\otimes e_{21}\in Q$
for all $k,\ell$,
we have
$w_kv_\ell^* z_{11}=z_{22}w_kv_\ell^*$.
Then we have
\begin{align*}
\sum_k
w_kv_\ell^* z_{11} v_\ell w_k^*
&=
\sum_k
z_{22}w_kv_\ell^* v_\ell w_k^*
=
\sum_k
z_{22}v_\ell^* \rho_1(w_kw_k^*)v_\ell
\\
&=
z_{22}v_\ell^*v_\ell
=
v_\ell^* \rho_1(z_{22})v_\ell.
\end{align*}
Hence
\begin{align*}
\sum_k
r \rho_1(\varphi^\omega)(\rho_1(w_k^* w_k) v_\ell v_\ell^* z_{11})
&=
\sum_k
r \rho_1(\varphi^\omega)(v_\ell w_k^* w_k v_\ell^* z_{11})
=
\sum_k
\varphi^\omega(w_k^* w_k v_\ell^* z_{11}v_\ell)
\\
&=
\sum_k
r\rho_2(\varphi^\omega)(w_kv_\ell^* z_{11} v_\ell w_k^*)
=
r\rho_2(\varphi^\omega)(v_\ell^* \rho_1(z_{22})v_\ell)
\\
&=
r^2\rho_1\rho_2(\varphi^\omega)(v_\ell v_\ell^* \rho_1(z_{22})).
\end{align*}
Summing up the equality with $\ell$,
we have
\begin{align*}
\sum_k
\rho_1(\varphi^\omega)(\rho_1(w_k^* w_k)z_{11})
&=
r\rho_1\rho_2(\varphi^\omega)(\rho_1(z_{22}))
=
r\rho_2(\varphi^\omega)(z_{22}).
\end{align*}
The left-hand side equals
\begin{align*}
\sum_k
\varphi^\omega(w_k^* w_k\phi_{\rho_1}(z_{11}))
&=
\sum_k
\varphi^\omega(w_k^* \rho_2(\phi_{\rho_1}(z_{11}))w_k)
\\
&=
\sum_k
r\rho_2(\varphi^\omega)(w_k w_k^* \rho_2(\phi_{\rho_1}(z_{11})))
\\
&=
r\rho_2(\varphi^\omega)(\rho_2(\phi_{\rho_1}(z_{11})))
=
r\rho_1(\varphi^\omega)(z_{11}).
\end{align*}
Thus we have
$\rho_1(\varphi^\omega)(z_{11})=\rho_2(\varphi^\omega)(z_{22})$,
and we are done.
\end{proof}

\subsection{Amenability}
\label{Amenability}
We will recall the amenability of $\sC$
in the sense of \cite{HY,HI,NeTu,Popa-acta,Popa-mathlett},
but $\sC$ is not necessarily finitely generated.
There are plenty of equivalent conditions of amenability.
We will use the F\o lner type condition as follows.
Let $\sigma$ be the measure on $\Irr(\sC)$
with $\sigma(X):=d(X)^2$ for $X\in\Irr(\sC)$.
For $\cF\subset\Irr(\sC)$
we denote by $|\cF|_\sigma$
the measure $\sigma(\cF)$,
that is,
$|\cF|_\sigma:=\sum_{X\in\cF}d(X)^2$.

Let $\cF,\cK\subset\Irr(\sC)$ be finite subsets and $\delta>0$.
We will say
$\cK$ is \emph{$(\cF,\delta)$-invariant}
when the inequality
$|(\cF\cdot\cK)\bigtriangleup\cK|_\sigma
\leq \delta|\cK|_\sigma$
holds,
where $\bigtriangleup$ denotes the symmetric difference
and $\cF\cdot \cK$ does the collection
of all $Z\in \Irr(\sC)$ such that
$Z\prec X\otimes Y$ for some $X\in\cF$ and $Y\in\cK$.

A rigid C$^*$-tensor category
is said to be \emph{amenable}
if
for any finite $\cF\subset\Irr(G)$ and $\delta>0$,
there exists an $(\cF,\delta)$-invariant finite
subset $\cK\subset\Irr(G)$.
We set
\[
p_X(Y,Z):=\frac{d(Z)}{d(X)d(Y)}N_{X,Y}^Z
\quad
\mbox{for }
X,Y,Z\in\Irr(\sC).
\]
Then $p_X(Y,Z)$ is regarded as the transition probability
from $Y$ to $Z$ after multiplying $X$ from the left.
The following inequalities
are shown in the proof of \cite[Theorem 4.6]{HI},
but let us present a proof for readers' convenience.

\begin{lem}
\label{lem:Folner}
Let $\cK$ be an $(\cF,\delta)$-invariant
finite subset as above.
Then the following inequality holds:
\begin{align*}
&\sum_{(X,Y,Z)\in\cF\times\cK\times \cK^c}
d(X)^2d(Y)^2p_X(Y,Z)
\\
&=
\sum_{(X,Y,Z)\in\cF\times\cK^c\times \cK}
d(X)^2d(Y)^2p_{\ovl{X}}(Y,Z)
\leq
\delta
|\cF|_\sigma|\cK|_\sigma,
\end{align*}
where $\cK^c$ denotes the complement of $\cK$ in $\Irr(\sC)$.
\end{lem}
\begin{proof}
The equality immediately follows from the Frobenius reciprocity
and the change of the variables $Y$ and $Z$.
On the inequality,
we have
\begin{align*}
\sum_{(X,Y,Z)\in \cF\times\cK\times \cK^c}
d(X)^2
d(Y)^2p_X(Y,Z)
&=
\sum_{X\in\cF}
\sum_{Z\in \cF\cdot \cK\setminus\cK}
\sum_{Y\in\cK}
d(X)^2
d(Y)^2p_X(Y,Z)
\\
&\leq
\sum_{X\in\cF}
\sum_{Z\in \cF\cdot \cK\setminus\cK}
\sum_{Y\in\Irr(\sC)}
d(X)^2
d(Y)^2p_X(Y,Z)
\\
&=
|\cF|_\sigma
|(\cF\cdot \cK)\setminus\cK|_\sigma
\leq
\delta
|\cF|_\sigma
|\cK|_\sigma.
\end{align*}
\end{proof}

\subsection{Actions and cocycle actions via endomorphisms}
\label{subsect:actions}
Readers are referred to \cite{Iz-near,Mas-Rob}
for (cocycle) actions of C$^*$-tensor categories on factors.
In this subsection,
$\sC$ and $M$ henceforth
denote a rigid C$^*$-tensor category
and
a properly infinite von Neumann algebra
with not necessarily separable predual,
respectively.
A \emph{cocycle action} of $\sC$
on $M$ means a unitary tensor functor
$(\alpha,c)$ from $\sC$ into the W$^*$-multitensor category
$\End(M)_0$.
Namely,
for each $X\in\sC$,
an endomorphism $\alpha_X\in\End(M)_0$ is assigned
and for each pair $X,Y\in\sC$,
we have a unitary $c_{X,Y}\in M$
such that for all $X,Y,Z\in\sC$,
\begin{itemize}
\item
$\alpha_\btr=\id_M$, $c_{\btr,X}=1=c_{X,\btr}$.

\item 
$\alpha_X\circ\alpha_Y=\Ad c_{X,Y}\circ\alpha_{X\otimes Y}$.

\item
$c_{X,Y}c_{X\otimes Y,Z}=\alpha_X(c_{Y,Z})c_{X,Y\otimes Z}$.
\end{itemize}

We call unitaries $c=(c_{X,Y})_{X,Y}$ a \emph{2-cocycle}.
We also have $\alpha(T)\in (\alpha_X,\alpha_Y)$
for each morphism $T$ in $\sC(X,Y)$ with $X,Y\in\sC$.
We will simply write $T^\alpha$ or $[T]^\alpha$ for $\alpha(T)$.
The computation rules of intertwiners are summarized as follows:
\[
[ST]^\alpha=S^\alpha T^\alpha,
\quad
(S^\alpha)^*=[S^*]^\alpha\
\quad
\mbox{for }
X,Y,Z\in\sC,
\
S\in\sC(Y,Z),\ T\in\sC(X,Y).
\]

The relationship between $c$ and intertwiners is described as follows:
for all
$X,Y,U,V\in\sC$,
$S\in \sC(U,X)$
and
$T\in \sC(V,Y)$,
\begin{equation}
\label{eq:cST}
c_{X,Y}[S\otimes T]^\alpha=S^\alpha\alpha_U(T^\alpha)c_{U,V}.
\quad
\end{equation}
We always assume $1_X^\alpha=1_M$ for $1_X\in\sC(X,X)$ with $X\in\sC$.
Note that
the map $\sC(X,Y)\ni T\mapsto T^\alpha\in M$
is necessarily norm-isometric.

If $c_{X,Y}=1$ for all $X,Y\in\sC$,
then $\alpha$ is called an \emph{action} of $\sC$ on $M$.

Consider a family of unitaries $v=(v_X)_{X\in \sC}$ in $M$.
Then the \emph{perturbed cocycle action} $(\alpha^v,c^v)$
of a cocycle action $(\alpha,c)$
by unitaries $v$ is defined as follows:
\begin{itemize}
\item
$\alpha_X^v:=\Ad v_X\circ\alpha_X$ for $X\in\sC$.

\item
$c_{X,Y}^v:=v_X\alpha_X(v_Y)c_{X,Y}v_{X\otimes Y}^*$
for $X,Y\in\sC$.

\item
$T^{\alpha^v}:=v_Y T^\alpha v_X^*$
for $X,Y\in\sC$ and $T\in\sC(X,Y)$.
\end{itemize}

For $X\in\sC$ and $x\in M$,
we put
\begin{equation}
\label{eq:phileftinv}
\phi_X^\alpha(x)
:=R_X^{\alpha*} c_{\ovl{X},X}^*\alpha_{\ovl{X}}(x)c_{\ovl{X},X}R_X^\alpha,
\end{equation}
which is independent of a choice of a conjugate object $\ovl{X}$
and a standard solution $(R_X,\ovl{R}_X)$ of the conjugate equations.
The map $\phi_X^\alpha$ is a faithful normal unital completely
positive map on $M$ satisfying the left inverse property
of $\alpha_X$, that is,
$\phi_X^\alpha\circ\alpha_X=\id_M$.
The conjugate equation $(R_X^*\otimes 1_{\ovl{X}})(1_{\ovl{X}}\otimes\ovl{R}_X)=d(X)^{-1}1_{\ovl{X}}$
shows
$\phi_X^\alpha(c_{X,\ovl{X}}\ovl{R}_X^\alpha\ovl{R}_X^{\alpha*}c_{X,\ovl{X}}^*)=d(X)^{-2}$.
Note that
$\phi_{X\otimes Y}^\alpha\circ\Ad c_{X,Y}^*=\phi_Y^\alpha\circ\phi_X^\alpha$
since a pair
of
$(1_{\overline{Y}}\otimes R_X\otimes 1_Y)\circ R_Y$
and
$(1_X\otimes \ovl{R}_Y\otimes 1_{\ovl{X}})\circ \ovl{R}_X$
gives a standard solution of the conjugate equations
for $X\otimes Y$ and $\ovl{Y}\otimes \ovl{X}$ as usual.

Recall $P_X^Z$,
the projection to the $Z$-component of $X$
defined in Section \ref{subsect:tensorcat}.
Then we have
$\phi_X^\alpha([P_X^Z]^\alpha)=d(Z)d(X)^{-1}N_X^Z$.
This implies
\begin{equation}
\label{eq:PXYZ}
\phi_{X\otimes Y}^\alpha([P_{X\otimes Y}^Z]^\alpha)
=
p_X(Y,Z)
\quad
\mbox{for all }
X,Y,Z\in\Irr(\sC).
\end{equation}

Let $N\subset M$ be a von Neumann subalgebra
that is not necessarily globally $\alpha$-invariant.
The \emph{fixed point algebra} $N^\alpha$ means
the collection of all $x\in N$
with $\alpha_X(x)=x$ for all $X\in\sC$.
The following result is shown in \cite[Lemma A.2]{M-CH},
but we will present a proof for readers' convenience.

\begin{lem}
\label{lem:linv}
Let $X,Y,Z\in\sC$.
Suppose that $X$ is the direct product of
$Y$ and $Z$ with isometries $S\in\sC(Y,X)$ and $T\in\sC(Z,X)$.
Then
\[
d(X)
\phi_X^\alpha(x)
=d(Y)\phi_Y^\alpha(S^{\alpha*}xS^\alpha)
+d(Z)\phi_Z^\alpha(T^{\alpha*}xT^\alpha)
\quad
\mbox{for all }
x\in M.
\]
\end{lem}
\begin{proof}
This is direct from
$d(X)^{1/2}R_X
=d(Y)^{1/2}(\overline{S}\otimes S)R_Y
+d(Z)^{1/2}(\overline{T}\otimes T)R_Z$,
where $\overline{S}\in\sC(\overline{Y},\overline{X})$
and $\overline{T}\in\sC(\overline{Z},\overline{X})$
are isometries with
$\overline{S}\,\overline{S}^*+\overline{T}\,\overline{T}^*=1_{\ovl{X}}$.
\end{proof}

It is not difficult to show the following result
by using the previous result and the irreducible decompositions
of $X$ and $Y$.

\begin{lem}
\label{lem:linv2}
One has
$d(X)\phi_X^\alpha(xS^\alpha)=d(Y)\phi_Y^\alpha(S^\alpha x)$
for all
$X,Y\in\sC$,
$S\in\sC(X,Y)$ and $x\in M$.
\end{lem}

We set
$\alpha_X(\varphi):=\varphi\circ\phi_X^\alpha$ for $X\in\sC$
and $\varphi\in M_*$.
Then for $a,b\in M$ and $X,Y\in\sC$,
we have
\begin{equation}
\label{eq:alpha-bimod-varphi}
\alpha_X(a\varphi b)
=
\alpha_X(a)\alpha_X(\varphi)\alpha_X(b),
\quad
\alpha_X(\alpha_Y(\varphi))
=c_{X,Y}\alpha_{X\otimes Y}(\varphi)c_{X,Y}^*.
\end{equation}
The lemma above shows the following formula:
\begin{equation}
\label{eq:alxphi}
d(X)
S^\alpha\alpha_X(\varphi)
=
d(Y)
\alpha_Y(\varphi)S^\alpha
\quad
\mbox{for all }
X,Y\in\sC,\ S\in\sC(X,Y).
\end{equation}
In particular,
we have
\[
d(X)d(Y)
\alpha_{X\otimes Y}(\varphi)
=
\sum_{Z\in\Irr(\sC)}
\sum_{S\in \ONB(Z,X\otimes Y)}
d(Z)S^\alpha\alpha_Z(\varphi)S^{\alpha*}.
\]

For a finite $\cK\subset\Irr(\sC)$,
we will introduce the faithful normal unital completely positive map
$I_\cK^\alpha$
on $M$ as follows:
\[I_\cK^\alpha(x)
:=
\frac{1}{|\cK|_\sigma}
\sum_{Y\in\cK}d(Y)^2\phi_{\ovl{Y}}^\alpha(x)
\quad
\mbox{for }
x\in M.
\]
This is thought of the average of $\alpha_Y(x)$'s over $Y\in\cK$.

\begin{lem}
\label{lem:average}
Let $(\alpha,c)$ be a cocycle action of $\sC$ on $M$
and
$\cK$ a finite subset of $\Irr(\sC)$.
For any $X\in\sC$ and
$x\in M$ commuting with all morphisms in $\sC$
and with unitaries $c_{Y,Z}$ with $Y,Z\in\sC$,
one has
\begin{align*}
\phi_{\ovl{X}}^\alpha(I_\cK^\alpha(x))
&=
\frac{1}{|\cK|_\sigma}
\sum_{(Y,Z)\in \cK\times \Irr(\sC)}
d(Y)^2p_X(Y,Z)
\phi_{\ovl{Z}}^\alpha(x),
\\
I_\cK^\alpha(x)
&=
\frac{1}{|\cK|_\sigma}
\sum_{(Y,Z)\in \Irr(\sC)\times \cK}
d(Y)^2p_X(Y,Z)
\phi_{\ovl{Z}}^\alpha(x).
\end{align*}
\end{lem}
\begin{proof}
We compute
$d(Y)^2\phi_{\ovl{X}}^\alpha(\phi_{\ovl{Y}}^\alpha(x))$
for $Y\in\Irr(\sC)$
as follows:
\begin{align*}
d(Y)^2\phi_{\ovl{X}}^\alpha(\phi_{\ovl{Y}}^\alpha(x))
&=
d(Y)^2\phi_{\ovl{Y}\otimes\ovl{X}}^\alpha
(c_{\ovl{Y},\ovl{X}}^*x c_{\ovl{Y},\ovl{X}})
\\
&=
\sum_{Z\in\Irr(\sC)}
\sum_{S\in \ONB(\ovl{Z},\ovl{Y}\otimes\ovl{X})}
\frac{d(Y)d(Z)}{d(X)}
\phi_{\ovl{Z}}^\alpha(S^{\alpha*}xS^\alpha)
\quad
\mbox{by Lemma }
\ref{lem:linv}
\\
&=
\sum_{Z,S}
\frac{d(Y)d(Z)}{d(X)}
\phi_{\ovl{Z}}^\alpha(x)
\quad
\mbox{by }
xS^\alpha=S^\alpha x
\\
&=
\sum_{Z}
d(Y)^2p_X(Y,Z)
\phi_{\ovl{Z}}^\alpha(x).
\end{align*}
The second equality is trivial
since $\sum_{Y\in\Irr(\sC)} N_{X,Y}^Z d(Y)=d(X)d(Z)$.
\end{proof}

\subsection{Freeness of cocycle actions}
Let $\sC$ and $M$ be as in the previous subsection.
We will introduce the freeness of a cocycle action of $\sC$
on $M$.

\begin{defn}
\label{defn:freeness}
A cocycle action $(\alpha,c)$ of $\sC$ on $M$
is said to be \emph{free}
if $(\alpha_X,\alpha_Y)=\alpha(\sC(X,Y))Z(M)$
for all $X,Y\in\sC$,
where $\alpha(\sC(X,Y))Z(M)$
denotes the linear span
of the elements $T^\alpha z$ with $T\in\sC(X,Y)$ and $z\in Z(M)$.
\end{defn}

\begin{lem}
\label{lem:free}
Let $(\alpha,c)$ be a cocycle action of $\sC$
on $M$.
Then the following statements are equivalent:
\begin{enumerate}
\item
$(\alpha,c)$ is free.

\item
$(\alpha_X,\alpha_{\btr})=\{0\}$
for all $X\in\Irr(\sC)\setminus\{0\}$,
that is,
there exists no non-zero element $a\in M$
with $ax=\alpha_X(x)a$ for all $x\in M$.
\end{enumerate}
\end{lem}
\begin{proof}
(1) $\Rightarrow$ (2).
This implication is trivial.

(2) $\Rightarrow$ (1).
It suffices to show the equality in Definition \ref{defn:freeness}
for $X,Y$ being simple objects.
It is trivial the right-hand side is contained in the left-hand side.
Let $a\in (\alpha_X,\alpha_Y)$.
Then $\ovl{R}_Y^{\alpha*}c_{Y,\ovl{Y}}^*ac_{X,\ovl{Y}}$
is an element of $(\alpha_{X\otimes\ovl{Y}},\alpha_\btr)$.
By considering the irreducible decomposition of $X\otimes\ovl{Y}$,
we have
$(\alpha_{X\otimes\ovl{Y}},\alpha_\btr)
=\delta_{X,Y}\ovl{R}_X^{\alpha*}Z(M)$
from the assumption of (2).
Take $b\in Z(M)$ so that
$\ovl{R}_Y^{\alpha*}c_{Y,\ovl{Y}}^*ac_{X,\ovl{Y}}
=\delta_{X,Y}\ovl{R}_X^{\alpha*}b$.
Multiplying $c_{X,\ovl{Y}}^*\alpha_X(c_{\ovl{Y},Y}R_Y^\alpha)$
to the both sides from the right,
we have
$d(X)^{-1}a
=\delta_{X,Y}\ovl{R}_X^{\alpha*}c_{X,\ovl{Y}}^*\alpha_X(c_{\ovl{Y},Y}R_Y)b$.
Hence $a=0$ when $X\neq Y$
and
$a=b$ when $X=Y$.
\end{proof}

\begin{lem}
\label{lem:thetaalpha}
If a cocycle action $(\alpha,c)$ of $\sC$ on $M$ is free,
then the following statements hold:
\begin{enumerate}
\item
For each $X\in\Irr(\sC)$,
the restriction of $\alpha_X$ on $Z(M)$
is an automorphism
which will be denoted by $\theta_X^\alpha$
henceforth.

\item
If $X,Y,Z\in\Irr(\sC)$ satisfy $Z\prec X\otimes Y$,
then $\theta_X^\alpha\circ\theta_Y^\alpha=\theta_Z^\alpha$.
\end{enumerate}
\end{lem}
\begin{proof}
(1).
Let $X\in\Irr(\sC)$.
The freeness implies $\alpha_X(M)'\cap M=Z(M)$.
Hence $\alpha_X(Z(M))\subset Z(M)$.
Let $x\in Z(M)$.
We will show $x\in\alpha_X(Z(M))$.
Set $y:=\alpha_{\ovl{X}}(x)$.
Then $y\in \alpha_{\ovl{X}}(Z(M))\subset Z(M)$,
and we have $\alpha_X(y)\in Z(M)$.
Hence
\[
x
=\phi_{\ovl{X}}^\alpha(y)
=\ovl{R}_X^{\alpha*}c_{X,\ovl{X}}^*
\alpha_X(y)
c_{X,\ovl{X}}\ovl{R}_X^{\alpha}
=\alpha_X(y).
\]

(2).
Take an isometry $S\in\sC(Z,X\otimes Y)$.
For $x\in Z(M)$,
we have
\[
S^\alpha\theta_Z^\alpha(x)=S^\alpha\alpha_Z(x)
=\alpha_{X\otimes Y}(x)S^\alpha
=c_{X,Y}^*\alpha_X(\alpha_Y(x))c_{X,Y}S^\alpha
=
\theta_X^\alpha(\theta_Y^\alpha(x))S^\alpha.
\]
Multiplying $S^{\alpha*}$ to the above from the left,
we are done. 
\end{proof}

\subsection{Centrally free cocycle actions}

Let $(\alpha,c)$ be a cocycle action of $\sC$ on $M$ with separable predual
as before.
Since each $\alpha_X$ has a faithful normal left inverse,
we can define $\alpha_X^\omega\in \End(M^\omega)$
by putting
$\alpha_X^\omega((x_n)^\omega):=(\alpha_X(x_n))^\omega$
for $X\in\sC$ and $(x_n)^\omega\in M^\omega$
(see \cite[Lemma 3.2]{MT-minimal}).
Then we have a cocycle action $(\alpha^\omega,c)$
of $\sC$ on $M^\omega$.
Note that we use the same 2-cocycle $c$
and the same intertwiners $T^\alpha$
for morphisms $T$ in $\sC$.
Then the canonical left inverse of $\alpha_X^\omega$
is given by $(\phi_X^\alpha)^\omega$.
We often simply write $\alpha_X$ and $\phi_X^\alpha$
for $\alpha_X^\omega$ and $(\phi_X^{\alpha})^\omega$,
respectively.

Since each $\alpha_X$ is an endomorphism on $M$,
$\alpha_X(M_\omega)$ is not contained in $M_\omega$
in general
though we always have $\phi_X^\alpha(M_\omega)\subset M_\omega$.
We will denote by ${}_{\alpha_X}M_\omega$
the von Neumann algebra generated by
$M_\omega$ and $\alpha_X(M_\omega)$.
Then it is clear
that ${}_{\alpha_X}M_\omega\subset \alpha_X(M)'\cap M^\omega$.
Thus $\tau^\omega$ maps ${}_{\alpha_X}M_\omega$
onto $(\alpha_X,\alpha_X)$.
The following result will be frequently used
in this paper.

\begin{lem}
\label{lem:freetrace}
Let $(\alpha,c)$ be a cocycle action of $\sC$
on a von Neumann algebra $M$.
Then the following statements hold:
\begin{enumerate}
\item
One has
$x\alpha_X(\psi^\omega)=\alpha_X(\psi^\omega)x$
for all $x\in{}_{\alpha_X}M_\omega$,
$\psi\in M_*$ and $X\in\sC$.

\item
If $\alpha$ is free,
then
$\tau^\omega$
maps ${}_{\alpha_X}M_\omega$ onto $Z(M)$
for all $X\in\Irr(\sC)$.

\item
Suppose that
$(\alpha,c)$ is free
and
a faithful state $\varphi\in M_*$
is $\theta^\alpha$-invariant
on $Z(M)$.
Then one has
$\alpha_Y(\varphi^\omega)=\varphi^\omega$
on ${}_{\alpha_X}M_\omega$
for all $Y\in\Irr(\sC)$.
In particular,
$\varphi^\omega$ is tracial on ${}_{\alpha_X}M_\omega$.
\end{enumerate}
\end{lem}
\begin{proof}
(1), (2).
These statements are trivial.
(3).
By (2) and $\theta^\alpha$-invariance of $\varphi$,
we have
$\alpha_Y(\varphi^\omega)
=\varphi\circ\theta_{\ovl{Y}}^\alpha\circ\tau^\omega
=\varphi^\omega$
on ${}_{\alpha_X}M_\omega$.
The statement (1) implies the tracial property
of $\varphi^\omega$ on ${}_{\alpha_X}M_\omega$.
\end{proof}

\begin{defn}
Let $(\alpha,c)$ be a cocycle action
of a rigid C$^*$-tensor category $\sC$
on a von Neumann algebra $M$ with separable predual.
We will say that $(\alpha,c)$ is
\begin{itemize}
\item 
\emph{centrally free}
if $\alpha_X$ is \emph{properly centrally non-trivial}
for each $X\in \Irr(\sC)\setminus\{\btr\}$,
that is,
there exists no non-zero element $a\in M$
such that $ax=\alpha_X(x)a$ for all $x\in M_\omega$;

\item
\emph{strongly free}
if $\alpha_X$ is \emph{strongly outer}
for each $X\in \Irr(\sC)\setminus\{\btr\}$,
that is,
for any countably generated von Neumann subalgebra
$Q\subset M^\omega$,
there exists no non-zero element $a\in M^\omega$
such that
$ax=\alpha_X(x)a$ for all $x\in Q'\cap M_\omega$.
\end{itemize}
\end{defn}

\begin{rem}
\label{rem:cent-trivial}
A few remarks are in order.
\begin{enumerate}
\item 
The strong freeness
trivially implies the central freeness.
It is known that they are actually equivalent
\cite[Lemma 8.2]{MT-minimal}.
It is trivial that the central freeness implies the freeness.

\item
Let $X\in\Irr(\sC)\setminus\{\btr\}$.
When $M$ is a factor and $(\alpha_X,\alpha_X)=\C$,
the proper central non-triviality of $\alpha_X$
is equivalent
to a much weaker condition
of the \emph{central non-triviality}
\cite[Lemma 8.3]{MT-minimal},
that is,
$\alpha_X\neq\id$ on $M_\omega$.

\item
When $\Irr(\sC)=\{\btr\}$,
we think any action of $\sC$ is centrally free.
\end{enumerate}
\end{rem}

\begin{lem}
\label{lem:relcomm}
Let $(\alpha,c)$ be a centrally free cocycle action of $\sC$ on $M$.
Then for any $X,Y\in\sC$ and any countably generated
von Neumann subalgebra $Q\subset M^\omega$,
one has
\[
\set{a\in M^\omega|a\alpha_X(x)=\alpha_Y(x)a,
\ x\in Q'\cap M_\omega}
=
\alpha(\sC(X,Y))
\alpha_X((Q'\cap M_\omega)'\cap M^\omega).
\]
\end{lem}
\begin{proof}
It is trivial that the right-hand side
is contained in the left.
We will show the converse inclusion.
By the direct decompositions of $X$ and $Y$,
we may and do assume
$X,Y\in\Irr(\sC)$.
Suppose $a\in M^\omega$ satisfies
$a\alpha_X(x)=\alpha_Y(x)a$
for all $x\in Q'\cap M_\omega$.
Then for $x\in Q'\cap M_\omega$,
\[
c_{\ovl{X},Y}^*
\alpha_{\ovl{X}}(a)c_{\ovl{X},X}R_X^\alpha x
=\alpha_{\ovl{X}\otimes Y}(x)
c_{\ovl{X},Y}^*\alpha_{\ovl{X}}(a)
c_{\ovl{X},X}R_X^\alpha.
\]
Hence for any simple $Z\prec \ovl{X}\otimes Y$
and any isometry $T\in\sC(Z,\ovl{X}\otimes Y)$,
we have
$T^{\alpha*}
c_{\ovl{X},Y}^*
\alpha_{\ovl{X}}(a)c_{\ovl{X},X}R_X^\alpha x
=\alpha_{Z}(x)T^{\alpha*}
c_{\ovl{X},Y}^*
\alpha_{\ovl{X}}(a)c_{\ovl{X},X}R_X^\alpha$.
By central freeness,
we have
$T^{\alpha*}c_{\ovl{X},Y}^*
\alpha_{\ovl{X}}(a)c_{\ovl{X},X}R_X^\alpha=0$
if $Z\neq\btr$.
Hence if $X\neq Y$,
then
$\alpha_{\ovl{X}}(a)c_{\ovl{X},X}R_X^\alpha=0$.
This implies $a=0$.
If $X=Y$,
then we have
\begin{align*}
&\alpha_{\ovl{X}}(a)c_{\ovl{X},X}R_X^\alpha R_X^{\alpha*}c_{\ovl{X},X}^*
\\
&=
\sum_{Z\in\Irr(\sC)}
\sum_{T\in\ONB(Z,\ovl{X}\otimes X)}
c_{\ovl{X},X}T^\alpha T^{\alpha*}c_{\ovl{X},X}^*
\alpha_{\ovl{X}}(a)c_{\ovl{X},X}
R_X^\alpha R_X^{\alpha*}
c_{\ovl{X},X}^*
\\
&=
c_{\ovl{X},X}R_X^\alpha R_X^{\alpha*}
c_{\ovl{X},X}^*
\alpha_{\ovl{X}}(a)c_{\ovl{X},X} R_X^\alpha R_X^{\alpha*}c_{\ovl{X},X}^*
\\
&=
c_{\ovl{X},X}R_X^\alpha
\phi_X^\alpha(a)R_X^{\alpha*}c_{\ovl{X},X}^*
\\
&=
\alpha_{\ovl{X}}(\alpha_X(\phi_X^\alpha(a)))
c_{\ovl{X},X}R_X^\alpha R_X^{\alpha*}c_{\ovl{X},X}^*.
\end{align*}
Applying $\phi_{\ovl{X}}^\alpha$ in the both sides above,
we have
$a=\alpha_X(\phi_X^\alpha(a))$.
Put $b:=\phi_X^\alpha(a)$.
Then $bx=xb$ for all $x\in Q'\cap M_\omega$,
and we are done.
\end{proof}

\subsection{Local quantization}
In this subsection,
$\sC$ and $M$ still denote
a rigid C$^*$-tensor category and a properly infinite von Neumann algebra
with separable predual,
respectively.
We will prove a C$^*$-tensor category version
of \cite[Lemma 5.3]{MT-minimal}
which has been proved for approximately inner
actions of discrete Kac algebras.

\begin{lem}
\label{lem:qq}
Let $(\alpha,c)$ be a centrally free cocycle action
of $\sC$ on $M$.
Let $\varphi$ be a faithful normal state on $M$
such that $\varphi$
is $\theta^\alpha$-invariant on $Z(M)$.
Then for any countably generated von Neumann subalgebra
$Q\subset M^\omega$,
finite subset $\cF\subset \Irr(\sC)\setminus\{\btr\}$
and $0<\delta<1$,
there exist $n\in\N$ and a partition of unity
$\{q_r\}_{r=0}^n$ in $Q'\cap M_\omega$
such that
\begin{enumerate}
\item
$|q_0|_{\varphi^\omega}<\delta$.

\item
$\sum_{X\in \cF}|q_r\alpha_X(q_r)|_{\varphi^\omega}
<\delta|q_r|_{\varphi^\omega}$
for all $r=1,\dots,n$.
\end{enumerate}
\end{lem}
\begin{proof}
Let $\cF=\{X_1,\dots,X_m\}$,
where $X_j\neq X_k$ if $j\neq k$.
Set $X:=\btr\oplus X_1\oplus\cdots\oplus X_m$
with isometries $T_k\in\sC(X_k,X)$.
We put $X_0:=\btr$.
We consider the inclusion
$N:=\alpha_X(Q'\cap M_\omega)\subset M^\omega$.
Then $N$ is contained
in the centralizer of $\alpha_X(\varphi^\omega)$.
Hence there exists an $\alpha_X(\varphi^\omega)$-preserving
conditional expectation $E$ from $M^\omega$ onto
$N\vee (N'\cap M^\omega)$.
It follows from Lemma \ref{lem:relcomm}
that
$N'\cap M^\omega=\alpha(\sC(X,X))\alpha_X((Q'\cap M_\omega)'\cap M^\omega)$,
and we have
\[
N\vee (N'\cap M^\omega)
=
\alpha(\sC(X,X))\alpha_X\big{(}(Q'\cap M_\omega)
\vee ((Q'\cap M_\omega)'\cap M^\omega)\big{)}.
\]
In particular,
$\alpha(\sC(X,X))=\sum_k\C T_k^\alpha T_k^{\alpha*}$ is contained
in the center of $N\vee (N'\cap M^\omega)$.
Thus $E(T_0^\alpha T_k^{\alpha*})=0$ for $k\geq1$.
Thanks to \cite[Theorem A.1.2]{Popa-endo},
for any $\varepsilon>0$,
we can take a finite partition of unity $\{q_r\}_{r\in J}$
in $Q'\cap M_\omega$
such that
\[
\left\|
\sum_{r\in J}
\alpha_X(q_r)T_0^\alpha T_k^{\alpha*}
\alpha_X(q_r)
\right\|_{\alpha_X(\varphi^\omega)}
<\varepsilon
\quad\mbox{for all }
k\geq1.
\]
Using $\alpha_X(q_r)T_0^\alpha T_k^{\alpha*}
\alpha_X(q_r)
=T_0^\alpha q_r\alpha_{X_k}(q_r)
T_k^{\alpha*}$,
we obtain
\begin{align*}
\left\|
\sum_{r\in J}
\alpha_X(q_r)T_0^\alpha T_k^{\alpha*}
\alpha_X(q_r)
\right\|_{\alpha_X(\varphi^\omega)}^2
&=
\sum_{r\in J}
\alpha_X(\varphi^\omega)
(|\alpha_X(q_r)T_0^\alpha T_k^{\alpha*}
\alpha_X(q_r)|^2)
\\
&=
\sum_{r\in J}
\alpha_X(\varphi^\omega)
(T_k^\alpha | q_r\alpha_{X_k}(q_r)|^2 T_k^{\alpha*})
\\
&=
\sum_{r\in J}
\frac{d(X_k)}{d(X)}
\alpha_{X_k}(\varphi^\omega)
(T_k^{\alpha*}T_k^\alpha |q_r\alpha_{X_k}(q_r)|^2 )
\\
&=
\sum_{r\in J}
\frac{d(X_k)}{d(X)}
\|q_r\alpha_{X_k}(q_r)\|_{\alpha_{X_k}(\varphi^\omega)}^2,
\end{align*}
where in the third equality,
we have used (\ref{eq:alxphi}).
On the last term,
note that $|q_r\alpha_{X_k}(q_r)|^2$ is contained in
${}_{\alpha_{X_k}}M_\omega$.
By Lemma \ref{lem:freetrace},
we obtain
$\|q_r\alpha_{X_k}(q_r)\|_{\alpha_{X_k}(\varphi^\omega)}^2
=
\|q_r\alpha_{X_k}(q_r)\|_{\varphi^\omega}^2$.
Thus
we obtain
\[
\sum_{r\in J}
\frac{d(X_k)}{d(X)}
\|q_r\alpha_{X_k}(q_r)\|_{\varphi^\omega}^2
<\varepsilon^2
\quad\mbox{for all }k\geq1,
\]
and
\[
\sum_{r\in J}
\sum_{k=1}^m
\frac{d(X_k)}{d(X)}
\|q_r\alpha_{X_k}(q_r)\|_{\varphi^\omega}^2
<
m\varepsilon^2.
\]
Then we can prove this lemma in a similar way to
\cite[Lemma 5.2]{MT-minimal}
by the following inequality:
\begin{align*}
\varphi^\omega(|q_r\alpha_{X_k}(q_r)|)
&=\varphi^\omega(\alpha_{X_k}(q_r)|q_r\alpha_{X_k}(q_r)|)
\\
&\leq
\varphi^\omega(\alpha_{X_k}(q_r))^{1/2}
\varphi^\omega(|q_r\alpha_{X_k}(q_r)|^2)^{1/2}
\\
&=
\varphi^\omega(q_r)^{1/2}
\varphi^\omega(|q_r\alpha_{X_k}(q_r)|^2)^{1/2}.
\end{align*}
\end{proof}

We will prove the following lemma that
is the key of our approach to a Rohlin tower construction.
It is informed by Y. Arano that he also has independently
obtained the similar estimate of the support projection.

\begin{lem}
\label{lem:support-estimate}
Let $e\in M_\omega$ be a projection and $X\in\sC$.
Let $s(\phi_X^\alpha(e))$ be the support projection
of $\phi_X^\alpha(e)$.
Then for any state $\varphi\in M_*$,
one has
$\varphi^\omega(s(\phi_X^\alpha(e)))
\leq2d(X)^2\alpha_X(\varphi^\omega)(e)$.
\end{lem}
\begin{proof}
Set the projection $q_X:=c_{\ovl{X},X}R_X^\alpha R_X^{\alpha*}c_{\ovl{X},X}^*$
and
the self-adjoint unitary
$u:=2q_X-1$
and
$p:=\alpha_{\overline{X}}(e)\vee u\alpha_{\overline{X}}(e)u^*$
in $M^\omega$.
Then $pq_X=q_X p$.
Since $c_{\ovl{X},X}R_X^\alpha\in M$
and $\phi_X^\alpha(e)\in M_\omega$,
we have
\begin{align*}
p^\perp q_X\phi_X^\alpha(e)
&=
p^\perp c_{\ovl{X},X}R_X^\alpha\phi_X^\alpha(e)
R_X^{\alpha*}c_{\ovl{X},X}^*
=
p^\perp q_X
\alpha_{\overline{X}}(e)
q_X
\\
&=q_Xp^\perp
\alpha_{\overline{X}}(e)
q_X
=0.
\end{align*}
This shows
$s(\phi_X^\alpha(e))\leq 1-p^\perp q_X$,
and
$s(\phi_X^\alpha(e))q_X
\leq p q_X \leq p$.

Set $\psi:=\alpha_{\overline{X}}(\alpha_X(\varphi))\in M_*$.
Note that $\psi^\omega$ commutes with
$\alpha_{\overline{X}}(M_\omega)$ and $u$.
Hence, on the one hand, we have
\begin{align*}
\psi^\omega(p)
&=
\alpha_{\overline{X}}(\alpha_X(\varphi^\omega))(p)
\\
&\leq
\alpha_{\overline{X}}(\alpha_X(\varphi^\omega))(\alpha_{\overline{X}}(e))
+
\alpha_{\overline{X}}(\alpha_X(\varphi^\omega))(u\alpha_{\overline{X}}(e)u^*)
\\
&=
2\alpha_{\overline{X}}(\alpha_X(\varphi^\omega))(\alpha_{\overline{X}}(e))
\\
&=
2\alpha_X(\varphi^\omega)(e).
\end{align*}

On the other hand, from (\ref{eq:alxphi}) we have
\begin{align*}
\psi^\omega(s(\phi_X^\alpha(e))q_X)
&=
\big{(}
R_X^{\alpha*}c_{\ovl{X},X}^*
\alpha_{\ovl{X}}(\alpha_X(\varphi^\omega))\big{)}
(s(\phi_X^\alpha(e))c_{\ovl{X},X}R_X^\alpha)
\\
&=
\frac{1}{d(X)^2}
\varphi^\omega(R_X^{\alpha*}c_{\ovl{X},X}^*
s(\phi_X^\alpha(e))c_{\ovl{X},X}R_X^\alpha)
\\
&=\frac{1}{d(X)^2}\varphi^\omega(s(\phi_X^\alpha(e))).
\end{align*}
Hence we have
\begin{align*}
\varphi^\omega(s(\phi_X^\alpha(e)))
&=
d(X)^2\psi^\omega(s(\phi_X^\alpha(e))q_X)
\leq
d(X)^2\psi(p)
\\
&\leq
2d(X)^2\alpha_X(\varphi^\omega)(e).
\end{align*}
\end{proof}

\begin{lem}
\label{lem:ee}
Let $(\alpha,c)$, $M$, $\cF$, $\varphi$,
$Q$ and $\delta$
be as in Lemma \ref{lem:qq}.
Then
there exist $n\in\N$ and a partition of unity
$\{e_r\}_{r=0}^n$ in $Q'\cap M_\omega$
such that
\begin{enumerate}
\item
$|e_0|_{\varphi^\omega}<\delta$.

\item
$e_r\alpha_X(e_r)=0$
for all $r=1,\dots,n$ and $X\in \cF$.
\end{enumerate}
\end{lem}
\begin{proof}
Our proof is almost parallel to that of \cite[Lemma 5.3]{MT-minimal}.
We may and do assume that $Q$ is $\alpha$-invariant,
that is, $\alpha_X(Q)\subset Q$ for all $X\in\sC$.
We should be careful here since we do not want to think of
the cardinality of objects of $\sC$.
However, if we assume that $Q$ contains $M$ and $\alpha_X(Q)$
for all $X\in\sC_0$,
then any $T^\alpha$ for morphisms $T$ in $\sC$
is realized in $Q$.
Thus we see $\alpha_X(Q)\subset Q$ for all $X\in\sC$.
The following claim is proved
by using Lemma \ref{lem:qq}
in the same way as Step A
in the proof of \cite[Lemma 5.3]{MT-minimal}.

\setcounter{clam}{0}

\begin{clam}
\label{cl:f}
Let $\mu>0$ and $0\neq f\in Q'\cap M_\omega$ a projection.
Then there exists a non-zero projection
$f'\in Q'\cap M_\omega$
such that $f'\leq f$
and
$\sum_{X\in\cF}|f'\alpha_X(f')|_{\varphi^\omega}
<\mu|f'|_{\varphi^\omega}$.
\end{clam}

The following claim,
which corresponds to Step B of the proof of \cite[Lemma 5.3]{MT-minimal},
is crucial.

\begin{clam}
Let $\mu>0$ and $f\in Q'\cap M_\omega$ a non-zero projection.
Then there exists a non-zero projection $e\in Q'\cap M_\omega$
such that
\begin{enumerate}
\item
$e\leq f$.

\item
$\sum_{X\in\cF}
|e\alpha_X(e)|_{\varphi^\omega}
\leq\mu|e|_{\varphi^\omega}$.

\item
$|e|_{\varphi^\omega}
\geq
(1+4|\cF|_\sigma)^{-1}
|f|_{\varphi^\omega}$.
\end{enumerate}
\end{clam}
\begin{proof}
[Proof of Claim 2.]
In a similar way to the proof of \cite[Lemma 5.3]{MT-minimal},
we can show a maximal projection $e$ with the conditions
(1) and (2) in this claim, which is non-zero from Claim 1,
satisfies the following equality:
\begin{equation}
\label{eq:support}
e\vee
\bigvee_{X\in\cF}s(\phi_X^\alpha(e))
\vee
\bigvee_{X\in\cF}s(\phi_{\overline{X}}^\alpha(e))
\vee
(1-f)
=1.
\end{equation}
Using the previous lemma and the tracial property of $\varphi^\omega$
on $M_\omega$,
we have
\begin{align*}
1
&\leq
\varphi^\omega(e)
+
\sum_{X\in\cF}
\varphi^\omega(s(\phi_X^\alpha(e)))
+
\sum_{X\in\cF}
\varphi^\omega(s(\phi_{\ovl{X}}^\alpha(e)))
+
\varphi^\omega(f^\perp)
\\
&\leq
\varphi^\omega(e)
+
\sum_{X\in\cF}
2d(X)^2
\alpha_X(\varphi^\omega)(e)
+
\sum_{X\in\cF}
2d(X)^2
\alpha_{\ovl{X}}(\varphi^\omega)(e)
+
\varphi^\omega(f^\perp)
\\
&=
\big{(}
1+
4\sum_{X\in\cF}
d(X)^2
\big{)}
\varphi^\omega(e)
+
\varphi^\omega(f^\perp)
\end{align*}
This proves Claim 2.
\end{proof}
The remaining part is proved in the same way
as Step C and D in the proof of \cite[Lemma 5.3]{MT-minimal}.
\end{proof}

\section{Rohlin property}

In this section,
we will formulate the Rohlin property
for centrally free cocycle actions
of amenable rigid C$^*$-tensor categories.

\subsection{Construction of Rohlin towers}
We begin with the following result
(cf. \cite[Lemma 5.4]{MT-minimal}).

\begin{lem}
\label{lem:ortho}
Let $(\alpha,c)$ be a cocycle action
of a rigid C$^*$-tensor category $\sC$
on a properly infinite von Neumann algebra $N$
with not necessarily separable predual
and $\cK$ a finite subset of $\Irr(\sC)$.
Suppose that a projection $e\in N$
satisfies $e\alpha_X(e)=0$
for all $X\in (\ovl{\cK}\cdot \cK)\setminus\{\btr\}$
and moreover
$e$ commutes with all morphisms in $\sC$
and all unitaries $c_{X,Y}$ with $X,Y\in\sC$.
Then the following statements hold:
\begin{enumerate}
\item
$\alpha_X(e)\alpha_Y(e)=0$
for $X\neq Y\in\cK$.

\item
$d(X)^2\phi_{\ovl{X}}^\alpha(e)$
with
$X\in\cK$
are mutually
orthogonal projections.

\item
For each $X\in\cK$,
$d(X)^2\phi_{\ovl{X}}^\alpha(e)$
is the minimum of projections
$f\in N$ such that $f\alpha_X(e)=\alpha_X(e)$,
$fc_{X,\ovl{X}}\ovl{R}_{X}^\alpha
=c_{X,\ovl{X}}\ovl{R}_{X}^\alpha f$.

\item
For each $X\in\cK$,
$d(X)^2\alpha_X(e)
c_{X,\ovl{X}}\ovl{R}_X^\alpha\ovl{R}_X^{\alpha*}c_{X,\ovl{X}}^*\alpha_X(e)
=\alpha_X(e)$.
\end{enumerate}
\end{lem}
\begin{proof}
(1).
Since $\btr\not\prec\ovl{X}\otimes Y$,
$e\alpha_{\ovl{X}\otimes Y}(e)=0$.
This implies $e\alpha_{\ovl{X}}(\alpha_Y(e))=0$.
The equality of (1) follows from
\begin{align*}
\phi_X^\alpha(\alpha_X(e)\alpha_Y(e)\alpha_X(e))
&=
e
\phi_X^\alpha(\alpha_Y(e))
e
=
eR_X^{\alpha*}
c_{\ovl{X},X}^*\alpha_{\ovl{X}}(\alpha_Y(e))c_{\ovl{X},X}R_X^\alpha  e
\\
&=
R_X^{\alpha*}c_{\ovl{X},X}^*
e\alpha_{\ovl{X}}(\alpha_Y(e))e c_{\ovl{X},X}R_X^\alpha
=0.
\end{align*}

(2).
Applying $\phi_{\ovl{X}}^\alpha$ to the equality
$e\alpha_{\ovl{X}}(\alpha_X(e))
=c_{\ovl{X},X}R_X^\alpha R_X^{\alpha*}c_{\ovl{X},X}^*
\alpha_{\ovl{X}}(\alpha_X(e))$,
we have
\[
\phi_{\ovl{X}}^\alpha(e)\alpha_X(e)
=
d(X)^{-2}
\alpha_X(e).
\]
Note
$\phi_{\ovl{X}}^\alpha(e)$ commutes
with $T^\alpha$ for all morphisms
$T\in\sC(Y,Z)$ with $Y,Z\in\sC$
and
$c_{Y,Z}$ for all $Y,Z\in\sC$.
Indeed,
we have
$\alpha_{\ovl{X}}(c_{Y,Z})=c_{X,Y}c_{X\otimes Y,Z}c_{X,Y\otimes Z}^*$,
and
\[
\phi_{\ovl{X}}^\alpha(e)c_{Y,Z}
=
\phi_{\ovl{X}}^\alpha(e\alpha_{\ovl{X}}(c_{Y,Z}))
=
\phi_{\ovl{X}}^\alpha(\alpha_{\ovl{X}}(c_{Y,Z})e)
=
c_{Y,Z}\phi_{\ovl{X}}^\alpha(e).
\]
Next for $T\in\sC(Y,Z)$,
we have
$\alpha_{\ovl{X}}(T^\alpha)
=c_{\ovl{X},Z}[1_{\ovl{X}}\otimes T]^\alpha
c_{\ovl{X},Y}^*$,
and
\[
\phi_{\ovl{X}}^\alpha(e)T^\alpha
=
\phi_{\ovl{X}}^\alpha(e\alpha_{\ovl{X}}(T^\alpha))
=
\phi_{\ovl{X}}^\alpha(\alpha_{\ovl{X}}(T^\alpha)e)
=
T^\alpha\phi_{\ovl{X}}^\alpha(e).
\]
Hence we have
\[
\phi_{\ovl{X}}^\alpha(e)^2
=
\ovl{R}_X^{\alpha*} c_{X,\ovl{X}}^*
\phi_{\ovl{X}}^\alpha(e)\alpha_X(e)
c_{X,\ovl{X}}\ovl{R}_X^\alpha
=
d(X)^{-2}\phi_{\ovl{X}}^\alpha(e).
\]
This implies
$d(X)^2\phi_{\ovl{X}}^\alpha(e)$ is a projection.
The orthogonality
$\phi_{\ovl{X}}^\alpha(e)\phi_{\ovl{Y}}^\alpha(e)=0$
for $X\neq Y\in\cK$
holds
since $e\alpha_{\ovl{X}}(\phi_{\ovl{Y}}^\alpha(e))=0$,
which is proved as follows:
\begin{align*}
e\alpha_{\ovl{X}}(\phi_{\ovl{Y}}^\alpha(e))
&=
e\alpha_{\ovl{X}}(R_Y^{\alpha*}c_{\ovl{Y},Y}^*)
\alpha_{\ovl{X}}(\alpha_Y(e))
\alpha_{\ovl{X}}(c_{\ovl{Y},Y}R_Y^\alpha)
\\
&=
\alpha_{\ovl{X}}(R_Y^{\alpha*}c_{\ovl{Y},Y}^*)
e\alpha_{\ovl{X}}(\alpha_Y(e))
\alpha_{\ovl{X}}(c_{\ovl{Y},Y}R_Y^\alpha)
=0.
\end{align*}

(3).
By the proof of (2),
we know $d(X)^2\phi_{\ovl{X}}^\alpha(e)\geq\alpha_X(e)$
and $d(X)^2\phi_{\ovl{X}}^\alpha(e)$
commutes with all morphisms in $\sC$
and unitaries $c_{Y,Z}$ for all $Y,Z\in\sC$.
Suppose a projection $f\in M^\omega$
is given as in the statement of (3).
Then $\alpha_{\ovl{X}}(f)e=e$.
Indeed, we have
\[
\phi_{\ovl{X}}^\alpha(\alpha_{\ovl{X}}(f^\perp)e\alpha_{\ovl{X}}(f^\perp))
=
f^\perp\phi_{\ovl{X}}^\alpha(e)f^\perp
=
R_{\ovl{X}}^{\alpha*}c_{X,\ovl{X}}^*
f^\perp \alpha_X(e)f^\perp c_{X,\ovl{X}}R_{\ovl{X}}^\alpha
=0.
\]
Hence $f\phi_{\ovl{X}}^\alpha(e)=\phi_{\ovl{X}}^\alpha(e)$.

(4).
We know from (2)
that $d(X)\alpha_X(e)c_{X,\ovl{X}}\ovl{R}_X^\alpha$ is a partial isometry,
and
$p:=d(X)^2\alpha_X(e)
c_{X,\ovl{X}}\ovl{R}_X^\alpha\ovl{R}_X^{\alpha *}c_{X,\ovl{X}}^*\alpha_X(e)$
is a projection with $p\leq\alpha_X(e)$.
Since
$\phi_X^\alpha(p)=d(X)^2 e\phi_X^\alpha(c_{X,\ovl{X}}\ovl{R}_X^\alpha\ovl{R}_X^{\alpha *}c_{X,\ovl{X}}^*)e=e$,
we have $p=\alpha_X(e)$.
\end{proof}

We will prove that a centrally free action has the Rohlin property
that is studied in \cite{Mas-uni,MT-minimal,MT-III,MT-discrete,Oc}
assuming the action is probability measure preserving
on the center.
Recall the notation $\theta^\alpha$ introduced in Lemma \ref{lem:thetaalpha}.

\begin{thm}\label{thm:Rohlin-tensor}
Let $(\alpha,c)$ be a centrally free
cocycle action of
an amenable rigid C$^*$-tensor category $\sC$
on a properly infinite von Neumann algebra $M$ with separable predual.
Let $\varphi$ be a faithful normal state on $M$
such that $\varphi$
is $\theta^\alpha$-invariant on $Z(M)$.
Let $0<\de<1$ and $\cF\subset\Irr(\sC)$ a finite subset
with $\cF=\ovl{\cF}$.
Let $\cK$ be an $(\cF,\delta)$-invariant finite subset
of $\Irr(\sC)$.
Then for any countably generated von Neumann subalgebra
$Q\subset M^\omega$,
there exists a family of projections
$E_X\in M^\omega$
with
$X\in\sC$
satisfying the following conditions:
\begin{enumerate}
\item
\label{item:Roh-natural}
(Naturality)
$E_Y T^\alpha=T^\alpha E_X$
for all $X,Y\in\sC$ and $T\in\sC(X,Y)$.

\item
\label{item:Roh-support}
$E_X=0$ for all $X\in\Irr(\sC)\setminus\cK$.

\item
\label{item:Roh-commute}
$E_X\in\alpha_X(Q)'\cap {}_{\alpha_X}M_\omega$
for all $X\in\Irr(\sC)$.

\item
(Splitting property)
\label{item:Roh-split}
$\tau^\omega(E_X x)
=\tau^\omega(E_X)\tau^\omega(x)$
for all $X\in\sC$ and $x\in Q$.

\item
\label{item:Roh-orthogonal}
$\{d(X)^2\ovl{R}_X^{\alpha*}c_{X,\ovl{X}}^*
E_X
c_{X,\ovl{X}}\ovl{R}_X^\alpha\}_{X\in\cK}$
is a family of orthogonal projections in $Q'\cap M_\omega$.

\item
\label{item:Roh-partition}
(Approximate partition of unity)
\[
\sum_{X\in\cK}\varphi^\omega
(d(X)^2\ovl{R}_X^{\alpha*}
c_{X,\ovl{X}}^*E_Xc_{X,\ovl{X}}\ovl{R}_X^\alpha)
\geq1-\delta^{1/2}.
\]

\item
\label{item:Roh-equivariance}
(Approximate equivariance)
\[
\sum_{X\in\cF}
\sum_{Y\in\Irr(\sC)}
d(X)^2d(Y)^2
|\alpha_X(E_Y)
-
c_{X,Y}E_{X\otimes Y}c_{X,Y}^*
|_{\alpha_X(\alpha_Y(\varphi^\omega))}
\leq
6\delta^{1/2}|\cF|_\sigma.
\]

\item
\label{item:Roh-resonance}
(Resonance property)
\[E_X
c_{X,\ovl{X}}\ovl{R}_X^\alpha
E_Y
=
\frac{\delta_{X,Y}}{d(X)}\alpha_{X}(c_{\ovl{X},X}R_X^\alpha)E_X
\]
for all
$X,Y\in\Irr(\sC)$
and
$r\in\Lambda$,
where $\delta_{X,Y}$ denotes the Kronecker delta.
\end{enumerate}
\end{thm}

In what follows,
we assume that $Q$ contains $M$ and $\alpha_X(Q)\subset Q$
for all $X\in \sC$.
(Also see the proof of Lemma \ref{lem:ee} on the problem of
the cardinality of objects in $\sC$.)
We will call a family of projections $E=(E_X)_X$
described in Theorem \ref{thm:Rohlin-tensor}
a \emph{Rohlin tower} along with $\cK$ henceforth.

\begin{rem}
\label{rem:Rohlin-tensor}
Remarks are in order.
\begin{enumerate}
\item 
If $x\in Q'\cap M^\omega$,
then $\tau^\omega(x)\in Z(M)$.
This implies $\alpha_X(\varphi^\omega)=\varphi^\omega$
on $Q'\cap M^\omega$.
Hence when an element $x\in Q'\cap M^\omega$
commutes with $\alpha_X(\varphi^\omega)$
for some $X\in\Irr(\sC)$,
then
for all $y\in Q'\cap M^\omega$, we have
\[
\varphi^\omega(xy)=\alpha_X(\varphi^\omega)(xy)
=\alpha_X(\varphi^\omega)(yx)=\varphi^\omega(yx).
\]
\item
By Lemma \ref{lem:freetrace}
and Theorem \ref{thm:Rohlin-tensor} (\ref{item:Roh-commute}),
we know that
$\tau^\omega(E_X)$ is in $Z(M)$
for all $X\in\cK$.

\item
From Theorem \ref{thm:Rohlin-tensor} (\ref{item:Roh-natural})
and (\ref{item:Roh-commute}),
we see $E_X$ commutes with
$\alpha_X(Q)$
and
$\alpha_X(\psi^\omega)$ for all $\psi\in M_*$
for all $X\in\sC$.

\item
\label{item:rem-Roh-orthogonal}
It turns out from Theorem \ref{thm:Rohlin-tensor} (\ref{item:Roh-orthogonal})
that
$\{d(X)\ovl{R}_X^{\alpha*}c_{X,\ovl{X}}^*E_X\}_{X\in\Irr(\sC)}$
are partial isometries with orthogonal range projections.
Each initial projection
$d(X)^2E_X c_{X,\ovl{X}}\ovl{R}_X^\alpha \ovl{R}_X^{\alpha*}c_{X,\ovl{X}}^*
E_X$
is supported by $E_X$.
Actually they are equal.
Indeed, we have
\begin{align*}
&\alpha_X(\varphi^\omega)
(d(X)^2E_X c_{X,\ovl{X}}\ovl{R}_X^\alpha
\ovl{R}_X^{\alpha*}c_{X,\ovl{X}}^*E_X)
\\
&=
d(X)^2\alpha_X(\varphi^\omega)(E_X 
c_{X,\ovl{X}}\ovl{R}_X^\alpha \ovl{R}_X^{\alpha*}c_{X,\ovl{X}}^*)
\\
&=
d(X)^2\alpha_X(\varphi)(\tau^\omega(E_X)
c_{X,\ovl{X}}\ovl{R}_X^\alpha \ovl{R}_X^{\alpha*}c_{X,\ovl{X}}^*)
\\
&=
d(X)^2\varphi(\theta_{\ovl{X}}^\alpha(\tau^\omega(E_X))
\cdot\phi_X^\alpha
(c_{X,\ovl{X}}\ovl{R}_X^\alpha \ovl{R}_X^{\alpha*}c_{X,\ovl{X}}^*))
\\
&=
\alpha_X(\varphi^\omega)(E_X).
\end{align*}

\item
The left-hand side of the inequality in
Theorem \ref{thm:Rohlin-tensor} (\ref{item:Roh-equivariance})
is actually a finite sum.
Indeed,
$\alpha_X(E_Y)-E_{X\otimes Y}=0$
for $Y\nin \cK\cup (\ovl{\cF}\cdot\cK)$.

\item
Let $\varphi_1\in M_*$ be a faithful state 
such that $\varphi_1=\varphi$ on $Z(M)$.
Then the values stated in (\ref{item:Roh-partition})
and (\ref{item:Roh-equivariance})
do not change
since
$\tau^\omega
(d(X)^2\ovl{R}_X^{\alpha*}
c_{X,\ovl{X}}^*E_Xc_{X,\ovl{X}}\ovl{R}_X^\alpha)
$
and
$\phi_Y^\alpha(\phi_X^\alpha(
\tau^\omega(|\alpha_X(E_Y)
-c_{X,Y}E_{X\otimes Y}c_{X,Y}^*|)))$
are contained in $Z(M)$.

\item
To prove Theorem \ref{thm:Rohlin-tensor},
it is sufficient to treat only actions $(\alpha,1)$.
This is because
we can consider the action of the C$^*$-tensor category
generated by the range $\alpha(\sC)$ in $\End(M)_0$.
We, however, do not use such a trick
since it seems treating a 2-cocycle $c$ is not so involving task
and helps us to understand the role of $c$ in computation.

\item
In our preceding papers
\cite{MT-minimal,MT-III,MT-discrete},
the approximate innerness of actions is crucially used
for technical reasons,
but this is removed in Theorem \ref{thm:Rohlin-tensor}.
\end{enumerate}
\end{rem}

We will introduce the set $\mathcal{J}$ which is the collection
of a family of projections
$E:=(E_X)_{X\in\sC}$
satisfying
the conditions
of Theorem \ref{thm:Rohlin-tensor}
(\ref{item:Roh-natural}),
(\ref{item:Roh-support}),
(\ref{item:Roh-commute}),
(\ref{item:Roh-split}),
(\ref{item:Roh-orthogonal})
and
(\ref{item:Roh-resonance}).
Trivially, $E=(0)_X$ is a member of $\mathcal{J}$.
We set the functions $a,b$ from $\mathcal{J}$ into $\R_+$
by
\begin{align*}
a_E
&:=
\frac{1}{|\cF|_\sigma}
\sum_{X\in\cF}
\sum_{Y\in\Irr(\sC)}
d(X)^2d(Y)^2
|\alpha_X(E_Y)
-c_{X,Y}E_{X\otimes Y}c_{X,Y}^*|_{\alpha_{X}(\alpha_Y(\varphi^\omega))},
\\
b_E
&:=
\sum_{X\in\cK}d(X)^2\varphi^\omega(E_X).
\end{align*}
Note that
$b_E
=\varphi^\omega(\sum_X d(X)^2
\ovl{R}_X^{\alpha*}c_{X,\ovl{X}}^*E_X c_{X,\ovl{X}}\ovl{R}_X^\alpha)$
since $c_{X,\ovl{X}}\ovl{R}_X^\alpha\in M$ and $\tau^\omega(E_X)\in Z(M)$.
This implies $b_E\leq1$
from Theorem \ref{thm:Rohlin-tensor} (\ref{item:Roh-orthogonal}).

\begin{lem}
\label{lem:ab}
Let $E\in\mathcal{J}$
with $b_E<1-\delta^{1/2}$.
Then there exists $E'\in\mathcal{J}$
such that
\begin{itemize}
\item
$a_{E'}-a_E\leq 6\delta^{1/2}(b_{E'}-b_E)$.

\item
$0<(\delta^{1/2}/2)
\sum_{X\in\cK}d(X)^2\varphi^\omega(|E_X'-E_X|)
\leq b_{E'}-b_E$.
\end{itemize}
\end{lem}
\begin{proof}
Our proof is similar to the proof of \cite[Lemma 6.4]{Oc}.
Take $\de_1>0$ so that $b_E<(1-\delta_1)(1-\delta^{1/2})$.
Let $Q_1$ be the $\alpha$-invariant
von Neumann subalgebra of $M^\omega$
that is generated by $Q$ and all $\alpha_X(E_Y)$'s
for $X,Y\in \Irr(\sC)$.
Then $Q_1$ actually contains $\alpha_X(E_Y)$
for all $X,Y\in\sC$
since $T^\alpha\in M$ for all morphisms $T$ in $\sC$
and $E$ is natural.
We let $\cK_1:=\cK\cup (\cF\cdot \cK)$.
By Lemma \ref{lem:ee},
we can obtain
a partition of unity $\{e_r\}_{r=0}^n$ in $Q_1'\cap M_\omega$
which satisfies
$|e_0|_{\varphi^\omega}<\delta_1$
and
$e_r\alpha_X(e_r)=0$
for $X\in \ovl{\cK_1}\cdot\cK_1\setminus\{\btr\}$
and $r\geq1$.
By using the fast reindexation \cite[Lemma 3.10]{MT-minimal}
if necessary,
we may and do assume that
$\tau^\omega(x \alpha_X(e_r))=\tau^\omega(x)\tau^\omega(\alpha_X(e_r))$
for all $x\in Q_1$, $X\in\Irr(\sC)$ and $r=0,\dots,n$.
Note that this equality holds also for all $X\in\sC$.

\setcounter{clam}{0}
\begin{clam}
The inequality
\[
\sum_{X,Y\in\cK}
d(X)^2d(Y)^2
\varphi^\omega(E_X\phi_{\ovl{Y}}^\alpha(e_r))
<(1-\delta^{1/2})|\cK|_\sigma
\varphi^\omega(e_r)
\]
holds for some $r\geq1$.
\end{clam}
\begin{proof}[Proof of Claim 1]
Suppose that the following inequalities hold for all $r\geq1$:
\[
\sum_{X,Y\in\cK}
d(X)^2d(Y)^2
\varphi^\omega(E_X\phi_{\ovl{Y}}^\alpha(e_r))
\geq(1-\delta^{1/2})|\cK|_\sigma
\varphi^\omega(e_r).
\]
Using $\sum_{r\geq1}e_r=e_0^\perp$,
we have
\begin{align*}
|\cK|_\sigma b_E
&=
\sum_{X,Y\in\cK}
d(X)^2d(Y)^2
\varphi^\omega(E_X)
\\
&\geq
\sum_{X,Y\in\cK}
d(X)^2d(Y)^2
\varphi^\omega(E_X\phi_{\ovl{Y}}^\alpha(e_0^\perp))
\\
&\geq
(1-\delta^{1/2})|\cK|_\sigma
\varphi^\omega(e_0^\perp)
\\
&>
(1-\delta^{1/2})(1-\delta_1)|\cK|_\sigma,
\end{align*}
which contradicts to
$b_E<(1-\delta_1)(1-\delta^{1/2})$.
\end{proof}

Take $r\geq1$
satisfying the previous claim.
We write $e:=e_r$.
Then
\begin{equation}
\label{eq:XYK}
\sum_{X,Y\in\cK}
d(X)^2d(Y)^2
\varphi^\omega(E_X\phi_{\ovl{Y}}^\alpha(e))
<(1-\delta^{1/2})|\cK|_\sigma
\varphi^\omega(e).
\end{equation}

Recall the averaging map
$I_\cK^\alpha
=|\cK|_\sigma^{-1}
\sum_{X\in\cK}d(X)^2\phi_{\ovl{X}}^\alpha$
introduced in Section \ref{subsect:actions}.
We set $f:=|\cK|_\sigma I_\cK^\alpha(e)$
that is a projection in $Q_1'\cap M_\omega$
from Lemma \ref{lem:ortho}.
Since $\varphi$ is $\theta^\alpha$-invariant
on $Z(M)$,
we have
$\varphi^\omega(f)=|\cK|_\sigma\varphi^\omega(e)$.
Hence (\ref{eq:XYK}) is expressed as follows:
\begin{equation}
\label{eq:XKf}
\sum_{X\in\cK}
d(X)^2
\varphi^\omega(E_Xf)
<(1-\delta^{1/2})
\varphi^\omega(f).
\end{equation}

Recall the projection $P_X^\cK\in\sC(X,X)$
defined in Section \ref{subsect:tensorcat}.
Then we set a projection $E_X'$ as follows:
\[
E_X':=E_X f^\perp+\alpha_X(e)[P_X^\cK]^\alpha
\quad
\mbox{for }
X\in\sC.
\]

\begin{clam}
$E':=(E_X')_X$ is a member of $\mathcal{J}$.
\end{clam}
\begin{proof}[Proof of Claim 2]
Using Lemma \ref{lem:ortho} and the equality $E_Xf=fE_X$,
we see that
$E'$ satisfies
Theorem \ref{thm:Rohlin-tensor}
(\ref{item:Roh-natural}),
(\ref{item:Roh-support}),
(\ref{item:Roh-commute}),
(\ref{item:Roh-split})
and
(\ref{item:Roh-orthogonal}).
We will check $E'$ satisfies the resonance property
of Theorem \ref{thm:Rohlin-tensor} (\ref{item:Roh-resonance}).
Since $f\in M_\omega$ and $f^\perp \alpha_X(e)=0$
for $X\in \cK$,
it suffices to prove that
$\alpha_X(e)
c_{X,\ovl{X}}\ovl{R}_X^\alpha
\alpha_Y(e)
=\frac{\delta_{X,Y}}{d(X)}
\alpha_X(e)
\alpha_X(c_{\ovl{X},X}R_X^\alpha)$
for all $X,Y\in\Irr(\sC)$.
We have
\begin{align*}
\alpha_X(e)
c_{X,\ovl{X}}\ovl{R}_X^\alpha
\alpha_Y(e)
&=
\alpha_X(e\alpha_{\ovl{X}}(\alpha_Y(e)))c_{X,\ovl{X}}\ovl{R}_X^\alpha
\\
&=
\delta_{X,Y}
\alpha_X(e\alpha_{\ovl{X}}(\alpha_X(e)))c_{X,\ovl{X}}\ovl{R}_X^\alpha
\quad
\\
&=
\delta_{X,Y}
\alpha_X(ec_{\ovl{X},X}R_X^\alpha R_X^{\alpha*} c_{\ovl{X},X}^*)
c_{X,\ovl{X}}\ovl{R}_X^\alpha
\\
&=
\delta_{X,Y}
\alpha_X(e)
\alpha_X(c_{\ovl{X},X}R_X^\alpha R_X^{\alpha*})
c_{X,\ovl{X}\otimes X}c_{X\otimes \ovl{X},X}^*
\ovl{R}_X^\alpha
\\
&=
\delta_{X,Y}
\alpha_X(e)
\alpha_X(c_{\ovl{X},X}R_X^\alpha)
c_{X,\btr}[1_X\otimes R_X^*]^\alpha
[\ovl{R}_X\otimes1_X]^\alpha
\\
&=
\frac{\delta_{X,Y}}{d(X)}
\alpha_X(e)
\alpha_X(c_{\ovl{X},X}R_X^\alpha).
\end{align*}
\end{proof}

The difference of $E$ and $E'$ is estimated
as follows:
\begin{align}
b_{E'}
&=
\sum_{X\in\cK}d(X)^2
\varphi^\omega(E_X f^\perp+\alpha_X(e))
\notag
\\
&=
b_E
-
\sum_{X\in\cK}d(X)^2
\varphi^\omega(E_X f)
+
|\cK|_\sigma \varphi^\omega(e)
\notag
\\
&>
b_E
-(1-\delta^{1/2})\varphi^\omega(f)
+\varphi^\omega(f)
\quad
\mbox{by }
(\ref{eq:XKf})
\notag
\\
&=
b_E+\delta^{1/2}\varphi^\omega(f).
\label{eq:bEE'}
\end{align}
Thus we have
$b_{E'}-b_E>\delta^{1/2}\varphi^\omega(f)$,
and
\begin{align*}
\sum_{X\in\cK}d(X)^2
\varphi^\omega(|E_X'-E_X|)
&=
\sum_{X\in\cK}d(X)^2
\varphi^\omega(|-E_Xf+\alpha_X(e)|)
\\
&\leq
\sum_{X\in\cK}d(X)^2
\varphi^\omega(E_Xf)
+
\sum_{X\in\cK}d(X)^2
\varphi^\omega(\alpha_X(e)),
\end{align*}
where we have used the tracial property of $\varphi^\omega$
on ${}_{\alpha_X}M_\omega$ (see Lemma \ref{lem:freetrace}).
Using the splitting property, we have
\begin{align*}
\tau^\omega(\ovl{R}_X^{\alpha*}c_{X,\ovl{X}}^*E_X
c_{X,\ovl{X}}\ovl{R}_X^{\alpha}f)
&=
\tau^\omega(\ovl{R}_X^{\alpha*}c_{X,\ovl{X}}^*E_X
c_{X,\ovl{X}}\ovl{R}_X^{\alpha})
\tau^\omega(f)
\\
&=\tau^\omega(E_X)\tau^\omega(f)=\tau^\omega(E_Xf).
\end{align*}
Hence we have
\begin{align*}
&\sum_{X\in\cK}d(X)^2
\varphi^\omega(|E_X'-E_X|)
\\
&\leq
\varphi^\omega
(\sum_{X\in\cK}d(X)^2
\ovl{R}_X^{\alpha*}c_{X,\ovl{X}}^*E_X c_{X,\ovl{X}}\ovl{R}_X^\alpha f)
+
\sum_{X\in\cK}d(X)^2
\varphi^\omega(e)
\\
&\leq
\varphi^\omega(f)
+
|\cK|_\sigma\varphi^\omega(e)
\quad
\mbox{by Theorem \ref{thm:Rohlin-tensor} (\ref{item:Roh-orthogonal}})
\\
&=2\varphi^\omega(f).
\end{align*}
From this and (\ref{eq:bEE'}),
the second inequality of this lemma holds.

Next, we will show the first one.
For $X,Y\in\Irr(\sC)$,
we have
\begin{align}
&\alpha_X(E_Y')-c_{X,Y}E_{X\otimes Y}'c_{X,Y}^*
\notag
\\
&=
\alpha_X(E_Y)\alpha_X(f^\perp)
+\alpha_{X}(\alpha_Y(e))\alpha_X([P_Y^\cK]^\alpha)
\notag
\\
&\quad
-
c_{X,Y}
(E_{X\otimes Y}f^\perp+\alpha_{X\otimes Y}(e)
[P_{X\otimes Y}^\cK]^\alpha)
c_{X,Y}^*
\notag
\\
&=
\alpha_X(E_Y)(\alpha_X(f^\perp)-f^\perp)
+
(\alpha_X(E_Y)-c_{X,Y}E_{X\otimes Y}c_{X,Y}^*)f^\perp
\notag\\
&\quad
+
\alpha_{X}(\alpha_Y(e))(\alpha_X([P_Y^\cK]^\alpha)
-c_{X,Y}[P_{X\otimes Y}^\cK]^\alpha c_{X,Y}^*).
\label{eq:EYPX}
\end{align}
Note
$\alpha_{X}(\alpha_Y(\varphi^\omega))
(\alpha_X(E_Y)|\alpha_X(f)-f|)
=
\varphi^\omega
(\alpha_X(E_Y)|\alpha_X(f)-f|)$
since
\begin{align*}
\tau^\omega(\alpha_X(E_Y)|\alpha_X(f)-f|)
&=
\tau^\omega(\alpha_X(E_Y))\tau^\omega(|\alpha_X(f)-f|)
\\
&=
\theta_X^\alpha(\tau^\omega(E_Y))
\tau^\omega(|\alpha_X(f)-f|)
\in Z(M).
\end{align*}
Also note that from (\ref{eq:PXYZ})
we have
\begin{align*}
&\phi_Y^\alpha(\phi_{X}^\alpha
(|\alpha_X([P_Y^\cK]^\alpha)-c_{X,Y}[P_{X\otimes Y}^\cK]^\alpha c_{X,Y}^*|))
\\
&=
1_\cK(Y)
\sum_{Z\in\cK^c}
p_X(Y,Z)
+
1_{\cK^c}(Y)
\sum_{Z\in\cK}
p_X(Y,Z).
\end{align*}

Then by (\ref{eq:EYPX})
we have
\begin{align*}
&|\alpha_X(E_Y')
-c_{X,Y}E_{X\otimes Y}'c_{X,Y}^*|_{\alpha_X(\alpha_Y(\varphi^\omega))}
-
|\alpha_X(E_Y)
-c_{X,Y}E_{X\otimes Y}c_{X,Y}^*|_{\alpha_X(\alpha_Y(\varphi^\omega))}
\\
&\leq
\varphi^\omega(\alpha_X(E_Y)|\alpha_X(f)-f|)
\\
&\quad
+
\varphi^\omega(e)
\big{(}
1_\cK(Y)
\sum_{Z\in\cK^c}
p_X(Y,Z)
+
1_{\cK^c}(Y)
\sum_{Z\in\cK}
p_X(Y,Z)
\big{)}.
\end{align*}

Hence by Lemma \ref{lem:Folner},
we have
\begin{align*}
|\cF|_\sigma(a_{E'}-a_E)
&\leq
2\delta|\cF|_\sigma|\cK|_\sigma
\varphi^\omega(e)
\\
&\quad
+
\sum_{X\in\cF}
\sum_{Y\in\Irr(\sC)}
d(X)^2d(Y)^2
\varphi^\omega
(\alpha_X(E_Y)|\alpha_X(f)-f|).
\end{align*}
Using the splitting property as before,
we know
\begin{align*}
&\sum_{Y\in\Irr(\sC)}
d(Y)^2
\varphi^\omega
(\alpha_X(E_Y)|\alpha_X(f)-f|)
\\
&=
\varphi^\omega
\big{(}
\alpha_X
\big{(}
\sum_{Y\in\cK}
d(Y)^2R_{\ovl{Y}}^{\alpha*}c_{Y,\ovl{Y}}^*
E_Y c_{Y,\ovl{Y}}R_{\ovl{Y}}^\alpha
\big{)}
|\alpha_X(f)-f|
\big{)}
\\
&\leq
\varphi^\omega
(|\alpha_X(f)-f|).
\end{align*}
Thus
\begin{equation}
\label{eq:FaEf}
|\cF|_\sigma(a_{E'}-a_E)
\leq
2\delta|\cF|_\sigma\varphi^\omega(f)
+
\sum_{X\in\cF}
d(X)^2
\varphi^\omega
(|\alpha_X(f)-f|).
\end{equation}

\begin{clam}
The projection $f$ commutes with $\alpha_X(f)$ for all $X\in \cF$.
\end{clam}
\begin{proof}[Proof of Claim 3]
Recall that $f=\sum_{Y\in\cK}d(Y)^2\phi_{\ovl{Y}}^\alpha(e)$ holds
and $d(Y)^2\phi_{\ovl{Y}}^\alpha(e)$ are orthogonal projections
for $Y\in \cK_1$.
It suffices to show that
$f\alpha_{X\otimes Y}(e)=\alpha_{X\otimes Y}(e)f$
for all $X\in\cF$ and $Y\in \cK$.
This is equivalent to the equality
$f\alpha_Z(e)=\alpha_Z(e)f$ for all $Z\in\cF\cdot\cK$.
Let $Y\in\cK$.
Then
$d(Y)^2\phi_{\ovl{Y}}^\alpha(e)\alpha_Z(e)$
equals 0 if $Z\neq Y$
and
does $\alpha_Y(e)$
if $Z=Y$
from Lemma \ref{lem:ortho}.
Hence we are done.
\end{proof}

By the previous claim,
we have
$|\alpha_X(f)-f|=\alpha_X(f)f^\perp+\alpha_X(f^\perp)f$
and
\begin{align*}
\varphi^\omega
(|\alpha_X(f)-f|)
&=
\varphi^\omega(\alpha_X(f)f^\perp)+\varphi^\omega(\alpha_X(f^\perp)f)
\\
&=
\varphi^\omega(f\phi_X^\alpha(f^\perp))+\varphi^\omega(f^\perp\phi_X^\alpha(f))
\\
&=
\varphi^\omega(f-f\phi_X^\alpha(f))
+
\varphi^\omega(\phi_X^\alpha(f)-f\phi_X^\alpha(f))
\\
&=2\varphi^\omega(f-f\phi_X^\alpha(f))
\\
&\leq
2|f-\phi_X^\alpha(f)|_{\varphi^\omega}.
\end{align*}
From (\ref{eq:FaEf}),
we obtain
\begin{equation}
\label{FEfphi}
|\cF|_\sigma(a_{E'}-a_E)
\leq
2\delta|\cF|_\sigma|f|_{\varphi^\omega}
+
2\sum_{X\in\cF}d(X)^2|f-\phi_X^\alpha(f)|_{\varphi^\omega}.
\end{equation}
Applying Lemma \ref{lem:Folner}
and \ref{lem:average} to $f=|\cK|_\sigma I_{\cK}^\alpha(e)$,
we obtain
\[
\sum_{X\in\cF}d(X)^2
|\phi_X^\alpha(f)-f|_{\varphi^\omega}
\leq
2\de
|\cF|_\sigma
|\cK|_\sigma|e|_{\varphi^\omega}
=
2\delta|\cF|_\sigma
|f|_{\varphi^\omega},
\]
where we have also used the equality
$|\phi_Z^\alpha(e)|_{\varphi^\omega}=|e|_{\varphi^\omega}$
for $Z\in\Irr(\sC)$.
Thus from (\ref{eq:bEE'}) and (\ref{FEfphi}),
we obtain
\[
a_{E'}-a_E
\leq
6
\delta\varphi^\omega(f)
\leq
6\delta^{1/2}(b_{E'}-b_E).
\]
\end{proof}

\begin{proof}[Proof of Theorem \ref{thm:Rohlin-tensor}]
Let $\mathcal{I}$ be the subset of $\mathcal{J}$
which consists of $E=(E_X)_X$ satisfying
$a_E\leq6\delta^{1/2}b_E$.
We define the order on $\mathcal{I}$
by $E\leq E'$
if $E=E$ or the inequalities in Lemma \ref{lem:ab} hold.
Then this order is inductive as shown
in \cite[Proof of Theorem 5.9]{MT-minimal}
or \cite[p.54]{Oc}.
Take a maximal element $E=(E_X)_X$ in $\mathcal{I}$.
Then we have $b_{E}\geq1-\delta^{1/2}$ by Lemma \ref{lem:ab},
and we are done.
\end{proof}

\subsection{A slight generalization of the Rohlin tower theorem to semiliftable cocycle actions}
\label{subsect:slight}

Let $(\gamma,c^\gamma)$ be a cocycle action of $\sC$ on $M^\omega$.
We will say $(\gamma,c^\gamma)$ is \emph{semiliftable}
when for each $X\in\sC$, there exists a sequence of $\beta_X^n\in\End(M)_0$
with $n\in\N$
and $\beta_X\in\End(M)_0$
such that $\beta_X^n$ converges to $\beta_X$ in $\End(M)_0$
and
$\gamma_X(x)=(\beta_X^n(x_n))^\omega$
for $x=(x_n)^\omega\in M^\omega$.
(See \cite[Definition 3.4]{MT-minimal} or \cite[Chapter 5.2]{Oc}.)
We can actually show Theorem \ref{thm:Rohlin-tensor}
for centrally free semiliftable cocycle actions.
We, however, apply this version to a simple case of 
$(\gamma,c^\gamma)$ being a unitary perturbation of
a cocycle action $(\alpha,c^\alpha)$ on $M$
in Section \ref{subsect:approx}.
Hence we do not prove Theorem \ref{thm:Rohlin-tensor}
in full generality.
We will explain this more precisely as follows.

\begin{defn}
\label{defn:approx}
Two cocycle actions $(\alpha,c^\alpha)$ and $(\beta,c^\beta)$
of $\sC$ on $M$
are said to be \emph{approximately unitarily equivalent}
when
for each $X\in\sC$
there exists a unitary $u_X\in M^\omega$
such that
$u_X\alpha_X(\psi^\omega)u_X^{*}=\beta_X(\psi^\omega)$
for all $\psi\in M_*$.
\end{defn}

\begin{lem}
\label{lem:appunitequiv}
Let $(\alpha,c^\alpha)$ and $(\beta,c^\beta)$
be approximately unitarily equivalent cocycle actions.
There exists a family of unitaries $u_X\in M^\omega$
with $X\in\sC$
such that the following conditions hold:
\begin{itemize}
\item
$u_X\alpha_X(\psi^\omega)u_X^*=\beta_X(\psi^\omega)$
for all $X\in\sC$ and $\psi\in M_*$.

\item
$u_Y T^\alpha=T^\beta u_X$
for all $X\in\sC$ and $T\in\sC(X,Y)$.
\end{itemize}
\end{lem}
\begin{proof}
By approximate unitary equivalence,
we can take a unitary $u_X\in M^\omega$ with $X\in\Irr(\sC)$
so that
$u_X\alpha_X(\psi^\omega)u_X^*=\beta_X(\psi^\omega)$
for $\psi\in M_*$.
Next for a general $X\in\sC$,
we set
$u_X:=\sum_{Y\in\Irr(\sC)}
\sum_{T\in\ONB(Y,X)}T^\beta u_Y T^{\alpha*}$,
and we are done.
\end{proof}

Assuming the central freeness of $(\alpha,c^\alpha)$
and the amenability of $\sC$,
we will prove in Lemma \ref{lem:cocapprox}
that we can actually take $u_X$'s
so that they satisfy the conditions in the previous lemma
and moreover perturb $c^\alpha$ to $c^\beta$.

We will denote by $(\gamma,c^\gamma)$ the unitary perturbation
of $(\alpha,c^\alpha)$ by $u=(u_X)_X$
as in Lemma \ref{lem:appunitequiv}.
We know that $\gamma_X(\psi^\omega)=\beta_X(\psi^\omega)$
for all $X\in\sC$ and $\psi\in M_*$,
which implies $\gamma_X(x)=\beta_X(x)$ for all $x\in M$.

\begin{lem}
\label{lem:Rohlin-tensor-gamma}
Suppose that $(\alpha,c^\alpha)$ satisfies the assumption
of Theorem \ref{thm:Rohlin-tensor}.
Then the same statements of Theorem \ref{thm:Rohlin-tensor}
hold for $(\gamma,c^\gamma)$ in place of $(\alpha,c^\alpha)$
except the splitting condition (\ref{item:Roh-split}).
\end{lem}
\begin{proof}
We note that $\gamma=\theta^\alpha$ on $Z(M)$.
Let us assume that a countably generated von Neumann subalgebra
$Q\subset M^\omega$ contains all $\alpha_X(u_Y)$
with $X,Y\in\sC_0$.
Let $E^\alpha=(E_X^\alpha)_X$ be a Rohlin tower
as in Theorem \ref{thm:Rohlin-tensor}
for $\cF,\cK,\delta$ and $Q$ there.
We set $E^\gamma:=(E_X^\gamma)_X$
defined by $E_X^\gamma:=u_X E_X^\alpha u_X^*$
for $X\in\sC$.
Then the conditions Theorem \ref{thm:Rohlin-tensor}
(\ref{item:Roh-natural}), (\ref{item:Roh-support}), 
(\ref{item:Roh-commute}) are trivial.
On (\ref{item:Roh-orthogonal}),
we have the following for each $X\in\Irr(\sC)$:
\begin{align*}
d(X)^2
\ovl{R}_X^{\gamma*}c_{X,\ovl{X}}^{\gamma*}
E_X^\gamma c_{X,\ovl{X}}^{\gamma}\ovl{R}_X^{\gamma}
&=
d(X)^2
\ovl{R}_X^{\alpha*}
u_{X\otimes \ovl{X}}^*
c_{X,\ovl{X}}^{\gamma*}
u_X E_X^\alpha u_X^*
c_{X,\ovl{X}}^{\gamma}
u_{X\otimes\ovl{X}}
\ovl{R}_X^{\alpha}
\\
&=
d(X)^2
\ovl{R}_X^{\alpha*}
\alpha_X(u_{\ovl{X}}^*)
E_X^\alpha 
\alpha_X(u_{\ovl{X}})
\ovl{R}_X^{\alpha}
\\
&=
d(X)^2
\ovl{R}_X^{\alpha*}
E_X^\alpha 
\ovl{R}_X^{\alpha}
\quad
\mbox{since }
u_{\ovl{X}}\in Q.
\end{align*}
Thus (\ref{item:Roh-partition}) also holds.
We will check (\ref{item:Roh-equivariance}) as follows:
for $X,Y\in\Irr(\sC)$,
\begin{align*}
&
\gamma_X(E_Y^\gamma)
-
c_{X,Y}^\gamma E_{X\otimes Y}^\gamma c_{X,Y}^{\gamma*}
\\
&=
u_X\alpha_X(u_Y)\alpha_X(E_Y^\alpha)\alpha_X(u_Y^*)u_X^*
-
c_{X,Y}^\gamma
u_{X\otimes Y}
E_{X\otimes Y}^\alpha
u_{X\otimes Y}c_{X,Y}^{\gamma*}
\\
&=
u_X\alpha_X(u_Y)(\alpha_X(E_Y^\alpha)
-c_{X,Y}^\alpha E_{X\otimes Y}^\alpha c_{X,Y}^{\alpha*})
\alpha_X(u_Y^*)u_X^*,
\end{align*}
and
\[
\gamma_X(\gamma_Y(\varphi^\omega))
=
u_X\alpha_X(u_Y) \alpha_X(\alpha_Y(\varphi^\omega))\alpha_X(u_Y^*)u_X^*.
\]
Hence
\[
|\gamma_X(E_Y^\gamma)
-
c_{X,Y}^\gamma E_{X\otimes Y}^\gamma c_{X,Y}^{\gamma*}
|_{\gamma_X(\gamma_Y(\varphi^\omega))}
=
|
\alpha_X(E_Y^\alpha)
-c_{X,Y}^\alpha E_{X\otimes Y}^\alpha c_{X,Y}^{\alpha*}
|_{\alpha_X(\alpha_Y(\varphi^\omega))}.
\]
On (\ref{item:Roh-resonance}),
we have the following: for $X,Y\in\Irr(\sC)$,
\begin{align*}
E_X^\gamma c_{X,\ovl{X}}^\gamma\ovl{R}_X^\gamma E_Y^\gamma
&=
u_X E_X^\alpha \alpha_X(u_{\ovl{X}})
c_{X,\ovl{X}}^\alpha\ovl{R}_X^\alpha
u_Y E_Y^\alpha u_Y^*
\\
&=
u_X \alpha_X(u_{\ovl{X}}\alpha_{\ovl{X}}(u_Y))
E_X^\alpha 
c_{X,\ovl{X}}^\alpha\ovl{R}_X^\alpha
E_Y^\alpha u_Y^*
\quad
\mbox{since }
\alpha_{\ovl{X}}(u_Y)\in Q
\\
&=
u_X \alpha_X(u_{\ovl{X}}\alpha_{\ovl{X}}(u_Y))
\frac{\delta_{X,Y}}{d(X)}
\alpha_X(c_{\ovl{X},X}^\alpha R_X^\alpha)E_X^\alpha u_X^*
\\
&=
\frac{\delta_{X,Y}}{d(X)}
\gamma_X(c_{\ovl{X},X}^\gamma R_X^\gamma)E_X^\gamma.
\end{align*}
\end{proof}

In the following lemma,
let us keep the notation $(\gamma,c^\gamma)$ as above.

\begin{lem}
\label{lem:gaphiDe}
Suppose that $(\alpha,c^\alpha)$ satisfies the assumption
of Theorem \ref{thm:Rohlin-tensor}
and $\beta$ satisfies $\beta_X(M)'\cap M=Z(M)$
for all $X\in\Irr(\sC)$.
Let $\psi\in M_*$ be a faithful state
being $\alpha$-invariant on $Z(M)$.
Let $X,Y\in\Irr(\sC)$ and
$\Delta_{X,Y}\in M^\omega$
be a self-adjoint element
commuting with $\gamma_X(\gamma_Y(M_*))$.
Then for all $y\in M^\omega$,
one has
\[
|\gamma_X(\psi^\omega)
(y\Delta_{X,Y}\gamma_X(c_{Y,\ovl{Y}}^\gamma\ovl{R}_Y^\gamma))|
\leq
\|y\|
|\Delta_{X,Y}|_{\gamma_{X}(\gamma_Y(\psi^\omega))}.
\]
\end{lem}
\begin{proof}
We will simply write $c$ for $c^\gamma$.
Since $|\Delta_{X,Y}\gamma_X(c_{Y,\ovl{Y}}\ovl{R}_Y^\gamma)|$ commutes with
$\gamma_X(\psi^\omega)$,
we have
\[
|\gamma_X(\psi^\omega)(y\Delta_{X,Y}\gamma_X(c_{Y,\ovl{Y}}\ovl{R}_Y^\gamma))|
\leq
\|y\|
|\Delta_{X,Y}\gamma_X(c_{Y,\ovl{Y}}\ovl{R}_Y^\gamma)|_{\gamma_{X}(\psi^\omega)}.
\]
Hence we will estimate the right-hand side.
Let us consider the polar decomposition
$\Delta_{X,Y}\gamma_X(c_{Y,\ovl{Y}}\ovl{R}_Y^\gamma)
=w_{X,Y}|\Delta_{X,Y}\gamma_X(c_{Y,\ovl{Y}}\ovl{R}_Y^\gamma)|$.
Then the partial isometry $w_{X,Y}$ satisfies
\[
w_{X,Y}\gamma_X(x)
=\gamma_X(\gamma_Y(\gamma_{\ovl{Y}}(x)))w_{X,Y},
\quad
w_{X,Y}\gamma_X(\psi^\omega)
=\gamma_X(\gamma_Y(\gamma_{\ovl{Y}}(\psi^\omega)))w_{X,Y}
\]
for all $x\in M$ and $\psi\in M^\omega$.
Hence we have
\begin{align}
&|\Delta_{X,Y}\gamma_X(c_{Y,\ovl{Y}}\ovl{R}_Y^\gamma)
|_{\gamma_{X}(\psi^\omega)}
\notag
\\
&=
\gamma_{X}(\psi^\omega)(w_{X,Y}^*\Delta_{X,Y}
\gamma_X(c_{Y,\ovl{Y}}\ovl{R}_Y^\gamma))
\notag
\\
&=
\gamma_{X}(\psi^\omega)
(w_{X,Y}^*\Delta_{X,Y}
\gamma_X(c_{Y,\ovl{Y}}\ovl{R}_Y^\gamma\ovl{R}_Y^{\gamma*}c_{Y,\ovl{Y}}^*)
\gamma_X(c_{Y,\ovl{Y}}\ovl{R}_Y^\gamma))
\notag\\
&=
d(Y)^2
\gamma_X(\gamma_Y(\gamma_{\ovl{Y}}(\psi^\omega)))
(\gamma_X(c_{Y,\ovl{Y}}\ovl{R}_Y^\gamma)w_{X,Y}^*\Delta_{X,Y}
\gamma_X(c_{Y,\ovl{Y}}\ovl{R}_Y^\gamma\ovl{R}_Y^{\gamma*}c_{Y,\ovl{Y}}^*))
\notag\\
&\leq
d(Y)^2
\gamma_X(\gamma_Y(\gamma_{\ovl{Y}}(\psi^\omega)))
(\gamma_X(c_{Y,\ovl{Y}}\ovl{R}_Y^\gamma)w_{X,Y}^*|\Delta_{X,Y}|
w_{X,Y}\gamma_X(\ovl{R}_Y^{\gamma*}c_{Y,\ovl{Y}}^*))^{1/2}
\notag\\
&\quad
\cdot
\gamma_X(\gamma_Y(\gamma_{\ovl{Y}}(\psi^\omega)))
(\gamma_X(c_{Y,\ovl{Y}}\ovl{R}_Y^\gamma
\ovl{R}_Y^{\gamma*}c_{Y,\ovl{Y}}^*)|\Delta_{X,Y}|
\gamma_X(c_{Y,\ovl{Y}}\ovl{R}_Y^\gamma
\ovl{R}_Y^{\gamma*}c_{Y,\ovl{Y}}^*))^{1/2},
\label{eq:deltaxy}
\end{align}
where, in the last inequality,
we have used the self-adjointness of $\Delta_{X,Y}$
and the Cauchy--Schwarz inequality.
On the one hand,
we have
\begin{align*}
&\gamma_X(\gamma_Y(\gamma_{\ovl{Y}}(\psi^\omega)))
(\gamma_X(c_{Y,\ovl{Y}}\ovl{R}_Y^\gamma)w_{X,Y}^*|\Delta_{X,Y}|
w_{X,Y}\gamma_X(\ovl{R}_Y^{\gamma*}c_{Y,\ovl{Y}}^*))
\\
&=
\gamma_X(\gamma_Y(\gamma_{\ovl{Y}}(\psi^\omega)))
\bigl(w_{X,Y}^*\gamma_X(\gamma_Y(\gamma_{\ovl{Y}}(c_{Y,\ovl{Y}}
\ovl{R}_Y^\gamma)))
|\Delta_{X,Y}|
\gamma_X(\gamma_Y(\gamma_{\ovl{Y}}(\ovl{R}_Y^{\gamma*}c_{Y,\ovl{Y}}^*)))
w_{X,Y}\bigr)
\\
&=
\gamma_X(\gamma_Y(\gamma_{\ovl{Y}}
(\gamma_Y(\gamma_{\ovl{Y}}(\psi^\omega)))))
\bigl(w_{X,Y}w_{X,Y}^*
\gamma_X(\gamma_Y(\gamma_{\ovl{Y}}(c_{Y,\ovl{Y}}\ovl{R}_Y^\gamma)))
|\Delta_{X,Y}|
\\
&\quad
\cdot
\gamma_X(\gamma_Y(\gamma_{\ovl{Y}}(\ovl{R}_Y^{\gamma*}c_{Y,\ovl{Y}}^*))\bigr)
\\
&\leq
\gamma_X(\gamma_Y(\gamma_{\ovl{Y}}
(\gamma_Y(\gamma_{\ovl{Y}}(\psi^\omega)))))
\bigl(\gamma_X(\gamma_Y(\gamma_{\ovl{Y}}(c_{Y,\ovl{Y}}\ovl{R}_Y^\gamma)))
|\Delta_{X,Y}|
\gamma_X(\gamma_Y(\gamma_{\ovl{Y}}(\ovl{R}_Y^{\gamma*}c_{Y,\ovl{Y}}^*)))
\bigr)
\\
&=
\frac{1}{d(Y)^2}
\gamma_X(\gamma_Y(\gamma_{\ovl{Y}}(\psi^\omega)))
(\gamma_X(\gamma_Y(\gamma_{\ovl{Y}}(\ovl{R}_Y^{\gamma*}\ovl{R}_Y^\gamma)))
|\Delta_{X,Y}|)
\\
&=
\frac{1}{d(Y)^2}
|\Delta_{X,Y}|_{\gamma_X(\gamma_Y(\gamma_{\ovl{Y}}(\psi^\omega)))}
\\
&=
\frac{1}{d(Y)^2}
|\Delta_{X,Y}|_{\gamma_X(\gamma_Y(\psi^\omega))},
\end{align*}
where in the last equality,
we note that $\tau^\omega(\phi_Y^\gamma(\phi_{X}^\gamma(|\Delta_{X,Y}|)))$
is contained in $M'\cap M=Z(M)$
and $\psi$ is $\theta^\alpha$-invariant.
On the other hand, we have
\begin{align*}
&\gamma_X(\gamma_Y(\gamma_{\ovl{Y}}(\psi^\omega)))
(\gamma_X(c_{Y,\ovl{Y}}\ovl{R}_Y^\gamma\ovl{R}_Y^{\gamma*}
c_{Y,\ovl{Y}}^*)|\Delta_{X,Y}|
\gamma_X(c_{Y,\ovl{Y}}\ovl{R}_Y^\gamma\ovl{R}_Y^{\gamma*}c_{Y,\ovl{Y}}^*))
\\
&=
\gamma_X(\gamma_Y(\gamma_{\ovl{Y}}(\psi^\omega)))
(\gamma_X(c_{Y,\ovl{Y}}\ovl{R}_Y^\gamma\ovl{R}_Y^{\gamma*}c_{Y,\ovl{Y}}^*)|
\Delta_{X,Y}|)
\\
&=
\frac{1}{d(Y)^2}
\gamma_{X}(\psi^\omega)
(\gamma_X(\ovl{R}_Y^{\gamma*}c_{Y,\ovl{Y}}^*)
|\Delta_{X,Y}|\gamma_X(c_{Y,\ovl{Y}}\ovl{R}_Y^\gamma))
\\
&=
\frac{1}{d(Y)^2}
\psi^\omega
(\ovl{R}_Y^{\gamma*}c_{Y,\ovl{Y}}^*
\phi_X^\gamma(|\Delta_{X,Y}|)c_{Y,\ovl{Y}}\ovl{R}_Y^\gamma)
\\
&=
\frac{1}{d(Y)^2}
\psi^\omega
(\phi_X^\gamma(|\Delta_{X,Y}|))
\\
&=
\frac{1}{d(Y)^2}
|\Delta_{X,Y}|_{\gamma_X(\gamma_Y(\psi^\omega))},
\end{align*}
where we have used the fact that
$\tau^\omega(\phi_X^\gamma(|\Delta_{X,Y}|))$
is contained in $\beta_Y(M)'\cap M=Z(M)$
since $\gamma=\beta$ on $M$.
Thus we obtain
$(\ref{eq:deltaxy})
\leq |\Delta_{X,Y}|_{\gamma_X(\gamma_Y(\psi^\omega))}$.
\end{proof}

\subsection{Fixed point subalgebras of $M_\omega$}

By averaging technique over a large finite subset of $\Irr(\sC)$,
we can show the following result.

\begin{thm}
\label{thm:nontrivial}
Let $M$ be a properly infinite von Neumann algebra
with separable predual
such that $M_\omega$ has the continuous part.
Let $\al$ be a centrally free action of
an amenable rigid C$^*$-tensor category $\sC$
on $M$ such that
$M$ has a faithful normal state
being $\theta^\alpha$-invariant on $Z(M)$.
Then $M_\omega^\alpha$ is not trivial.
If, moreover, $M_\omega$ is of type II$_1$,
then so is $M_\omega^\alpha$.
\end{thm}
\begin{proof}
If $\Irr(\sC)=\{\btr\}$, then the statement is trivial.
Hence we assume $\sC$ is non-trivial. 
Let $M_{\rm d}$ and $M_{\rm c}$ be the atomic part
and the diffuse part of $M$, respectively.
When $M_{\rm d}$ is non-trivial,
then its central support projection
is invariant under $\theta^\alpha$.
Hence we may and do assume that $M$ is diffuse.
We further decompose $M$ as $M=M_0\oplus M_1$
so that $M_0$ is a direct sum of full factors
and $M_1$ does not contain a direct summand of a full factor.
Since $\theta^\alpha$ preserves $M_0$ and $M_1$,
we may and do assume that $M=M_1$.
Then
$M_\omega$ has a countably generated
diffuse von Neumann subalgebra $Q$.
Let $\varphi$ be
a faithful normal state on $M$
that is
$\theta^\alpha$-invariant on $Z(M)$.
Let $0<\delta<1/4$ and $\cF=\ovl{\cF}$
a finite subset in $\Irr(\sC)$.
Take an $(\cF,\delta)$-invariant finite subset $\cK$
in $\Irr(\sC)$.
Then we can obtain a Rohlin tower $E=(E_X)_X$ along with $\cK$
as in Theorem \ref{thm:Rohlin-tensor}
for $\cF,\delta,\cK$ and $Q$ as defined above.

We set $\pi(x)
:=\sum_{X\in\cK}
d(X)^2\ovl{R}_X^{\alpha*}
\alpha_X(x)E_X\ovl{R}_X^\alpha$
for $x\in Q$.
It turns out from Theorem \ref{thm:Rohlin-tensor}
and Remark \ref{rem:Rohlin-tensor} (4)
that
$\pi$ is a (not necessarily unital)
normal $*$-homomorphism from $Q$
into $M_\omega$.
Note that $\varphi^\omega(\pi(1))=b_E\geq1-\delta^{1/2}$.

Recall the following formula (\ref{eq:leftinv}):
for $X,Z\in\Irr(\sC)$,
we have
\begin{align*}
1_{\ovl{Z}}\otimes 1_X
&=
\sum_Y\sum_S
d(Y)^2(1_{\ovl{Z}}\otimes 1_X\otimes \ovl{R}_Y^*)
(1_{\ovl{Z}}\otimes S\otimes 1_{\ovl{Y}})
(R_Z\otimes 1_{\ovl{Y}})
\notag
\\
&
\hspace{80pt}\cdot
(R_Z^*\otimes 1_{\ovl{Y}})
(1_{\ovl{Z}}\otimes S^*\otimes 1_{\ovl{Y}})
(1_{\ovl{Z}}\otimes 1_X\otimes \ovl{R}_Y),
\end{align*}
where the summation is taken for $Y\in\Irr(\sC)$
and $S\in\ONB(Z,X\otimes Y)$.
In $M$,
we obtain
\[
1=
\sum_Y\sum_S
d(Y)^2
\alpha_{\ovl{Z}}(\alpha_X(\ovl{R}_Y^{\alpha*})S^\alpha)
R_Z^\alpha R_Z^{\alpha*}
\alpha_{\ovl{Z}}(S^{\alpha*}\alpha_X(\ovl{R}_Y^\alpha)).
\]
Multiplying $\alpha_Z$ and $\alpha_Z(x)E_Z$ with $x\in Q$
to the both sides,
we obtain
\begin{align*}
\pi(x)
&
=\sum_{Z\in\Irr(\sC)}
d(Z)^2
\ovl{R}_Z^{\alpha*}
\alpha_Z(x)E_Z
\ovl{R}_Z^\alpha
\\
&=
\sum_{Y,Z}
\sum_{S\in\ONB(Z,X\otimes Y)}
d(Y)^2
d(Z)^2
\ovl{R}_Z^{\alpha*}
\alpha_Z(x)E_Z
\\
&\hspace{120pt}
\cdot
\alpha_{Z\otimes\ovl{Z}}(\alpha_X(\ovl{R}_Y^{\alpha*})S^\alpha)
\alpha_Z(R_ZR_Z^*)
\alpha_{Z\otimes\ovl{Z}}(S^{\alpha*}\alpha_X(\ovl{R}_Y^\alpha))
\ovl{R}_Z^\alpha
\\
&=
\sum_{Y,Z,S}
d(Y)^2
d(Z)^2
\alpha_X(\ovl{R}_Y^{\alpha*})S^\alpha\cdot
\ovl{R}_Z^{\alpha*}
\alpha_Z(x)E_Z
\alpha_Z(R_Z^\alpha R_Z^{\alpha*})
\ovl{R}_Z^\alpha\cdot
S^{\alpha*}\alpha_X(\ovl{R}_Y^\alpha)
\\
&=
\sum_{Y,Z,S}
d(Y)^2
d(Z)^2
\alpha_X(\ovl{R}_Y^{\alpha*})S^\alpha\cdot
\ovl{R}_Z^{\alpha*}\alpha_Z(R_Z^\alpha)
\alpha_Z(x)E_Z
\alpha_Z(R_Z^{\alpha*})
\ovl{R}_Z^\alpha
S^{\alpha*}\alpha_X(\ovl{R}_Y^\alpha)
\\
&=
\sum_{Y,Z,S}
d(Y)^2
\alpha_X(\ovl{R}_Y^{\alpha*})S^\alpha
\alpha_Z(x)E_Z
S^{\alpha*}\alpha_X(\ovl{R}_Y^\alpha)
\\
&=
\sum_{Y}
d(Y)^2
\alpha_X(\ovl{R}_Y^{\alpha*})
\alpha_{X\otimes Y}(x)E_{X\otimes Y}
\alpha_X(\ovl{R}_Y^\alpha),
\end{align*}
where the summations are taken for $Y,Z\in\Irr(\sC)$
and $S\in\ONB(Z,X\otimes Y)$.
Hence we have the following equality
for all $X\in\sC$ and $x\in Q$:
\begin{align*}
\alpha_X(\pi(x))-\pi(x)
&=
\sum_{Y\in\Irr(\sC)}
d(Y)^2
\alpha_X(\ovl{R}_Y^{\alpha*})
\alpha_{X\otimes Y}(x)
(\alpha_X(E_Y)-E_{X\otimes Y})
\alpha_X(\ovl{R}_Y^\alpha).
\end{align*}
By Lemma \ref{lem:gaphiDe} applied to $(\gamma,c)=(\alpha,1)$
and $\Delta_{X,Y}=\alpha_X(E_Y)-E_{X\otimes Y}$,
we obtain
\[
\sum_{X\in\cF}
d(X)^2|\alpha_X(\pi(x))-\pi(x)|_{\alpha_X(\varphi^\omega)}
\notag
\leq
6\delta^{1/2}|\cF|_\sigma
\|x\|
\quad
\mbox{for all }
x\in Q.
\]

Let $\cF_n$, $n\geq1$
be a sequence of increasing finite subsets of $\Irr(\sC)$
which covers $\Irr(\sC)$.
Set $\delta_n:=6^{-2}n^{-2}|\cF_n|_\sigma^{-2}<1/4$.
Then we can take a normal $*$-homomorphism
$\pi_n\colon Q\to M_\omega$,
which is not necessarily unital,
as in the preceding discussion
for each $\cF_n$ and $\delta_n$,
that is,
\[
\sum_{X\in\cF_n}d(X)^2
|\alpha_X(\pi_n(x))-\pi_n(x)|_{\alpha_X(\varphi^\omega)}
\leq
6\delta_n^{1/2}|\cF_n|_\sigma
\|x\|
=\|x\|/n
\quad
\mbox{for all }
x\in Q.
\]

Let $Q_0$ be a $\sigma$-weakly dense separable C$^*$-subalgebra
of $Q$.
Let $C$ be the separable $\alpha$-invariant
C$^*$-subalgebra in $\ell^\infty(M_\omega)$
generated by
$(\pi_n(x))_n$
for all $x\in Q_0$.
Applying the index selection for $C$
(see \cite[Lemma 3.11]{MT-minimal} and \cite[Lemma 5.5]{Oc}),
we obtain a C$^*$-homomorphism $\pi\colon Q_0\to M_\omega$
such that
$\alpha_X(\pi(x))=\pi(x)$ for all $X\in\Irr(\sC)$.
This equality actually holds for all $X\in\sC$
since $\pi(x)\in M_\omega$.
Since $\varphi^\omega(\pi_n(1))\geq1-\delta_n^{1/2}$,
we see $\pi$ is unital.
When $M_\omega$ is of type II$_1$,
we may and do assume that $Q_0$ contains a 2-by-2 matrix units
$e_{ij}$, $i,j=1,2$
with $e_{11}+e_{22}=1$.
Thus $\pi(e_{ij})$ is non-zero for all $i,j$,
and we see $M_\omega^\alpha$ is of type II$_1$.

We will check the non-triviality of $\pi(Q_0)$ in a general situation.
We may assume that
$Q_0$ has non-zero projections $e_n$, $n\in\N$
with
$\varphi^\omega(\pi_n(e_n))
=\varphi^\omega(\pi_n(1))/2$
since $Q$ is diffuse.
We let $C_1$ be the C$^*$-subalgebra generated by $C$
and $(\pi_n(e_n))_n$ in $\ell^\infty(M_\omega)$.
By index selection,
we get a projection $p\in M_\omega^\alpha$
with $\varphi^\omega(p)=1/2$.
\end{proof}

\begin{rem}
In the statement of Theorem \ref{thm:nontrivial},
any assumption of $\alpha$ might not be necessary
when $M$ is diffuse and amenable
though it has not been solved yet.
Namely,
let $\alpha$ be any action of an amenable rigid C$^*$-tensor category
$\sC$ on a diffuse amenable von Neumann algebra $M$.
Then is $M_\omega^\alpha$ non-trivial?
Readers are referred to \cite{Mar2} for the related topics
for amenable discrete group actions.
\end{rem}

\section{Discrete irreducible subfactors}
\label{sect:discrete}
In this section,
we will study a discrete inclusion $N\subset M$ with separable preduals
through a centrally free action of a C$^*$-tensor category.
In particular, we investigate the hereditary property of fullness
from $M$ to $N$.

\subsection{Discrete inclusions}
We will quickly review the notion of a discrete inclusion.
We freely use the notations and the terminology
introduced in \cite{ILP} and \cite[Section 2]{T-Gal}.
Let $N\subset M$ be an inclusion of factors
with $N'\cap M=\C$ and a faithful normal
conditional expectation $E$ from $M$ onto $N$.
The inclusion $N\subset M$ is said to be
\emph{discrete}
if the dual operator valued weight
$\widehat{E}$ from $M_1$ to $M$ is semifinite
on $N'\cap M_1$.
Then $N'\cap M_1$ is the direct sum of matrix algebras
$A_\xi$ as follows \cite[Proposition 2.8]{ILP}:
\[
N'\cap M_1\cong \bigoplus_{\xi\in\Xi}A_\xi,
\]
where $A_\xi$ is a type I$_{n_\xi}$ factor
with $n_\xi\in\N$.

Suppose $N$ is an infinite factor.
For each $\xi\in\Xi$,
we take a minimal projection $e_\xi\in A_\xi$.
Let $\rho_\xi$ be an irreducible endomorphism on $N$
that is associated
to the irreducible $N$-$N$-bimodule ${}_N e_\xi L^2(M)_N$.

Let $\cH_\xi$ be the subspace of $M$
defined as follows:
\[
\cH_\xi:=\{V\in M\mid Vx=\rho_\xi(x)V\mbox{ for all }x\in N\}.
\]
Then $\cH_\xi$ is a Hilbert space in $M$ of dimension $n_\xi$.
By \cite[Theorem 3.3]{ILP},
we know that $A_\xi=\cH_\xi^*e_N\cH_\xi$,
where $e_N$ denotes the Jones projection of $N\subset M$.

We now let $\sC$ be the rigid C$^*$-tensor category
that is the full subcategory
generated by all $\rho_\xi$'s in $\End(N)_0$.
Thus $\Xi$ is a subset of $\Irr(\sC)$.
Let us regard the fully faithful embedding
$\rho$ from $\sC$ into $\End(N)_0$
as an action of $\sC$ on $N$.

\begin{thm}
\label{thm:discretefull}
Let $N\subset M$ be a discrete inclusion of factors
with separable preduals
such that $\sC$ is amenable
and the action of $\sC$ on $N$ is centrally free.
If $M$ is a full factor,
then $N$ actually equals $M$.
\end{thm}
\begin{proof}
If $N_\omega$ were non-trivial,
it would follow from Theorem \ref{thm:nontrivial}
that $N_\omega^\rho$ were non-trivial.
Since $M$ is the $\sigma$-weak closure of
the linear span of $N\cH_\xi$, $\xi\in\Xi$,
we see $N_\omega^\rho\subset M_\omega$.
This is a contradiction.
Hence $N$ is a full factor.
From the central freeness of the action $\sC$,
it turns out that $\sC$ is trivial, that is, $N=M$.
\end{proof}

It has been not clear yet
if a similar result to the above also holds
when
the action of $\sC$ is only assumed to be free.
Let us consider this situation for a while.
Let $\Xi_0$ be the subset of $\Xi$
which consists of all centrally trivial $\rho_\xi$'s.
We set the intermediate subfactor $M_{\rm cnt}$
of $N\subset M$ that is $\sigma$-weakly spanned by
$N$ and $N\cH_\xi$, $\xi\in\Xi_0$.
Thanks to \cite[Lemma 3.8]{ILP},
we have a faithful
normal conditional expectation
from $M$ onto $M_{\rm cnt}$.
Let us call $M_{\rm cnt}$
the \emph{centrally trivial part}
of the discrete inclusion $N\subset M$.

\begin{prop}
\label{prop:cntpart}
Suppose
the inclusion $M_{\rm cnt}\subset M$ is discrete.
Let $\{\sigma_\lambda\}_{\lambda\in\Lambda}$
be the system of the irreducible endomorphisms on $M_{\rm cnt}$
associated with
the $M_{\rm cnt}$-$M_{\rm cnt}$-subbimodules of $L^2(M)$.
Then each $\sigma_\lambda$ is centrally non-trivial
if $\sigma_\lambda\neq\id$.
\end{prop}
\begin{proof}
Put $L:=M_{\rm cnt}$.
Note that $N_\omega'\cap M=L$.
The inclusion $L\subset N_\omega'\cap M$ is trivial
since each $\rho_\xi$, $\xi\in\Xi_0$, is centrally trivial.
We will show the converse.
Take $x\in N_\omega'\cap M$.
Suppose for some $\xi\in \Xi$, we have $E_N^M(x\cH_\xi^*)\neq\{0\}$.
Take $V\in\cH_\xi$ so that $a:=E_N^M(xV^*)\neq0$.
Then for any $y\in N_\omega$,
we have
$ya=a\rho_\xi(y)$.
This implies $\rho_\xi$ is centrally trivial
(see Remark \ref{rem:cent-trivial}),
and $\xi\in\Xi_0$.
Hence $x\in L$.
The equality $N_\omega'\cap M=L$ implies $L_\omega'\cap M=L$.

Let $\mathcal{K}_\lambda$ be the Hilbert space
in $M$ which implements $\sigma_\lambda$.
If $\sigma_\lambda$ is centrally trivial
for some $\lambda$,
then $\mathcal{K}_\lambda\subset L_\omega'\cap M=L$.
Thus $\sigma_\lambda$ must be the identity map.
\end{proof}

Assume that $M_{\rm cnt}\subset M$ is a discrete inclusion.
The action of $\sD$, the associated rigid C$^*$-tensor category
of $M_{\rm cnt}\subset M$, is not necessarily centrally free.
Assuming its amenability and central freeness,
which would be possibly superfluous,
we immediately obtain the following result
from Theorem \ref{thm:discretefull} and Proposition \ref{prop:cntpart}.

\begin{cor}
\label{cor:cntpartfull}
Let $M$ be a full factor with separable predual
and $N\subset M_{\rm cnt}\subset M$
as above.
If the inclusion $M_{\rm cnt}\subset M$ is discrete
and the associated rigid C$^*$-tensor category $\sD$ is amenable
and gives a centrally free action on $M_{\rm cnt}$,
then $\sD$ is trivial, that is, $\Irr(\sD)=\{\btr\}$.
In particular, $M_{\rm cnt}=M$ and $N$ is also a full factor.
\end{cor}

\subsection{Minimal actions}

For general theory of (locally) compact quantum groups,
readers are referred to \cite{KuVa,NeTu,T-Gal,W}.
Let $G$ be a compact quantum group
with von Neumann algebra $L^\infty(G)$
and its coproduct
$\delta\colon L^\infty(G)\to L^\infty(G)\otimes L^\infty(G)$.
An \emph{action}
of $G$ on a von Neumann algebra $M$
means a unital faithful normal $*$-homomorphism
$\alpha\colon M\to M\otimes L^\infty(G)$
such that $(\alpha\otimes\id)\circ\alpha=(\id\otimes\delta)\circ\alpha$.
(Let us denote by $\alpha$ an action of $G$ for a moment.)
An action is said to be \emph{faithful}
when $\set{(\phi\otimes\id)(\alpha(M))|\phi\in M_*}$
generates the whole von Neumann algebra $L^\infty(G)$.
We will say $\alpha$ is \emph{minimal}
if $\alpha$ is faithful and $(M^\alpha)'\cap M=\C$.
It is known that $\alpha$ is minimal
if and only if $M'\cap (M\rtimes_\alpha G)=\C$.

It is known that if $\alpha$ is minimal,
then the inclusion $N:=M^\alpha\subset M$ is discrete
(see \cite{ILP,T-Gal}).
We assume $N$ is infinite.
Each irreducible endomorphism of $N$
associated with this inclusion
corresponds to the dual action
$\widehat{\alpha}_U\colon M\rtimes_\alpha G
\to
(M\rtimes_\alpha G)\otimes L^\infty(\widehat{G})$,
where $U$ denotes an irreducible unitary
representation of $G$
and $\widehat{G}$ the dual discrete quantum group.
By $\Irr(G)$,
we denote a complete set of irreducible unitary representations
of $G$.
We suppose $\Irr(G)$ is at most countable.

Let $\sC$ be the C$^*$-tensor category associated with
$N\subset M$ as before.
Then the central freeness of the action of $\sC$ on $N$
is equivalent to that of
the action $\widehat{\alpha}$
of the discrete quantum group $\widehat{G}$
in the sense of \cite[Section 8]{MT-minimal}.
Let us give a brief proof for readers' convenience
as follows.

Let $\alpha$ be a minimal
action of $G$ on a factor $M$ as before.
We suppose that $M^\alpha$ is infinite.
Then $\alpha$ is dual,
that is,
there exist an action
$\beta\colon M^\alpha\to M^\alpha
\otimes L^\infty(\widehat{G})$
and an isomorphism
$\Psi\colon M^\alpha\rtimes_\beta\widehat{G}\to M$
such that
$\Psi(\pi_\beta(x))=x$
for $x\in M^\alpha$
and $(\Psi\otimes\id)\circ\widehat{\beta}=\alpha\circ\Psi$,
where $\beta$ denotes the dual action of $G$
on $M^\alpha\rtimes_\beta\widehat{G}$.
By duality theory,
we know that $\widehat{\alpha}$ is cocycle conjugate
to $\beta$.

Note the crossed product $M^\alpha\rtimes_\beta\widehat{G}$
is generated by the canonical copy $\pi_\beta(M^\alpha)$ of $M^\alpha$
and $\C\otimes L^\infty(G)$.
We set
$\lambda_U^\alpha
:=1\otimes w_U\in (M^\alpha\rtimes_\beta \widehat{G})\otimes B(H_U)$
for $U\in \Irr(G)$,
where $w_U\in C(G)\otimes B(H_U)$ be an irreducible unitary representation
of $G$.
Then we have the covariance relation
$\lambda_U^\beta(\pi_\beta(x)\otimes1)
=(\pi_\beta\otimes\id)(\beta(x))\lambda_U^\beta$.

Let $\{\nu_k^U\}_{k\in I_U}$ be a system of isometries in $M^\alpha$
with $\sum_k \nu_k^U(\nu_k^U)^*=1$.
We set $V_j^U:=\sum_k \pi_\beta(\nu_k^U) (\lambda_U^\beta)_{kj}$
for $j\in I_U$.
Then $V_j^U$'s are isometries such that
$\sum_j V_j^U(V_j^U)^*=1$
and
$\widehat{\beta}(V_j^U)=\sum_k V_k^U\otimes (w_U)_{kj}$.
Hence the endomorphism $\rho_U$ on $M^\alpha$
associated with $M^\alpha\subset M$
is given by
$\rho_U(x)=\sum_{j}\Psi(V_j^U)x\Psi(V_j^U)^*$
for $x\in M^\alpha$.
Then for $x\in M^\alpha$,
we have
\begin{align*}
\sum_j V_j^U \pi_\beta(x)(V_j^U)^*
&=
\sum_{j,k}
\pi_\beta(\nu_k^U)(\lambda_U^\beta)_{kj}
\pi_\beta(x)(V_j^U)^*
\\
&=
\sum_{j,k,\ell}
\pi_\beta(\nu_k^U)
\pi_\beta(\beta_U(x)_{k\ell})
(\lambda_U^\beta)_{\ell j}
(V_j^U)^*
\\
&=
\sum_{k,\ell}
\pi_\beta(\nu_k^U
\beta_U(x)_{k\ell}
\nu_\ell^U)^*.
\end{align*}
Hence we obtain
$\rho_U(x)=\sum_{k,\ell}
\nu_k^U
\beta_U(x)_{k\ell}
(\nu_\ell^U)^*$
for $x\in M^\alpha$.
This means $\rho_U=\pi_U\circ\beta_U$,
where $\pi_U$ is the isomorphism from $M^\alpha \otimes B(H_U)$
onto $M^\alpha$
satisfying $\pi_U(x\otimes e_{k\ell})
=\nu_k^U x(\nu_{\ell}^U)^*$
for $x\in M^\alpha$.
Hence $\widehat{\alpha}$ is centrally free if and only
if $\rho$ gives a centrally free action.

We recall the coamenability of a compact quantum group $G$.
See \cite{BeCoTu,BeMuTu1,BeMuTu2,NeTu,T-am}
for many equivalent conditions.
We will say that $G$ is \emph{coamenable}
when the reduced C$^*$-algebra $C(G)$ has the norm bounded counit.
This is equivalent to the amenability of $\hat{G}$,
that is, the existence of a left invariant mean.
Note the amenability of the representation category of $G$
implies $G$ is of Kac type.

\begin{thm}
Let $G$ be a coamenable compact quantum group of Kac type
and $\alpha$ a minimal action on a full factor $M$
with separable predual
such that $\hat{\alpha}$ is centrally free.
Then $G$ is actually trivial.
\end{thm}
\begin{proof}
The assumption that $G$ is coamenable and of Kac type
implies the amenability of its representation category
(see \cite[Chapter 2.7]{NeTu} for example).
It follows from Theorem \ref{thm:discretefull}
that $M^\alpha=M$, which implies $G=\{\btr\}$.
\end{proof}

Let us restate the same result in terms of
the dual discrete quantum group of Kac type.
Note that existence of a centrally free action
of a non-trivial discrete quantum group on a von Neumann algebra $M$
implies that $M_\omega$ is not one-dimensional,
that is, $M$ is not a full factor.

\begin{thm}\label{thm:discKacfull}
Let $\alpha$ be an action
of a non-trivial
amenable discrete quantum group $\hG$ of Kac type
on a factor $M$ with separable predual.
If $\alpha$ is centrally free,
then the crossed product $M\rtimes_\alpha\hG$ is never full.
\end{thm}

Although we have not succeeded in removing the central freeness
from the assumptions of the results stated above,
we can for compact groups as discussed below.
Some parts of the proof of the following lemma
holds for a more general $\hG$,
but we only treat $G$ being a compact group.

Let us fix our notation.
Put $L^\infty(\hG):=R(G)$, the group von Neumann algebra
associated with the right regular representation
$\rho(s)$, $s\in G$ as usual.
The coproduct
$\Delta\colon L^\infty(\hG)
\to L^\infty(\hG)\otimes L^\infty(\hG)$
is defined so that $\Delta(\rho(s))=\rho(s)\otimes \rho(s)$
for $s\in G$.
Let $\alpha$ be an action of $\hG$ on a von Neumann algebra $M$.
Namely, $\alpha$ is a unital faithful normal $*$-homomorphism
from $M$ into $M\otimes L^\infty(\hG)$
such that $(\alpha\otimes\id)\circ\alpha=(\id\otimes\Delta)\circ\alpha$.

Let $\Irr(G)$ be a complete set of irreducible unitary representations
of $G$.
For each $U\in\Irr(G)$,
we fix the corresponding unitary representation
$w_U\colon G\to B(H_U)$.
Regard $w_U$ as an element of $L^\infty(G)\otimes B(H_U)$.
From Peter--Weyl theory,
$L^\infty(\hG)$ is isomorphic to the direct sum of
$B(H_U)$ for $U\in\Irr(G)$.
We denote by $\alpha_U$ the composition of $\alpha$
and the projection map $M\otimes L^\infty(\hG)$
onto $M\otimes B(H_U)$.

The crossed product von Neumann algebra $M\rtimes_\alpha\hG$
is the von Neumann subalgebra of $M\otimes B(L^2(G))$
that is generated by
$\pi_\alpha(x):=\alpha(x)\in M\otimes L^\infty(\hG)$
for $x\in M$
and $\C\otimes L^\infty(G)$.
Set $\lambda_U^\alpha:=1_M\otimes w_U$
that is an operator in $(M\rtimes_\alpha\hG)\otimes L^\infty(\hG)$.
Then we have the covariance relation
$\lambda_U^\alpha(\pi_\alpha(x)\otimes 1_U)
=(\pi_\alpha\otimes\id_U)(\alpha_U(x))\lambda_U^\alpha$
for $U\in\Irr(G)$ and $x\in M$.
The dual action $\hal$ of $G$ on $M\rtimes_\alpha G$
has the following properties:
for $s\in G$, $x\in M$ and $U\in\Irr(G)$,
\[
\hal_s(\pi_\alpha(x))=\pi_\alpha(x),
\quad
(\hal_s\otimes\id)(\lambda_U^\alpha)=\lambda_U^\alpha(1\otimes w_U(s)).
\]

\begin{lem}
Let $G$ be a compact group
and $\alpha$ a free action of $\hG$
on a von Neumann algebra $M$.
Suppose that
the associated action $\theta^\alpha$
of $\hG$ on $Z(M)$ is strongly ergodic
and $Z(M)\neq\C$.
Then
there exist a central element $s\in G\setminus\{e\}$
and a unitary $u\in Z(M)$
such that $\hal_s=\Ad \pi_\alpha(u)$ on $M\rtimes_\alpha\hG$.
\end{lem}
\begin{proof}
Let $\theta^\alpha\colon \Irr(G)\to \Aut(Z(M))$ be the associated
map with $\alpha$,
that is,
$\alpha_U(x)=\theta_U^\alpha(x)\otimes 1_U$
for all $U\in \Irr(G)$ and $x\in Z(M)$
(see \cite[Lemma 2.9]{MT-minimal}).
Let us introduce the equivalence relation on $\Irr(G)$
such that
$U\sim V$ if and only if
$\theta_U^\alpha=\theta_V^\alpha$.
Then the quotient set $\Gamma:=\Irr(G)/{\sim}$ is naturally identified
with the subgroup
$\{\theta_U^\alpha\mid U\in\Irr(G)\}$ in $\Aut(Z(M))$
that is abelian because the fusion rule of $G$ is commutative.
Hence we regard $\Gamma$ as a discrete abelian group.
Let us denote the equivalence class of $U\in\Irr(G)$
by $[U]\in\Gamma$.

The strong ergodicity
implies $Z(M)$ is discrete
(see \cite[Theorem 4.75]{KL} and \cite[Theorem 2.4]{Sch}).
Hence we can find a subgroup $\Gamma_1$ of $\Gamma$
such that $\Gamma_1\neq\Gamma$
and $\theta^\alpha$ is equivariantly isomorphic
to the left transformation
on $\ell^\infty(\Gamma/\Gamma_1)$.

Let $\chi$ be an arbitrary non-zero element
in the Pontryagin dual of $\Gamma/\Gamma_1$.
Let $u_\chi$ be the corresponding unitary in $Z(M)$.
Through the above equivariant identification,
we obtain
$\theta_U^\alpha(u_\chi)=\ovl{\langle \chi,[U]+\Gamma_1\rangle} u_\chi$
for $U\in\Irr(G)$.
Hence we have
$(\Ad \pi_\alpha(u_\chi)\otimes\id)(\lambda_U^\alpha)
=
\langle \chi,[U]+\Gamma_1\rangle
\lambda_U^\alpha
$.

The assignment
$\Irr(G)\ni U\mapsto \langle \chi,[U]+\Gamma_1\rangle1_U\in B(H_U)$
defines a group-like element in $L^\infty(\hG)$,
and
we get a central element $s_\chi\in G$
such that $w_U(s_\chi)=\langle \chi,[U]+\Gamma_1\rangle 1_U$
for all $U\in\Irr(G)$ and $t\in G$.
Then we have $\widehat{\alpha}_{s_\chi}=\Ad \pi_\alpha(u_\chi)$.
\end{proof}

We will prove the enhanced version of Theorem \ref{thm:discKacfull}
for a compact group.

\begin{thm}
\label{thm:enhance}
Let $G$ be a compact group
and $\alpha$ a free action of $\hG$
on a von Neumann algebra $M$ with separable predual.
Suppose that
the crossed product $M\rtimes_\alpha\hG$ is a full factor
and the dual action $\hal$ of $G$ is pointwise outer.
Then $M$ is a full factor.
\end{thm}
\begin{proof}
By considering $M\otimes B(\ell^2)$ if necessary,
we may and do assume that $M$
is properly infinite.
When $G$ is a trivial group,
we have nothing to prove.
We assume that $G$ is non-trivial.

The fullness of $M\rtimes_\alpha\hG$ implies the strong ergodicity
of $\alpha$ on $Z(M)$,
and the factoriality of $M$ follows from the previous lemma
and the pointwise outerness of $\hal$.
Let $\Lambda_\alpha$ be the collection of $U\in\Irr(G)$
such that $\alpha_U$ is centrally trivial.
Let $K_\alpha$ be the collection of $s\in G$ such that
$w_U(s)=1_U$ for all $U\in\Lambda_\alpha$.
Then $K_\alpha$ is a closed normal subgroup of $G$
such that $C(G/K_\alpha)$ is equivariantly identified with
the C$^*$-subalgebra generated by the matrix elements
of $\{w_U\}_{U\in \Lambda_\alpha}$
inside $C(G)$.

Set $N:=M\rtimes_\alpha\hG$.
Then we obtain the inclusions $M\subset N^{K_\alpha}\subset N$,
where $N^{K_\alpha}$ denotes the $K_\alpha$-fixed point subalgebra
of $N$ by the restriction of $\hal$.

Note that $\hal$ is a minimal action of $G$,
and the inclusion $R:=N^{K_\alpha}\subset N$ is given
by the minimal action $\hal\!\upharpoonright_{K_\alpha}$
on $N$.
Then $R$ is generated by $M$ and the matrix elements
of $\lambda_U^\alpha$ for $U\in\Lambda_\alpha$,
which means that $R$ is the centrally trivial part of $M\subset N$.

Let $\{\sigma_j\}_j$ be the system of the irreducible endomorphisms on $R$
that is associated with $R\subset N$.
It turns out from Proposition \ref{prop:cntpart}
that each non-identity $\sigma_j$ is centrally non-trivial.
Furthermore,
the rigid C$^*$-tensor category associated with the inclusion
$R\subset N$ is nothing but the representation category of
$K_\alpha$,
and each simple object is of the form $\sigma_j$ for some $j$.
Therefore,
the action of the representation category of $K_\alpha$ on $R$
is centrally free.
Since any compact group is coamenable,
the representation category is amenable.
Hence it follows from Corollary \ref{cor:cntpartfull}
that $M$ is a full factor.
\end{proof}

The following result
is a kind of a generalization of \cite[Proposition 4]{TU}
to a non-commutative compact group.

\begin{thm}\label{thm:minimal}
Let $\alpha$ be a minimal action
of a compact group $G$ on a full factor $M$ with separable predual.
Then the both factors
$M\rtimes_\alpha G$ and $M^\alpha$ are full.
\end{thm}
\begin{proof}
We put $N:=M\rtimes_\alpha G$.
Applying the previous result to the dual action $\widehat{\alpha}$ on $N$,
we see that $N$ is a full factor
from the duality
$N\rtimes_{\widehat{\alpha}}\hG\cong M\otimes B(L^2(G))$.
Since $M^\alpha$ is a corner von Neumann subalgebra of $N$,
$M^\alpha$ is also full.
\end{proof}

It is not known if we can replace the minimality
with the pointwise outerness
as in the following problem asked by Marrakchi:

\begin{prob}
\label{prob:outerfull}
Let $\alpha$ be an action of a locally compact group
$G$ on a full factor $M$.
Then is $M\rtimes_\alpha G$ a full factor
if $G$ is homeomorphic to the quotient image of
$\alpha(G)$ in $\Out(M)$?
\end{prob}

It is not even known whether
the pointwise outerness would imply
the minimality for a compact group in the problem above.

\section{Classification of centrally free cocycle actions of amenable rigid C$^*$-tensor categories}
\label{sect:classification}
In this section,
we will apply the Rohlin tower construction (Theorem \ref{thm:Rohlin-tensor})
to classification of cocycle actions of amenable rigid C$^*$-tensor categories.

\subsection{Approximation of cocycle actions by cocycle perturbations in $M^\omega$}
\label{subsect:approx}
We will state our main result of this subsection
as follows.

\begin{lem}
\label{lem:cocapprox}
Let $(\alpha,c^\alpha)$
and $(\beta,c^\beta)$
be centrally free cocycle action of an amenable rigid C$^*$-tensor category
$\sC$ on a properly infinite von Neumann algebra $M$ with separable predual.
Suppose that
$\alpha_X$ and $\beta_X$ are approximately unitarily equivalent
for all $X\in\sC$
and there exists a $\theta^\alpha$-invariant
faithful normal state on $Z(M)$.
Then one can take unitaries $\nu_X$ in $M^\omega$ with $X\in\sC$
so that the following equalities hold
for all $X,Y\in\sC$:
\begin{itemize}
\item
$\nu_X\alpha_X(\psi^\omega)\nu_X^*=\beta_X(\psi^\omega)$
for all $\psi\in M_*$.

\item
$\nu_X\alpha_X(\nu_Y)c_{X,Y}^\alpha \nu_{X\otimes Y}^*
=c_{X,Y}^\beta$.

\item
$\nu_YT^\alpha=T^\beta \nu_X$ for all $T\in\sC(X,Y)$.
\end{itemize}
\end{lem}

Let $(\alpha,c^\alpha)$ and $(\beta,c^\beta)$
be actions
of $\sC$ on a von Neumann algebra $M$
as in the statement of the previous theorem.
Note by approximate unitary equivalence,
we have $\theta^\alpha=\theta^\beta$
on $Z(M)$.
We recall the contents of Section \ref{subsect:slight}.
By Lemma \ref{lem:appunitequiv},
we can take a unitary $u_X\in M^\omega$ for each $X\in\sC$
such that
\begin{itemize}
\item
$u_X\alpha_X(\psi^\omega)u_X^*=\beta_X(\psi^\omega)$
for all $X\in\sC$ and $\psi\in M_*$.

\item
$u_Y T^\alpha=T^\beta u_X$
for all $X,Y\in\sC$ and $T\in\sC(X,Y)$.
\end{itemize}

We set $\gamma_X:=\Ad u_X\circ\alpha_X$
and $c_{X,Y}:=u_X\alpha_X(u_Y)c_{X,Y}^\alpha u_{X\otimes Y}^*$
for $X,Y\in\sC$.
Then $(\gamma,c)$ is a semiliftable centrally free cocycle action
of $\sC$ on $M^\omega$ with $T^\gamma:=u_Y T^\alpha u_X^*=T^\beta$
for all $X,Y\in\sC$ and $T\in\sC(X,Y)$.
Note that
$\gamma_X(\psi^\omega)=\beta_X(\psi^\omega)$
and $\gamma_X(x)=\beta_X(x)$
for all $\psi\in M_*$ and $x\in M$. 
We will denote by $d_{X,Y}$ the 2-cocycle $c_{X,Y}^\beta$
for simplicity.
By definition, we also have
\begin{align}
&d_{X,Y}c_{X,Y}^*\gamma_X(\gamma_Y(\psi^\omega))
=\beta_X(\beta_Y(\psi^\omega))d_{X,Y}c_{X,Y}^*,
\label{eq:dcgamgampsi}
\\
&d_{X,Y}c_{X,Y}^*\gamma_X(T^\gamma)
=
\beta_X(T^\beta)d_{X,Z}c_{X,Z}^*.
\label{eq:dcgamT}
\end{align}
for all $X,Y,Z\in\sC$, $T\in\sC(Z,Y)$ and $\psi\in M_*$,
where we actually have $\gamma_X(T^\gamma)=\beta_X(T^\beta)$
and
$\gamma_X(\gamma_Y(\psi^\omega))
=\beta_X(\beta_Y(\psi^\omega))$.

Let $\cF=\ovl{\cF}$ be a finite subset of $\Irr(\sC)$
and $\delta>0$.
Let $\cK$ be an $(\cF,\delta)$-invariant
finite subset of $\Irr(\sC)$.
Let $Q$ be a $\gamma$-invariant
countably generated von Neumann subalgebra
of $M^\omega$ such that $Q$ contains $M$
and $c_{X,Y}$ for all $X,Y\in\sC_0$.
Actually $Q$ contains $c_{X,Y}$ for all $X,Y\in\sC$.
This follows from (\ref{eq:cST})
and the fact that
any $T^\gamma=T^\beta$ is an element in $M$
for all morphisms $T$ in $\sC$.
Let $\varphi$ be a faithful normal state on $M$
such that $\varphi$ is $\alpha$-invariant
on $Z(M)$.

Let $E:=(E_Y^\gamma)_Y$
be a Rohlin tower along with $\cK$
with respect to $(\gamma,c),\cF,\delta,\cK$ and $Q$
as in Theorem \ref{thm:Rohlin-tensor}
and Lemma \ref{lem:Rohlin-tensor-gamma}.
We set the following element in $M^\omega$:
\begin{equation}
\label{eq:defnv}
v_X
:=\sum_{Y\in\Irr(\sC)}
d(Y)^2
\ovl{R}_Y^{\gamma*} c_{Y,\ovl{Y}}^*
\gamma_Y(d_{\ovl{Y},X}c_{\ovl{Y},X}^*)
E_Y^\gamma
c_{Y,\ovl{Y}}\ovl{R}_Y^\gamma
+p
\quad
\mbox{for }
X\in\sC,
\end{equation}
where $p:=1-\sum_{Y\in\cK}d(Y)^2
\ovl{R}_Y^{\gamma*} c_{Y,\ovl{Y}}^*
E_Y^\gamma
c_{Y,\ovl{Y}}\ovl{R}_Y^\gamma$ that is a projection
in $Q'\cap M_\omega$.
We should note each $v_X$ is a unitary
which follows from
Remark \ref{rem:Rohlin-tensor} (\ref{item:rem-Roh-orthogonal})
and the fact that all $c_{X,Y}$'s are contained in $Q$.
Note here that $E_Y^\gamma$ commutes with
$\gamma_Y(d_{\ovl{Y},X}c_{\ovl{Y},X}^*)$.
We set the partial isometry $v_X^0:=v_X-p$.
We will simply write $E:=E^\gamma$.
Let us keep our notation introduced above throughout this subsection.

\begin{lem}
\label{lem:vXcommute}
The following equalities hold:
\begin{enumerate}
\item
$v_X\gamma_X(\psi^\omega)=\gamma_X(\psi^\omega)v_X$
for all $X\in\sC$ and $\psi\in M_*$.

\item
$v_Y T^\gamma=T^\beta v_X$
for all
$X,Y\in\sC$ and $T\in\sC(X,Y)$.
\end{enumerate}
\end{lem}
\begin{proof}
(1).
Since $p\in M_\omega$, $p$
commutes with
$\gamma_X(\psi^\omega)=\beta_X^\omega(\psi^\omega)
=(\beta_X(\psi))^\omega$.
Hence it suffices to show $v_X^0$ commutes
with $\gamma_X(\psi^\omega)$
for all $X\in\sC$.
We have
\begin{align*}
&v_X^0\gamma_X(\psi^\omega)
\\
&=
\sum_{Y\in\Irr(\sC)}
d(Y)^2
\ovl{R}_Y^{\gamma*} c_{Y,\ovl{Y}}^*
\gamma_Y(d_{\ovl{Y},X}c_{\ovl{Y},X}^*)
E_Y
c_{Y,\ovl{Y}}\ovl{R}_Y^\gamma
\gamma_X(\psi^\omega)
\\
&=
\sum_{Y\in\Irr(\sC)}
d(Y)^4
\ovl{R}_Y^{\gamma*} c_{Y,\ovl{Y}}^*
\gamma_Y(d_{\ovl{Y},X}c_{\ovl{Y},X}^*)
E_Y
\gamma_Y(\gamma_{\ovl{Y}}(\gamma_X(\psi^\omega)))
c_{Y,\ovl{Y}}\ovl{R}_Y^\gamma
\\
&=
\sum_{Y\in\Irr(\sC)}
d(Y)^4
\ovl{R}_Y^{\gamma*} c_{Y,\ovl{Y}}^*
\gamma_Y(d_{\ovl{Y},X}c_{\ovl{Y},X}^*
\gamma_{\ovl{Y}}(\gamma_X(\psi^\omega)))
E_Y
c_{Y,\ovl{Y}}\ovl{R}_Y^\gamma
\\
&=
\sum_{Y\in\Irr(\sC)}
d(Y)^4
\ovl{R}_Y^{\gamma*} c_{Y,\ovl{Y}}^*
\gamma_Y(
\gamma_{\ovl{Y}}(\gamma_X(\psi^\omega))
d_{\ovl{Y},X}c_{\ovl{Y},X}^*)
E_Y
c_{Y,\ovl{Y}}\ovl{R}_Y^\gamma
\quad
\mbox{by }
(\ref{eq:dcgamgampsi})
\\
&=
\gamma_X(\psi^\omega)v_X^0.
\end{align*}

(2).
Note $p$ commutes with $T^\gamma=T^\beta\in M$.
We have
\begin{align*}
v_Y^0 T^\gamma
&=
\sum_{Z\in\Irr(\sC)}
d(Z)^2
\ovl{R}_Z^{\gamma*} c_{Z,\ovl{Z}}^*
\gamma_Z(d_{\ovl{Z},Y}c_{\ovl{Z},Y}^*)
E_Z
\gamma_Z(\gamma_{\ovl{Z}}(T^\gamma))
c_{Z,\ovl{Z}}\ovl{R}_Z^\gamma
\\
&=
\sum_{Z\in\Irr(\sC)}
d(Z)^2
\ovl{R}_Z^{\gamma*} c_{Z,\ovl{Z}}^*
\gamma_Z(d_{\ovl{Z},Y}c_{\ovl{Z},Y}^*\gamma_{\ovl{Z}}(T^\gamma))
E_Z
c_{Z,\ovl{Z}}\ovl{R}_Z^\gamma
\\
&=
\sum_{Z\in\Irr(\sC)}
d(Z)^2
\ovl{R}_Z^{\gamma*} c_{Z,\ovl{Z}}^*
\gamma_Z(\beta_{\ovl{Z}}(T^\gamma)
d_{\ovl{Z},X}c_{\ovl{Z},X}^*)
E_Z
c_{Z,\ovl{Z}}\ovl{R}_Z^\gamma
\quad
\mbox{by }
(\ref{eq:dcgamT})
\\
&=
T^\gamma v_X^0.
\end{align*}
\end{proof}

\begin{lem}
\label{lem:vperturb}
For all $Y\in\sC_0$,
one has
\[
\sum_{X\in\cF}d(X)^2
|v_X\gamma_X(v_Y)
-d_{X,Y}v_{X\otimes Y}c_{X,Y}^*
|_{\gamma_X(\gamma_Y(\varphi^\omega))}
\leq
8\delta^{1/2}|\cF|_\sigma.
\]
\end{lem}
\begin{proof}
(1).
Let $X\in\cF$ and $Y\in\sC_0$.
Then we have
$\gamma_X(v_Y^0)
=
A_{X,Y}
+
B_{X,Y}$,
where
\begin{align*}
A_{X,Y}
&:=
\sum_{Z\in\Irr(\sC)}
d(Z)^2
\gamma_X
(\ovl{R}_Z^{\gamma*} c_{Z,\ovl{Z}}^*)
\cdot
\gamma_X(\gamma_Z(d_{\ovl{Z},Y}c_{\ovl{Z},Y}^*))
\Delta_{X,Z}\cdot
\gamma_X(c_{Z,\ovl{Z}}\ovl{R}_Z^\gamma),
\\
B_{X,Y}
&:=
\sum_{Z\in\Irr(\sC)}
d(Z)^2
\gamma_X
(\ovl{R}_Z^{\gamma*} c_{Z,\ovl{Z}}^*)
\cdot
\gamma_X(\gamma_Z(d_{\ovl{Z},Y}c_{\ovl{Z},Y}^*))
c_{X,Z}E_{X\otimes Z}c_{X,Z}^*
\cdot
\gamma_X(c_{Z,\ovl{Z}}\ovl{R}_Z^\gamma),
\\
\Delta_{X,Z}
&:=\gamma_X(E_Z)-c_{X,Z}E_{X\otimes Z}c_{X,Z}^*.
\end{align*}

We will compute $B_{X,Y}$ as follows
\begin{align*}
B_{X,Y}
&=
\sum_{Z}
d(Z)^2
\gamma_X
(\ovl{R}_Z^{\gamma*} c_{Z,\ovl{Z}}^*)
\cdot
c_{X,Z}
\gamma_{X\otimes Z}(d_{\ovl{Z},Y}c_{\ovl{Z},Y}^*)
E_{X\otimes Z}c_{X,Z}^*
\cdot
\gamma_X(c_{Z,\ovl{Z}}\ovl{R}_Z^\gamma)
\notag
\\
&=
\sum_{V,Z,S}
d(Z)^2
\gamma_X
(\ovl{R}_Z^{\gamma*} c_{Z,\ovl{Z}}^*)
\cdot
c_{X,Z}S^\gamma
\gamma_{V}(d_{\ovl{Z},Y}c_{\ovl{Z},Y}^*)
E_V S^{\gamma*}c_{X,Z}^*
\cdot
\gamma_X(c_{Z,\ovl{Z}}\ovl{R}_Z^\gamma),
\end{align*}
where the summation is taken
for all $V,Z\in \Irr(\sC)$
and $S\in\ONB(V,X\otimes Z)$.

Applying the following equality to $v_X^0 B_{X,Y}$:
\[
E_Uc_{U,\ovl{U}}\ovl{R}_U^\gamma x
=\gamma_U(\gamma_{\ovl{U}}(x))E_Uc_{U,\ovl{U}}\ovl{R}_U^\gamma
\quad
\mbox{for all }
U\in\Irr(\sC),
\
x\in Q,
\]
we have
\begin{align*}
v_X^0B_{X,Y}
&=
\sum_{U,V,Z}
\sum_{S}
d(U)^2d(Z)^2
\ovl{R}_U^{\gamma*} c_{U,\ovl{U}}^*
\gamma_U(d_{\ovl{U},Y}c_{\ovl{U},X}^*)
\\
&\quad
\cdot
\gamma_U(\gamma_{\ovl{U}}(
\gamma_X
(\ovl{R}_Z^{\gamma*} c_{Z,\ovl{Z}}^*)
c_{X,Z}S^\gamma
\gamma_{V}(d_{\ovl{Z},Y}c_{\ovl{Z},Y}^*)))
\cdot
E_U
c_{U,\ovl{U}}\ovl{R}_U^\gamma
E_V
\\
&\quad
\cdot
S^{\gamma*}c_{X,Z}^*
\gamma_X(c_{Z,\ovl{Z}}\ovl{R}_Z^\gamma).
\end{align*}

Using Theorem \ref{thm:Rohlin-tensor} (\ref{item:Roh-resonance})
and Lemma \ref{lem:Rohlin-tensor-gamma}
for $(\gamma,c)$
in the above,
we obtain
\begin{align*}
v_X^0
B_{X,Y}
&=
\sum_{U,Z,S}
d(U)d(Z)^2
\ovl{R}_U^{\gamma*} c_{U,\ovl{U}}^*
\gamma_U(d_{\ovl{U},Y}c_{\ovl{U},X}^*)
\notag
\\
&\quad
\cdot
\gamma_U\big{(}\gamma_{\ovl{U}}(
\gamma_X
(\ovl{R}_Z^{\gamma*} c_{Z,\ovl{Z}}^*)
c_{X,Z}S^\gamma
\gamma_{U}(d_{\ovl{Z},Y}c_{\ovl{Z},Y}^*))\big{)}
\cdot
\gamma_{U}(c_{\ovl{U},U}R_U^\gamma)E_U
\notag
\\
&\quad
\cdot
S^{\gamma*}c_{X,Z}^*
\gamma_X(c_{Z,\ovl{Z}}\ovl{R}_Z^\gamma),
\end{align*}
where the summation is taken for all $U,Z\in\Irr(\sC)$
and $S\in\ONB(V,X\otimes Z)$.

Recall the right Frobenius map $F_r$ defined in (\ref{eq:rightFrob}).
For $T:=F_r(S)$, we obtain
\begin{align*}
\frac{d(U)^{1/2}}{d(X)^{1/2}d(Z)^{1/2}}
c_{U,\ovl{Z}}T^\gamma
&=
c_{U,\ovl{Z}}[S^*\otimes1_{\ovl{Z}}]^\gamma
[1_X\otimes\ovl{R}_Z]^\gamma
\\
&=
S^{\gamma*}
c_{X\otimes Z,\ovl{Z}}[1_X\otimes\ovl{R}_Z]^\gamma
\\
&=
S^{\gamma*}
c_{X,Z}^*\gamma_X(c_{Z,\ovl{Z}})c_{X,Z\otimes\ovl{Z}}
[1_X\otimes\ovl{R}_Z]^\gamma
\\
&=
S^{\gamma*}
c_{X,Z}^*\gamma_X(c_{Z,\ovl{Z}}\ovl{R}_Z^\gamma).
\end{align*}

Then we have
\begin{align*}
v_X^0 B_{X,Y}
&=
\sum_{U,Z,T}
\frac{d(U)^2d(Z)}{d(X)}
\ovl{R}_U^{\gamma*} c_{U,\ovl{U}}^*
\gamma_U(d_{\ovl{U},Y}c_{\ovl{U},X}^*)
\notag
\\
&\quad
\cdot
\gamma_U\big{(}\gamma_{\ovl{U}}(
T^{\gamma*}c_{U,\ovl{Z}}^*
\gamma_{U}(d_{\ovl{Z},Y}c_{\ovl{Z},Y}^*))\big{)}
\cdot
\gamma_{U}(c_{\ovl{U},U}R_U^\gamma)E_U
\cdot
c_{U,\ovl{Z}}T^\gamma,
\label{eq:v0BXY}
\end{align*}
where the summation is taken
for all $U,Z\in\Irr(\sC)$
and $T\in\ONB(X,U\otimes\ovl{Z})$.
Then we have
\begin{align*}
&v_X^0 B_{X,Y}c_{X,Y}v_{X\otimes Y}^{0*}
\notag
\\
&=
\sum_{U,Z,T}
\frac{d(U)^2d(Z)}{d(X)}
\ovl{R}_U^{\gamma*} c_{U,\ovl{U}}^*
\gamma_U\big{(}
d_{\ovl{U},Y}c_{\ovl{U},X}^*
\gamma_{\ovl{U}}(
T^{\gamma*}c_{U,\ovl{Z}}^*
\gamma_{U}(d_{\ovl{Z},Y}c_{\ovl{Z},Y}^*))\big{)}
\notag
\\
&\quad\cdot
\gamma_{U}(c_{\ovl{U},U}R_U^\gamma)E_U
\cdot
c_{U,\ovl{Z}}T^\gamma
c_{X,Y}v_{X\otimes Y}^{0*}.
\end{align*}
Since
\begin{align*}
&\gamma_{U}(c_{\ovl{U},U}R_U^\gamma)E_U
c_{U,\ovl{Z}}T^\gamma
c_{X,Y}v_{X\otimes Y}^{0*}
\\
&=
\gamma_{U}(c_{\ovl{U},U}R_U^\gamma)E_U
c_{U,\ovl{Z}}T^\gamma
c_{X,Y}
\sum_{V\in\Irr(\sC)}
d(V)^2
\ovl{R}_V^{\gamma*}
c_{V,\ovl{V}}^*
\gamma_V(c_{\ovl{V},X\otimes Y}d_{\ovl{V},X\otimes Y}^*)
E_V
c_{V,\ovl{V}}\ovl{R}_V^\gamma
\\
&=
\sum_V
d(V)^2
\gamma_{U}(c_{\ovl{U},U}R_U^\gamma)
E_U
\ovl{R}_V^{\gamma*}
c_{V,\ovl{V}}^*
E_V
\\
&\quad
\cdot
\gamma_V(\gamma_{\ovl{V}}(
c_{U,\ovl{Z}}T^\gamma
c_{X,Y}
))
\gamma_V(c_{\ovl{V},X\otimes Y}d_{\ovl{V},X\otimes Y}^*)
c_{V,\ovl{V}}\ovl{R}_V^\gamma
\\
&=
d(U)
\gamma_{U}(c_{\ovl{U},U}R_U^\gamma
R_U^{\gamma*}c_{\ovl{U},U}^*)E_U
\quad
\mbox{by Theorem \ref{thm:Rohlin-tensor} (\ref{item:Roh-resonance})
and Lemma \ref{lem:Rohlin-tensor-gamma}}
\\
&\quad
\cdot
\gamma_U(\gamma_{\ovl{U}}(
c_{U,\ovl{Z}}T^\gamma
c_{X,Y}
))
\gamma_U(c_{\ovl{U},X\otimes Y}d_{\ovl{U},X\otimes Y}^*)
c_{U,\ovl{U}}\ovl{R}_U^\gamma,
\end{align*}
we have
\begin{align*}
&v_X^0 B_{X,Y}c_{X,Y}v_{X\otimes Y}^{0*}
\\
&=
\sum_{U,Z\in\Irr(\sC)}
\sum_{T\in\ONB(X,U\otimes\ovl{Z})}
\frac{d(U)^3d(Z)}{d(X)}
\ovl{R}_U^*c_{U,\ovl{U}}^*
\gamma_U(C_{U,Y,Z,T})E_U
c_{U,\ovl{U}}\ovl{R}_U^\gamma,
\end{align*}
where
\begin{align*}
C_{U,Y,Z,T}
&:=
d_{\ovl{U},X}
c_{\ovl{U},X}^*
\gamma_{\ovl{U}}(
T^{\gamma*}c_{U,\ovl{Z}}^*
\gamma_{U}(d_{\ovl{Z},Y}c_{\ovl{Z},Y}^*))
\\
&\quad\cdot
c_{\ovl{U},U}
R_U^\gamma
R_U^{\gamma*}
c_{\ovl{U},U}^*
\\
&\quad\cdot
\gamma_{\ovl{U}}(
c_{U,\ovl{Z}}T^\gamma
c_{X,Y}
)
c_{\ovl{U},X\otimes Y}d_{\ovl{U},X\otimes Y}^*.
\end{align*}
By (\ref{eq:dcgamT}), we have
\begin{align*}
&d_{\ovl{U},X}c_{\ovl{U},X}^*
\gamma_{\ovl{U}}(
T^{\gamma*}c_{U,\ovl{Z}}^*
\gamma_{U}(d_{\ovl{Z},Y}c_{\ovl{Z},Y}^*))
\\
&=
\beta_{\ovl{U}}(T^{\beta*})
d_{\ovl{U},U\otimes \ovl{Z}}
c_{\ovl{U},U\otimes \ovl{Z}}^*
\gamma_{\ovl{U}}(c_{U,\ovl{Z}}^*
\gamma_{U}(d_{\ovl{Z},Y}c_{\ovl{Z},Y}^*))
\\
&=
\beta_{\ovl{U}}(T^{\beta*})
d_{\ovl{U},U\otimes \ovl{Z}}
c_{\ovl{U}\otimes U,\ovl{Z}}^*c_{\ovl{U},U}^*
\gamma_{\ovl{U}}(\gamma_{U}(d_{\ovl{Z},Y}c_{\ovl{Z},Y}^*))
\\
&=
\beta_{\ovl{U}}(T^{\beta*})
d_{\ovl{U},U\otimes \ovl{Z}}c_{\ovl{U}\otimes U,\ovl{Z}}^*
\gamma_{\ovl{U}\otimes U}(d_{\ovl{Z},Y}c_{\ovl{Z},Y}^*)
c_{\ovl{U},U}^*
\end{align*}
and
\begin{align*}
&\gamma_{\ovl{U}}(
c_{U,\ovl{Z}}T^\gamma
c_{X,Y})
c_{\ovl{U},X\otimes Y}d_{\ovl{U},X\otimes Y}^*
\\
&=
\gamma_{\ovl{U}}(
c_{U,\ovl{Z}}
c_{U\otimes \ovl{Z},Y}[T\otimes 1_Y]^\gamma)
c_{\ovl{U},X\otimes Y}d_{\ovl{U},X\otimes Y}^*
\\
&=
\gamma_{\ovl{U}}(
c_{U,\ovl{Z}}
c_{U\otimes \ovl{Z},Y})
c_{\ovl{U},U\otimes\ovl{Z}\otimes Y}
[1_{\ovl{U}}\otimes T\otimes 1_Y]^\gamma
d_{\ovl{U},X\otimes Y}^*
\\
&=
\gamma_{\ovl{U}}(
\gamma_U(c_{\ovl{Z},Y})
c_{U,\ovl{Z}\otimes Y})
c_{\ovl{U},U\otimes\ovl{Z}\otimes Y}
[1_{\ovl{U}}\otimes T\otimes 1_Y]^\gamma
d_{\ovl{U},X\otimes Y}^*
\\
&=
c_{\ovl{U},U}
\gamma_{\ovl{U}\otimes U}(c_{\ovl{Z},Y})
c_{\ovl{U},U}^*
\gamma_{\ovl{U}}(c_{U,\ovl{Z}\otimes Y})
c_{\ovl{U},U\otimes\ovl{Z}\otimes Y}
[1_{\ovl{U}}\otimes T\otimes 1_Y]^\gamma
d_{\ovl{U},X\otimes Y}^*
\\
&=
c_{\ovl{U},U}
\gamma_{\ovl{U}\otimes U}(c_{\ovl{Z},Y})
c_{\ovl{U}\otimes U,\ovl{Z}\otimes Y}
[1_{\ovl{U}}\otimes T\otimes 1_Y]^\gamma
d_{\ovl{U},X\otimes Y}^*.
\end{align*}
Thus we have
\begin{align*}
&C_{U,Y,Z,T}
\\
&=
\beta_{\ovl{U}}(T^\beta)
d_{\ovl{U},U\otimes\ovl{Z}}
c_{\ovl{U}\otimes U,\ovl{Z}}^*
\gamma_{\ovl{U}\otimes U}(d_{\ovl{Z},Y}c_{\ovl{Z},Y}^*)
c_{\ovl{U},U}^*
\cdot
c_{\ovl{U},U}
R_U^\gamma
R_U^{\gamma*}
c_{\ovl{U},U}^*
\\
&\quad
\cdot
c_{\ovl{U},U}
\gamma_{\ovl{U}\otimes U}(c_{\ovl{Z},Y})
c_{\ovl{U}\otimes U,\ovl{Z}\otimes Y}
[1_{\ovl{U}}\otimes T\otimes 1_Y]^\gamma
d_{\ovl{U},X\otimes Y}^*
\\
&=
\beta_{\ovl{U}}(T^\beta)
d_{\ovl{U},U\otimes\ovl{Z}}
[R_U\otimes 1_{\ovl{Z}}]^\beta
d_{\ovl{Z},Y}c_{\ovl{Z},Y}^*
c_{\ovl{Z},Y}
[R_U^*\otimes 1_{\ovl{Z}}\otimes 1_Y]^\gamma
[1_{\ovl{U}}\otimes T\otimes 1_Y]^\gamma
d_{\ovl{U},X\otimes Y}^*
\\
&=
d_{\ovl{U},X}
[1_{\ovl{U}}\otimes T^*]^\beta
[R_U\otimes 1_{\ovl{Z}}]^\beta
[R_U^*\otimes 1_{\ovl{Z}}]^\gamma
d_{\ovl{U}\otimes U\otimes \ovl{Z},Y}
[1_{\ovl{U}}\otimes T\otimes 1_Y]^\gamma
d_{\ovl{U},X\otimes Y}^*
\\
&=
d_{\ovl{U},X}
[1_{\ovl{U}}\otimes T^*]^\beta
[R_U\otimes 1_{\ovl{Z}}]^\beta
[R_U^*\otimes 1_{\ovl{Z}}]^\gamma
[1_{\ovl{U}}\otimes T]^\gamma
d_{\ovl{U}\otimes X,Y}
d_{\ovl{U},X\otimes Y}^*.
\end{align*}
Recall the following formula of $\sC$
(\ref{eq:leftinv}):
\[
1_{\ovl{U}}\otimes 1_X
=
\sum_{Z\in\Irr(\sC)}
\sum_{T\in\ONB(X,U\otimes \ovl{Z})}
\frac{d(U)d(Z)}{d(X)}
(1_{\ovl{U}}\otimes T^*)
(R_U\otimes 1_{\ovl{Z}})
(R_U^*\otimes 1_{\ovl{Z}})
(1_{\ovl{U}}\otimes T).
\]
This implies the following:
\[
\sum_{Z\in\Irr(\sC)}
\sum_{T\in\ONB(X,U\otimes \ovl{Z})}
\frac{d(U)d(Z)}{d(X)}
C_{U,Y,Z,T}
=
d_{\ovl{U},X}
d_{\ovl{U}\otimes X,Y}
d_{\ovl{U},X\otimes Y}^*
=
\gamma_{\ovl{U}}(d_{X,Y}).
\]
Hence
\[
v_X^0 B_{X,Y}c_{X,Y}v_{X\otimes Y}^{0*}
=
\sum_{U\in\Irr(\sC)}
d(U)^2
\ovl{R}_U^{\gamma*}c_{U,\ovl{U}}^*
\gamma_U(\gamma_{\ovl{U}}(d_{X,Y}))
E_Uc_{U,\ovl{U}}\ovl{R}_U^\gamma
=
d_{X,Y}(1-p).
\]
This shows that
$v_X^0 B_{X,Y}=d_{X,Y}v_{X\otimes Y}^0 c_{X,Y}^*$.
Note $E_U p=0$ for all $U\in\Irr(\sC)$,
and $E_Vp=0$ for all $V\in\sC$
since any morphisms in $\sC$ commute with $p\in M_\omega$.
By definition of $B_{X,Y}$,
we have $B_{X,Y}p=0=pB_{X,Y}$.
Thus $v_X B_{X,Y}=d_{X,Y}v_{X\otimes Y} c_{X,Y}^*(1-p)$.
Hence we have
\begin{align*}
v_X\gamma_X(v_Y)
&=
v_X\gamma_X(v_Y^0)
+
v_X\gamma_X(p)
\\
&=
v_X A_{X,Y}
+
v_X B_{X,Y}
+
v_X\gamma_X(p)
\\
&=
v_X A_{X,Y}
+
d_{X,Y}v_{X\otimes Y} c_{X,Y}^*
(1-p)
+
v_X\gamma_X(p).
\end{align*}
Let
$v_X\gamma_X(v_Y)-d_{X,Y}v_{X\otimes Y}c_{X,Y}^*
=
w_{X,Y}|v_X\gamma_X(v_Y)-d_{X,Y}v_{X\otimes Y}c_{X,Y}^*|$
be the polar decomposition.
This implies
\begin{align*}
&|v_X\gamma_X(v_Y)-d_{X,Y}v_{X\otimes Y}c_{X,Y}^*
|_{\gamma_X(\gamma_Y(\varphi^\omega))}
\\
&=
\gamma_X(\gamma_Y(\varphi^\omega))(w_{X,Y}^*v_XA_{X,Y})
-
\gamma_X(\gamma_Y(\varphi^\omega))(w_{X,Y}^*d_{X,Y}v_{X\otimes Y}c_{X,Y}^* p)
\\
&\quad+
\gamma_X(\gamma_Y(\varphi^\omega))
(w_{X,Y}^*v_X\gamma_X(p)).
\end{align*}

Applying Lemma \ref{lem:gaphiDe} to
$\psi:=\beta_Y(\varphi)$
(note that $\gamma_Y(\varphi^\omega)=\beta_Y(\varphi^\omega)$),
we obtain
\begin{align*}
|\gamma_X(\gamma_Y(\varphi^\omega))(w_{X,Y}^*v_XA_{X,Y})|
&\leq
\sum_{Z\in\Irr(\sC)}
d(Z)^2
|\Delta_{X,Z}|_{\gamma_X(\gamma_Z(\gamma_Y(\varphi^\omega)))}
\\
&=
\sum_{Z\in\Irr(\sC)}
d(Z)^2
|\Delta_{X,Z}|_{\gamma_X(\gamma_Z(\varphi^\omega))}.
\end{align*}
Since $p\in M_\omega$ commutes with
$\gamma_{X}(\gamma_Y(\varphi^\omega))$,
we have
\[
|\gamma_X(\gamma_Y(\varphi^\omega))(
w_{X,Y}^*d_{X,Y}v_{X\otimes Y}c_{X,Y}^*p)|
\leq
\gamma_X(\gamma_Y(\varphi^\omega))(p)
=
\varphi^\omega(p)
\leq
\delta^{1/2}.
\]
Similarly, we have
\[
|\gamma_X(\gamma_Y(\varphi^\omega))(w_{X,Y}^*v_X\gamma_X(p))|
\leq
\gamma_X(\gamma_Y(\varphi^\omega))(\gamma_X(p))
=\varphi^\omega(p)
\leq
\delta^{1/2}.
\]
Thus we are done.
\end{proof}

\begin{proof}[Proof of Lemma \ref{lem:cocapprox}]
Let $\cF_n=\ovl{\cF_n}$ be an increasing sequence of finite subsets
of $\Irr(\sC)$ such that their union equals $\Irr(\sC)$.
Let $0<\delta_n<n^{-2}|\cF_n|_\sigma^{-2}$.
Let $\cK_n$ be an $(\cF_n,\delta_n)$-invariant
finite subset of $\Irr(\sC)$.
For each $n\geq1$,
we take a unitary $v_X^n$ with $X\in\sC$
as defined in (\ref{eq:defnv}).
Then by Lemma \ref{lem:vperturb},
we know that
\[
\sum_{X\in\cF_n}
d(X)^2
|v_X^n\gamma_X(v_Y^n)
-d_{X,Y}v_{X\otimes Y}^n c_{X,Y}^*
|_{\gamma_X(\gamma_Y(\varphi^\omega))}
<
8/n
\quad
\mbox{for all }
Y\in\sC_0.
\]

Applying the index selection (\cite[Lemma 5.5]{Oc})
to the inequality above
and the equalities in Lemma \ref{lem:vXcommute},
we obtain a unitary $w_X\in M^\omega$ with $X\in\sC_0$
such that
\begin{itemize}
\item
$c_{X,Y}w_{X\otimes Y}^* d_{X,Y}^* w_X\gamma_X(w_Y)=1$
for all 
$X\in\Irr(\sC)$ and $Y\in\sC_0$.

\item
$w_X\gamma_X(\psi^\omega)=\gamma_X(\psi^\omega)w_X$
for all $X\in\sC_0$ and $\psi\in M_*$.

\item
$w_Y T^\gamma=T^\beta w_X$ for all $X,Y\in\sC_0$
and $T\in\sC(X,Y)$.
\end{itemize}

Readers should be careful
because there is an unclear point in the proof of \cite[Lemma 5.5]{Oc}.
However,
we can add the condition of
the equality presented in Lemma \ref{lem:vXcommute} (1)
to his proof,
and
we see the index selection of $v_X^n$'s indeed gives
a unitary which is a multiplier of $\mathcal{T}^\omega(M)$.
Now we will show
\begin{equation}
\label{eq:wXYC0}
w_X\gamma_X(w_Y)=d_{X,Y}w_{X\otimes Y}c_{X\otimes Y}^*
\quad
\mbox{for all }
X,Y\in\sC_0.
\end{equation}
Let $X,Y\in\sC_0$.
Then for $Z\in\Irr(\sC)$ and $S\in \sC(Z,X)$,
we have
\begin{align*}
d_{X,Y}w_{X\otimes Y}c_{X\otimes Y}^*S^\gamma
&=
d_{X,Y}w_{X\otimes Y}[S\otimes 1_Y]^\gamma c_{Z\otimes Y}^*
\\
&=
d_{X,Y}[S\otimes1_Y]^\beta w_{Z\otimes Y}c_{Z\otimes Y}^*
\\
&=
S^\beta d_{Z,Y}w_{Z\otimes Y}c_{Z\otimes Y}^*
\\
&=
S^\beta w_Z\gamma_Z(w_Y)
\\
&=
w_X\gamma_X(w_Y)S^\gamma.
\end{align*}
This shows the equality (\ref{eq:wXYC0}).
Next we will define $w_X$ for a general $X\in\sC$.
For $X\in\sC$,
we take $X_0\in\sC_0$ and a unitary $S\in\sC(X_0,X)$.
Then we set $w_X:=S^\gamma w_{X_0} S^{\gamma*}$.
This does not depend on $X_0$ or $S$.
Indeed, take another $X_1\in\sC_0$ and unitary $T\in\sC(X_1,X)$.
Since $T^*S\in\sC(X_0,X_1)$,
we have
\[
S^\gamma w_{X_0} S^{\gamma*}
=
T^\gamma [T^*S]^\gamma
w_{X_0} S^{\gamma*}
=
T^\gamma 
w_{X_1} [T^*S]^\gamma
S^{\gamma*}
=
T^\gamma w_{X_1} T^{\gamma*}.
\]
We will check (\ref{eq:wXYC0})
for $X,Y\in\sC$.
Take $X_0,Y_0\in\sC_0$
and unitaries $S\in\sC(X_0,X)$
and $T\in\sC(Y_0,Y)$.
Then we have
\begin{align*}
w_X\gamma_X(w_Y)
&=
S^\gamma w_{X_0} S^{\gamma*}
\gamma_X(T^\gamma w_{Y_0} T^{\gamma*})
\\
&=
S^\gamma w_{X_0} 
\gamma_{X_0}(T^\gamma w_{Y_0} T^{\gamma*})
S^{\gamma*}
\\
&=
S^\beta
\beta_{X_0}(T^\beta)
w_{X_0}\gamma_{X_0}(w_{Y_0})
\gamma_{X_0}( T^{\gamma*})
S^{\gamma*}
\\
&=
S^\beta
\beta_{X_0}(T^\beta)
d_{X_0,Y_0}
w_{X_0\otimes Y_0}c_{X_0,Y_0}^*
\gamma_{X_0}( T^{\gamma*})
S^{\gamma*}
\\
&=
d_{X,Y}
[S\otimes T]^\beta
w_{X_0\otimes Y_0}
[S^*\otimes T^*]^\gamma
c_{X,Y}
\\
&=
d_{X,Y}w_{X\otimes Y}c_{X,Y}.
\end{align*}

Recall the unitary $u_X$ which perturbs $(\alpha,c^\alpha)$
to $(\gamma,c)$.
Setting $\nu_X:=w_Xu_X$,
we are done.
\end{proof}
 
\subsection{On the first cohomology vanishing type result}

Let $(\alpha,c^\alpha)$ and $(\beta,c^\beta)$
be centrally free cocycle actions
of an amenable rigid C$^*$-tensor category
$\sC$ on a properly infinite von Neumann algebra $M$
with separable predual.
Suppose $\alpha$ has a faithful normal state $\varphi$ on $M$
such that $\varphi$ is $\alpha$-invariant
on $Z(M)$ as before.
Also suppose we have unitaries $w_X\in M^\omega$ with $X\in\sC$
which are satisfying the following properties:
for all $X,Y\in\sC$,
\begin{itemize}
\item
$w_X\alpha_X(\psi^\omega)w_X^*=\beta_X(\psi^\omega)$
for all $\psi\in M_*$.

\item
$w_X\alpha_X(w_Y)c_{X,Y}^\alpha w_{X\otimes Y}^*
=c_{X,Y}^\beta$.

\item
$w_YT^\alpha=T^\beta w_X$ for all $T\in\sC(X,Y)$.
\end{itemize}

In general,
$T^\alpha$ may not equal $T^\beta$,
and we will introduce the following unitary:
\[
\theta_X^{\beta,\alpha}
:=
\sum_{Z\in\Irr(\sC)}
\sum_{T\in\ONB(Z,X)}
T^\beta T^{\alpha*}
\quad
\mbox{for }
X\in\sC.
\]
Note that $\theta_X^{\beta,\alpha}$
is close to 1 in the strong$*$ topology
if and only if
each morphism $T^\alpha$ is to $T^\beta$.

Let $\cF=\ovl{\cF}$ be a finite subset of $\Irr(\sC)$
and $\delta>0$.
Let $\cK$ be an $(\cF,\delta)$-invariant finite subset
of $\Irr(\sC)$.
Let $Q$ be an $\alpha$-invariant
countably generated von Neumann subalgebra
of $M^\omega$ which contains $M$.
Let $E^\alpha=(E_X^\alpha)_X$ with $X\in\sC$
be a Rohlin tower
with respect to $(\alpha,c^\alpha),\cF,\cK,\delta$ and $Q$
as in Theorem \ref{thm:Rohlin-tensor}.
Then we will introduce the following Shapiro type element:
\[
\mu
:=\sum_{Z\in\Irr(\sC)}
d(Y)^2 \ovl{R}_Z^{\beta*}c_{Z,\ovl{Z}}^{\beta*}
w_Z E_Z^\alpha
c_{Z,\ovl{Z}}^\alpha \ovl{R}_Z^\alpha
+p,
\]
where
and $p:=1-\sum_{Z\in\cK}
d(Z)^2\ovl{R}_Z^{\alpha*}c_{Z,\ovl{Z}}^{\alpha*}
E_Z^\alpha
c_{Z,\ovl{Z}}^\alpha\ovl{R}_Z^\alpha$
that is a projection in $Q'\cap M_\omega$.
Note that $\varphi^\omega(p)\leq\delta^{1/2}$.
As explained in Section \ref{subsect:slight},
we can construct a Rohlin tower $E^\gamma$ of
$(\gamma,c^\gamma):=(\alpha^w,c^{\alpha^w})$
by setting
$E^\gamma:=(E_X^\gamma)_X$
with $E_X^\gamma:=w_X E_X^\alpha w_Y^*$ for $X\in\sC$.
Then it turns out from Remark \ref{rem:Rohlin-tensor}
(\ref{item:Roh-orthogonal}) for both
$(\alpha,c^\alpha)$ and $(\gamma,c^\gamma)$
that $\mu$ is a unitary.
We set $\mu_0:=\mu-p$.

\begin{lem}
\label{lem:wgadec}
Let $\mu$ be as above.
Then the following inequalities hold:
\begin{enumerate}
\item
\begin{align*}
&\sum_{X\in\cF}
d(X)^2|w_X\alpha_X(\mu)-\mu|_{\alpha_X(\varphi^\omega)}
\\
&\leq
8\delta^{1/2}|\cF|_\sigma
+
\sum_{X\in\cF}\sum_{Z\in\cK}
d(X)^2d(Z)^2
\|c_{\ovl{Z},X}^\beta
\theta_{\ovl{Z}\otimes X}^{\beta,\alpha}
c_{\ovl{Z},X}^{\alpha*}
-1\|_{\alpha_{\ovl{Z}}(\alpha_X(\varphi^\omega))}.
\end{align*}

\item
For all $\psi\in M_*$,
\[
\|\mu\psi^\omega-\psi^\omega\mu\|
\leq
\sum_{Y\in\cK}
d(Y)^4
\|\alpha_{\ovl{Y}}(\psi^\omega)-\beta_{\ovl{Y}}(\psi^\omega)\|.
\]
\end{enumerate}
\end{lem}
\begin{proof}
We will write $c_{X,Y}:=c_{X,Y}^\alpha$
and $d_{X,Y}:=c_{X,Y}^\beta$ for simplicity.
For $X\in\cF$, we have
\begin{align*}
w_X\alpha_X(\mu_0)
&=
\sum_{Y\in\Irr(\sC)}
d(Y)^2 
w_X\alpha_X(\ovl{R}_Y^\beta d_{Y,\ovl{Y}}^*)
\cdot
\alpha_X(w_Y E_Y^\alpha)
\cdot
\alpha_X(c_{Y,\ovl{Y}}\ovl{R}_Y^\alpha)
\\
&=
\sum_{Y}
d(Y)^2 
\beta_X(\ovl{R}_Y^\beta d_{Y,\ovl{Y}}^*)
w_X
\cdot
\alpha_X(w_Y) \alpha_X(E_Y^\alpha)
\cdot
\alpha_X(c_{Y,\ovl{Y}}\ovl{R}_Y^\alpha).
\end{align*}

Then we have
\[
w_X\alpha_X(\mu_0)
=
\sum_{Y}
d(Y)^2 
\beta_X(\ovl{R}_Y^\beta d_{Y,\ovl{Y}}^*)
w_X
\alpha_X(w_Y) \Delta_{X,Y}
\alpha_X(c_{Y,\ovl{Y}}\ovl{R}_Y^\alpha)
\notag
+
A_X,
\]
where
\begin{align*}
&\Delta_{X,Y}
:= \alpha_X(E_Y^\alpha)-c_{X,Y}E_{X\otimes Y}^\alpha c_{X,Y}^*,
\\
&
A_X=\sum_{Y\in\Irr(\sC)}
d(Y)^2 
\beta_X(\ovl{R}_Y^\beta d_{Y,\ovl{Y}}^*)
w_X\alpha_X(w_Y)
c_{X,Y}E_{X\otimes Y}^\alpha c_{X,Y}^*
\alpha_X(c_{Y,\ovl{Y}}\ovl{R}_Y^\alpha).
\end{align*}
We will compute $A_X$ as follows:
\begin{align}
A_X
&=
\sum_{Y}
d(Y)^2 
\beta_X(\ovl{R}_Y^\beta)
d_{X,Y\otimes\ovl{Y}}d_{X\otimes Y,\ovl{Y}}^*
w_{X\otimes Y}E_{X\otimes Y}^\alpha
c_{X\otimes Y,\ovl{Y}}
c_{X,Y\otimes \ovl{Y}}^*\alpha_X(\ovl{R}_Y^\alpha)
\notag\\
&=
\sum_{Y}
d(Y)^2
[1_X\otimes \ovl{R}_Y^*]^\beta
d_{X\otimes Y,\ovl{Y}}^*
w_{X\otimes Y}
E_{X\otimes Y}^\alpha
c_{X\otimes Y,\ovl{Y}}
[1_X\otimes \ovl{R}_Y]^\alpha.
\end{align}

By irreducible decomposition of $X\otimes Y$,
we have
\begin{align*}
A_X
&=
\sum_{Y,Z,S}
d(Y)^2 
[1_X\otimes \ovl{R}_Y^*]^\beta
d_{X\otimes Y,\ovl{Y}}^*
w_{X\otimes Y}
S^\alpha E_{Z}^\alpha S^{\alpha *}
c_{X\otimes Y,\ovl{Y}}
[1_X\otimes \ovl{R}_Y]^\alpha
\notag
\\
&=
\sum_{Y,Z,S}
d(Y)^2 
[1_X\otimes \ovl{R}_Y^*]^\beta
d_{X\otimes Y,\ovl{Y}}^*
S^\beta w_Z E_{Z}^\alpha c_{Z,\ovl{Y}}
[S^*\otimes 1_{\ovl{Y}}]^\alpha
[1_X\otimes \ovl{R}_Y]^\alpha
\notag
\\
&=
\sum_{Y,Z,S}
d(Y)^2 
[(1_X\otimes \ovl{R}_Y^*)(S\otimes 1_{\ovl{Y}})]^\beta
d_{Z,\ovl{Y}}^*
w_Z E_{Z}^\alpha c_{Z,\ovl{Y}}
[(S^*\otimes 1_{\ovl{Y}})(1_X\otimes \ovl{R}_Y)]^\alpha,
\end{align*}
where the summation is taken for all
$Y,Z\in\Irr(\sC)$ and $S\in\ONB(Z,X\otimes Y)$.

Using the right Frobenius map (\ref{eq:rightFrob})
with $T:=F_r(S)$,
we have 
\[
A_X
=
\sum_{Y,Z\in\Irr(\sC)}
\sum_{T\in\ONB(X,Z\otimes\ovl{Y})}
\frac{d(Y)d(Z)}{d(X)}
T^{\beta*}
d_{Z,\ovl{Y}}^*
w_Z E_Z^\alpha c_{Z,\ovl{Y}}
T^\alpha.
\]
Next by using the left Frobenius map (\ref{eq:leftFrob})
with $U:=F_\ell(T)$,
we have
\[
T=
\frac{d(X)^{1/2}d(Z)^{1/2}}{d(Y)^{1/2}}
(1_Z\otimes U^*)(\ovl{R}_Z\otimes1_X)
\]
and
\begin{align*}
A_X
&=
\sum_{Y,Z}
\sum_{U\in\ONB(\ovl{Y},\ovl{Z}\otimes X)}
d(Z)^2
[(\ovl{R}_Z^*\otimes 1_X)(1_Z\otimes U)]^\beta
\\
&\hspace{140pt}
\cdot
d_{Z,\ovl{Y}}^*
w_Z E_Z^\alpha c_{Z,\ovl{Y}}
[(1_Z\otimes U^*)(\ovl{R}_Z\otimes 1_X)]^\alpha
\\
&=
\sum_{Y,Z,U}
d(Z)^2
[\ovl{R}_Z^*\otimes 1_X]^\beta
d_{Z,\ovl{Z}\otimes X}^*
\beta_Z(U^\beta)
w_Z E_Z^\alpha
\alpha_Z(U^{\alpha *})
c_{Z,\ovl{Z}\otimes X}
[\ovl{R}_Z\otimes 1_X]^\alpha
\\
&=
\sum_{Y,Z,U}
d(Z)^2
[\ovl{R}_Z^*\otimes 1_X]^\beta
d_{Z,\ovl{Z}\otimes X}^*
w_Z E_Z^\alpha
\alpha_Z(U^\beta U^{\alpha *})c_{Z,\ovl{Z}\otimes X}
[\ovl{R}_Z\otimes 1_X]^\alpha
\\
&=
\sum_{Z}
d(Z)^2
[\ovl{R}_Z^*\otimes 1_X]^\beta
d_{Z,\ovl{Z}\otimes X}^*
w_Z E_Z^\alpha
\alpha_Z(\theta_{\ovl{Z}\otimes X}^{\beta,\alpha})
c_{Z,\ovl{Z}\otimes X}
[\ovl{R}_Z\otimes 1_X]^\alpha
\\
&=
\sum_{Z}
d(Z)^2
[\ovl{R}_Z^*\otimes 1_X]^{\beta}
d_{Z\otimes \ovl{Z},X}^*d_{Z,\ovl{Z}}^*\beta_Z(d_{\ovl{Z},X})
w_Z E_Z^\alpha
\\
&\quad
\hspace{30pt}\cdot
\alpha_Z(\theta_{\ovl{Z}\otimes X}^{\beta,\alpha}c_{\ovl{Z},X}^*)
c_{Z,\ovl{Z}}c_{Z\otimes\ovl{Z},X}
[\ovl{R}_Z\otimes 1_X]^\alpha
\\
&=
\sum_{Z}
d(Z)^2
\ovl{R}_Z^{\beta *}d_{Z,\ovl{Z}}^*
w_Z E_Z^\alpha
\alpha_Z(d_{\ovl{Z},X}
\theta_{\ovl{Z}\otimes X}^{\beta,\alpha}c_{\ovl{Z},X}^*)c_{Z,\ovl{Z}}
\ovl{R}_Z^\alpha.
\end{align*}
Thus we have
\begin{align*}
w_X\alpha_X(\mu)-\mu
&=
\sum_{Y\in\cK}
d(Y)^2 
\beta_X(\ovl{R}_Y^\beta d_{Y,\ovl{Y}}^*)
w_X
\alpha_X(w_Y) \Delta_{X,Y}
\alpha_X(c_{Y,\ovl{Y}}\ovl{R}_Y^\alpha)
\\
&\quad
+
\sum_{Z\in\cK}
d(Z)^2
\ovl{R}_Z^{\beta *}d_{Z,\ovl{Z}}^*
w_Z E_Z^\alpha
\alpha_Z(d_{\ovl{Z},X}\theta_{\ovl{Z}\otimes X}^{\beta,\alpha}c_{\ovl{Z},X}^*-1)c_{Z,\ovl{Z}}
\ovl{R}_Z^\alpha
\\
&\quad
+
w_X\alpha_X(p)-p.
\end{align*}

Let $w_X\alpha_X(\mu)-\mu=u_X|w_X\alpha_X(\mu)-\mu|$
be the polar decomposition.
Then
\begin{align}
&|w_X\alpha_X(\mu)-\mu|_{\alpha_X(\varphi^\omega)}
\notag
\\
&=
\sum_{Y\in\cK}
d(Y)^2 
\alpha_X(\varphi^\omega)
\big{(}
u_X^*
\beta_X(\ovl{R}_Y^\beta d_{Y,\ovl{Y}}^*)
w_X
\alpha_X(w_Y) \Delta_{X,Y}
\alpha_X(c_{Y,\ovl{Y}}\ovl{R}_Y^\alpha)
\big{)}
\label{eq:wgawDe}
\\
&\quad
+
\sum_{Z\in\cK}
d(Z)^2
\alpha_X(\varphi^\omega)
\big{(}
u_X^*
\ovl{R}_Z^{\beta *}d_{Z,\ovl{Z}}^*
w_Z E_Z^\alpha
\alpha_Z(d_{\ovl{Z},X}\theta_{\ovl{Z}\otimes X}^{\beta,\alpha}c_{\ovl{Z},X}^*-1)c_{Z,\ovl{Z}}
\ovl{R}_Z^\alpha
\big{)}
\label{eq:wEga}
\\
&\quad
+
\alpha_X(\varphi^\omega)(u_X^*(w_X\alpha_X(p)-p)).
\label{eq:gaphip}
\end{align}
From Lemma \ref{lem:gaphiDe},
we obtain
\[
|(\ref{eq:wgawDe})|
\leq
\sum_{Y\in\cK}
d(Y)^2
|\Delta_{X,Y}|_{\alpha_X(\alpha_Y(\varphi^\omega))}.
\]

We will roughly estimate (\ref{eq:wEga})
as follows:
\begin{align*}
|(\ref{eq:wEga})|
&\leq
\sum_{Z\in\cK}
d(Z)^2
\alpha_X(\varphi^\omega)
\big{(}
\ovl{R}_Z^{\alpha*}
c_{Z,\ovl{Z}}
\alpha_Z(|d_{\ovl{Z},X}\theta_{\ovl{Z}\otimes X}^{\beta,\alpha}c_{\ovl{Z},X}^*-1|^2)c_{Z,\ovl{Z}}
\ovl{R}_Z^\alpha
\big{)}^{1/2}
\\
&=
\sum_{Z\in\cK}
d(Z)^2
\alpha_{\ovl{Z}}(\alpha_X(\varphi^\omega))
(|d_{\ovl{Z},X}\theta_{\ovl{Z}\otimes X}^{\beta,\alpha}
c_{\ovl{Z},X}^*-1|^2)^{1/2}.
\end{align*}

Since $\alpha_X(p)$ and $p$ commutes with $\alpha_X(\varphi^\omega)$,
we have
\[
|(\ref{eq:gaphip})|
\leq
\alpha_X(\varphi^\omega)(\alpha_X(p))
+
\alpha_X(\varphi^\omega)(p)
=2\varphi^\omega(p)
\leq
2\delta^{1/2}.
\]
Hence we obtain the inequality of the statement of this lemma.

(2).
The following equality yields the inequality (2):
\[
\mu\psi^\omega-\psi^\omega\mu
=
\sum_{Y\in\Irr(\sC)}
d(Y)^4 \ovl{R}_Y^{\beta*}
d_{Y,\ovl{Y}}^{*}
\beta_Z(\alpha_{\ovl{Y}}(\psi^\omega)
-\beta_{\ovl{Y}}(\psi^\omega))
w_Y E_Y^\alpha 
c_{Y,\ovl{Y}}^\alpha \ovl{R}_Y^\alpha.
\]
\end{proof}

Taking a unitary representing sequence
of $\mu$ introduced in the previous lemma,
we obtain the following result
by Lemma \ref{lem:ultra-functional}.

\begin{lem}
\label{lem:app1cohvan}
Let $(\alpha,c^\alpha)$ be a centrally free cocycle action
of an amenable rigid C$^*$-tensor category $\sC$
on a properly infinite von Neumann algebra $M$ with separable predual.
Let $\varphi$ be a faithful normal state on $M$
being $\alpha$-invariant on $Z(M)$.
Let $\cF=\ovl{\cF}\subset\Irr(\sC)$ be a finite subset
and $\delta,\varepsilon>0$.
Let $\cK$ be an $(\cF,\delta)$-invariant finite subset
of $\Irr(\sC)$.
Let $u_X\in M$ with $X\in\sC$ be unitaries.
Let $(\beta,c^\beta)$ be the perturbed cocycle action
of $(\alpha,c^\alpha)$ by $u$.
Suppose
the following inequalities hold:
\[
\|c_{\ovl{Y},X}^\beta
\theta_{\ovl{Y}\otimes X}^{\beta,\alpha}
c_{\ovl{Y},X}^{\alpha*}-1\|_{\alpha_{\ovl{Y}}(\alpha_X(\varphi))}
<\varepsilon
\quad
\mbox{for all }
X\in\cF,\ Y\in\cK,
\]
and
\[
\|\beta_{\ovl{Y}}(\psi)
-\alpha_{\ovl{Y}}(\psi)\|<\varepsilon
\quad
\mbox{for all }
Y\in\cK,
\ \psi\in \Psi.
\]
Then for any finite subset $\Psi\subset M_*$,
there exists a unitary $\nu\in M$ such that
\[
\sum_{X\in\cF}
d(X)^2
|u_X\alpha_X(\nu)-\nu|_{\alpha_X(\varphi)}
<
8\delta^{1/2}|\cF|_\sigma
+
\varepsilon|\cF|_\sigma|\cK|_\sigma
\]
and
\[
\|\nu\psi-\psi\nu\|
<
\varepsilon
\sum_{Y\in\cK}
d(Y)^4
\quad
\mbox{for all }
\psi\in\Psi.
\]
\end{lem}

\subsection{Intertwining argument}
Let us recall the cocycle conjugacy of cocycle actions
studied in \cite{Iz-near,Mas-Rob}.

\begin{defn}
Let $(\alpha,c^\alpha)$ and $(\beta,c^\beta)$
be cocycle actions of a rigid C$^*$-tensor category $\sC$
on a properly infinite von Neumann algebra $M$.
We will say that they are
\begin{itemize}
\item
\emph{cocycle conjugate}
if there exist a family of unitaries
$v_X\in M$ with $X\in\sC$ and $\theta\in\Aut(M)$
such that
\begin{itemize}
\item
$\Ad v_X\circ\alpha_X=\theta\circ\beta_X\circ\theta^{-1}$
for all $X\in\sC$,

\item
$v_X\alpha_X(v_Y)c_{X,Y}^\alpha v_{X\otimes Y}^*
=\theta(c_{X,Y}^\beta)$
for all $X,Y\in\sC$,

\item
$v_YT^\alpha=\theta(T^\beta)v_X$
for all $X,Y\in\sC$ and $T\in\sC(X,Y)$;
\end{itemize}

\item
\emph{strongly cocycle conjugate}
if we can take a family of unitaries $v_X\in M$
with $X\in \sC$ and an approximately inner automorphism
$\theta$ on $M$ which satisfy the conditions above.
\end{itemize}
\end{defn}

We will state our main result of this section.

\begin{thm}
\label{thm:classification}
Let $\sC$ be an amenable rigid C$^*$-tensor category.
Let $(\alpha,c^\alpha)$ and $(\beta,c^\beta)$
be centrally free and approximately unitarily equivalent
cocycle actions
of $\sC$ on a properly infinite von Neumann algebra $M$
with separable predual.
Suppose that
there exists an $\alpha$-invariant
faithful normal state on $Z(M)$.
Then $(\alpha,c^\alpha)$ and $(\beta,c^\beta)$ are strongly cocycle conjugate.
\end{thm}
\begin{proof}
We put $\mathcal{S}_0:=\{\btr\}$.
Let $\mathcal{S}_n=\ovl{\mathcal{S}_n}$ with $n\geq0$
be an increasing sequence of finite subsets
of $\Irr(\sC)$ whose union equals $\Irr(\sC)$.
We can inductively construct a sequence of finite subsets
$\cF_n=\ovl{\cF_n},\cK_n\subset\Irr(\sC)$ and $\delta_n>0$
so that for $n\geq1$,
\begin{itemize}
\item
$\cF_0:=\mathcal{S}_0=:\cK_0$, $\delta_0:=3$.

\item
$\cF_n:=\cF_{n-1}\cup\cK_{n-1}\cup\ovl{\cK_{n-1}}\cup\mathcal{S}_n$.

\item
$\delta_n:=16^{-n}|\cF_n|_\sigma^{-2}$.

\item
$\cK_n$ is $(\cF_n,\delta_n)$-invariant.
\end{itemize}
Take $\varepsilon_n>0$ with $n\geq0$
such that
$\varepsilon_n<\delta_n|\cK_n|_\sigma^{-2}$
and $\varepsilon_n<\varepsilon_{n-1}<1$.

Let $\mathcal{G}_n$ be an increasing sequence of finite subsets
in $\sC_0$ whose union equals $\sC_0$.
Let $\varphi$ be a faithful normal state on $M$
such that $\varphi$ is $\alpha$-invariant on $Z(M)$.
Put $\Phi_0:=\{\varphi\}$.
Let $\Phi_n$ with $n\geq0$ be an increasing sequence of finite subsets
of $M_*$ whose union is norm-dense in $M_*$.
We set $\Psi_{-1}:=\Phi_0=:\Psi_0$,
$(\gamma^{-1},c^{-1}):=(\alpha,c^\alpha)$,
$(\gamma^0,c^0):=(\beta,c^\beta)$,
$w_{-1}:=1=:w_0$,
$\theta_{-1}:=\id_M=:\theta_0$
and
$u_X^{-1}:=1=:u_X^0$
and
$\ovl{u}_X^{-1}:=1=:\ovl{u}_X^0$
for all $X\in\sC$.
We will inductively construct the following members with $n\geq1$:

\begin{itemize}
\item
$\Psi_n$:
a finite subset of $M_*$;

\item
$(\gamma^n,c^n)$:
a centrally free cocycle action of $\sC$ on $M$;

\item
$u_X^n$:
a unitary in $M$ for $X\in\sC$;

\item
$w_n$:
a unitary in $M$;

\item
$\theta_n$: an inner automorphism on $M$;

\item
$\ovl{u}_X^n$:
a unitary in $M$ for $X\in\sC$;
\end{itemize}
such that
\begin{enumerate}
\renewcommand{\labelenumi}{$(n.\arabic{enumi})$}
\item
\label{item:Psin}
\begin{align*}
\Psi_{n}
&:=\Phi_{n}
\cup
\Psi_{n-1}
\cup
\theta_{n-1}(\Phi_n)
\\
&\hspace{20pt}
\cup \bigcup_{X\in\cF_{n}}
\{\gamma_X^{n-1}(\Phi_{n}),
\gamma_X^{n-1}(\varphi)\ovl{u}_X^{n-1},
\ovl{u}_X^{n-1}\gamma_X^{n-1}(\varphi),
\ovl{u}_X^{n-1*}\gamma_X^{n-1}(\varphi)\ovl{u}_X^{n-1}
\}.
\end{align*}

\item
\label{item:gammanun}
The unitaries $u^n$ perturb
$(\Ad w_{n}\circ\gamma^{n-2}\circ\Ad w_{n}^*,w_n c^{n-2} w_n^*)$
to
$(\gamma^{n},c^n)$.

\item
\label{item:gammanpsi}
$\|
\gamma_{\ovl{X}}^{n}(\gamma_{Y}^{n}(\psi))
-\gamma_{\ovl{X}}^{n-1}(\gamma_{Y}^{n-1}(\psi))
\|<\varepsilon_{n}$
for all
$X,Y\in \cF_{n+1}$
and $\psi\in \Psi_{n}$.

\item
\label{item:Tgammanpsi}
$\|(T^{\gamma^{n}}-T^{\gamma^{n-1}})\varphi\|
+
\|\varphi\cdot (T^{\gamma^{n}}-T^{\gamma^{n-1}})\|<\varepsilon_n$
for all $T\in\sC(X,Y)$ with $X,Y\in \mathcal{G}_{n}$.

\item
\label{item:cocyclen}
$\|(c_{X,Y}^n-c_{X,Y}^{n-1})\varphi\|
+
\|\varphi\cdot(c_{X,Y}^n-c_{X,Y}^{n-1})\|<\varepsilon_n$
for all $X,Y\in \mathcal{G}_{n}$.

\item
\label{item:thetaYX}
$
\|c_{\ovl{Y},X}^n
\theta_{\ovl{Y}\otimes X}^{\gamma^n,\gamma^{n-1}}
c_{\ovl{Y},X}^{n-1*}-1
\|_{\gamma_{\ovl{Y}}^{n-1}(\gamma_X^{n-1}(\varphi))}<\varepsilon_n$
for all $X,Y\in\cF_{n+1}$.

\item
\label{item:sumuX}
$\sum_{X\in\cF_{n-1}}
|u_X^{n}-1|_{\Ad w_n\circ \gamma_X^{n-2}\circ\Ad w_n^*(\varphi)}
<
13/4^{n-1}$
with $n\geq 2$.

\item
\label{item:wnpsi}
$\|w_n\psi-\psi w_n\|<2\delta_{n-1}$
for $\psi\in\Psi_{n-1}$.

\item
\label{item:thetan}
$\theta_n:=\Ad w_n\circ\theta_{n-2}$.

\item
\label{item:ovlu}
$\ovl{u}_X^n:=u_X^n w_n\ovl{u}_X^{n-2} w_n^*$
for $X\in\sC$.
\end{enumerate}

\noindent
\textbf{Step of $n=1$}.
We set $\Psi_1$ as in (1.\ref{item:Psin}).
Using Lemma \ref{lem:cocapprox},
we obtain unitaries $\lambda_X\in M^\omega$ with $X\in\sC$
such that
for all $X,Y\in\sC$:
\begin{itemize}
\item
$\lambda_X\gamma_X^{-1}(\psi^\omega)\lambda_X^*
=\gamma_X^0(\psi^\omega)$
for all $\psi\in M_*$.

\item
$\lambda_X\gamma_X^{-1}(\lambda_Y)c_{X,Y}^{-1} 
\lambda_{X\otimes Y}^*
=c_{X,Y}^0$.

\item
$\lambda_Y T^{\gamma^{-1}}=T^{\gamma^0} \lambda_X$ for all $T\in\sC(X,Y)$.
\end{itemize}

For each $X\in\sC$,
we take a representing sequence of $\lambda_X=(\lambda_X^m)^\omega$
with each $\lambda_X^m$ being a unitary in $M$.
Then we can check the conditions
(1.\ref{item:gammanpsi}),
(1.\ref{item:Tgammanpsi}), (1.\ref{item:cocyclen})
and (1.\ref{item:thetaYX}) hold
for sufficiently large $m_0$
by putting $\gamma_X^1:=\Ad \lambda_X^m\circ\gamma_X^{-1}$
for $X\in\sC$.
Let $u^1:=\lambda^{m_0}$.
Put $w_1:=1$.
The conditions (1.\ref{item:gammanun}) and (1.\ref{item:wnpsi}) are trivial.
We set $\theta_1:=\id$ and $\ovl{u}_X^1:=u_X^1$ for $X\in\sC$,
and then (1.\ref{item:thetan}) and (1.\ref{item:ovlu}) hold.

\noindent
\textbf{Step of $n+1$}.
Suppose our construction is done up to the $n^{\rm th}$-step.
We set $\Psi_{n+1}$ as in ($n+1$.\ref{item:Psin}).
Again by Lemma \ref{lem:cocapprox},
we obtain unitaries $\lambda_X\in M^\omega$ with $X\in\sC$
such that
for all $X,Y\in\sC$:
\begin{itemize}
\item
$\lambda_X\gamma_X^{n-1}(\psi^\omega)\lambda_X^*
=\gamma_X^n(\psi^\omega)$
for all $\psi\in M_*$.

\item
$\lambda_X\gamma_X^{n-1}(\lambda_Y)c_{X,Y}^{n-1} 
\lambda_{X\otimes Y}^*
=c_{X,Y}^n$.

\item
$\lambda_Y T^{\gamma^{n-1}}=T^{\gamma^n} \lambda_X$ for all $T\in\sC(X,Y)$.
\end{itemize}
Thinking of a representing sequence of $\lambda$
as above,
we obtain unitaries $v_X^{n+1}\in M$ with $X\in \sC$
such that
we have the three conditions
($n+1$,\ref{item:gammanpsi}), ($n+1$,\ref{item:Tgammanpsi}),
($n+1$.\ref{item:cocyclen})
and ($n+1$,\ref{item:thetaYX})
by putting $\gamma^{n+1}:=\Ad v^{n+1}\circ\gamma^{n-1}$
and $c^{n+1}:=(c^{n-1})^{v^{n+1}}$.

Let $X\in\cF_n$ and $Y\in\cK_n$.
Then $X,\ovl{Y}\in\cF_{n+1}$.
Using the equality
\[
c_{\ovl{Y},X}^{n+1}
\theta_{\ovl{Y}\otimes X}^{\gamma^{n+1},\gamma^{n-1}}c_{\ovl{Y},X}^{n-1*}
=c_{\ovl{Y},X}^{n+1}
\theta_{\ovl{Y}\otimes X}^{\gamma^{n+1},\gamma^{n}}c_{\ovl{Y},X}^{n*}
\cdot c_{\ovl{Y},X}^n
\theta_{\ovl{Y}\otimes X}^{\gamma^{n},\gamma^{n-1}}c_{\ovl{Y},X}^{n-1*},
\]
we have
\begin{align*}
&\|c_{\ovl{Y},X}^{n+1}
\theta_{\ovl{Y}\otimes X}^{\gamma^{n+1},\gamma^{n-1}}
c_{\ovl{Y}, X}^{n-1*}
-1\|_{\gamma_{\ovl{Y}}^{n-1}(\gamma_X^{n-1}(\varphi))}
\\
&\leq
\|c_{\ovl{Y}, X}^{n+1}
\theta_{\ovl{Y}\otimes X}^{\gamma^{n+1},\gamma^{n}}
c_{\ovl{Y}, X}^{n*}
(c_{\ovl{Y}, X}^n\theta_{\ovl{Y}\otimes X}^{\gamma^{n},\gamma^{n-1}}
c_{\ovl{Y}, X}^{n-1*}-1)
\|_{\gamma_{\ovl{Y}}^{n-1}(\gamma_X^{n-1}(\varphi))}
\\
&\quad
+
\|
c_{\ovl{Y}, X}^{n+1}\theta_{\ovl{Y}\otimes X}^{\gamma^{n+1},\gamma^{n}}
c_{\ovl{Y}, X}^{n*}-1
\|_{\gamma_{\ovl{Y}}^{n-1}(\gamma_X^{n-1}(\varphi))}
\\
&<
\|
c_{\ovl{Y}, X}^n\theta_{\ovl{Y}\otimes X}^{\gamma^{n},\gamma^{n-1}}
c_{\ovl{Y}, X}^{n-1*}-1
\|_{\gamma_{\ovl{Y}}^{n-1}(\gamma_X^{n-1}(\varphi))}\\
&\quad
+
2
\|\gamma_{\ovl{Y}}^{n-1}(\gamma_X^{n-1}(\varphi))
-\gamma_{\ovl{Y}}^{n}(\gamma_X^{n}(\varphi))\|^{1/2}
+
\|c_{\ovl{Y}, X}^{n+1}\theta_{\ovl{Y}\otimes X}^{\gamma^{n+1},\gamma^{n}}
c_{\ovl{Y}, X}^{n*}-1
\|_{\gamma_{\ovl{Y}}^{n}(\gamma_X^{n}(\varphi))}
\\
&<
\varepsilon_n+2\varepsilon_n^{1/2}+\varepsilon_{n+1}
\quad
\mbox{by }
(n.\ref{item:thetaYX}),
\ (n.\ref{item:gammanpsi}),\ (n+1.\ref{item:thetaYX})
\\
&<4\varepsilon_n^{1/2}.
\end{align*}

Applying Corollary \ref{lem:app1cohvan}
to $\gamma^{n-1}$ and $v^{n+1}$,
we get a unitary $w_{n+1}\in M$ such that
\begin{align}
\sum_{X\in\cF_n}
d(X)^2
|v_X^{n+1}\gamma_X^{n-1}(w_{n+1})-w_{n+1}|_{\gamma_X^{n-1}(\varphi)}
&<
8\delta_n^{1/2}|\cF_n|_\sigma
+
4\varepsilon_n^{1/2}|\cF_n|_\sigma|\cK_n|_\sigma
\notag
\\
&<12/4^{n}
\label{eq:vgawn}
\end{align}
and for $\psi\in \Psi_n$,
\begin{align*}
\|w_{n+1}\psi-\psi w_{n+1}\|
&<\sum_{Z\in\cK_n}
d(Z)^4\|\gamma_{\ovl{Z}}^{n+1}(\psi)-\gamma_{\ovl{Z}}^{n-1}(\psi)\|
\\
&\leq
\sum_{Z\in\cK_n}
d(Z)^4\|\gamma_{\ovl{Z}}^{n+1}(\psi)-\gamma_{\ovl{Z}}^{n}(\psi)\|
+
\sum_{Z\in\cK_n}
d(Z)^4\|\gamma_{\ovl{Z}}^{n}(\psi)-\gamma_{\ovl{Z}}^{n-1}(\psi)\|
\\
&<
\varepsilon_{n+1}|\cK_n|_\sigma^2
+
\varepsilon_{n}|\cK_n|_\sigma^2
\quad
\mbox{by }
(n.\ref{item:gammanpsi}),
\
(n+1,\ref{item:gammanpsi})
\\
&<2\delta_n.
\end{align*}
Thus we obtain ($n+1$,\ref{item:wnpsi}).

Put
$u_X^{n+1}:=v_X^{n+1}\gamma_X^{n-1}(w_{n+1})w_{n+1}^*$
for $X\in\sC$.
Then the condition ($n+1$,\ref{item:gammanun}) holds,
and we have
\begin{align*}
&\sum_{X\in\cF_{n}}
d(X)^2
|u_X^{n+1}
-1|_{\Ad w_{n+1} \circ \gamma_X^{n-1}\circ\Ad w_{n+1}^*(\varphi)}
\\
&=
\sum_{X\in\cF_{n}}
d(X)^2
|v_X^{n+1}\gamma_X^{n-1}(w_{n+1})
-w_{n+1}|_{\gamma_X^{n-1}(w_{n+1}^*\varphi w_{n+1})}
\\
&\leq
\sum_{X\in\cF_{n}}
d(X)^2
|v_X^{n+1}\gamma_X^{n-1}(w_{n+1})
-w_{n+1}|_{\gamma_X^{n-1}(\varphi)}
+
\sum_{X\in\cF_{n}}
2d(X)^2
\|w_{n+1}^*\varphi w_{n+1}-\varphi\|
\\
&<
12/4^{n}
+4\delta_n|\cF_n|_\sigma
\quad
\mbox{by }
(\ref{eq:vgawn}),
\
\varphi\in\Phi_n\subset \Psi_n
\\
&<
13/4^n,
\end{align*}
which shows ($n+1$,\ref{item:sumuX}).
We set $\theta_{n+1}$ and $\ovl{u}^{n+1}$
as in
($n+1$.\ref{item:thetan})
and
($n+1$.\ref{item:ovlu}).
Then our inductive construction is done.

\setcounter{clam}{0}
\begin{clam}
\label{clam:theta10}
The limits $\ovl{\theta}_{-1}:=\lim_{m\to\infty}\theta_{2m+1}$
and $\ovl{\theta}_0:=\lim_{m\to\infty}\theta_{2m}$ exist in $\Aut(M)$
with respect to the $u$-topology.
\end{clam}
\begin{proof}
It follows from ($n$,\ref{item:wnpsi}) and ($n$.\ref{item:thetan}).
\end{proof}

It is not difficult to show the following claim by induction.
Denote by $\epsilon$ the quotient map from $\Z$
onto $\Z/2\Z$ that is identified with $\{-1,0\}$.

\begin{clam}
\label{clam:gautheta}
The unitaries $\ovl{u}^n$ perturbs
the cocycle action
$(\theta_n\circ\gamma^{\epsilon(n)}\circ\theta_n^{-1},
\theta_n(c^{\epsilon(n)}))$
to
the cocycle action $(\gamma^n,c^{n})$
for $n\geq1$.
\end{clam}

\begin{clam}
\label{clam:ovthgamma}
The following inequality hold:
\[
\|\ovl{\theta}_{\epsilon(n)}\circ\gamma_X^{\epsilon(n)}
\circ\ovl{\theta}_{\epsilon(n)}^{-1}(\varphi)
-\theta_{n-2}\circ\gamma_X^{-1}\circ\theta_{n-2}^{-1}(\varphi)\|
\leq
8/16^{n-1}
\quad
\mbox{ for all }
X\in\cF_{n-1},\ n\geq1.
\]
\end{clam}
\begin{proof}[Proof of Claim \ref{clam:ovthgamma}]
We have the following for odd $n\geq1$ and $X\in\cF_{n-1}$:
\begin{align*}
&\|\theta_n\circ\gamma_X^{-1}\circ\theta_n^{-1}(\varphi)
-\theta_{n-2}\circ\gamma_X^{-1}\circ\theta_{n-2}^{-1}(\varphi)\|
\\
&\leq
\|\theta_n^{-1}(\varphi)-\theta_{n-2}^{-1}(\varphi)\|
+
\|\theta_{n}\circ\gamma_X^{-1}\circ\theta_{n-2}^{-1}(\varphi)
-\theta_{n-2}\circ\gamma_X^{-1}\circ\theta_{n-2}^{-1}(\varphi)
\|
\\
&=
\|w_n\varphi w_n^*-\varphi\|
+
\|w_n \ovl{u}_X^{n-2*}\gamma_X^{n-2}(\varphi)
\ovl{u}_X^{n-2}w_n^*
-\ovl{u}_X^{n-2*}\gamma_X^{n-2}(\varphi)
\ovl{u}_X^{n-2}\|
\\
&<
4\delta_{n-1}
\quad
\mbox{by }
(n,\ref{item:wnpsi}),
\
\varphi,\ovl{u}_X^{n-2*}\gamma_X^{n-2}(\varphi)
\ovl{u}_X^{n-2}\in\Psi_{n-1}.
\end{align*}
We obtain a similar estimate for even $n\geq2$,
and we are done.
\end{proof}

\begin{clam}
\label{clam:limhatu}
The strong$*$ limits
$\hat{u}_X^{-1}:=\lim_m \ovl{u}_X^{2m+1}$
and $\hat{u}_X^0:=\lim_m\ovl{u}_X^{2m}$
exist for all $X\in\sC_0$.
\end{clam}
\begin{proof}[Proof of Claim \ref{clam:limhatu}]
In the following,
when two functionals $\psi_1,\psi_2\in M_*$
satisfy
$\|\psi_1-\psi_2\|\leq\varepsilon$ for an $\varepsilon>0$,
we write
$\psi_1\approx_\varepsilon \psi_2$.
Then for $X\in\cF_{n-1}$,
we have
\begin{align*}
&\ovl{u}_X^n
\cdot
\theta_{n-2}\circ\gamma_X^{\epsilon(n)}\circ\theta_{n-2}^{-1}(\varphi)
\\
&=
u_X^n w_n\ovl{u}_X^{n-2}w_n^*
\cdot
\ovl{u}_X^{n-2*} \gamma_X^{n-2}(\varphi)\ovl{u}_X^{n-2}
\\
&
\approx_{2\delta_{n-1}}
u_X^n w_n\ovl{u}_X^{n-2}
\cdot
\ovl{u}_X^{n-2*} \gamma_X^{n-2}(\varphi)\ovl{u}_X^{n-2}
w_n^*
\quad
\mbox{by }
(n.\ref{item:wnpsi}),
\
\ovl{u}_X^{n-2*} \gamma_X^{n-2}(\varphi)\ovl{u}_X^{n-2}
\in\Psi_{n-1}
\\
&=
u_X^n w_n
\gamma_X^{n-2}(\varphi)\ovl{u}_X^{n-2}
w_n^*
\\
&\approx_{2\delta_{n-1}}
u_X^n w_n
\gamma_X^{n-2}(w_n^*\varphi w_n)\ovl{u}_X^{n-2}
w_n^*
\quad
\mbox{by }
(n.\ref{item:wnpsi})
\\
&\approx_{6/2^{n-1}}
w_n
\gamma_X^{n-2}(w_n^*\varphi w_n)\ovl{u}_X^{n-2}
w_n^*
\quad
\mbox{by }
(n,\ref{item:sumuX})
\\
&\approx_{2\delta_{n-1}}
w_n
\gamma_X^{n-2}(\varphi)\ovl{u}_X^{n-2}
w_n^*
\quad
\mbox{by }
(n.\ref{item:wnpsi})
\\
&\approx_{2\delta_{n-1}}
\gamma_X^{n-2}(\varphi)\ovl{u}_X^{n-2}
\quad
\mbox{by }
(n.\ref{item:wnpsi}),
\
\gamma_X^{n-2}(\varphi)\ovl{u}_X^{n-2}
\in\Psi_{n-1}
\\
&=
\ovl{u}_X^{n-2}
\cdot
\theta_{n-2}\circ\gamma_X^{\epsilon(n)}\circ\theta_{n-2}^{-1}(\varphi).
\end{align*}
Thus we have
\[
\ovl{u}_X^n
\cdot
\theta_{n-2}\circ\gamma_X^{\epsilon(n)}\circ\theta_{n-2}^{-1}(\varphi)
\approx_{14/2^{n-1}}
\ovl{u}_X^{n-2}
\cdot
\theta_{n-2}\circ\gamma_X^{\epsilon(n)}\circ\theta_{n-2}^{-1}(\varphi).
\]
From Claim \ref{clam:ovthgamma} and the inequality above,
we have the following: for all $X\in\cF_{n-1}$ and odd $n\geq1$.
\[
\ovl{u}_X^n
\ovl{\theta}_{-1}\circ\gamma_X^{-1}\circ\ovl{\theta}_{-1}^{-1}(\varphi)
\approx_{14/2^{n-1}+8/16^{n-1}}
\ovl{u}_X^{n-2}
\ovl{\theta}_{-1}\circ\gamma_X^{-1}\circ\ovl{\theta}_{-1}^{-1}(\varphi).
\]
Since
$\ovl{\theta}_{-1}\circ\gamma_X^{-1}\circ\ovl{\theta}_{-1}^{-1}(\varphi)$
is a faithful normal state on $M$,
we see $\ovl{u}_X^{2m+1}$
converges in the strong topology for all $X\in\Irr(\sC)$.
By similar estimate of
$\ovl{\theta}_{-1}
\circ\gamma_X^{-1}\circ\ovl{\theta}_{-1}^{-1}(\varphi)
\ovl{u}_X^n$,
we can show the convergence in the strong$*$ topology.
We can also show the strong$*$ convergence of $\ovl{u}_X^{2m}$
for all $X\in\Irr(\sC)$.

For a general $X\in\sC_0$,
we have
\[
\ovl{u}_X^n=\sum_{Z\in\Irr(\sC)}\sum_{T\in\ONB(Z,X)}
T^{\gamma^n}
\ovl{u}_Z^n
w_n T^{\gamma^{n-2}*}w_n^*.
\]
It follows from ($n$.\ref{item:Tgammanpsi}) and ($n$.\ref{item:wnpsi})
that each $T^{\gamma^{n}}$
has the strong$*$ limit
and $w_n$ is a central sequence.
From this we see that $\ovl{u}_X^n$ has the strong$*$ limit.
\end{proof}

Then by ($n$.\ref{item:gammanpsi}),
Claim \ref{clam:gautheta} and Claim \ref{clam:limhatu},
we have
\[
\Ad\hat{u}_X^{-1}\circ\ovl{\theta}_1\circ
\gamma_X^{-1}\circ\ovl{\theta}_1^{-1}
=
\Ad\hat{u}_X^0\circ\ovl{\theta}_0\circ\gamma_X^{0}\circ\ovl{\theta}_0^{-1}
\quad
\mbox{for all }
X\in\sC_0.
\]
The condition ($n$.\ref{item:cocyclen})
shows $c_{X,Y}^n$ converges in the strong $*$-topology
for each $X,Y\in\sC_0$,
and
we see that
$\hat{u}_X^{-1}$
and
$\hat{u}_X^{0}$
perturb the 2-cocycles
$\ovl{\theta}_{-1}(c^{-1})$
and
$\ovl{\theta}_{0}(c^{0})$
to the common 2-cocycle.
On morphisms,
using the equality
$T^{\gamma^n}=\ovl{u}_Y^n \theta_n(T^{\epsilon(n)})\ovl{u}_X^{n*}$
for all $X,Y\in\sC_0$ and $T\in\sC(X,Y)$ and $n\geq1$,
we have
\[
\hat{u}_Y^{-1} \ovl{\theta}_1(T^{\gamma^{-1}})\hat{u}_X^{-1}
=
\hat{u}_Y^0 \ovl{\theta}_0(T^{\gamma^{0}})\hat{u}_X^0.
\]
Finally, we can extend the domain of $\hat{u}^{-1}$ and $\hat{u}^0$
from $\sC_0$ to $\sC$ as in the proof of Lemma \ref{lem:cocapprox}.
\end{proof}

We will state a generalization of the preceding works
by Izumi and Masuda
\cite[Theorem 2.2]{Iz-near}
and \cite[Theorem 3.4]{Mas-Rob}
to possibly non-injective von Neumman algebras.

A unitary tensor functor $(F,L)$ from a strict C$^*$-tensor category
$\sC$ into a strict C$^*$-tensor category $\sD$
means $F\colon \sC\to\sD$ is a C$^*$-functor
between C$^*$-categories
and $L=(L_{X,Y})_{X,Y\in\sC}$ are natural unitary isomorphisms
such that $L_{X,Y}\colon F(X\otimes Y)\to F(X)\otimes F(Y)$.

\begin{cor}
\label{cor:CDclass}
Let $\sC$ and $\sD$ be amenable rigid C$^*$-tensor categories
with a unitary tensor equivalence $(F,L)$ from $\sC$ into $\sD$.
Let $(\alpha,c^\alpha)$ and $(\beta,c^\beta)$ be centrally free cocycle actions
of $\sC$ and $\sD$ on a properly infinite
von Neumann algebra $M$ with separable predual,
respectively
such that $\alpha_X$ and $\beta_{F(X)}$ are approximately unitarily equivalent
for all $X\in\sC$
and $\alpha$ has an invariant faithful normal state on $Z(M)$.
Then there exist
a family of unitaries $v=(v_X)_{X\in\sC}$ in $M$
and an approximately inner automorphism $\theta$ on $M$
such that
\begin{itemize}
\item
$\Ad v_X\circ\alpha_X=\theta\circ\beta_{F(X)}\circ\theta^{-1}$
for all $X\in \sC$.

\item
$v_X \alpha_X(v_Y) c_{X,Y}^\alpha v_{X\otimes Y}^*
=\theta(c_{F(X),F(Y)}^\beta [L_{X,Y}]^\beta)$
for all $X,Y\in\sC$.

\item
$v_Y T^\alpha=\theta(F(T)^\beta)v_X$
for all $X,Y\in\sC$ and $T\in\sC(X,Y)$.
\end{itemize}
\end{cor}
\begin{proof}
We introduce a new cocycle action $(\gamma,c^\gamma)$
of $\sC$ on $M$ as follows:
\begin{itemize}
\item 
$\gamma_X:=\beta_{F(X)}$ for all $X\in\sC$.

\item
$c_{X,Y}^\gamma:=c_{F(X),F(Y)}^\beta [L_{X,Y}]^\beta$
for all $X,Y\in\sC$.

\item
$T^\gamma:=F(T)^\beta$ for all $X,Y\in\sC$ and $T\in\sC(X,Y)$.
\end{itemize}
From Theorem \ref{thm:classification},
$(\alpha,c^\alpha)$ and  $(\gamma,c^\gamma)$ are strongly cocycle conjugate.
Namely, there exist a family of unitaries $v=(v_X)_{X\in\sC}$ in $M$
and an approximately inner automorphism $\theta$ on $M$
such that
\begin{itemize}
\item
$\Ad v_X\circ\alpha_X=\theta\circ\gamma_X\circ\theta^{-1}$
for all $X\in \sC$.

\item
$v_X \alpha_X(v_Y) c_{X,Y}^\alpha v_{X\otimes Y}^*
=\theta(c_{X,Y}^\gamma)$
for all $X,Y\in\sC$.

\item
$v_Y T^\alpha=\theta(T^\gamma)v_X$
for all $X,Y\in\sC$ and $T\in\sC(X,Y)$.
\end{itemize}
Thus we are done.
\end{proof}

An action of a C$^*$-tensor category on a finite von Neumann algebra
is introduced in terms of bimodule categories.
Readers are referred to \cite[Appendix A]{Mas-Rob}.
The following result is an immediate corollary of
Theorem \ref{thm:classification}.
This generalizes the main results of \cite{MT-minimal,MT-III,MT-discrete}.

\begin{cor}
Let $G$ be a coamenable compact quantum group of Kac type
with its complete set of irreducible representations $\Irr(G)$
being at most countable.
Let $(\alpha,c^\alpha)$ and $(\beta,c^\beta)$
be centrally free cocycle actions of $\widehat{G}$
on a von Neumann algebra $M$ with separable predual.
Suppose that $\alpha$ has an invariant faithful normal state on $Z(M)$
and $(\alpha,c^\alpha)$ and $(\beta,c^\beta)$
are approximately unitarily equivalent.
Then they are strongly cocycle conjugate.
\end{cor}
\begin{proof}
If $M$ is properly infinite,
then we have nothing to prove.
Let us consider the case of $M$ being finite.
Let $\tau$ be a faithful normal tracial state on $M$
such that $\tau$ is $\alpha$-invariant on $Z(M)$.
Let us use notations introduced in \cite{MT-minimal}.
The freeness of $\alpha$ implies $\alpha_X(\tau)=\tau\otimes\tau_X$
for all $X\in\Irr(G)$,
where $\tau_X$ is the tracial state on $B(H_X)$.
This enables us to use the trick of taking the tensor product
$B(\ell^2)\otimes M$.
Hence $(\id\otimes\alpha,1\otimes c^\alpha)$
and $(\id\otimes\beta,1\otimes c^\beta)$
are strongly cocycle conjugate.
Then the intertwining automorphism is approximately inner,
and we obtain the strong cocycle conjugacy of
$(\alpha,c^\alpha)$ and $(\beta,c^\beta)$.
\end{proof}

\begin{rem}
Theorem \ref{thm:classification}
is actually sufficient to show the uniqueness part of
Popa's classification of strongly amenable subfactors
stated in \cite[Theorem 5.1]{Popa-endo}
(see Theorem \ref{thm:M1P1isom}).
We, however, postpone to present our proof
since it is better to understand a subfactor
as an action of a rigid C$^*$-2-category.
\end{rem}

The following result is a direct consequence
of the realization result due to Hayashi--Yamagami.
Note we do not assume the finitely generating property
of the set of simple objects.
Let us use the formulation in terms of endomorphisms.

\begin{prop}
\label{prop:HY-real}
Let $\sC$ be an amenable rigid C$^*$-tensor category
and $M$ the injective type II$_\infty$ factor with separable predual.
Then there exists a free cocycle
action $(\alpha,c)$ of $\sC$ on $M$ such that
$\alpha_X$ is an approximately inner endomorphism of rank $d(X)$
for all $X\in\Irr(\sC)$.
\end{prop}
\begin{proof}
Let $N$ be the injective type II$_1$ factor with separable predual.
Thanks to \cite[Theorem 7.6]{HY},
we have a fully faithful unitary tensor functor
$F$ from $\sC$ into ${}_N\sB_N$,
where ${}_N\sB_N$ denotes the bimodule category of $N$
whose objects are $N$-$N$-bimodules with finite indices.
By taking the tensor product $N\otimes B(\ell^2)=:M$
as explained in \cite[Appendix]{Mas-Rob},
we have a fully faithful unitary tensor functor $\alpha\colon \sC\to\End(M)_0$.
In other words,
the functor $\alpha$ is a free cocycle action of $\sC$
with a 2-cocycle $c$.
Note that the freeness and the central freeness are equivalent properties
for the injective type II$_\infty$ factors
(see \cite[Theorem 4.12]{MT-app} or Corollary \ref{cor:modapp}).
We will adjust the rank of the approximate innerness of $\alpha$ as follows.

Let $\tau$ be a faithful normal tracial weight on $M$.
For each $X\in\Irr(\sC)$, $(\alpha_X,\alpha_X)=\alpha(\sC(X,X))=\C$,
and we have a scalar $\lambda(X)>0$, which is called the module of $\alpha_X$,
such that $\alpha_X(\tau)=\lambda(X)d(X)^{-1}\tau$,
where $\alpha_X(\tau):=\tau\circ\phi_X^\alpha$.
It is not difficult to see $\lambda(X)\lambda(Y)=\lambda(Z)$
for $X,Y,Z\in\Irr(\sC)$ with $Z\prec X\otimes Y$.
Let $\Gamma$ be the subgroup generated by $\lambda(X)$ with $X\in\Irr(\sC)$
inside the multiplicative group $\R_+^*$ with the discrete topology.
Take a trace scaling action $\beta$ of $\R_+^*$ on $M$
with $\beta_t(\tau)=t^{-1}\tau$ for $t\in\R_+^*$.
Then we consider
$\gamma_X:=\beta_{\lambda(X)}\otimes \alpha_X \in \End(M\otimes M)_0$
for $X\in\Irr(\sC)$.
For a general $X\in\sC$,
we set
\[
\gamma_X(x):=
\sum_{Y\in\Irr(\sC)}\sum_{T\in\ONB(Y,X)}(1\otimes T^\alpha)
\gamma_Y(x)(1\otimes T^{\alpha*})
\quad
\mbox{for }
x\in M\otimes M.
\]
For a morphism $T$ in $\sC$,
we set $T^\gamma:=1\otimes T^\alpha$.
Then we see $\gamma$ is a free cocycle action of $\sC$ on $M\otimes M$
with the 2-cocycle $1\otimes c$.
By definition, we have $\gamma_X(\tau\otimes\tau)=d(X)^{-1}\tau\otimes\tau$
for all $X\in\Irr(\sC)$.
It turns out from \cite[Theorem 3.15]{MT-app} that
$\gamma_X$ is an approximately inner endomorphism of rank $d(X)$.
\end{proof}

\section{Classification of centrally free cocycle actions of amenable rigid C$^*$-2-categories}

In this section,
we will discuss a cocycle action of a rigid C$^*$-2-categories
on a system of properly infinite von Neumann algebras.
Our references for 2-categories
are \cite{Bena,Hayas-Real,Lack,NY-Tube,Yam-Frobalg},
but we use the tensor product notation to denote the compositions
of 1-morphisms.
Let $\Lambda$ be a set and $\sC:=(\sC_{rs})_{r,s\in\Lambda}$
a rigid C$^*$-2-category.
Namely, $\sC_{rs}$ denotes a non-zero C$^*$-category whose object
is called a 1-morphism $s\to r$.
These $\sC_{rs}$ have the bifunctors
$\otimes$ from $\sC_{rs}\times\sC_{st}$ into $\sC_{rt}$
which are assumed to be strict,
that is, we consider the associators are the identity maps.
Each $\sC_{rr}$ has the tensor unit $\btr$ for the bifunctor $\otimes$.
We assume $\sC_{rr}(\btr,\btr)=\C1$ for all $r\in\Lambda$.
Each object $X\in\sC_{rs}$ has a conjugate object $\ovl{X}\in\sC_{sr}$.
Let us fix a standard solution of the conjugate equations
$(R_X,\ovl{R}_X)$ as in Section \ref{subsect:tensorcat}.
By $d(X)$ we denote the intrinsic dimension of $X$.
Let us denote by $\Irr(\sC)$ the disjoint union of $\Irr(\sC_{rs})$
with $r,s\in\Lambda$ which are assumed to be at most countable.
The measure $\sigma$ on $\Irr(\sC)$ is defined by
$\sigma(X):=d(X)^2$ as before.
In this section,
we only consider the case of $\Lambda$ being the two-point set $\{0,1\}$.

\subsection{Cocycle actions of rigid C$^*$-2-categories}
We will introduce the notion of a cocycle action
of a rigid C$^*$-2-category on a system of properly infinite von Neumann algebras.
Recall our notation $\Mor(N,M)_0$
introduced in Section \ref{subsect:notation}.

\begin{defn}
Let $M:=(M_r)_{r\in\Lambda}=(M_0,M_1)$
be a system of properly infinite von Neumann algebras.
A \emph{cocycle action} $(\alpha,c)$ of $\sC$
on $M$
consists of 
a family of unitary tensor functors
$\alpha=(\alpha_{rs})_{r,s\in\Lambda}$
with
$\alpha^{rs}\colon \sC_{rs}\to \Mor(M_s,M_r)_0$,
unitary elements $c_{X,Y}\in M_r$
with $(X,Y)\in\sC_{rs}\times \sC_{st}$
and $T^\alpha\in (\alpha_X^{rs},\alpha_Y^{rs})\subset M_r$
for $T\in\sC_{rs}(X,Y)$
such that
\begin{itemize}
\item
$\alpha_X^{rs}\circ\alpha_Y^{st}=\Ad c_{X,Y}\circ\alpha_{X\otimes Y}^{rt}$
for $(X,Y)\in\sC_{rs}\times\sC_{st}$ with $r,s,t\in\Lambda$.

\item
$c_{X,Y}c_{X\otimes Y,Z}=\alpha_X^{rs}(c_{Y,Z})c_{X,Y\otimes Z}$
for $(X,Y,Z)\in\sC_{rs}\times\sC_{st}\times\sC_{tu}$
with $r,s,t,u\in\Lambda$.

\item
$c_{X,Y}[S\otimes T]^{\alpha}
=S^{\alpha}\alpha_V^{rs}(T^{\alpha})c_{V,W}$
for $(X,Y,V,W)\in \sC_{rs}\times\sC_{st}\times\sC_{rs}\times\sC_{st}$,
$S\in\sC_{rs}(V,X)$ and $T\in\sC_{st}(W,Y)$.
\end{itemize}
\end{defn}

We will simply write $\alpha$ for $\alpha^{rs}$ if there is no danger
of confusion.
Let $(\alpha,c)$ be as above
and $v_X\in M_r$ unitaries for $X\in\sC_{rs}$
with $r,s\in\Lambda$.
Then the \emph{perturbed cocycle action}
$(\alpha^v,c^v)$ on $M=(M_r)_{r\in\Lambda}$
is defined as follows:
\begin{itemize}
\item
$\alpha_X^v:=\Ad v_X\circ\alpha_X$
for $X\in\sC_{rs}$ with $r,s\in\Lambda$.

\item
$c_{X,Y}^v:=v_X\alpha_X(v_Y)c_{X,Y}v_{X\otimes Y}^*$
for $(X,Y)\in\sC_{rs}\times\sC_{st}$
with $r,s,t\in\Lambda$.

\item
$T^{\alpha^v}:=v_Y T^\alpha v_X^*$
for $X,Y\in\sC_{rs}$
and $T\in\sC_{rs}(X,Y)$
with $r,s\in\Lambda$.
\end{itemize}

\begin{defn}
Two cocycle actions $(\alpha,c^\alpha)$ and $(\beta,c^\beta)$
of a rigid C$^*$-2-category
$\sC=(\sC_{r,s})_{r,s\in\Lambda}$
on $M=(M_r)_{r\in\Lambda}$
are said to be
\begin{itemize}
\item
\emph{conjugate}
if there exists a system of automorphisms
$\theta=(\theta_{r}\colon M_r\to M_r)_{r\in\Lambda}$
such that
\begin{itemize}
\item
$\beta_X=\theta_r\circ\alpha_X\circ\theta_s^{-1}$
for $X\in\sC_{rs}$ with $r,s\in\Lambda$,

\item
$c_{X,Y}^\beta=\theta_r(c_{X,Y}^\alpha)$
for $(X,Y)\in\sC_{rs}\times\sC_{st}$
with $r,s,t\in\Lambda$,

\item
$T^\beta=\theta_r(T^\alpha)$
for $X,Y\in\sC_{rs}$
and $T\in\sC_{rs}(X,Y)$
with $r,s\in\Lambda$;
\end{itemize}

\item
\emph{cocycle conjugate}
if
there exists a unitary perturbation
of $(\alpha,c^\alpha)$ that is conjugate to $(\beta,c^\beta)$;

\item
\emph{strongly cocycle conjugate}
if there exists a unitary perturbation
of $(\alpha,c^\alpha)$ that is conjugate to $(\beta,c^\beta)$
as above
such that each $\theta_r$ is approximately inner.
\end{itemize}
\end{defn}

\subsection{Amenability}
We will say that $\sC=(\sC_{rs})_{r,s\in\Lambda}$
is \emph{amenable}
if the C$^*$-tensor category $\sC_{00}$ is amenable.
This implies the amenability of $\sC_{11}$ as well
(see \cite[Proposition 1.3.4]{Hayas-Real}).

For $\cF\subset \Irr(\sC_{rs})$
and $\cK\subset \Irr(\sC_{st})$,
$\cF\cdot \cK$ denotes the collection of all $Z\in\Irr(\sC_{rt})$
such that $Z\prec X\otimes Y$ for some $X\in\cF$ and $Y\in\cK$.
The following elementary lemma is probably well-known among experts,
but we will present a proof for readers' convenience.

\begin{lem}
\label{lem:prodfinsets}
The measure $\sigma$ on $\Irr(\sC)$ is submultiplicative.
Namely, let
$\cF\subset \Irr(\sC_{rs})$
and $\cG\subset \Irr(\sC_{st})$
finite subsets with $r,s,t\in\Lambda$.
Then one has the inequality
$|\cF\cdot\cG|_\sigma\leq|\cF|_\sigma|\cG|_\sigma$.
\end{lem}
\begin{proof}
We have
$|\cF\cdot\cG|_\sigma\leq\sum_{(X,Y)\in\cF\times\cG}|\{X\}\cdot\{Y\}|_\sigma$
by subadditivity of the measure $\sigma$.
For $(X,Y)\in\cF\times\cG$,
we have
$d(X)d(Y)=\sum_{Z\in \{X\}\cdot\{Y\}}N_{X,Y}^Z d(Z)$,
and
$|\{X\}\cdot\{Y\}|_\sigma
=\sum_{Z\in \{X\}\cdot\{Y\}}d(Z)^2\leq d(X)^2 d(Y)^2$.
Thus we are done.
\end{proof}

\begin{lem}
\label{lem:FrsKrs}
For any
finite subsets
$\cF_{rs}=\ovl{\cF_{sr}}$
in $\Irr(\sC_{rs})$ with $r,s\in\Lambda$
such that $\btr\in \cF_{00}$ and $\btr\in\cF_{11}$
and $\delta>0$,
there exist finite subsets
$\cK_{00}\subset\Irr(\sC_{00})$
and
$\cK_{10}\subset\Irr(\sC_{10})$
such that
\begin{itemize}
\item
$\cK_{00}\subset \cF_{01}\cdot\cK_{10}$
and
$\cF_{10}\cdot\cK_{00}\subset\cK_{10}$.

\item
$|(\cF_{00}\cdot\cK_{00})\setminus\cK_{00}|_\sigma
<\delta|\cK_{00}|_\sigma$.

\item
$|(\cF_{01}\cdot\cK_{10})\setminus\cK_{00}|_\sigma
<\delta|\cK_{00}|_\sigma$.

\item
$|(\cF_{11}\cdot\cK_{10})\setminus\cK_{10}|_\sigma
<\delta|\cK_{10}|_\sigma$.
\end{itemize}
In particular,
one has $|(\cF_{rs}\cdot\cK_{s0})\setminus\cK_{r0}|_\sigma
<\delta|\cK_{r0}|_\sigma$
for all $r,s\in\Lambda$.
\end{lem}
\begin{proof}
For such $\cF_{rs}$'s,
we can take a finite subset $\cG_{10}$ of $\Irr(\sC_{10})$
so that, putting $\cG_{01}:=\ovl{\cG_{10}}$,
we have $\cF_{00}\subset \cG_{01}\cdot \cG_{10}$,
$\cF_{10}\subset \cG_{10}$,
$\cF_{01}\subset \cG_{01}$
and
$\cF_{11}\subset \cG_{10}\cdot \cG_{01}$.

Let $\varepsilon>0$ such that $\varepsilon<\delta|\cG_{10}|_\sigma^{-2}$.
By amenability of $\sC_{00}$,
we can take a $(\cG_{01}\cdot \cG_{10},\varepsilon)$-invariant
finite subset $\cK_{00}$ of $\Irr(\sC_{00})$.
Put $\cK_{10}:=\cG_{10}\cdot\cK_{00}$.
Then we have
\[
\cK_{00}\subset
\cF_{01}\cdot\ovl{\cF_{01}}\cdot\cK_{00}
\subset
\cF_{01}\cdot\cK_{10}
\subset
\cG_{01}\cdot\cG_{10}\cdot\cK_{00},
\]
and
$|(\cF_{01}\cdot\cK_{10})\setminus \cK_{00}|_\sigma
<\varepsilon|\cK_{00}|_\sigma$.
Next we have
\[
\cK_{10}\subset
\cF_{11}\cdot \cK_{10}
=
\cF_{11}\cdot \cG_{10}\cdot \cK_{00}
\subset
\cG_{10}\cdot\cG_{01}\cdot \cG_{10}\cdot \cK_{00},
\]
and
\[
(\cF_{11}\cdot \cK_{10})\setminus \cK_{10}
\subset
\cG_{10}\cdot
\big{(}(\cG_{01}\cdot \cG_{10}\cdot \cK_{00})\setminus\cK_{00}\big{)}.
\]
Hence
\begin{align*}
|(\cF_{11}\cdot \cK_{10})\setminus \cK_{10}|_\sigma
&\leq
|\cG_{10}|_\sigma
|(\cG_{01}\cdot \cG_{10}\cdot \cK_{00})\setminus\cK_{00}|_\sigma
\quad
\mbox{by Lemma }
\ref{lem:prodfinsets}
\\
&\leq
|\cG_{10}|_\sigma
\cdot\varepsilon
|\cK_{00}|_\sigma
\\
&\leq
\varepsilon|\cG_{10}|_\sigma^2|\cK_{10}|_\sigma,
\end{align*}
where, in the last inequality,
we have used $\cK_{00}\subset \cG_{01}\cdot \cK_{10}$
and Lemma \ref{lem:prodfinsets}.
\end{proof}

\subsection{Freeness, Central freeness and the Rohlin property}
Let $(\alpha,c)$ be an action of a rigid C$^*$-2-category
$\sC=(\sC_{rs})_{r,s\in\Lambda}$
on a system of properly infinite
von Neumann algebras $M=(M_r)_{r\in\Lambda}$
whose preduals are not necessarily separable.

Each $\alpha_X$,
$X\in\sC_{rs}$
has
the faithful normal left inverse map
$\phi_X^\alpha \colon M_r\to M_s$ defined as follows:
\[
\phi_X^\alpha(x)
:=R_X^{\alpha*}
c_{\ovl{X},X}^*\alpha_{\ovl{X}}(x)
c_{\ovl{X},X}R_X^\alpha
\quad
\mbox{for }
x\in M_r.
\]
We put $\alpha_X(\psi):=\psi\circ\phi_X^\alpha$
for $X\in\sC_{rs}$ and $\psi\in (M_s)_*$.
Hence $\alpha_X$ induces a map from $(M_s)_*$
into $(M_r)_*$.
In particular, when $M_r$'s have separable preduals,
the unital normal $*$-homomorphism
$\alpha_X^\omega\colon M_s^\omega\to M_r^\omega$
for each $X\in\sC_{rs}$, $r,s\in\Lambda$
is well-defined.
The computation rules
(\ref{eq:alpha-bimod-varphi})
and
(\ref{eq:alxphi}) are also available.

\begin{lem}
\label{lem:free2}
Let $(\alpha,c)$ be a cocycle action of $\sC$ on $M$
as before.
Then the following equalities are equivalent:
\begin{enumerate}
\item
The cocycle action
$(\alpha^{00},c)$ of $\sC_{00}$ on $M_0$ is free.

\item
$(\alpha_X,\alpha_Y)=\alpha(\sC_{01}(X,Y))Z(M_0)$
for all $X,Y\in \sC_{01}$.

\item
$(\alpha_X,\alpha_Y)=\alpha(\sC_{10}(X,Y))\alpha_X(Z(M_0))$
for all $X,Y\in \sC_{10}$.
\end{enumerate}
In particular, if one of the equivalent conditions above holds,
then
$Z(M_1)\subset\alpha_X(Z(M_0))$
and
$\alpha_Y(Z(M_1))\subset Z(M_0)$
for all $(X,Y)\in\Irr(\sC_{10})\times\Irr(\sC_{01})$.
\end{lem}
\begin{proof}
(1) $\Rightarrow$ (2).
It suffices to show the equality for simple $X$ and $Y$.
It is trivial that
the right-hand side is contained in
the left-hand side.
Let $a\in(\alpha_X,\alpha_Y)$.
Then $\ovl{R}_Y^{\alpha*}c_{Y,\ovl{Y}}^*a c_{X,\ovl{Y}}$
is contained in $(\alpha_{X\otimes\ovl{Y}},\alpha_\btr)$
that is equal to
$\alpha(\sC_{00}(X\otimes\ovl{Y},\btr))Z(M_0)
=\delta_{X,Y} \ovl{R}_X^{\alpha*} Z(M_0)$
from the freeness of $(\alpha^{00},c)$.
Take $b\in Z(M_0)$ so that
$\ovl{R}_Y^{\alpha*}c_{Y,\ovl{Y}}^*a c_{X,\ovl{Y}}
=\delta_{X,Y} \ovl{R}_X^{\alpha*}b$.
Then we can check (2) as follows:
\begin{align*}
d(Y)^{-1}a
&=
\ovl{R}_Y^{\alpha*}c_{Y,\ovl{Y}}^*
\alpha_Y(c_{\ovl{Y},Y}R_Y^\alpha)
a
=
\ovl{R}_Y^{\alpha*}c_{Y,\ovl{Y}}^*
a
\alpha_X(c_{\ovl{Y},Y}R_Y^\alpha)
\\
&=
\delta_{X,Y} \ovl{R}_X^{\alpha*}bc_{X,\ovl{Y}}^*
\alpha_X(c_{\ovl{Y},Y}R_Y^\alpha)
\\
&=
\delta_{X,Y}b\ovl{R}_X^{\alpha*}c_{X,\ovl{X}}^*
\alpha_X(c_{\ovl{X},X}R_X^\alpha)
=
\delta_{X,Y}
d(X)^{-1}b.
\end{align*}

(2) $\Rightarrow$ (3).
Let $X,Y\in\sC_{10}$ and $a\in (\alpha_X,\alpha_Y)$.
Put
$b:=\alpha_{\ovl{X}}(\ovl{R}_Y^{\alpha*}c_{Y,\ovl{Y}}^* a)
c_{\ovl{X},X}R_X^\alpha$.
Then $b$ is contained in
$(\alpha_{\ovl{Y}},\alpha_{\ovl{X}})
=\alpha(\sC_{01}(\ovl{Y},\ovl{X}))Z(M_0)$
from the assumption of (2).
Since
$a=d(X)d(Y) \ovl{R}_X^{\alpha*}c_{X,\ovl{X}}^*
\alpha_X(b c_{\ovl{Y},Y}R_Y^\alpha)$,
it suffices to show
$c:=\ovl{R}_X^{\alpha*}c_{X,\ovl{X}}^*
\alpha_X(T^\alpha z c_{\ovl{Y},Y}R_Y^\alpha)$
is in $\alpha(\sC_{10}(X,Y))\alpha_X(Z(M_0))$
for all $T\in \sC_{01}(\ovl{Y},\ovl{X})$ and $z\in Z(M_0)$.
Indeed, we have
$c=[(\ovl{R}_X^*\otimes 1_Y)(1_X\otimes T\otimes 1_Y)
(1_X\otimes R_Y)]^\alpha \alpha_X(z)$.

(3) $\Rightarrow$ (1).
Suppose that $X\in\Irr(\sC_{00})\setminus\{\btr\}$.
By Lemma \ref{lem:free},
we will check $(\alpha_X,\alpha_\btr)=\{0\}$.
Take $U\in\sC_{10}$ so that $X\prec\ovl{U}\otimes U$.
It suffices to show that
$(\alpha_{\ovl{U}\otimes U},\alpha_\btr)$
equals
$\alpha(\sC_{00}(\ovl{U}\otimes U,\btr))Z(M_0)$.
Let $a\in (\alpha_{\ovl{U}\otimes U},\alpha_\btr)$.
Then $\alpha_U(a)c_{U,\ovl{U}\otimes U}$
is an element of
$(\alpha_{U\otimes \ovl{U}\otimes U},\alpha_U)$
which equals
$\alpha(\sC_{10}(U\otimes \ovl{U}\otimes U,U))
\alpha_{U\otimes \ovl{U}\otimes U}(Z(M_0))$
by assumption (3).
Thus we will compute
$\phi_U^\alpha
(T^\alpha \alpha_{U\otimes \ovl{U}\otimes U}(z)
c_{U,\ovl{U}\otimes U}^*)$
for $T\in \sC_{10}(U\otimes \ovl{U}\otimes U,U)$
and $z\in Z(M_0)$
as follows:
\begin{align*}
\phi_U^\alpha
(T^\alpha \alpha_{U\otimes \ovl{U}\otimes U}(z)
c_{U,\ovl{U}\otimes U}^*)
&=
zR_U^{\alpha*}c_{\ovl{U},U}^*
\alpha_{\ovl{U}}
(T^\alpha
c_{U,\ovl{U}\otimes U}^*)
c_{\ovl{U},U}R_U^\alpha
\\
&=
z
R_U^{\alpha*}[1_{\ovl{U}}\otimes T]^\alpha
c_{\ovl{U},U\otimes \ovl{U}\otimes U}^*
\alpha_{\ovl{U}}
(c_{U,\ovl{U}\otimes U}^*)
c_{\ovl{U},U}R_U^\alpha
\\
&=
z
R_U^{\alpha*}[1_{\ovl{U}}\otimes T]^\alpha
c_{\ovl{U}\otimes U,\ovl{U}\otimes U}^*R_U^\alpha
\\
&=
z[R_U^*(1_{\ovl{U}}\otimes T)
(R_U\otimes 1_{\ovl{U}}\otimes 1_U)]^\alpha.
\end{align*}
Hence $a\in \alpha(\sC_{00}(\ovl{U}\otimes U,\btr))Z(M_0)$.
\end{proof}

\begin{lem}
\label{lem:00-11free}
Let $(\alpha,c)$ be a cocycle action of $\sC$ on $M$
as before.
Then the following statements are equivalent:
\begin{enumerate}
\item
$(\alpha^{00},c)$ is free
and $Z(M_0)=\alpha_X(Z(M_1))$
for all $X\in\Irr(\sC_{01})$.

\item
$(\alpha^{00},c)$ is free
and $\alpha_Y(Z(M_0))=Z(M_1)$
for all $Y\in\Irr(\sC_{10})$.

\item
$(\alpha^{00},c)$ and $(\alpha^{11},c)$ are free.
\end{enumerate}
In particular,
one of the equivalent conditions holds,
then the restriction of $\alpha_X$,
which will be denoted by $\theta_X^\alpha$,
induces
an isomorphism from $Z(M_s)$ onto $Z(M_r)$
for $X\in\Irr(\sC_{rs})$ with $r,s\in\Lambda$.
\end{lem}
\begin{proof}
(1) $\Rightarrow$ (2).
Let $Y\in\Irr(\sC_{10})$.
By Lemma \ref{lem:free2},
we see $Z(M_1)\subset \alpha_Y(Z(M_0))$.
We will show the converse inclusion.
Let $x\in Z(M_0)$.
By (1), we can take $y\in Z(M_1)$
with $x=\alpha_{\ovl{Y}}(y)$.
Using the centrality of $x$,
we have $\phi_Y^\alpha(y)=x$
and $\phi_Y^\alpha(y^*y)=x^*x$.
Hence $\alpha_Y(x)=y$ since
$\phi_Y^\alpha
((\alpha_Y(x)-y)^*
(\alpha_Y(x)-y))
=0$.

(2) $\Rightarrow$ (3).
This implication follows from Lemma \ref{lem:free2} by symmetry.
The implication (3) $\Rightarrow$ (1) is trivial
from Lemma \ref{lem:free2}.
\end{proof}

\begin{defn}
\label{defn:free2}
Let $(\alpha,c)$ be a cocycle action of a rigid C$^*$-2-category
on a system of properly infinite von Neumann algebras as before.
We will say $(\alpha,c)$ is \emph{free}
when one of the equivalent conditions in Lemma \ref{lem:00-11free}
holds.
\end{defn}

\begin{rem}
By the previous lemma,
if $(\alpha,c)$ is free,
then $M_0$ is a factor if and only if $M_1$ is.
Note $\alpha$ preserves the type components of von Neumann algebras.
\end{rem}

The following lemma is an application of Lemma \ref{lem:ee}
to $(\alpha^{00},c)$.

\begin{lem}
\label{lem:ee2}
Let $(\alpha,c)$ be a cocycle action
of a rigid C$^*$-2-category $\sC=(\sC_{rs})_{r,s\in\Lambda}$
on a system of properly infinite von Neumann algebras
$M=(M_r)_{r\in\Lambda}$ with separable preduals.
Let $\cF_{r0}$ be a finite subset of $\Irr(\sC_{r0})$
with $r\in\Lambda$.
Suppose that $(\alpha^{00},c)$ is centrally free
and 
there exists a faithful normal state
$\varphi_0\in (M_{0})_*$
being $\alpha^{00}$-invariant on $Z(M_0)$.
Let $Q_0$ be a countably generated von Neumann subalgebra
in $M_0^\omega$
and $\delta>0$.
Then
there exist $n\in\N$ and a partition of unity
$\{e_k\}_{k=0}^n$ in $Q_0'\cap (M_0)_\omega$
such that
\begin{enumerate}
\item
$|e_0|_{\varphi^\omega}<\delta$.

\item
$e_k\alpha_{\ovl{X}\otimes Y}(e_k)=0$
for all $k=1,\dots,n$ and $X,Y\in \cF_{r0}$
with $X\neq Y$ and $r\in\Lambda$.
\end{enumerate}
\end{lem}

The proof of Lemma \ref{lem:ortho} is applicable
to a rigid C$^*$-2-category
and we have the following result.

\begin{lem}
\label{lem:ortho2}
Let $(\alpha,c)$ be a cocycle action
of a rigid C$^*$-2-category $\sC=(\sC_{rs})_{r,s\in\Lambda}$
on a system of properly infinite von Neumann algebras
$N=(N_r)_{r\in\Lambda}$
with not necessarily separable preduals.
Let $\cK_{r0}$ be a finite subset of $\Irr(\sC_{r0})$
with $r\in\Lambda$.
Suppose that a projection $e\in N_0$
satisfies $e\alpha_X(e)=0$
for all $X\in (\ovl{\cK_{r0}}\cdot \cK_{r0})\setminus\{\btr\}$
and moreover
$e$ commutes with all morphisms in $\sC_{00}\sqcup \sC_{01}$
and all unitaries $c_{X,Y}$, $(X,Y)\in\sC_{0t}\times \sC_{tu}$
with $t,u\in\Lambda$.
Then the following statements hold for all $r\in\Lambda$:
\begin{enumerate}
\item
$\alpha_X(e)\alpha_Y(e)=0$
for $X\neq Y\in\cK_{r0}$.

\item
$d(X)^2\phi_{\ovl{X}}^\alpha(e)$
with
$X\in\cK_{r0}$
are mutually
orthogonal projections in $(N_r)_\omega$.

\item
For each $X\in\cK_{r0}$,
$d(X)^2\phi_{\ovl{X}}^\alpha(e)$
is the minimum of projections
$f\in N_r$ such that $f\alpha_X(e)=\alpha_X(e)$,
$fc_{X,\ovl{X}}\ovl{R}_{X}^\alpha
=c_{X,\ovl{X}}\ovl{R}_{X}^\alpha f$.

\item
For each $X\in\cK_{r0}$,
$d(X)^2\alpha_X(e)
c_{\ovl{X},X}\ovl{R}_X^\alpha\ovl{R}_X^{\alpha*}c_{\ovl{X},X}^*\alpha_X(e)
=\alpha_X(e)$.
\end{enumerate}
\end{lem}

Let us state our main theorem of this subsection.

\begin{thm}\label{thm:Rohlin-tensor2}
Let $(\alpha,c)$ be a cocycle action of
an amenable rigid C$^*$-2-category $\sC=(\sC_{rs})_{r,s\in\Lambda}$
on a system of properly infinite von Neumann algebras
$M=(M_r)_{r\in\Lambda}$ with separable preduals.
Suppose the following conditions hold:
\begin{itemize}
\item 
$\alpha_X(Z(M_s))=Z(M_r)$
for $X\in\Irr(\sC_{rs})$ with $r,s\in\Lambda$.

\item
There exist faithful normal states $\varphi_r$ on $M_r$
with $r\in\Lambda$
such that $\alpha_X(\varphi_s)=\varphi_r$ on $Z(M_r)$
for all $X\in\sC_{rs}$ and $r,s\in\Lambda$.

\item
$(\alpha^{00},c)$ is centrally free.
\end{itemize}

Let $0<\delta<1$ and $\cF_{rs},\cK_{r0}$ with $r,s\in\Lambda$
be finite subsets as in Lemma \ref{lem:FrsKrs}.
Then for any countably generated von Neumann subalgebras
$Q_r\subset M_r^\omega$ with $r\in\Lambda$,
there exist two families of projections
$E^{00}:=(E_X^{00})_{X\in\sC_{00}}$ in $M_0^\omega$
and
$E^{10}:=(E_X^{10})_{X\in\sC_{10}}$ in $M_1^\omega$
satisfying the following conditions:
\begin{enumerate}
\item
\label{item:Roh-natural2}
(Naturality)
$E_Y^{r0} T^\alpha=T^\alpha E_X^{r0}$
for all $r\in\Lambda$,
$X,Y\in\sC_{r0}$ and $T\in\sC_{r0}(X,Y)$.

\item
\label{item:Roh-support2}
$E_X^{r0}=0$ for all $r\in\Lambda$,
$X\in\Irr(\sC_{r0})\setminus\cK_{r0}$.

\item
\label{item:Roh-commute2}
$E_X^{r0}\in\alpha_X(Q_0)'\cap {}_{\alpha_X}(M_0)_\omega$
for all $r\in\Lambda$ and $X\in\Irr(\sC_{r0})$,
where
${}_{\alpha_X}(M_0)_\omega:=\alpha_X((M_0)_\omega)\vee (M_r)_\omega$.

\item
(Splitting property)
\label{item:Roh-split2}
$\tau^\omega(E_X^{r0} x)
=\tau^\omega(E_X^{r0})\tau^\omega(x)$
for all $r\in\Lambda$,
$X\in\sC_{r0}$ and $x\in Q_r$.

\item
\label{item:Roh-orthogonal2}
$\{d(X)^2\ovl{R}_X^{\alpha*}c_{X,\ovl{X}}^*
E_X^{r0}
c_{X,\ovl{X}}\ovl{R}_X^\alpha\}_{X\in\cK_{r0}}$
is a family of orthogonal projections
in $Q_r'\cap (M_r)_\omega$.

\item
\label{item:Roh-partition2}
(Approximate partition of unity)
\begin{align*}
&\sum_{X\in\cK_{r0}}\varphi_r^\omega
(d(X)^2\ovl{R}_X^{\alpha*}
c_{X,\ovl{X}}^*E_X^{r0}c_{X,\ovl{X}}\ovl{R}_X^\alpha)
\geq
1-7\delta^{1/2}
\quad
\mbox{for all }
r\in\Lambda.
\end{align*}

\item
\label{item:Roh-equivariance2}
(Approximate joint equivariance)
For all $r,s\in\Lambda$,
\[
\sum_{X\in\cF_{rs}}
\sum_{Y\in\Irr(\sC_{s0})}
d(X)^2d(Y)^2
|\alpha_X(E_Y^{s0})
-
c_{X,Y}E_{X\otimes Y}^{r0}c_{X,Y}^*
|_{\alpha_X(\alpha_Y(\varphi_0^\omega))}
\leq
6\delta^{1/2}|\cF_{rs}|_\sigma.
\]
\item
\label{item:Roh-resonance2}
(Resonance property)
\[E_X^{r0}
c_{X,\ovl{X}}\ovl{R}_X^\alpha
E_Y^{r0}
=
\frac{\delta_{X,Y}}{d(X)}
\alpha_{X}(c_{\ovl{X},X}R_X^\alpha)E_X^{r0}
\]
for all
$X,Y\in\Irr(\sC_{r0})$
and
$r\in\Lambda$.
\end{enumerate}
\end{thm}

\begin{rem}
\label{rem:EZ}
Let $Z\in\Irr(\sC_{r0})$ with $r\in\Lambda$.
Since $(\alpha^{00},c)$ is free,
$\tau^\omega(E_Z^{r0})\in(\alpha_Z,\alpha_Z)=\alpha_Z(Z(M_0))$,
which also equals $Z(M_1)$ by our assumption.
Thus
$\alpha_Z(\varphi_0^\omega)(E_Z^{r0})=\varphi_r^\omega(E_Z^{r0})$.
\end{rem}

Our proof of the previous theorem
is almost parallel to that of Theorem \ref{thm:Rohlin-tensor}.
Let us introduce the set $\mathcal{J}$
that is the collection of a pair of families of projections
$E:=(E^{00},E^{10})$,
where
$E^{r0}=(E_X^{r0})_{X\in\sC_{r0}}$
which satisfies the conditions of
Theorem \ref{thm:Rohlin-tensor2}
(\ref{item:Roh-natural2}),
(\ref{item:Roh-support2}),
(\ref{item:Roh-commute2}),
(\ref{item:Roh-split2}),
(\ref{item:Roh-orthogonal2})
and
(\ref{item:Roh-resonance2}).
Trivially, the pair of $E^{00}=(0)_X$ and $E^{10}=(0)_Y$
is a member of $\mathcal{J}$.
We set the functions $a^{r,s},b^{r}$
with $r,s\in\Lambda$
from $\mathcal{J}$ into $\R_+$
by
\begin{align*}
a_E^{r,s}
&:=
\frac{1}{|\cF_{rs}|_\sigma}
\sum_{X\in\cF_{rs}}
\sum_{Y\in\Irr(\sC_{s0})}
d(X)^2d(Y)^2
|\alpha_X(E_Y^{s0})
-c_{X,Y}E_{X\otimes Y}^{r0}c_{X,Y}^*
|_{\alpha_{X}(\alpha_Y(\varphi_0^\omega))},
\\
b_E^{r}
&:=
\sum_{X\in\cK_{r0}}d(X)^2\varphi_r^\omega(E_X^{r0}).
\end{align*}
Note that we see $b_E^r\leq1$ as before.

\begin{lem}
\label{lem:bsrbrr}
One has
$|b_E^{0}-b_E^{1}|\leq a_E^{1,0}$
for $E\in \mathcal{J}$.
\end{lem}
\begin{proof}
In the defining equality of $a_E^{1,0}$,
$\alpha_X(E_Y^{00})$
and
$c_{X,Y}E_{X\otimes Y}^{10}c_{X,Y}^*$
are
commuting with
$\alpha_X(\alpha_Y(\varphi_0^\omega))$.
Hence
\begin{align*}
|\cF_{10}|_\sigma
a_E^{1,0}
&\geq
\bigg{|}
\sum_{X\in\cF_{10}}
\sum_{Y\in\Irr(\sC_{00})}
d(X)^2d(Y)^2
\alpha_{X}(\alpha_Y(\varphi_0^\omega))
\big{(}
\alpha_X(E_Y^{00})
-c_{X,Y}E_{X\otimes Y}^{10}c_{X,Y}^*
\big{)}
\bigg{|}
\\
&=
\bigg{|}
\sum_{X\in\cF_{10}}
\sum_{Y\in\Irr(\sC_{00})}
d(X)^2d(Y)^2
\varphi_0^\omega(E_Y^{00})
-\alpha_{X\otimes Y}(\varphi_0^\omega)(E_{X\otimes Y}^{10})
\bigg{|}
\\
&=
\bigg{|}
|\cF_{10}|_\sigma b_E^0
-
\sum_{X\in\cF_{10}}
\sum_{Y\in\Irr(\sC_{00})}
d(X)^2d(Y)^2
\alpha_{X\otimes Y}(\varphi_0^\omega)(E_{X\otimes Y}^{10})
\bigg{|}.
\end{align*}
Using the following equality, we are done:
\begin{align*}
&\sum_{X\in\cF_{10}}
\sum_{Y\in\Irr(\sC_{00})}
d(X)^2d(Y)^2
\alpha_{X\otimes Y}(\varphi_0^\omega)(E_{X\otimes Y}^{10})
\\
&=
\sum_{X\in\cF_{10}}
\sum_{Y\in\Irr(\sC_{00})}
\sum_{Z\in\Irr(\sC_{10})}
d(X)^2d(Y)^2
p_X(Y,Z)\alpha_Z(\varphi_0^\omega)(E_Z^{10})
\\
&=
\sum_{X\in\cF_{10}}
\sum_{Z\in\cK_{10}}
d(X)^2d(Z)^2
\alpha_Z(\varphi_0^\omega)(E_Z^{10})
\\
&=
|\cF_{10}|_\sigma b_E^1
\quad
\mbox{by Remark }\ref{rem:EZ}.
\end{align*}
\end{proof}

\begin{lem}
\label{lem:ab2}
Suppose that an element
$E=(E^{00},E^{10})$ of $\mathcal{J}$
satisfies the inequality
$\sum_{r\in\Lambda}|\cK_{r0}|_\sigma b_E^{r}
<(1-\delta^{1/2})\sum_{r\in\Lambda}|\cK_{r0}|_\sigma$.
Then there exists $E'=(E'^{00},E'^{10})$
in $\mathcal{J}$
such that
\begin{enumerate}
\item
$a_{E'}^{r,s}-a_E^{r,s}
\leq 3\delta^{1/2}
\sum_{t\in\Lambda}(b_{E'}^t-b_E^t)$
for all $r,s\in\Lambda$.

\item
$0<(\delta^{1/2}/2)
\sum_{r\in\Lambda}
\sum_{Y\in\cK_{r0}}
d(Y)^2\varphi_r^\omega(|E_Y'^{r0}-E_Y^{r0}|)
\leq
\sum_{r\in\Lambda}
(b_{E'}^{r}-b_E^{r})$.
\end{enumerate}
\end{lem}
\begin{proof}
Take $\delta_1>0$ so that
$\sum_{r\in\Lambda}|\cK_{r0}|_\sigma b_E^{r}
<(1-\delta_1)(1-\delta^{1/2})\sum_{r\in\Lambda}|\cK_{r0}|_\sigma$.
We may and do assume that
each $Q_r$ contains $M_r$
and
$\alpha_X(E^{s0})$
for $s\in\Lambda$ and $X\in\Irr(\sC_{rs})$
and moreover that $\alpha_X(Q_s)\subset Q_r$
for all $X\in\Irr(\sC_{rs})$ and $r,s\in\Lambda$.
We let
$\cL_r:=\cK_{r0}\cup \bigcup_{s\in\Lambda}(\cF_{rs}\cdot\cK_{s0})$
that is a finite subset of $\Irr(\sC_{r0})$.
By Lemma \ref{lem:ee2},
we can take a partition of unity $\{e_t\}_{t=0}^n$
in $Q_0'\cap (M_0)_\omega$
so that
$|e_0|_{\varphi_0^\omega}<\delta_1$
and
$e_k\alpha_X(e_k)=0$
for all
$X\in \bigcup_{r\in\Lambda}(\ovl{\cL_r}\cdot\cL_r)\setminus\{\btr\}$.
By fast reindexation,
we may and do assume that
$\tau^\omega(x\alpha_X(e_k))=\tau^\omega(x)\tau^\omega(\alpha_X(e_k))$
for all $r\in\Lambda$, $x\in Q_r$, $X\in\Irr(\sC_{r0})$
and $k=0,\dots,n$.
Then as Claim 1 in the proof of Lemma \ref{lem:ab},
we can show the following inequality holds for some $k\geq1$:
\[
\sum_{r\in\Lambda}
\sum_{X,Y\in\cK_{r0}}
d(X)^2d(Y)^2
\varphi_r^\omega(E_X^{r0}\phi_{\ovl{Y}}^\alpha(e_k))
<
(1-\delta^{1/2})
\sum_{r\in\Lambda}|\cK_{r0}|_\sigma
\varphi_0^\omega(e_k).
\]
We put $e:=e_k\in Q_0'\cap (M_0)_\omega$.
We let
\[
f_r:=\sum_{X\in\cK_{r0}}
d(X)^2\phi_{\ovl{X}}^\alpha(e)
\quad
\mbox{for }
r\in\Lambda,
\]
which are projections in $Q_r'\cap (M_r)_\omega$.
Then $\varphi_r^\omega(f_r)=|\cK_{r0}|_\sigma\varphi_0^\omega(e)$,
and
\begin{equation}
\label{eq:XKf2}
\sum_{r\in\Lambda}
\sum_{X\in\cK_{r0}}
d(X)^2
\varphi_r^\omega(E_X^{r0} f_r)
<
(1-\delta^{1/2})
\sum_{r\in\Lambda}\varphi_r^\omega(f_r).
\end{equation}

For $X\in\sC_{r0}$ with $r\in\Lambda$,
we put
\[
E_X'^{r0}:=E_X^{r0}f_r^\perp+\alpha_X(e)
[P_X^{\cK_{r0}}]^\alpha.
\]
Then we can see the pair $E':=(E'^{00},E'^{10})$
is indeed the member of $\mathcal{J}$.
We have
\begin{align}
\sum_{r\in\Lambda}
b_{E'}^r
&=
\sum_{r\in\Lambda}
\sum_{X\in\cK_{r0}}
d(X)^2
\varphi_r^\omega(E_X^{r0}f_r^\perp+\alpha_X(e))
\notag
\\
&=
\sum_{r\in\Lambda}
b_{E}^r
-
\sum_{r\in\Lambda}
\sum_{X\in\cK_{r0}}
d(X)^2
\varphi_r^\omega(E_X^{r0}f_r)
+
\sum_{r\in\Lambda}
\varphi_r^\omega(f_r)
\notag
\\
&>
\sum_{r\in\Lambda}
b_{E}^r
-
(1-\delta^{1/2})
\sum_{r\in\Lambda}
\varphi_r^\omega(f_r)
+
\sum_{r\in\Lambda}
\varphi_r^\omega(f_r)
\quad
\mbox{by }
(\ref{eq:XKf2})
\notag
\\
&=
\sum_{r\in\Lambda}
b_{E}^r
+
\delta^{1/2}
\sum_{r\in\Lambda}
\varphi_r^\omega(f_r).
\label{eq:bEE'2}
\end{align}
Now we have
\begin{align*}
\sum_{Y\in\cK_{r0}}
d(Y)^2\varphi_r^\omega(|E_Y'^{r0}-E_Y^{r0}|)
&=
\sum_{Y\in\cK_{r0}}
d(Y)^2\varphi_r^\omega(|-E_Y^{r0}f_r+\alpha_Y(e)|)
\\
&\leq
\sum_{Y\in\cK_{r0}}
d(Y)^2\varphi_r^\omega(E_Y^{r0}f_r+\alpha_Y(e)).
\end{align*}
Applying the splitting property of $\tau^\omega$
to $E_Y^{r0}$ and $f_r$,
we obtain
$\varphi_r^\omega(E_Y^{r0}f_r)
=\varphi_r^\omega
(\ovl{R}_Y^{\alpha*}c_{Y,\ovl{Y}}^*
E_Y^{r0}c_{Y,\ovl{Y}}\ovl{R}_Y^{\alpha}
f_r)$.
Hence
\begin{align*}
&\sum_{Y\in\cK_{r0}}
d(Y)^2\varphi_r^\omega(|E_Y'^{r0}-E_Y^{r0}|)
\\
&\leq
\varphi_r^\omega
(\sum_{Y\in\cK_{r0}}d(Y)^2
\ovl{R}_Y^{\alpha*}
c_{Y,\ovl{Y}}^*
E_Y^{r0}c_{Y,\ovl{Y}}\ovl{R}_Y^{\alpha}f_r)
+
\sum_{Y\in\cK_{r0}}
d(Y)^2\varphi_0^\omega(e)
\\
&\leq
2\varphi_r^\omega(f_r).
\end{align*}
From (\ref{eq:bEE'2}) and the inequality above,
we have the inequality in (2) of this lemma.

Let $r,s\in\Lambda$.
For $X\in\Irr(\sC_{rs})$ and $Y\in\Irr(\sC_{s0})$,
we have
\begin{align*}
&\alpha_X(E_Y'^{s0})-c_{X,Y}E_{X\otimes Y}'^{r0}c_{X,Y}^*
\\
&=
\alpha_X(E_Y^{s0})(\alpha_X(f_s^\perp)-f_r^\perp)
+
(\alpha_X(E_Y^{s0})-c_{X,Y}E_{X\otimes Y}^{r0}c_{X,Y}^*)f_r^\perp
\\
&\quad
+\alpha_X(\alpha_Y(e))(\alpha_X([P_{Y}^{\cK_{s0}}]^\alpha)
-c_{X,Y}[P_{X\otimes Y}^{\cK_{r0}}]^\alpha c_{X,Y}^*).
\end{align*}
Then
\begin{align*}
&|\alpha_X(E_Y'^{s0})
-c_{X,Y}E_{X\otimes Y}'^{r0}c_{X,Y}^*|_{\alpha_X(\alpha_Y(\varphi_0^\omega))}
-
|\alpha_X(E_Y^{s0})
-c_{X,Y}E_{X\otimes Y}^{r0}c_{X,Y}^*|_{\alpha_X(\alpha_Y(\varphi_0^\omega))}
\\
&\leq
\varphi_r^\omega(\alpha_X(E_Y^{s0})|\alpha_X(f_s)-f_r|)
\\
&\quad
+
\varphi_0^\omega(e)
\big{(}
1_{\cK_{s0}}(Y)\sum_{Z\in\cK_{r0}^c}p_X(Y,Z)
+1_{\cK_{s0}^c}(Y)\sum_{Z\in\cK_{r0}}p_X(Y,Z)
\big{)},
\end{align*}
where we have used our assumption
that
$\alpha_U(\varphi_1^\omega)=\varphi_0^\omega$
on $Z(M_0)$ for all $U\in\Irr(\sC_{01})$
and $\cK_{r0}^c$ denotes the complement of $\cK_{r0}$
in $\Irr(\sC_{r0})$.
Using Lemma \ref{lem:FrsKrs},
We have the following inequalities:
\begin{align*}
&\sum_{(X,Y,Z)\in\cF_{rs}\times\cK_{s0}\times\cK_{r0}^c}
d(X)^2d(Y)^2 p_X(Y,Z)
\leq
\delta |\cF_{rs}|_\sigma|\cK_{r0}|_\sigma,
\\
&\sum_{(X,Y,Z)\in\cF_{rs}\times\cK_{s0}^c\times\cK_{r0}}
d(X)^2d(Y)^2 p_X(Y,Z)
\leq
\delta |\cF_{rs}|_\sigma|\cK_{s0}|_\sigma.
\end{align*}
Hence we obtain
\begin{align*}
|\cF_{rs}|_\sigma(a_{E'}^{r,s}-a_{E}^{r,s})
\leq
\delta|\cF_{rs}|_\sigma
\sum_{t\in\Lambda}\varphi_t^\omega(f_t)
+
\sum_{X\in\cF_{rs}}
d(X)^2\varphi_r^\omega(|\alpha_X(f_s)-f_r|).
\end{align*}
It turns out from the definition of $\mathcal{L}_r$
that $\alpha_X(f_s)$ and $f_r$ are commuting operators
for $X\in\cF_{rs}$,
and $|\alpha_X(f_s)-f_r|=\alpha_X(f_s)f_r^\perp
+\alpha_X(f_s^\perp)f_r$
(cf. Claim 4 in the proof of Lemma \ref{lem:ab}).
Thus
\begin{align*}
\varphi_r^\omega(|\alpha_X(f_s)-f_r|)
&=
\varphi_s^\omega(f_s-f_s\phi_X^\alpha(f_r))
+
\varphi_s^\omega(\alpha_X(f_s^\perp)f_r)
\\
&\leq
\varphi_s^\omega(|f_s-\phi_X^\alpha(f_r)|)
+
\varphi_r^\omega(\phi_{\ovl{X}}^\alpha(f_s^\perp)f_r)
\\
&\leq
\varphi_s^\omega(|f_s-\phi_X^\alpha(f_r)|)
+
\varphi_r^\omega(|f_r-\phi_{\ovl{X}}^\alpha(f_s)|).
\end{align*}
Using
\begin{align*}
&\phi_X^\alpha(f_r)
=
\sum_{(U,V)\in\cK_{r0}\times\Irr(\sC_{s0})}
d(U)^2p_{\ovl{X}}(U,V)\phi_{\ovl{V}}^\alpha(e),
\\
&\phi_{\ovl{X}}^\alpha(f_s)
=
\sum_{(U,V)\in\Irr(\sC_{r0})\times\cK_{s0}}
d(V)^2p_X(V,U)\phi_{\ovl{U}}^\alpha(e),
\end{align*}
and
\begin{align*}
&f_r=
\sum_{(U,V)\in\cK_{r0}\times\Irr(\sC_{s0})}
d(V)^2p_{X}(V,U)\phi_{\ovl{U}}^\alpha(e),
\\
&f_s=
\sum_{(U,V)\in\Irr(\sC_{r0})\times\cK_{s0}}
d(U)^2p_{\ovl{X}}(U,V)\phi_{\ovl{V}}^\alpha(e),
\end{align*}
we obtain
\begin{align*}
&\sum_{X\in\cF_{rs}}
d(X)^2
\varphi_s^\omega(|f_s-\phi_X^\alpha(f_r)|)
\\
&\leq
\sum_{(X,U,V)\in\cF_{rs}\times \cK_{r0}^c\times\cK_{s0}}
d(X)^2d(U)^2p_{\ovl{X}}(U,V)\varphi_0^\omega(e)
\\
&\quad
+
\sum_{(X,U,V)\in\cF_{rs}\times \cK_{r0}\times\cK_{s0}^c}
d(X)^2d(U)^2p_{\ovl{X}}(U,V)\varphi_0^\omega(e)
\\
&\leq
|\cF_{rs}|_\sigma|(\cF_{rs}\cdot\cK_{s0})\setminus\cK_{r0}|_\sigma
\varphi_0^\omega(e)
\\
&\quad
+
|\cF_{rs}|_\sigma|(\ovl{\cF_{rs}}\cdot\cK_{r0})\setminus\cK_{s0}|_\sigma
\varphi_0^\omega(e)
\\
&\leq
\delta|\cF_{rs}|_\sigma|\cK_{r0}|_\sigma\varphi_0^\omega(e)
+
\delta|\cF_{rs}|_\sigma|\cK_{s0}|_\sigma\varphi_0^\omega(e)
\quad
\mbox{by Lemma \ref{lem:FrsKrs}}
\\
&=
\delta|\cF_{rs}|_\sigma
\sum_{t\in\Lambda} \varphi_t^\omega(f_t)
\end{align*}
and
\begin{align*}
&\sum_{X\in\cF_{rs}}
d(X)^2
\varphi_r^\omega(|f_r-\phi_{\ovl{X}}^\alpha(f_s)|)
\\
&=
\sum_{(X,U,V)\in\cF_{rs}\times\cK_{r0}^c\times\cK_{s0}}
d(X)^2d(V)^2p_X(V,U)\varphi_0^\omega(e)
\\
&\quad
+
\sum_{(X,U,V)\in\cF_{rs}\times\cK_{r0}\times\cK_{s0}^c}
d(X)^2d(V)^2p_X(V,U)\varphi_0^\omega(e)
\\
&\leq
|\cF_{rs}|_\sigma
|(\cF_{rs}\cdot\cK_{s0})\setminus\cK_{r0}|_\sigma
\varphi_0^\omega(e)
\\
&\quad
+
|\cF_{rs}|_\sigma
|(\ovl{\cF_{rs}}\cdot\cK_{r0})\setminus\cK_{s0}|_\sigma
\varphi_0^\omega(e)
\\
&\leq
\delta|\cF_{rs}|_\sigma|\cK_{r0}|_\sigma\varphi_0^\omega(e)
+
\delta|\cF_{rs}|_\sigma|\cK_{s0}|_\sigma\varphi_0^\omega(e)
\quad
\mbox{by Lemma \ref{lem:FrsKrs}}
\\
&=
\delta|\cF_{rs}|_\sigma
\sum_{t\in\Lambda} \varphi_t^\omega(f_t).
\end{align*}
Hence by (\ref{eq:bEE'2}),
we have
\[
(a_{E'}^{r,s}-a_E^{r,s})
\leq
3\delta
\sum_{t\in\Lambda}
\varphi_t^\omega(f_t)
<
3\delta^{1/2}
\sum_{t\in\Lambda}
(b_{E'}^t-b_E^t).
\]
\end{proof}

\begin{proof}[Proof of Theorem \ref{thm:Rohlin-tensor2}]
Let $\mathcal{I}$ be the subset of $\mathcal{J}$
which consists of $E=(E^{00},E^{10})$
such that $a_E^{r,s}\leq 3\delta^{1/2}\sum_{t\in\Lambda}b_E^t$
for all $r,s\in\Lambda$.
We give the order on $\mathcal{I}$ by $E\leq E'$
if $E=E'$ or the inequalities in Lemma \ref{lem:ab2} hold.
Then the order is inductive as shown in \cite[Proof of Theorem 5.9]{Oc}.
Take a maximal element $E=(E^{00},E^{10})\in\mathcal{I}$.
Then
$\sum_{r\in\Lambda}
|\cK_{r0}|_\sigma
b_E^{r}\geq(1-\delta^{1/2})\sum_{r\in\Lambda}
|\cK_{r0}|_\sigma$
from Lemma \ref{lem:ab2}.
Thus $b_E^{0}\geq 1-\delta^{1/2}$
or $b_E^{1}\geq 1-\delta^{1/2}$.
When $b_E^{0}\geq 1-\delta^{1/2}$,
we have the following estimate from Lemma \ref{lem:bsrbrr}:
\[
b_{E}^{1}\geq -a_E^{1,0}+b_E^{0}
\geq -3\delta^{1/2}
\sum_{t\in\Lambda}b_E^t+b_E^{0}
\geq
1-7\delta^{1/2}.
\]
Similarly, when $b_E^{1}\geq 1-\delta^{1/2}$,
we also have $b_E^{0}\geq 1-7\delta^{1/2}$.
\end{proof}

\subsection{Approximation of cocycle actions by cocycle perturbations
in $M^\omega$}

We will state the following rigid C$^*$-2-category version
of Lemma \ref{lem:cocapprox}.

\begin{lem}
\label{lem:cocapprox2}
Let $(\alpha,c^\alpha)$
and $(\beta,c^\beta)$
be cocycle actions of an amenable rigid C$^*$-2-category
$\sC=(\sC_{r,s})_{r,s\in\Lambda}$
on a system of properly infinite von Neumann algebras
$M=(M_r)_{r\in\Lambda}$ with separable preduals.
Suppose that the following four conditions hold:
\begin{itemize}
\item 
$\alpha_X(Z(M_s))=Z(M_r)$
for $X\in\Irr(\sC_{rs})$ with $r,s\in\Lambda$.

\item
There exist faithful normal states $\varphi_r$ on $M_r$
with $r\in\Lambda$
such that $\alpha_X(\varphi_s)=\varphi_r$ on $Z(M_r)$
for all $X\in\sC_{rs}$ and $r,s\in\Lambda$.

\item
$(\alpha^{00},c)$ is centrally free.

\item
$\alpha_X$ and $\beta_X$ are approximately unitarily equivalent
for all $X\in\sC$.
\end{itemize}
Then one can take unitaries $\nu_X$ in $M_r^\omega$ with $X\in\sC_{rs}$
and $r,s\in\Lambda$
so that the following equalities hold:
\begin{itemize}
\item
$\nu_X\alpha_X(\psi^\omega)\nu_X^*=\beta_X(\psi^\omega)$
for all $X\in\sC$ and $\psi\in M_*$.

\item
$\nu_X\alpha_X(\nu_Y)c_{X,Y}^\alpha \nu_{X\otimes Y}^*
=c_{X,Y}^\beta$
for all $X\in\sC_{rs}$ and $Y\in\sC_{st}$
with $r,s,t\in\Lambda$.

\item
$\nu_YT^\alpha=T^\beta \nu_X$ for all $X,Y\in\sC_{rs}$,
$T\in\sC(X,Y)$
and $r,s\in\Lambda$.
\end{itemize}
\end{lem}
\begin{proof}
Let $\cF_{rs}$, $\cK_{rs}$ and $E^{r0}$
be as in Theorem \ref{thm:Rohlin-tensor2}.
Then the same proof of Lemma \ref{lem:cocapprox} and Lemma \ref{lem:vperturb}
is applicable to
the unitaries in place of $v_X$ in (\ref{eq:defnv}) as follows:
for $X\in\sC_{rs}$ with $r,s\in\Lambda$,
\[
v_X
:=\sum_{Y\in\Irr(\sC_{r0})}
d(Y)^2
\ovl{R}_Y^{\gamma*} c_{Y,\ovl{Y}}^*
\gamma_Y(d_{\ovl{Y},X}c_{\ovl{Y},X}^*)
E_Y^{\gamma,r0}
c_{Y,\ovl{Y}}\ovl{R}_Y^\gamma
+p_r,
\]
where $p_r:=1-\sum_{Y\in\cK_{r0}}d(Y)^2
\ovl{R}_Y^{\gamma*} c_{Y,\ovl{Y}}^*
E_Y^{\gamma,r0}
c_{Y,\ovl{Y}}\ovl{R}_Y^\gamma$.
Note here that $(\gamma,c^\gamma)$
is a unitary perturbation of $(\alpha,c)$ by unitaries $u_X$'s
as introduced in Lemma \ref{lem:appunitequiv}
and $E_Y^{\gamma,r0}$ of course denotes $u_Y E_Y^{r0} u_Y^*$.
In the proof of Lemma \ref{lem:vperturb},
we have used Lemma \ref{lem:gaphiDe},
which should be now changed as follows.
Let $\psi_t\in (M_t)_*$ be a faithful normal state with
$t\in\Lambda$ such that
$\gamma_X(\psi_s)=\psi_r$ on $Z(M_r)$
for all $X\in \sC_{rs}$ and $r,s\in\Lambda$.
Let $\Delta_{X,Y}\in M_r^\omega$ be a self-adjoint element
commuting with $\gamma_X(\gamma_Y(M_t))$
and $\gamma_X(\gamma_Y(\psi_t^\omega))$
for $(X,Y)\in\sC_{rs}\times\sC_{st}$ and $r,s,t\in\Lambda$.
Then for all $y\in M_r^\omega$,
we have
\[
|\gamma_X(\psi_s^\omega)
(y\Delta_{X,Y}\gamma_X(c_{Y,\ovl{Y}}^\gamma\ovl{R}_Y^\gamma))|
\leq
\|y\|
|\Delta_{X,Y}|_{\gamma_{X}(\gamma_Y(\psi_t^\omega))}.
\]
\end{proof}

\subsection{On the first cohomology vanishing type result}
Let $(\alpha,c^\alpha)$ and $(\beta,c^\beta)$
be cocycle actions
of $\sC=(\sC_{rs})_{r,s\in\Lambda}$ on $M=(M_r)_{r\in\Lambda}$
such that the four conditions in Lemma \ref{lem:cocapprox2}
are satisfied.
Also suppose we have unitaries $w_X\in M_r^\omega$ with $X\in\sC_{rs}$
which are satisfying the following properties:
for all $r,s\in \Lambda$ and $X,Y\in\sC_{rs}$,
\begin{itemize}
\item
$w_X\alpha_X(\psi^\omega)w_X^*=\beta_X(\psi^\omega)$
for all $\psi\in (M_s)_*$.

\item
$w_X\alpha_X(w_Y)c_{X,Y}^\alpha w_{X\otimes Y}^*
=c_{X,Y}^\beta$.

\item
$w_YT^\alpha=T^\beta w_X$ for all $T\in\sC_{rs}(X,Y)$.
\end{itemize}

Let us introduce the following unitary:
\begin{equation}
\label{eq:thetabeal}
\theta_X^{\beta,\alpha}
:=
\sum_{Z\in\Irr(\sC_{rs})}
\sum_{T\in\ONB(Z,X)}
T^\beta T^{\alpha*}
\quad
\mbox{for }
X\in\sC_{rs},
\
r,s\in\Lambda.
\end{equation}

Let $\delta>0$.
Take finite subsets
$\cF_{rs}$ and $\cK_{r0}$ as in Lemma \ref{lem:FrsKrs}.
Let $Q_r$
be a countably generated von Neumann subalgebras of $M_r^\omega$
with $r\in\Lambda$.
We may and do assume that
each $Q_r$ is an $\alpha^{rr}$-invariant
and contains $M_r$ for $r=0,1$.
Let $E^\alpha=(E^{00},E^{10})$
be a Rohlin tower
with respect to $(\alpha,c)$ as in Theorem \ref{thm:Rohlin-tensor2}.
Then we will introduce the following Shapiro type unitary element:
\[
\mu_r
:=\sum_{Z\in\Irr(\sC_{r0})}
d(Y)^2 \ovl{R}_Z^{\beta*}c_{Z,\ovl{Z}}^{\beta*}
w_Z E_Z^{r0}
c_{Z,\ovl{Z}}^\alpha \ovl{R}_Z^\alpha
+p_{r},
\]
where
and $p_r:=1-\sum_{Z\in\cK_{r0}}
d(Z)^2\ovl{R}_Z^{\alpha*}c_{Z,\ovl{Z}}^{\alpha*}
E_Z^{r0}
c_{Z,\ovl{Z}}^\alpha\ovl{R}_Z^\alpha$
that is a projection in $Q_r'\cap (M_r)_\omega$.
Then the same proof of Lemma \ref{lem:wgadec} is applicable
to a cocycle action of a rigid C$^*$-2-category.
Note that $\varphi^\omega(p_r)\leq7\delta^{1/2}$.

\begin{lem}
\label{lem:wgadec2}
Let $\mu=(\mu_r)_{r\in\Lambda}$ be as above.
Then the following inequalities hold:
\begin{enumerate}
\item
For all $r,s\in\Lambda$,
\begin{align*}
&\sum_{X\in\cF_{rs}}
d(X)^2|w_X\alpha_X(\mu_s)-\mu_r|_{\alpha_X(\varphi_s^\omega)}
\\
&\leq
20\delta^{1/2}|\cF_{rs}|_\sigma
+
\sum_{X\in\cF_{rs}}\sum_{Y\in\cK_{r0}}
d(X)^2d(Y)^2
\|c_{\ovl{Y},X}^\beta
\theta_{\ovl{Y}\otimes X}^{\beta,\alpha}
c_{\ovl{Y},X}^{\alpha*}
-1\|_{\alpha_{\ovl{Y}}(\alpha_X(\varphi_s^\omega))}.
\end{align*}

\item
For all $r\in\Lambda$
and
$\psi\in (M_r)_*$,
\[
\|\mu_r\psi^\omega-\psi^\omega\mu_r\|
\leq
\sum_{Y\in\cK_{r0}}
d(Y)^4
\|\alpha_{\ovl{Y}}(\psi^\omega)-\beta_{\ovl{Y}}(\psi^\omega)\|.
\]
\end{enumerate}
\end{lem}

Then we obtain the following result corresponding
to Lemma \ref{lem:app1cohvan}.

\begin{lem}
\label{lem:app1cohvan2}
Let $(\alpha,c^\alpha)$ be a centrally free cocycle action
of $\sC=(\sC_{rs})_{r,s\in\Lambda}$
on $M=(M_r)_{r\in\Lambda}$
such that the four conditions in Lemma \ref{lem:cocapprox2}
are satisfied.
Let $\delta>0$ and $\cF_{rs},\cK_{r0}\subset\Irr(\sC)$
be as in Lemma \ref{lem:FrsKrs}.
Let $\varepsilon>0$.
Let $u_X\in M$ with $X\in\sC$ be unitaries
and $(\beta,c^\beta)$ the perturbed cocycle action
of $(\alpha,c^\alpha)$ by the unitary $u$.
Suppose the following inequality holds:
\[
\|c_{\ovl{Y},X}^\beta
\theta_{\ovl{Y}\otimes X}^{\beta,\alpha}
c_{\ovl{Y},X}^{\alpha*}-1\|_{\alpha_{\ovl{Y}}(\alpha_X(\varphi_s))}
<\varepsilon
\quad
\mbox{for all }
r,s\in\Lambda,
\
X\in\cF_{rs},
\
Y\in\cK_{r0}
\]
and
\[
\|\beta_{\ovl{Y}}(\psi)
-\alpha_{\ovl{Y}}(\psi)\|<\varepsilon
\quad
\mbox{for all }
Y\in\cK_{r0},
\ \psi\in \Psi_r,
\ r\in\Lambda.
\]
Then for any finite subset $\Psi_r\subset (M_r)_*$ with $r\in\Lambda$,
there exists a unitary $\nu_r\in M_r$ with $r\in\Lambda$
such that for all $r,s\in\Lambda$,
the following inequalities hold:
\[
\sum_{X\in\cF_{rs}}
d(X)^2
|u_X\alpha_X(\nu_s)-\nu_r|_{\alpha_X(\varphi_s)}
<
20\delta^{1/2}|\cF_{rs}|_\sigma
+
\varepsilon|\cF_{rs}|_\sigma|\cK_{r0}|_\sigma
\]
and
\[
\|\nu_r\psi-\psi\nu_r\|
<
\varepsilon
\sum_{Y\in\cK_{r0}}
d(Y)^4
\quad
\mbox{for all }
\psi\in\Psi_r.
\]
\end{lem}

\subsection{Intertwining argument}
Lemma \ref{lem:cocapprox2} and Lemma \ref{lem:app1cohvan2}
are all we need to proceed the intertwining argument.
Our main result of this section is as follows.

\begin{thm}
\label{thm:classification2}
Let $\sC=(\sC_{rs})_{r,s\in\Lambda}$
be an amenable rigid C$^*$-2-category
with $\Lambda=\{0,1\}$.
Let $(\alpha,c^\alpha)$ and $(\beta,c^\beta)$
be cocycle actions
of $\sC$ on a system of properly infinite von Neumann algebras
$M=(M_r)_{r\in\Lambda}$
with separable preduals.
Suppose that the following conditions hold:
\begin{itemize}
\item 
$\alpha_X(Z(M_s))=Z(M_r)$
for $X\in\Irr(\sC_{rs})$ with $r,s\in\Lambda$.

\item
There exist faithful normal states $\varphi_r$ on $M_r$
with $r\in\Lambda$
such that $\alpha_X(\varphi_s)=\varphi_r$
on $Z(M_r)$
for all $X\in\sC_{rs}$ and $r,s\in\Lambda$.

\item
$(\alpha^{00},c^\alpha)$ is centrally free.

\item
$\alpha_X$ and $\beta_X$ are approximately unitarily equivalent
for all $X\in\sC$.
\end{itemize}
Then $(\alpha,c^\alpha)$ and $(\beta,c^\beta)$ are strongly cocycle conjugate.
\end{thm}
\begin{proof}
For each $r,s\in\Lambda$,
we take an increasing sequence of finite subsets
$\mathcal{S}_n^{rs}$ in $\Irr(\sC_{rs})$
such that
the union of $\mathcal{S}_n^{rs}$ with $n\geq0$
equals $\Irr(\sC_{rs})$
and $\ovl{\mathcal{S}_n^{rs}}=\mathcal{S}_n^{sr}$.
We may and do assume that $\mathcal{S}_0^{00}=\{\btr\}$
and $\mathcal{S}_0^{11}=\{\btr\}$.
By Lemma \ref{lem:FrsKrs},
we can inductively construct a sequence of finite subsets
$\cF_n^{rs}=\ovl{\cF_n^{sr}}\subset\Irr(\sC_{rs})$
and $\cK_n^{r0}\subset\Irr(\sC_{r0})$
with $r,s\in\Lambda$
and $\delta_n>0$ for $n\geq0$
so that
\begin{itemize}
\item
$\cF_0^{rr}:=\mathcal{S}_0^{rr}=:\cK_0^{rr}$ with $r\in\Lambda$,
$\cF_0^{01}=\ovl{\cF_0^{10}}\neq\emptyset$ and $\delta_0:=3$.

\item
$\cF_n^{00}
:=\cF_{n-1}^{00}\cup\cK_{n-1}^{00}\cup\ovl{\cK_{n-1}^{00}}
\cup\mathcal{S}_n^{00}$
for $n\geq1$.

\item
$\cF_n^{10}
:=\cF_{n-1}^{10}\cup\cK_{n-1}^{10}
\cup\mathcal{S}_n^{10}$
for $n\geq1$.

\item
$\cF_n^{01}:=\ovl{\cF_n^{10}}$
for $n\geq1$.

\item
$\cF_n^{11}
:=\cF_{n-1}^{11}\cup\mathcal{S}_n^{11}$
for $n\geq1$.

\item
$\delta_n:=16^{-n}\min_{r,s\in\Lambda}|\cF_n^{rs}|_\sigma^{-2}$
for $n\geq1$.

\item
$|(\cF_n^{rs}\cdot\cK_n^{s0})\setminus\cK_n^{r0}|_\sigma
<\delta_n|\cK_n^{r0}|_\sigma$
for $n\geq1$.
\end{itemize}

Take $\varepsilon_n>0$ with $n\geq0$
such that
$\varepsilon_n
<\delta_n
\min_{r\in\Lambda} |\cK_n^{r0}|_\sigma^{-2}$
and $\varepsilon_n<\varepsilon_{n-1}<1$.
Let $\sD=(\sD_{rs})_{rs\in\Lambda}$ be dense subsets in $\sC$
such that each $\sD_{rs}$ is dense in $\sC_{rs}$
and $\sD$ is closed under the conjugation and the tensor products.
Let $\mathcal{G}_n=(\mathcal{G}_n^{rs})_{r,s\in\Lambda}$
be an increasing sequence of finite subsets
in $\sD$ such that $\mathcal{G}_n^{rs}\subset\sD_{rs}$
and the union of $\mathcal{G}_n^{rs}$ with $n\geq0$ equals $\sD_{rs}$.

For $r\in\Lambda$,
we set $\Phi_0^{r}:=\{\varphi_r\}$.
Let $\Phi_n^r$ with $n\geq0$ be an increasing sequence of finite subsets
of $(M_r)_*$ whose union is norm-dense in $(M_r)_*$.
We set $\Psi_{-1}^r:=\Phi_0^r=:\Psi_0^r$,
$(\gamma^{-1},c^{-1}):=(\alpha,c^\alpha)$,
$(\gamma^0,c^0):=(\beta,c^\beta)$,
$w_{-1}^r:=1=:w_0^r$,
$\theta_{-1}^r:=\id_{M_r}=:\theta_0^r$
and
$u_X^{-1}:=1=:u_X^0\in M_r$
and
$\ovl{u}_X^{-1}:=1=:\ovl{u}_X^0\in M_r$
for all $X\in\sC_{rs}$ with $r,s\in\Lambda$.
We can inductively construct the following members with $n\geq1$
in the same way as Theorem \ref{thm:classification2}:
\begin{itemize}
\item
$\Psi_n^r$:
a finite subset of $(M_r)_*$ with $r\in\Lambda$;

\item
$(\gamma^n,c^n)$:
a centrally free cocycle action of $\sC$
on $M=(M_r)_{r\in\Lambda}$;

\item
$u_X^{n}$:
a unitary in $M_r$ for $X\in\sC_{rs}$
with $r,s\in\Lambda$;

\item
$w_n^r$:
a unitary in $M_r$ with $r\in\Lambda$;

\item
$\theta_n^r$:
an inner automorphism on $M_r$ with $r\in\Lambda$;

\item
$\ovl{u}_X^{n}$:
a unitary in $M_r$
for $X\in\sC_{rs}$
with $r,s\in\Lambda$;
\end{itemize}
such that
\begin{enumerate}
\renewcommand{\labelenumi}{$(n.\arabic{enumi})$}
\item
\label{item:Psin2}
For $r\in\Lambda$,
\begin{align*}
\Psi_{n}^r
&:=\Phi_{n}^r
\cup
\Psi_{n-1}^r
\cup
\theta_{n-1}^r(\Phi_n^r)
\\
&\hspace{20pt}
\cup \bigcup_{s=0}^1\bigcup_{X\in\cF_{n}^{rs}}
\{\gamma_X^{n-1}(\Phi_{n}^r),
\gamma_X^{n-1}(\varphi_s)\ovl{u}_X^{n-1},
\ovl{u}_X^{n-1}\gamma_X^{n-1}(\varphi_s),
\ovl{u}_X^{n-1*}\gamma_X^{n-1}(\varphi_s)\ovl{u}_X^{n-1}
\}.
\end{align*}

\item
\label{item:gammanun2}
The unitaries $u^n$ perturbs
$(\Ad w_{n}\circ\gamma^{n-2}\circ\Ad w_{n}^*,w_n c^{n-2} w_n^*)$
to
$(\gamma^{n},c^n)$,
where
$\Ad w_{n}\circ\gamma^{n-2}\circ\Ad w_{n}^*$
means a system of maps
$\Ad w_{n}^{r}\circ\gamma_X^{n-2}\circ\Ad w_{n}^{s*}$
from $M_{s}$ into $M_{r}$
for $X\in\sC_{rs}$ with $r,s\in\Lambda$.

\item
\label{item:gammanpsi2}
$\|
\gamma_{\ovl{X}}^{n}(\gamma_{Y}^{n}(\psi))
-\gamma_{\ovl{X}}^{n-1}(\gamma_{Y}^{n-1}(\psi))
\|<\varepsilon_{n}$
for all
$(X,Y)\in \cF_{n+1}^{rs}\times\cF_{n+1}^{rt}$
and $\psi\in \Psi_{n}^t$
with $r,s,t\in \Lambda$.

\item
\label{item:Tgammanpsi2}
$\|(T^{\gamma^{n}}-T^{\gamma^{n-1}})\varphi_r\|
+
\|\varphi_r\cdot (T^{\gamma^{n}}-T^{\gamma^{n-1}})\|<\varepsilon_n$
for all $T\in\sC(X,Y)$
with $X,Y\in \mathcal{G}_{n}^{rs}$
and $r,s\in\Lambda$.

\item
\label{item:cocyclen2}
$\|(c_{X,Y}^n-c_{X,Y}^{n-1})\varphi_r\|
+
\|\varphi_r\cdot(c_{X,Y}^n-c_{X,Y}^{n-1})\|<\varepsilon_n$
for all $X,Y\in \mathcal{G}_{n}^{rs}$
with $r,s\in\Lambda$.

\item
\label{item:thetaYX2}
$
\|c_{\ovl{Y},X}^n
\theta_{\ovl{Y}\otimes X}^{\gamma^n,\gamma^{n-1}}
c_{\ovl{Y},X}^{n-1*}-1
\|_{\gamma_{\ovl{Y}}^{n-1}(\gamma_X^{n-1}(\varphi_s))}<\varepsilon_n$
for all $X\in\cF_{n+1}^{rs}$ and $Y\in\cK_{n+1}^{r0}$
with $r,s\in\Lambda$.

\item
\label{item:sumuX2}
$\sum_{X\in\cF_{n-1}^{rs}}
|u_X^{n}-1|_{\Ad w_n^r\circ \gamma_X^{n-2}\circ\Ad w_n^{s*}(\varphi_s)}
<
24/4^n$
with $r,s\in\Lambda$ and $n\geq 2$.

\item
\label{item:wnpsi2}
$\|w_n^r\psi-\psi w_n^r\|<2\delta_{n-1}$
for $\psi\in\Psi_{n-1}^r$ with $r\in\Lambda$.

\item
\label{item:thetan2}
$\theta_n^r:=\Ad w_n^r\circ\theta_{n-2}^r$ with $r\in\Lambda$.

\item
\label{item:ovlu2}
$\ovl{u}_X^n:=u_X^n w_n\ovl{u}_X^{n-2} w_n^*$
for $X\in\sC$.
\end{enumerate}
Then it is not so difficult to see that
the intertwining method similar to
the proof of Theorem \ref{thm:classification}
works, and we are done.
\end{proof}

We will state a rigid C$^*$-2-category version of
Corollary \ref{cor:CDclass}.
The proof is straightforward.

\begin{cor}
\label{cor:CDclass2}
Let $\sC=(\sC_{rs})_{r,s\in\Lambda}$
and $\sD=(\sD_{rs})_{r,s\in\Lambda}$
be amenable rigid C$^*$-2-categories
with a unitary tensor equivalence $(F,L)$ from $\sC$ into $\sD$.
Let $(\alpha,c^\alpha)$ and $(\beta,c^\beta)$
be cocycle actions
of $\sC$ and $\sD$ on a system of properly infinite
von Neumann algebras with separable preduals
$M=(M_r)_{r\in\Lambda}$.
Suppose that the following conditions hold:
\begin{itemize}
\item 
$\alpha_X(Z(M_s))=Z(M_r)$
for $X\in\Irr(\sC_{rs})$ with $r,s\in\Lambda$.

\item
There exist faithful normal states $\varphi_r$ on $M_r$
with $r\in\Lambda$
such that $\alpha_X(\varphi_s)=\varphi_r$
on $Z(M_r)$
for all $X\in\sC_{rs}$ and $r,s\in\Lambda$.

\item
$(\alpha^{00},c^\alpha)$ is centrally free.

\item
$\alpha_X$ and $\beta_{F(X)}$ are approximately unitarily equivalent
for all $X\in\sC$.
\end{itemize}
Then there exist
a family of unitaries $v=(v_X)_{X\in\sC}$ in $M$
and an approximately inner automorphism $\theta_r$ on $M_r$
with $r\in\Lambda$
such that
\begin{itemize}
\item
$\Ad v_X\circ\alpha_X=\theta_r\circ\beta_{F(X)}\circ\theta_s^{-1}$
for all $X\in \sC_{rs}$
with $r,s\in\Lambda$.

\item
$v_X \alpha_X(v_Y) c_{X,Y}^\alpha v_{X\otimes Y}^*
=\theta(c_{F(X),F(Y)}^\beta [L_{X,Y}]^\beta)$
for all $(X,Y)\in\sC_{rs}\times\sC_{st}$
with $r,s,t\in\Lambda$.

\item
$v_Y T^\alpha=\theta(F(T)^\beta)v_X$
for all $X,Y\in\sC$ and $T\in\sC(X,Y)$.
\end{itemize}
\end{cor}

\subsection{Isomorphic property}
In Theorem \ref{thm:classification2},
we have classified cocycle actions of a rigid C$^*$-2-category
such that $(\alpha^{00},c^\alpha)$ is centrally free on $M_0$.
We will show $M_1$ is actually determined by
the data of $(\alpha^{00},c^\alpha)$ and $\sC$.

\begin{thm}
\label{thm:M1P1isom}
Let $(\alpha,c^\alpha)$ and $(\beta,c^\beta)$
be cocycle actions of an amenable rigid C$^*$-2-category
$\sC=(\sC_{rs})_{r,s\in\Lambda}$
on systems of properly infinite von Neumann algebras
$M=(M_r)_{r\in\Lambda}$
and $P=(P_r)_{r\in\Lambda}$ with separable preduals,
respectively.
Suppose that
$(\alpha^{00},c^\alpha)$ and $(\beta^{00},c^\beta)$
are cocycle conjugate,
that is,
there exist a $*$-isomorphism
$\pi_0\colon M_0\to P_0$ and
unitaries $u_X\in P_0$ with $X\in\sC_{00}$
such that
\begin{itemize}
\item
$\pi_0\circ\alpha_X\circ\pi_0^{-1}
=\Ad u_X\circ\beta_X$
for all $X\in\sC_{00}$.

\item
$\pi_0(c_{X,Y}^\alpha)
=u_X\beta_X(u_Y)c_{X,Y}^\beta u_{X\otimes Y}^*$
for all $X,Y\in\sC_{00}$.

\item
$\pi_0(T^\alpha)
=u_Y T^\beta u_X^*$
for all $X,Y\in\sC_{00}$
and $T\in\sC_{00}(X,Y)$.
\end{itemize}
Then for each $Z\in\sC_{10}$,
there exists an isomorphism $\pi_Z\colon M_1\to P_1$
such that
\begin{itemize}
\item 
$\pi_Z\circ\alpha_Z=\beta_Z\circ\pi_0$.

\item
$\pi_Z(T^\alpha)=T^\beta$
for all $T\in\sC_{10}(Z,Z)$.
\end{itemize}
\end{thm}
\begin{proof}
Let $\varphi\in M_0$ be a faithful normal state.
Fix $Z\in\sC_{10}$.
In the following,
we will compare the inclusions $\alpha_Z(M_0)\subset M_1$
to $\beta_Z(P_0)\subset P_1$ via GNS Hilbert spaces
(cf. discrete decompositions treated in Section \ref{sect:discrete}).

Let $E_Z^\alpha:=\alpha_Z\circ\phi_Z^\alpha$
be a conditional expectation from $M_1$ onto $\alpha_Z(M_0)$.
Then $\psi:=\alpha_Z(\varphi)\in (M_1)_*$
satisfies $\psi\circ E_Z^\alpha=\psi$.
Let $L^2(M_1,\psi)$ be the GNS Hilbert space of $M_1$
with respect to $\psi$.
The GNS cyclic vector is denoted by $\xi_\psi$.
Then the Jones projection $e_Z^\alpha$ on $L^2(M_1,\psi)$
satisfies
$e_Z x\xi_\psi=E_Z^\alpha(x)\psi$ for $x\in M_1$.

Similarly, we put $E_Z^\beta:=\beta_Z\circ\phi_Z^\beta$
and $\chi:=\beta_Z(\pi_0(\varphi))\in (P_1)_*$.
Then $L^2(P_1,\chi)$, $\xi_\chi$ and $e_Z^\beta$
denote the GNS Hilbert space, the GNS cyclic vector
and the Jones projection, respectively.

For each $X\in\Irr(\sC_{00})$
with $X\prec\ovl{Z}\otimes Z$,
we take an orthonormal base $\{S_j^X\}_{j\in I_X}$
of $\sC_{00}(X,\ovl{Z}\otimes Z)$.
We put
\[
T_j^X:=d(Z)(1_Z\otimes S_j^{X*})(\ovl{R}_Z\otimes1_Z)
\in \sC_{10}(Z,Z\otimes X).
\]

Then we let
\[
v_j^X:=c_{Z,X}^\alpha[T_j^X]^\alpha
=d(Z)\alpha_Z([S_j^{X*}]^\alpha c_{\ovl{Z},Z}^{\alpha *})
c_{Z,\ovl{Z}}^\alpha\ovl{R}_Z^\alpha
\in (\alpha_Z,\alpha_Z\circ \alpha_X)
\subset M_1,
\]
and
\[
w_j^X:=c_{Z,X}^\beta
[T_j^X]^\beta
=d(Z)\beta_Z([S_j^{X*}]^\beta c_{\ovl{Z},Z}^{\beta *})
c_{Z,\ovl{Z}}^\beta\ovl{R}_Z^\beta
\in (\beta_Z,\beta_Z\circ\beta_X)
\subset P_1.
\]

\setcounter{clam}{0}
\begin{clam}
\label{clam:eZalpha}
The following statements hold:
\begin{enumerate}
\item
For $X,Y\in\Irr(\sC_{00})$
contained in $\ovl{Z}\otimes Z$,
one has
$\phi_Z^\alpha(v_j^X v_k^{Y*})=\delta_{X,Y}\delta_{j,k}$
and
$\phi_Z^\beta(w_j^X w_k^{Y*})=\delta_{X,Y}\delta_{j,k}$
for $j\in I_X$ and $k\in I_Y$.

\item
The equalities
$\sum_{X}\sum_{j\in I_X}v_j^{X*}e_Z^\alpha v_j^X=1$
and
$\sum_{X}\sum_{j\in I_X}w_j^{X*}e_Z^\beta w_j^X=1$
hold,
where the summation is taken for
$X\in\Irr(\sC_{00})$
with $X\prec\ovl{Z}\otimes Z$.

\item
The linear span of $v_j^{X*}\alpha_Z(M_0)$
for $X\in\Irr(\sC_{00})$ with $X\prec\ovl{Z}\otimes Z$
equals $M_1$.
Similarly,
the linear span of $w_j^{X*}\beta_Z(M_0)$
for the same $X$'s
equals $P_1$.
\end{enumerate}
\end{clam}
\begin{proof}[Proof of Claim 1]
We will verify the statements only for $(\alpha,c^\alpha)$.
(1).
We have
\begin{align*}
\phi_Z^\alpha(v_j^X v_k^{Y*})
&=
d(Z)^2
\phi_Z^\alpha(
\alpha_Z([S_j^{X*}]^\alpha)c_{Z,\ovl{Z}}^\alpha\ovl{R}_Z^\alpha
\ovl{R}_Z^{\alpha*}c_{Z,\ovl{Z}}^{\alpha*}\alpha_Z([S_k^{Y}]^\alpha))
\\
&=
d(Z)^2
[S_j^{X*}]^\alpha
\phi_Z^\alpha(
c_{Z,\ovl{Z}}^\alpha\ovl{R}_Z^\alpha
\ovl{R}_Z^{\alpha*}c_{Z,\ovl{Z}}^{\alpha*})
[S_k^{Y}]^\alpha
\\
&=
[S_j^{X*}]^\alpha[S_k^{Y}]^\alpha
=
\delta_{X,Y}\delta_{j,k}.
\end{align*}

(2).
By (1).
$p^\alpha:=\sum_{X}\sum_{j\in I_X}v_j^{X*}e_Z^\alpha v_j^X$
is a projection in the basic extension $\langle M_1,e_Z^\alpha\rangle$.
Let $\widehat{E}_Z^\alpha$ be the dual operator valued weight
of $E_Z^\alpha$.
Using $\widehat{E}_Z^\alpha(e_Z^\alpha)=1$
and the $M_1$-$M_1$-bimodularity of $\widehat{E}_Z^\alpha$,
we get
\begin{align*}
\widehat{E}_Z^\alpha(p^\alpha)
&=
\sum_{X}\sum_{j\in I_X}v_j^{X*}v_j^X
=
\sum_{X}\sum_{j\in I_X}
d(Z)^2
\ovl{R}_Z^{\alpha *}c_{Z,\ovl{Z}}^{\alpha*}
\alpha_Z([S_j^XS_j^{X*}]^\alpha)
c_{Z,\ovl{Z}}^\alpha\ovl{R}_Z^{\alpha}
\\
&=
\sum_{X}
d(Z)^2
\phi_{\ovl{Z}}^\alpha([P_{\ovl{Z}\otimes Z}^X]^\alpha)
=d(Z)^2.
\end{align*}
It is elementary to check
$b:=d(Z)\ovl{R}_Z^{\alpha *}c_{Z,\ovl{Z}}^{\alpha*}$
gives a Pimsner--Popa basis of $\alpha_Z(M_0)\subset M_1$,
and we have $be_Z^\alpha b^*=1$.
This implies that
$\widehat{E}_Z^\alpha(1)
=\widehat{E}_Z^\alpha(be_Z^\alpha b^*)=bb^*=d(Z)^2$.
Hence $p^\alpha=1$.

(3). This is trivial from (2).
\end{proof}

It turns out from the previous claim
that
$L^2(M_1,\psi)$ and $L^2(P_1,\chi)$
have the orthogonal decomposition to
$v_j^{X*} \alpha_Z(M_0)\xi_\psi$
and
$w_j^{X*} \beta_Z(M_0)\xi_\chi$
with $X\in \Irr(\sC_{00})$, $X\prec\ovl{Z}\otimes Z$
and $j\in I_X$,
respectively.

Let us compare the structures of the inner products.
For $x,y\in M_0$,
$X,Y\in\Irr(\sC_{00})$
with $X,Y\prec\ovl{Z}\otimes Z$,
$j\in I_X$ and $k\in I_Y$,
we have
\begin{align*}
\psi(\alpha_Z(y^*)v_k^Yv_j^{X*} \alpha_Z(x))
&=
\varphi(y^* \phi_Z^\alpha(v_k^Yv_j^{X*})x)
\\
&=
\delta_{X,Y}\delta_{j,k}
\varphi(y^*x)
\quad
\mbox{by Claim \ref{clam:eZalpha} (1)}
\\
&=
\chi(\beta_Z(\pi_0(y^*)u_Y^*)w_k^Y w_j^{X*} \beta_Z(u_X^*\pi_0(x))).
\end{align*}

Hence we have shown the following claim.

\begin{clam}
There exists a unitary $V$ from $L^2(M_1,\psi)$
onto $L^2(P_1,\chi)$
such that
$V v_j^{X*} \alpha_Z(x)\xi_\psi
=
w_j^{X*} \beta_Z(u_X^*\pi_0(x))\xi_\chi$
for all $x\in M_0$,
$X\in\Irr(\sC_{00})$ with $X\prec\ovl{Z}\otimes Z$
and $j\in I_X$.
\end{clam}

Now we will show that
$V$ implements the isomorphism
between the inclusions
$\alpha_Z(M_0)\subset M_1$
and
$\beta_Z(M_0)\subset P_1$.
For $x\in M_0$,
we have $V\alpha_Z(x)=\beta_Z(\pi_0(x))V$.
Indeed for $y\in M_0$,
$X\in\Irr(\sC_{00})$ with $X\prec\ovl{Z}\otimes Z$
and $j\in I_X$,
we have
\begin{align*}
V \alpha_Z(x)v_j^{X*} \alpha_Z(y)\xi_\psi
&=
V v_j^{X*} \alpha_Z(\alpha_X(x)y)\xi_\psi
=
w_j^{X*}\beta_Z(u_X^*\pi_0(\alpha_X(x)y))\xi_\chi
\\
&=
w_j^{X*}
\beta_Z\big{(}\beta_X(\pi_0(x))u_X^*\pi_0(y)\big{)}\xi_\chi
\\
&=
\beta_Z(\pi_0(x))w_j^{X*}\beta_Z(u_X^*\pi_0(y))\xi_\chi.
\end{align*}
Next let
$Y\in\Irr(\sC_{00})$ and $k\in I_Y$.
Then for $x\in M_0$,
$X\in\Irr(\sC_{00})$ with $X\prec Z\otimes\ovl{Z}$
and $j\in I_X$,
we have
\begin{align*}
v_k^{Y*} v_j^{X*}\alpha_Z(x)\xi_\psi
&=
\sum_{U}\sum_{\ell\in I_U}
v_\ell^{U*}e_Z^\alpha v_\ell^U
v_k^{Y*} v_j^{X*}\alpha_Z(x)\xi_\psi
\quad
\mbox{by Claim \ref{clam:eZalpha} (2)}
\\
&=
\sum_{U}\sum_{\ell\in I_U}
v_\ell^{U*}
E_Z^\alpha(v_\ell^U v_k^{Y*} v_j^{X*})\alpha_Z(x)\xi_\psi,
\end{align*}
where the summation is taken for
$U\in\Irr(\sC_{00})$
with
$U\prec\ovl{Z}\otimes Z$.
Thus we have
\[
Vv_k^{Y*} v_j^{X*}\alpha_Z(x)\xi_\psi
=
\sum_{U}\sum_{\ell\in I_U}
w_\ell^{U*}
\beta_Z
\big{(}u_U^*\pi_0(\phi_Z^\alpha(v_\ell^U v_k^{Y*} v_j^{X*})x)\big{)}\xi_\chi.
\]
Now we have
\begin{align*}
v_\ell^U v_k^{Y*} v_j^{X*}
&=
d(Z)^2
\alpha_Z([S_\ell^{U*}]^\alpha c_{\ovl{Z},Z}^{\alpha*})c_{Z,\ovl{Z}}^\alpha
\ovl{R}_Z^\alpha
\ovl{R}_Z^{\alpha*}c_{Z,\ovl{Z}}^{\alpha*}
\alpha_Z(c_{\ovl{Z},Z}^\alpha[S_k^{Y}]^\alpha)
v_j^{X*}
\\
&=
d(Z)^2
\alpha_Z([S_\ell^{U*}]^\alpha c_{\ovl{Z},Z}^{\alpha*})
c_{Z,\ovl{Z}}^\alpha
\ovl{R}_Z^\alpha\ovl{R}_Z^{\alpha*}
c_{Z,\ovl{Z}}^{\alpha*}
v_j^{X*}
\alpha_Z(\alpha_X(c_{\ovl{Z},Z}^\alpha[S_k^Y]^\alpha)).
\end{align*}
This implies
\begin{align*}
&\phi_Z^\alpha(v_\ell^U v_k^{Y*} v_j^{X*})
\\
&=
d(Z)^2
[S_\ell^{U*}]^\alpha c_{\ovl{Z},Z}^{\alpha*}
\phi_Z^\alpha(c_{Z,\ovl{Z}}^\alpha
\ovl{R}_Z^\alpha\ovl{R}_Z^{\alpha*}
c_{Z,\ovl{Z}}^{\alpha*}
v_j^{X*})
\alpha_X(c_{\ovl{Z},Z}^\alpha[S_k^Y]^\alpha)
\\
&=
d(Z)
[S_\ell^{U*}]^\alpha c_{\ovl{Z},Z}^{\alpha*}
\alpha_{\ovl{Z}}(\ovl{R}_Z^{\alpha*}
c_{Z,\ovl{Z}}^{\alpha*} v_j^{X*})
c_{\ovl{Z},Z}^\alpha
R_Z^\alpha
\alpha_X(c_{\ovl{Z},Z}^\alpha[S_k^Y]^\alpha)
\\
&=
d(Z)^2
[S_\ell^{U*}]^\alpha c_{\ovl{Z},Z}^{\alpha*}
\alpha_{\ovl{Z}}
(\ovl{R}_Z^{\alpha*}c_{Z,\ovl{Z}}^{\alpha*}
\ovl{R}_Z^{\alpha*}c_{Z,\ovl{Z}}^{\alpha*})
c_{\ovl{Z},Z}^\alpha
R_Z^\alpha
c_{\ovl{Z},Z}^\alpha
[S_j^X]^\alpha
\alpha_X(c_{\ovl{Z},Z}^\alpha[S_k^Y]^\alpha)
\\
&=
d(Z)
[S_\ell^{U*}]^\alpha c_{\ovl{Z},Z}^{\alpha*}
\alpha_{\ovl{Z}}
(\ovl{R}_Z^{\alpha*}c_{Z,\ovl{Z}}^{\alpha*})
c_{\ovl{Z},Z}^\alpha
[S_j^X]^\alpha
\alpha_X(c_{\ovl{Z},Z}^\alpha[S_k^Y]^\alpha)
\\
&=
d(Z)[S_\ell^{U*}(1_{\ovl{Z}}\otimes \ovl{R}_Z^*\otimes 1_Z)
(S_j^X\otimes 1_{\ovl{Z}}\otimes1_Z)
(1_X\otimes S_k^Y)]^\alpha
c_{X,Y}^{\alpha*}.
\end{align*}
Thus
\begin{align*}
&\pi_0(\phi_Z^\alpha(v_\ell^U v_k^{Y*} v_j^{X*}))
\\
&=
d(Z)
\pi_0\big{(}
[S_\ell^{U*}(1_{\ovl{Z}}\otimes \ovl{R}_Z^*\otimes 1_Z)
(S_j^X\otimes 1_{\ovl{Z}}\otimes1_Z)
(1_X\otimes S_k^Y)]^\alpha\big{)}
\pi_0(c_{X,Y}^{\alpha*})
\\
&=
d(Z)
u_U[S_\ell^{U*}(1_{\ovl{Z}}\otimes \ovl{R}_Z^*\otimes 1_Z)
(S_j^X\otimes 1_{\ovl{Z}}\otimes1_Z)
(1_X\otimes S_k^Y)]^\beta u_{X\otimes Y}^*
\\
&\quad
\cdot
u_{X\otimes Y}c_{X,Y}^{\beta*}\beta_X(u_Y^*)u_X^*
\\
&=
u_U
\phi_Z^\beta(w_\ell^U w_k^{Y*} w_j^{X*})
\beta_X(u_Y^*)u_X^*.
\end{align*}
Therefore,
we obtain
\begin{align*}
Vv_k^{Y*} v_j^{X*}\alpha_Z(x)\xi_\psi
&=
\sum_{U}\sum_{\ell\in I_U}
w_\ell^{U*}
\beta_Z\big{(}
u_U^*u_U \phi_Z^\beta(w_\ell^U w_k^{Y*} w_j^{X*})
\beta_X(u_Y^*)u_X^*\pi_0(x)
\big{)}\xi_\chi
\\
&=
\sum_{U}\sum_{\ell\in I_U}
w_\ell^{U*}e_Z^\beta w_\ell^U
w_k^{Y*} w_j^{X*}
\beta_Z(\beta_X(u_Y^*)u_X^*\pi_0(x))\xi_\chi
\\
&=
w_k^{Y*} w_j^{X*}\beta_Z(\beta_X(u_Y^*)u_X^*\pi_0(x))\xi_\chi
=w_k^{Y*}\beta_Z(u_Y^*)
Vv_j^{X*}\alpha_Z(x)\xi_\psi.
\end{align*}
Hence $V v_k^{Y*}=w_k^{Y*}\beta_Z(u_Y^*)V$.
Putting $\pi_Z:=\Ad V$,
we see $\pi_Z$ gives an isomorphism from $M_1$ onto $P_1$
satisfying $\pi_Z\circ\alpha_Z=\beta_Z\circ\pi_0$
by Claim \ref{clam:eZalpha} (3).

Let $T\in\sC_{10}(Z,Z)$.
For $X\in\Irr(\sC_{00})$
with $X\prec\ovl{Z}\otimes Z$
and $j\in I_X$,
we have
\begin{align*}
T_j^XT
&=
d(Z)(1_Z\otimes S_j^{X*})(1_Z\otimes 1_{\ovl{Z}}\otimes T)
(\ovl{R}_Z\otimes 1_Z)
\\
&=
\sum_Y\sum_{k\in I_Y}
d(Z)
\big{(}
1_Z\otimes (S_j^{X*}(1_{\ovl{Z}}\otimes T)S_k^Y S_k^{Y*})
\big{)}
(\ovl{R}_Z\otimes 1_Z)
\\
&=
\sum_{k\in I_X}
\mu_{jk}^X(T)T_k^X,
\end{align*}
where the summation is taken for $Y\in\Irr(\sC_{00})$
with $Y\prec\ovl{Z}\otimes Z$
and $\mu_{jk}^X(T):=S_j^{X*}(1_{\ovl{Z}}\otimes T)S_k^X
\in\sC_{00}(X,X)=\C$.
Note here that
$S_j^{X*}(1_{\ovl{X}}\otimes T)S_k^Y=0$
if $X\neq Y$.
Hence we have
$v_j^X T^\alpha=\sum_{k\in I_X}\mu_{jk}^X(T)v_k^X$
and
$w_j^X T^\beta=\sum_{k\in I_X}\mu_{jk}^X(T)w_k^X$
for all $T\in\sC_{10}(Z,Z)$.
This implies the following for all $x\in M_0$:
\begin{align*}
V T^\alpha v_j^{X*}\alpha_Z(x)\xi_\psi
&=
\sum_{k\in I_X}
\ovl{\mu_{jk}^X(T^*)}
V v_k^{X*}\alpha_Z(x)\xi_\psi
\\
&=
\sum_{k\in I_X}
\ovl{\mu_{jk}^X(T^*)}
w_k^{X*}\beta_Z(u_X^*\pi_0(x))\xi_\chi\\
&=
T^\beta w_j^{X*}\beta_Z(u_X^*\pi_0(x))\xi_\chi.
\end{align*}
Hence $\pi_Z(T^\alpha)=T^\beta$.
\end{proof}

\subsection{Cores and canonical extensions}
Let us quickly review the notion of the core of a von Neumann algebra
\cite{Fal-Tak}
and the canonical extension of
an endomorphism due to Izumi \cite{Iz-can}.

The \emph{core} $\tM$ of a von Neumann algebra $M$
is generated by the copy of $M$ and one-parameter unitary groups
$\lambda^\varphi(t)$, $t\in\R$,
where $\varphi$ denotes any faithful normal semifinite weights.
Their computation rules are described as follows:
\[
\lambda^\varphi(t)x=\sigma_t^\varphi(x)\lambda^\varphi(t),
\quad
\lambda^\varphi(t)=[D\varphi:D\psi]_t \lambda^\psi(t)
\]
for all $t\in\R$, $x\in M$ and faithful semifinite normal weights
$\varphi$ and $\psi$ on $M$.
Then it is known that $\tM$ is canonically isomorphic to
the crossed product $M\rtimes_{\sigma^\varphi}\R$.
There exists a unique $\R$-action $\theta^M$ on $\tM$
such that $\theta_s^M(x)=x$
and $\theta_s^M(\lambda^\varphi(t))=e^{-ist}\lambda^\varphi(t)$
for all $x\in M$, $s,t\in\R$
and faithful normal semifinite weights $\varphi$.
The restriction of $\theta^M$ on the center $Z(\tM)$ is called
the \emph{flow of weights} \cite{CT}.

In the following,
we are able to treat a normal $*$-homomorphism,
but we will think of only cocycle actions of $\sC$.
Then we do not need to discuss dimension theory of $*$-homomorphisms
between von Neumann algebras.

Let $(\alpha,c)$ be a cocycle action of a rigid C$^*$-2-category
$\sC=(\sC_{rs})_{r,s\in\Lambda}$
on a system of properly infinite von Neumann algebras
$M=(M_r)_{r\in\Lambda}$ with separable preduals as before.
Then $(\alpha,c)$ extends to the cocycle action
$(\tal,c)$ on $\tM:=(\tM_r)_{r\in\Lambda}$
as follows:
for $X\in\sC_{rs}$ with $r,s\in\Lambda$,
\[
\tal_X(x):=\alpha_X(x),
\quad
\tal_X(\lambda^\varphi(t))
:=
d(X)^{it}
\lambda^{\alpha_X(\varphi)}(t)
\]
for all $x\in M_s$ and faithful normal semifinite weights
$\varphi$ on $M_s$.
We will use the same $T^\alpha$ for a morphism $T$ in $\sC$
and the same 2-cocycle $c$, too.
We can see $(\tal,c)$ is indeed a cocycle action
from (\ref{eq:alxphi}) and the proof of \cite[Proposition 2.5]{Iz-can},
where the statistical dimension is used.
Note $\tal_X$ is commuting with the flows,
that is,
$\tal_X\circ\theta_t^{M_s}=\theta_t^{M_r}\circ\tal_X$
for all $X\in\sC_{rs}$ with $r,s\in\Lambda$ and $t\in\R$.

\begin{defn}
We will call the cocycle action $(\tal,c)$ introduced above
is called the \emph{canonical extension} of $(\alpha,c)$.
\end{defn}

\begin{rem}
When $\sC$ is amenable and $M_r$ are factors,
the statistical dimension of the endomorphism $\alpha_X$
and the intrinsic dimension of $X$ are equal
for all $X\in\sC$
(see \cite[Corollary 2.7.9]{NeTu}).
Thus our definition of $\tal_X$ coincide with the definition by Izumi
\cite[Theorem 2.4]{Iz-can}.
\end{rem}

\subsection{Modular freeness for injective von Neumann algebras}

The approximate innerness and the central triviality
seem to be technical properties,
but it is actually known that they are characterized
in terms of the canonical extension
for injective von Neumann algebras.
Let us recall the modular freeness introduced
in \cite[Definition 4.1]{MT-discrete} for automorphic actions.

\begin{defn}
Let $(\alpha,c)$ be a cocycle action of a rigid C$^*$-2-category
$\sC=(\sC_{rs})_{r,s\in\Lambda}$
on a system of properly infinite von Neumann algebras
$M=(M_r)_{r\in\Lambda}$ with separable preduals as before.
We will say that $(\alpha,c)$ is \emph{modularly free}
when its canonical extension $(\tal,c)$ is free.
\end{defn}

\begin{rem}
\label{rem:modular-corner}
A few remarks are in order.
\begin{enumerate}
\item
By Lemma \ref{lem:00-11free},
we see that
$(\alpha,c)$ is modularly free
if and only if $(\tal^{00},c)$ is free
and $\tal_X(Z(\tM_0))=Z(\tM_1)$ for all $X\in\Irr(\sC_{10})$.
(cf. the preceding paragraph of \cite[Theorem 3.7]{Iz-can}
on the graph change for subfactors with the common flow of
weights.)

\item
For $X\in\Irr(\sC_{rs})$,
we will say $\alpha_X$ has the \emph{Connes--Takesaki module}
when
$\tal_X$ induces the isomorphism denoted by $\mo(\alpha_X)$
from $Z(\tM_s)$ onto $Z(\tM_r)$.
If $(\alpha,c)$ is modularly free,
all $\alpha_X$ with $X\in\Irr(\sC)$ have the Connes--Takesaki module.
\end{enumerate}
\end{rem}

\begin{lem}
\label{lem:cent-modular}
Let $(\alpha,c)$ be a cocycle action of $\sC=(\sC_{rs})_{r,s\in\Lambda}$
on $M=(M_r)_{r\in\Lambda}$
as before.
Consider the following properties:
\begin{enumerate}
\item
$(\alpha^{00},c)$ is centrally free.

\item
$(\tal^{00},c)$ is free.

\item
$(\alpha^{00},c)$ is free.
\end{enumerate}
Then the implications (1) $\Rightarrow$ (2) $\Rightarrow$ (3) hold.
If $M_0$ is injective,
then (1) and (2) are equivalent.
\end{lem}
\begin{proof}
(1) $\Rightarrow$ (2).
Suppose $(\tal^{00},c)$ is not free.
Then there exists $X\in\Irr(\sC_{00})\setminus\{\btr\}$
such that $(\tal_\btr,\tal_X)\neq\{0\}$.
Take a non-zero partial isometry $v$ in $(\tal_\btr,\tal_X)$.
Recall $(M_0)_\omega$ is naturally contained in $(\tM_0)_\omega$
(see the proof of \cite[Lemma 4.11]{MT-app} or \cite[Lemma 3.3]{T-ultra}).
Hence we have
$\al_X^\omega(x)v=xv$ for all $x\in (M_0)_\omega$.
Since $\tal$ and $\theta^{M_0}$ are commuting,
we have
$\al_X^\omega(x)\theta_t^{M_0}(vv^*)=x\theta_t^{M_0}(vv^*)$
for all $x\in (M_0)_\omega$ and $t\in\R$.
Let $e$ be the supremum of the projections
$\theta_t^{M_0}(vv^*)$ with $t\in\R$.
Then $e$ is contained in the $\theta^{M_0}$-fixed point algebra
of $\tM_0$ which equals $M_0$.
Since
$\al_X^\omega(x)e=xe$ for all $x\in (M_0)_\omega$,
$\alpha_X$ is not properly centrally non-trivial.

(2) $\Rightarrow$ (3).
This implication is trivial.

(2) $\Rightarrow$ (1) for injective $M_0$ and $M_1$.
For each $X\in\Irr(\sC_{00})$,
we take the maximal projection $e_X\in M_0$
such that $\alpha_X(x)e_X=xe_X$ for all $x\in (M_0)_\omega$.
The maximality shows $e_X\in \alpha_X(M_0)'\cap M_0$,
and $e_X$ is a central projection by freeness of $(\alpha^{00},c)$.

\begin{clm}
We have $\alpha_X(e_X)=e_X=e_{\ovl{X}}$.
\end{clm}
\begin{proof}[Proof of Claim]
For $x\in (M_0)_\omega$, we have
\begin{align*}
x\alpha_{\ovl{X}}(e_X)c_{\ovl{X},X}R_X^\alpha
&=
\alpha_{\ovl{X}}(e_X)c_{\ovl{X},X}R_X^\alpha x
=
\alpha_{\ovl{X}}(e_X\alpha_X(x))c_{\ovl{X},X}R_X^\alpha
\\
&=
\alpha_{\ovl{X}}(e_X x)c_{\ovl{X},X}R_X^\alpha
=
\alpha_{\ovl{X}}(x)\alpha_{\ovl{X}}(e_X)c_{\ovl{X},X}R_X^\alpha.
\end{align*}
Thus
the support projection
of $\alpha_{\ovl{X}}(e_X)
c_{\ovl{X},X}R_X^\alpha R_X^{\alpha*}c_{\ovl{X},X}^*\alpha_{\ovl{X}}(e_X)$
is less than or equal to $e_{\ovl{X}}$.
Applying the left inverse $\phi_{\ovl{X}}^\alpha$,
we see $e_X\leq \alpha_X(e_{\ovl{X}})$.
Similarly, we have $e_{\ovl{X}}\leq \alpha_{\ovl{X}}(e_X)$,
and we have
$e_X=\alpha_X(e_{\ovl{X}})$
since $\alpha_{\ovl{X}}\circ\alpha_X=\id$
on $Z(M_0)$.
Then we have
$e_X=\alpha_X(e_{\ovl{X}})e_X
=
e_{\ovl{X}}e_X$,
which shows $e_X\leq e_{\ovl{X}}$.
Similarly, the converse inequality holds,
and we have the equality $e_X=e_{\ovl{X}}$.
\end{proof}

Suppose that $(\alpha^{00},c)$ is not properly centrally non-trivial.
Take $X\in\Irr(\sC_{00})\setminus\{\btr\}$
so that $e:=e_X$ is non-zero.
Then $\alpha_X$ is centrally trivial on $M_0 e$.
Let $M_0 =\int_\Omega (M_0)_\gamma\,d\mu(\gamma)$ be the disintegration
over $Z(M_0)=L^\infty(\Omega,\mu)$,
where $(\Omega,\mu)$ denotes a standard measure space.
Let $E\subset \Omega$ be a measurable set associated with $e$.
Then the fibers $\alpha_X^\gamma \in \End(M_0^\gamma)$
are centrally trivial
for almost every $\gamma\in E$
(see the proof of \cite[Theorem 9.14]{MT-Roh}).
Note the fibers $(M_0)^\gamma$
are injective infinite factors with separable predual
for almost every $\gamma$.
From \cite[Theorem 4.12]{MT-app},
$\alpha_X^\gamma$ is modular,
and $\tal_X$ has the inner part by disintegration,
which shows $(\tal^{00},c)$ is not free.
\end{proof}

\begin{rem}
\label{rem:cent-modular}
A few remarks are in order.
\begin{enumerate}
\item 
If a cocycle action $(\alpha,c)$ on $M=(M_0,M_1)$ is given,
then $M_0$ is an injective von Neumann algebra
if and only if $M_1$ is
since the inclusion $\alpha_X(M_s)\subset M_r$
has a conditional expectation $\alpha_X\circ\phi_X^\alpha$
for $X\in\sC_{rs}$ with $r,s\in\Lambda$.

\item
Thanks to classification of injective factors with separable preduals
due to Connes, Haagerup and Krieger \cite{Co-inj,Co-III1,Ha-III1,Kri},
we see
if $(\alpha,c)$ is modularly free
and $M_0$ and $M_1$ are injective properly infinite
von Neumann algebras of type III with separable preduals,
then $M_0$ and $M_1$ are isomorphic
(see also \cite[Theorem 6.17]{MT-Roh} for non-factorial case).
Hence when we have interest
in a classification of such cocycle actions
on a system of injective properly infinite
von Neumann algebras $M=(M_0,M_1)$ of type III with separable preduals,
we may assume $M_0=M_1$.
\end{enumerate}
\end{rem}

\subsection{Approximate unitary equivalence of cocycle actions}
\label{subsect:appuec}
We will discuss the approximate unitary equivalence
of two centrally free cocycle actions
in terms of the Connes--Takesaki module.

\begin{lem}
\label{lem:CTmodule2}
Let $(\alpha,c^\alpha)$ and $(\beta,c^\beta)$
be cocycle actions of a rigid C$^*$-2-category $\sC=(\sC_{rs})_{r,s\in\Lambda}$
on a system of injective infinite factors
$M=(M_r)_{r\in\Lambda}$ with separable preduals.
Suppose that $(\alpha_X,\alpha_X)=\C$
and $\tal_X(Z(\tM_s))=Z(\tM_r)=\widetilde{\beta}_X(Z(\tM_s))$
for all $X\in\Irr(\sC_{rs})$ with $r,s\in\Lambda$.
Then for all $X\in\Irr(\sC_{rs})$ with $r,s\in\Lambda$,
$\alpha_X$ and $\beta_X$
are approximately unitarily equivalent
if and only if
$\mo(\alpha_X)=\mo(\beta_X)$.
\end{lem}
\begin{proof}
We may and do assume that
$M_0=M_1$ as mentioned in Remark \ref{rem:cent-modular} (3).
The ``only if" part follows from
the continuity of the normalized canonical extension \cite[Lemma 3.7]{MT-app}.
We will show the ``if" part.
Since the Connes--Takesaki module of
$\beta_{\ovl{X}}\circ\alpha_X$ is trivial,
$\beta_{\ovl{X}}\circ\alpha_X$ is an approximately inner endomorphism
of rank $d(X)^2$ from \cite[Theorem 3.15]{MT-app}.
From Proposition \ref{prop:rhosigma},
$\alpha_X$ and $\beta_X$ are approximately unitarily equivalent.
Otherwise, we can use a more general result proved by Shimada
\cite[Theorem 1]{Shima}.
\end{proof}

The next result follows from
Theorem \ref{thm:classification2},
Lemma \ref{lem:cent-modular}
and
Lemma \ref{lem:CTmodule2}.

\begin{cor}
\label{cor:modapp}
If $(\alpha,c^\alpha)$ and $(\beta,c^\beta)$
are free cocycle actions of an amenable rigid C$^*$-2-category
$\sC=(\sC_{rs})_{r,s\in\Lambda}$
on a system of injective infinite factors
$M=(M_r)_{r\in\Lambda}$
with separable preduals.
Then they are strongly cocycle conjugate
if both following conditions are satisfied:
\begin{enumerate}
\item
$(\alpha,c^\alpha)$ and $(\beta,c^\beta)$
are modularly free.

\item
$\mo(\alpha_X)=\mo(\beta_X)$ for all $X\in\Irr(\sC)$.
\end{enumerate}
\end{cor}

If injective factors $M_r$ are of type III$_1$,
then any irreducible modular endomorphism
is unitarily equivalent to a modular automorphism
\cite[Remark 3.8]{Iz-can}
and any endomorphism with finite index
is an approximately inner endomorphism of rank $1$
\cite[Corollary 3.16]{MT-app}.
Thus we have the following simple statement.

\begin{cor}
A modularly free cocycle action of an amenable rigid C$^*$-2-category
$\sC=(\sC_{rs})_{r,s\in\Lambda}$
on a system of injective type III$_1$ factors
$M=(M_r)_{r\in\Lambda}$ with separable preduals
is unique up to strong cocycle conjugacy.
\end{cor}

When we study a cocycle action of a C$^*$-tensor category
or a rigid C$^*$-2-category on finite von Neumann algebras,
we should treat functors into bimodule categories
(see \cite[Appendix A]{Mas-Rob}).
Then our classification holds for the injective type II$_1$
factor
as proved in \cite[Theorem A.1]{Mas-Rob}).

The following result is a rigid C$^*$-2-category version
of Proposition \ref{prop:HY-real}.
The proof is identical to that of Proposition \ref{prop:HY-real},
where the result due to Hayashi \cite[Theorem 3.2]{Hayas-Real}
should be used.

\begin{prop}
Let $\sC=(\sC_{rs})_{r,s\in\Lambda}$
be an amenable rigid C$^*$-2-category.
Let $(N,N)$ be a system of the injective
type II$_\infty$ factor $N$ with separable preduals.
Then there exists a free cocycle action $(\alpha,c)$
of $\sC$ on $(N,N)$ such that
$\alpha_X$ is an approximately inner endomorphism of rank $d(X)$
for all $X\in\Irr(\sC)$.
\end{prop}

\section{Classification of amenable subfactors}
In this section,
we will discuss how our classification
of (cocycle) actions of C$^*$-tensor categories
(Theorem \ref{thm:classification} and Theorem \ref{thm:classification2})
can be applied to classification of amenable subfactors.

\subsection{Standard invariants}
We will quickly review the notion of the standard invariant
introduced by Popa.
Our references are
\cite{Popa-acta,Popa-WorldSci,Popa-mathlett,Popa-axiom,Popa-endo,Popa-some}
and \cite{Iz-near,Mas-Rob}.
Let $\sC=(\sC_{rs})_{r,s\in\Lambda}$
be a rigid C$^*$-2-category
that is generated by an object $\rho\in \sC_{10}$,
where $\rho$ implicitly means a $*$-homomorphism.
Namely,
any object in $\sC$ is isomorphic to the tensor products
of an alternative words in $\rho$ and $\orho$.
Let us choose and fix a conjugate object of $\orho\in\sC_{01}$
and a standard solution $(R_\rho,\ovl{R}_\rho)$
of the conjugate equations for $\rho$ and $\orho$.
We will often omit the tensor symbol $\otimes$.
Then for $n\geq0$,
we set
\begin{align*}
A_{00}^n
&:=\sC_{00}((\orho\rho)^n,(\orho\rho)^n),
\\
A_{01}^n
&:=\sC_{01}((\orho\rho)^n\orho,(\orho\rho)^n\orho).
\\
A_{10}^n
&:=\sC_{10}
((\rho\orho)^n\rho,(\rho\orho)^n\rho),
\\
A_{11}^n
&:=\sC_{11}((\rho\orho)^n,(\rho\orho)^n).
\end{align*}
Then they form the following commutative diagram of C$^*$-algebras:
\[
\xymatrix{
&
A_{00}^0
\ar[d]_{1_\rho\otimes}
\ar[r]^(0.50){\otimes1_{\orho}}
&
A_{01}^0
\ar[d]_{1_\rho\otimes}
\ar[r]^(0.50){\otimes1_{\rho}}
&
\cdots
\ar[r]^(0.50){\otimes1_{\rho}}
&
A_{00}^n
\ar[d]_{1_{\rho}\otimes}
\ar[r]^(0.50){\otimes1_{\ovl{\rho}}}
&
A_{01}^n
\ar[d]_{1_{\rho}\otimes}
\ar[r]^(0.50){\otimes1_{\rho}}
&
A_{00}^{n+1}
\ar[d]_{1_{\rho}\otimes}
\ar[r]^(0.50){\otimes1_{\ovl{\rho}}}
&
\cdots
\\
A_{11}^0
\ar[r]_(0.50){\otimes1_{\rho}}
&
A_{10}^0
\ar[r]_(0.50){\otimes1_{\orho}}
&
A_{11}^1
\ar[r]_(0.50){\otimes1_{\rho}}
&
\cdots
\ar[r]_(0.50){\otimes1_{\rho}}
&
A_{10}^n
\ar[r]_(0.50){\otimes1_{\ovl{\rho}}}
&A_{11}^{n+1}
\ar[r]_(0.50){\otimes1_{\rho}}
&
A_{10}^{n+1}
\ar[r]_(0.50){\otimes1_{\ovl{\rho}}}
&
\cdots
}
\]
Each vertical embedding $1_\rho\otimes \cdot$
has the left inverse map $\tr_\rho\otimes\id$,
which is defined by
$(\tr_\rho\otimes\id)(T)
:=(R_\rho^*\otimes1)(1_{\orho}\otimes T)(R_\rho\otimes1)$.
The left inverses for the embedding $\cdot\otimes 1_{\orho}$
and $\cdot \otimes 1_\rho$ are similarly introduced.

For $n\geq0$,
we set the projections called the Jones projections
as follows:
\begin{align*}
e_{00}^n
:=
1_{(\orho\rho)^n}\otimes R_\rho R_\rho^* \in A_{00}^{n+1},
&
\quad
e_{01}^n
:=
1_{\orho(\rho\orho)^n}\otimes \ovl{R}_\rho \ovl{R}_\rho^* \in A_{01}^{n+1},
\\
e_{10}^n
:=1_{\rho(\orho\rho)^n}\otimes R_\rho R_\rho^* \in A_{10}^{n+1},
&
\quad
e_{11}^n
:=1_{(\rho\orho)^n}\otimes \ovl{R}_\rho \ovl{R}_\rho^* \in A_{11}^{n+1}.
\end{align*}
Actually, $e_{01}^n$ are $e_{10}^n$ are determined by $e_{00}^n$
and $e_{11}^n$ as $e_{01}^n=(\tr_\rho\otimes\id)(e_{11}^{n+1})$
and $e_{10}^n=1_\rho\otimes e_{00}^n$.

The \emph{standard invariant}
$\mathcal{G}_{\rho,\orho}(\sC)$ of a rigid C$^*$-2-category $\sC$
with respect to $(\rho,\orho)$
consists of
the nest of finite dimensional C$^*$-algebras $A_{rs}^n$
as above
and the left inverses $\tr_\rho\otimes\id$, $\id\otimes\tr_\rho$,
$\id\otimes\tr_{\orho}$
and the Jones projections $e_{rs}^n$.
When we have interest in study of subfactors
with non-minimal expectations,
we need to add data of another Jones projections
to the standard invariant introduced above.
We will explain this more precisely.
Let $(S_\rho,\ovl{S}_\rho)$ be another solution of
the conjugate equations of $\rho$ and $\orho$.
Then we set the following projections:
\begin{align*}
e_{00}^{n,S}
:=
1_{(\orho\rho)^n}\otimes S_\rho S_\rho^* \in A_{00}^{n+1},
&
\quad
e_{01}^{n,S}
:=
1_{\orho(\rho\orho)^n}\otimes \ovl{S}_\rho \ovl{S}_\rho^* \in A_{01}^{n+1},
\\
e_{10}^{n,S}
:=1_{\rho(\orho\rho)^n}\otimes S_\rho S_\rho^* \in A_{10}^{n+1},
&
\quad
e_{11}^{n,S}
:=1_{(\rho\orho)^n}\otimes \ovl{S}_\rho \ovl{S}_\rho^* \in A_{11}^{n+1}.
\end{align*}
Then the standard invariant $\mathcal{G}_{\rho,\orho,S_\rho,\ovl{S}_\rho}(\sC)$ of $\sC$
with respect to $(\rho,\orho,S_\rho,\ovl{S}_\rho)$
consists of
$\mathcal{G}_{\rho,\orho}(\sC)$
and the projections $e_{rs}^{n,S}$.

Let $\sC$ and $\sD$ be rigid C$^*$-2-categories.
Suppose that $\rho\in\sC_{10}$ generates $\sC$
and $\sigma\in\sD_{10}$ does $\sD$.
Choose their conjugate objects $\orho$ and $\osi$
and the standard solutions of the conjugate equations as before.
We will write $A_{rs}^{\rho,n}$, $A_{rs}^{\sigma,n}$,
$e_{rs}^{\rho,n}$ and $e_{rs}^{\sigma,n}$
for $A_{rs}^n$ and $e_{rs}^n$
with respect to $\rho$ and $\sigma$.

An isomorphism 
from $\mathcal{G}_{\rho,\orho}(\sC)$
to $\mathcal{G}_{\sigma,\ovl{\sigma}}(\sD)$
consists of
$*$-isomorphisms
$\pi_{rs}^n\colon A_{rs}^{\rho,n}\to A_{rs}^{\sigma,n}$
which satisfies $\pi_{rs}^n(e_{rs}^{\rho,n})=e_{rs}^{\sigma,n}$
and preserves the commutative diagram
and the left inverses.
Similarly, we can introduce an isomorphism
from $\mathcal{G}_{\rho,\orho,S_\rho,\ovl{S}_\rho}(\sC)$
to $\mathcal{G}_{\sigma,\ovl{\sigma},S_\sigma,\ovl{S}_\sigma}(\sD)$.
Note that
if we choose another conjugate object of $\rho$,
then the resulting standard invariant is canonically isomorphic
to the original one.

\subsection{From subfactors to rigid C$^*$-2-categories}
Our standard references are \cite{Iz-SubI,Iz-SubII}.
Let $N,M$ be infinite factors with separable preduals.
Let $\rho\in \Mor(N,M)_0$,
that is,
$\rho$ is a faithful normal unital $*$-homomorphism from $N$ into $M$
with a conjugate map $\orho\in\Mor(M,N)_0$.
Then $\rho$ makes the rigid C$^*$-2-category $\sC^\rho$
as follows.
Let $\sC_{00}^\rho$
be the full subcategory of $\End(N)_0$
whose object is isomorphic
to a direct summand of $(\ovl{\rho}\rho)^n$
with some $n\geq0$.
$\sC_{11}^\rho$ is defined similarly via $(\rho\ovl{\rho})^n$.
Next let $\sC_{10}^\rho$
be the full subcategory of $\Mor(N,M)_0$
similarly defined via a direct summand of $\rho(\ovl{\rho}\rho)^n$
with some $n\geq0$.
The $\sC_{01}^\rho$ is defined similarly by using $\orho(\rho\orho)^n$.
Note that each category $\sC_{rs}^\rho$ does not depend on a choice
of the conjugate object $\orho$.

Thus we obtain the system
of the C$^*$-categories
$\sC^\rho:=(\sC_{rs}^\rho)_{r,s\in\Lambda}$.
The tensor product operation of objects
is defined by the composition of maps.
Then $\sC^\rho$ is a rigid C$^*$-2-category.
Hence we can consider the standard invariant
$\mathcal{G}_{\rho,\orho}(\sC^\rho)$
that is called the standard invariant
of the subfactor $\rho(N)\subset M$.
Note $\sC^\rho$ has the canonical free action $\alpha^\rho$
on the system $(N,M)=(M_0,M_1)$.
Namely, for $\lambda\in\sC_{rs}^\rho$,
we have $\alpha_\lambda^\rho(x)=\lambda(x)$
for $x\in M_s$.

Let $\phi_\rho(\cdot)=R_\rho^*\orho(\cdot)R_\rho$
and $\phi_{\orho}(\cdot)=\ovl{R}_\rho^*\rho(\cdot)\ovl{R}_\rho$
be the left inverses as before.
Then $E_\rho:=\rho\circ\phi_\rho$
and $E_{\orho}:=\orho\circ\phi_{\orho}$
are the minimal conditional expectations of $\rho(N)\subset M$
and $\orho(M)\subset N$, respectively.
They are dual to each other.
Indeed, the projection $\ovl{R}_\rho\ovl{R}_\rho^*\in M$
gives the Jones projection for
$\rho\orho(M)\subset\rho(N)$ with respect to the conditional expectation
$\rho\circ E_\orho\circ\rho^{-1}$.

Next we consider the subfactor $\rho(N)\stackrel{E}{\subset}M$,
where $E$ is a fixed (not necessarily minimal)
faithful normal conditional expectation $E$.
Then there uniquely exists a positive invertible operator
$h_\rho \in \rho(N)'\cap M=(\rho,\rho)$
such that $E(\cdot)=E_\rho(h_\rho\cdot h_\rho)$
(see the proof of \cite[Theorem 1]{Hi} for example).

Denote by $\Ind E$ the Jones index of $E$
introduced by Kosaki to subfactors of arbitrary type \cite{Jo,Ko}.
Then the pair
$S_\rho:=\orho(h_\rho)R_\rho$
and $\ovl{S}_\rho:=d(\rho)(\Ind E)^{-1/2}h_\rho^{-1}\ovl{R}_\rho$
solves the conjugate equations for $\rho$ and $\orho$,
that is,
$S_\rho^*\orho(\ovl{S}_\rho)=(\Ind E)^{-1/2}1_\orho$
and
$\ovl{S}_\rho^*\rho(S_\rho)=(\Ind E)^{-1/2}1_\rho$.
The associated left inverses of $\rho$ and $\orho$
are defined by the following, respectively:
\[
\phi_\rho^S(x)=S_\rho^*\orho(x)S_\rho,
\quad
\phi_{\orho}^S(y)=\ovl{S}_{\rho}^*\rho(y)\ovl{S}_{\rho}
\quad
\mbox{for }
x\in M,\ y\in N.
\]
Then $E=\rho\circ\phi_\rho^S$ and its dual expectation
from $N\to\orho(M)$ is given by $\orho\circ\phi_\orho^S$.
The standard invariant of the inclusion $\rho(N)\stackrel{E}{\subset}M$
means $\mathcal{G}_{\rho,\orho,S_\rho,\ovl{S}_\rho}(\sC^\rho)$.

Consider two inclusions of infinite factors
$\rho\colon N\to M$
and
$\sigma\colon Q\to P$.
Let $\sC^\rho$ and $\sC^\sigma$ be the associated
rigid C$^*$-2-categories.
Suppose that we have a unitary tensor equivalence
$(G,c)\colon \sC^\rho\to\sC^\sigma$
such that $G(\rho)=\sigma$ and
$c_{\lambda,\mu}\colon G(\lambda\mu)\to G(\lambda)G(\mu)$
are natural unitary 2-cocycles.
For a word $w=(\lambda_1,\dots,\lambda_n)\in (\sC^\rho)^n$
such that $\lambda_1\lambda_2\cdots\lambda_n$ is well-defined,
we can define a natural unitary morphism
$c_w\colon G(\lambda_1\cdots\lambda_n)
\to G(\lambda_1)\cdots G(\lambda_n)$
which extends $c_{\lambda,\mu}$ by coherence theorem.
We will see $G$ implements an isomorphism
between $\mathcal{G}_{\rho,\orho,S_\rho,\ovl{S}_\rho}(\sC^\rho)$
and $\mathcal{G}_{\sigma,\ovl{\sigma},S_\sigma,\ovl{S}_\sigma}(\sC^\sigma)$.

The pair $(c_{\orho,\rho}G(R_\rho),c_{\rho,\orho}G(\ovl{R}_\rho))$
is a solution of
the conjugate equations of $(G(\rho),G(\orho))=(\sigma,G(\orho))$.
Thus $G(\orho)$ is isomorphic to $\ovl{\sigma}$.
Since $G$ is an equivalence, we have $d(\rho)=d(\sigma)$.
Hence the pair above is actually a standard solution
of the conjugate equations,
and there exists a unique unitary $u\in Q$ such that
\[
c_{\orho,\rho}G(R_\rho)
=
uR_\sigma,
\quad
c_{\rho,\orho}G(\ovl{R}_\rho)
=
\sigma(u)\ovl{R}_\sigma.
\]

Replacing $\osi$ with $\Ad u\circ\osi$,
we may and do assume that $u=1$ and $G(\orho)=\osi$.
Recall $A_{rs}^{\rho,n}$ defined as before.
We introduce the isomorphism $\pi_{rs}^n$
from $A_{rs}^{\rho,n}$ into $A_{rs}^{\sigma,n}$
as follows:
$\pi_{rs}^n(x):=c_{w_{rs}^n} G(x)c_{w_{rs}^n}^*$ for $x\in A_{rs}^{\rho,n}$,
where $w_{rs}^n$ denotes the alternating words
consisting of $\rho$ and $\orho$.
For instance, $w_{10}^2$ represents the object
$\rho(\orho\rho)^2\in\sC_{10}$, that is,
$w_{10}^2=(\rho,\orho,\rho,\orho,\rho)$.
Then it is not difficult to see $\pi_{rs}^n$'s
gives an isomorphism from
$\mathcal{G}_{\rho,\orho}(\sC^\rho)$
and $\mathcal{G}_{\sigma,\ovl{\sigma}}(\sC^\sigma)$.

Next let us consider
the inclusions $\rho(N)\stackrel{E}{\subset}M$
and $\sigma(Q)\stackrel{F}{\subset}P$.
Take $(S_\rho,\ovl{S}_\rho)$ and $(S_\sigma,\ovl{S}_\sigma)$
as before.
Suppose we are given a unitary tensor equivalence
$(G,c)\colon \sC^\rho\to\sC^\sigma$ with $G(\rho)=\sigma$
as above.
We may assume that
$G(\orho)=\osi$,
$R_\sigma=c_{\orho,\rho}G(R_\rho)$
and
$\ovl{R}_\sigma=c_{\rho,\orho}G(\ovl{R}_\rho)$.
We further assume
$S_\sigma=c_{\orho,\rho}G(S_\rho)$
and
$\ovl{S}_\sigma=c_{\rho,\orho}G(\ovl{S}_\rho)$.
This assumption is equivalent to say $G(h_\rho)=h_\sigma$.
Then $G$ implements an isomorphism
from
$\mathcal{G}_{\rho,\orho,S_\rho,\ovl{S}_\rho}(\sC^\rho)$
into
$\mathcal{G}_{\sigma,\ovl{\sigma},S_\sigma,\ovl{S}_\sigma}(\sC^\sigma)$
as verified above.

We have seen a unitary tensor equivalence between rigid C$^*$-2-categories
implements an isomorphism between those standard invariants.
It is actually well-known among experts that the converse statement
holds.
Namely,
$\sC^\rho$ and $\sC^\sigma$ are unitarily tensor equivalent
via a unitary tensor functor mapping
$\rho$ to $\sigma$
if and only if
their standard invariants
$\mathcal{G}_{\rho,\orho}(\sC^\rho)$
and
$\mathcal{G}_{\sigma,\ovl{\sigma}}(\sC^\sigma)$
are isomorphic.
The similar statement for
$\mathcal{G}_{\rho,\orho,S_\rho,\ovl{S}_\rho}(\sC^\rho)$
and
$\mathcal{G}_{\sigma,\ovl{\sigma},S_\sigma,\ovl{S}_\sigma}(\sC^\sigma)$
also holds.

\subsection{Amenable subfactors}

In the series of papers
\cite{Popa-acta,Popa-WorldSci,Popa-mathlett,Popa-axiom,Popa-endo},
Popa has succeeded in proving the standard invariants are
complete invariants for strongly amenable subfactors.
Let us denote by $N\stackrel{E}{\subset} M$
an inclusion of factors $N$ and $M$
with a faithful normal conditional expectation $E$
from $M$ onto $N$.
The approximate innerness and the central freeness
of $N\stackrel{E}{\subset} M$
are introduced in \cite[Definition 2.1, 3.1]{Popa-endo}.
Note that the approximate innerness depends on the choice of $E$
while the central freeness does not.
Popa's celebrated classification result
is stated as follows \cite[Theorem 5.1]{Popa-endo}:

\begin{thm}[Popa]
\label{thm:Popasubfactor}
Let $N\stackrel{E}{\subset} M$ be an inclusion of factors
with separable preduals and finite index.
Suppose the following conditions hold:
\begin{itemize}
\item 
The extensions of $E$ implement a trace
on the higher relative commutants.

\item
$N$ is a McDuff factor.

\item
The standard invariant of $N\stackrel{E}{\subset} M$
is strongly amenable.

\item
$N\subset M$ is centrally free.

\item
$N\stackrel{E}{\subset} M$ is approximately inner.
\end{itemize}
Then $N\stackrel{E}{\subset} M$
is isomorphic to
$(N^{\rm st}\stackrel{F}{\subset} M^{\rm st})\otimes M$,
where $N^{\rm st}\subset M^{\rm st}$
denotes the standard inclusion
and
$F$ the trace preserving expectation.
\end{thm}

We will reformulate the previous result as follows.

\begin{thm}
\label{thm:subfactor2}
Let $N\stackrel{E}{\subset} M$
and $Q\stackrel{F}{\subset} P$ be an inclusion of factors
with separable preduals and finite indices.
Suppose the following conditions hold:
\begin{itemize}

\item
$N$ and $Q$ are isomorphic.

\item
The standard invariants
of $N\stackrel{E}{\subset} M$ and $Q\stackrel{F}{\subset} P$
are amenable and isomorphic.

\item
$N\subset M$ and $Q\subset P$ are centrally free.

\item
$N\stackrel{E}{\subset} M$ and $Q\stackrel{F}{\subset} P$
are approximately inner.
\end{itemize}
Then $N\stackrel{E}{\subset} M$
is isomorphic to $Q\stackrel{F}{\subset} P$.
\end{thm}

\begin{rem}
\label{rem:ergodicity}
If we also assume the ergodicity of the standard invariant
of $N\stackrel{E}{\subset} M$,
then \cite[Lemma 5.2]{Popa-endo} and Theorem \ref{thm:subfactor2}
imply Theorem \ref{thm:Popasubfactor}.
The fact that the ergodicity is unnecessary for the classification
of amenable subfactors has been announced by Popa
\cite[Problem 5.4.7]{Popa-acta}.
\end{rem}

Let us translate the centrally freeness
and
the approximate innerness of subfactors
in terms of actions of tensor categories.

Let $N$ and $M$ be infinite factors with separable preduals.
Let $\rho\colon N\to M$ be a normal $*$-homomorphism of finite index
and $E\colon M\to \rho(N)$ a faithful normal conditional expectation.
Take $h_\rho$, $S_\rho$, $\ovl{S}_\rho$, $\phi_\rho^S$
and $\phi_\orho^S$ as in the previous subsection.
We also prepare the positive invertible operator
$k_{\orho}\in (\orho,\orho)\subset N$
such that
$\orho(h_\rho)R_\rho=k_{\orho}^{-1}R_\rho$
and
$h_\rho^{-1}\ovl{R}_\rho=\rho(k_{\orho})\ovl{R}_\rho$.
Then
$\phi_\rho^S(x)=\phi_\rho(h_\rho x h_\rho)$
and
$\phi_{\orho}^S(y)=\phi_{\orho}(k_{\orho} yk_{\orho})$
for $x\in M$ and $y\in N$.

Set the the left inverse of $(\orho\rho)^n$
defined by $\phi_{(\orho\rho)^n}^S:=(\phi_\rho^S\circ\phi_{\orho}^S)^n$
for $n\geq0$.
Let us denote by $A_{\orho\rho}$ the C$^*$-subalgebra of $N$
that is the norm closure of the union
of $((\orho\rho)^n,(\orho\rho)^n)$ for $n\geq0$.
Let $\varphi^S$ be the state on $A_{\orho\rho}$ defined by
$\varphi^S(x):=\lim_{n}\phi_{(\orho\rho)^n}^S(x)$
for $x\in A_{\orho\rho}$.
Let us keep these notations in this subsection.

\begin{prop}
\label{prop:centappsubfactor}
Let $N$ and $M$ be infinite factors with separable preduals
and $\rho\colon N\to M$ an inclusion of finite index.
Then the following statements hold:
\begin{enumerate}
\item
The subfactor $\rho(N)\subset M$
is centrally free in the sense of
\cite[Definition 3.1]{Popa-endo}
if and only if
the action $\alpha^\rho$ of $\sC_{00}^{\rho}$
on $N$ is.

\item
Suppose that
for a faithful normal conditional expectation
$E\colon M\to \rho(N)$,
the subfactor $\rho(N)\stackrel{E}{\subset}M$
is approximately inner in the sense of
\cite[Definition 2.1]{Popa-endo}.
Take $(S_\rho,\ovl{S}_\rho)$ and $\varphi^S$ as above.
Then for each $\lambda\in\Irr(\sC_{00}^{\rho})$,
there exists an isometry $T\in\sC_{00}^\rho(\lambda,(\orho\rho)^n)$
with $n\in\N$ such that
$\lambda$ is an approximately inner endomorphism
of rank $d(\rho)^{2n} \varphi^S(TT^*)$
in the sense of \cite[Definition 2.4]{MT-app}.
\end{enumerate}
\end{prop}

\begin{rem}
If we further assume the condition of
Theorem \ref{thm:Popasubfactor} (1)
in the statement of Proposition \ref{prop:centappsubfactor} (2),
then the positive operators $a_{\orho\rho}^n$, which will be introduced
in the following proof of (2),
are central elements in $((\orho\rho)^n,(\orho\rho)^n)$.
This implies a rank of the approximate innerness does not depend on a choice
of an isometry.
Note that a rank is in general not necessarily uniquely determined
for an endomorphism (see \cite{MT-app}).
\end{rem}

\begin{proof}
(1).
We will show the ``only if" part.
Suppose $\rho(N)\subset M$ is a centrally free inclusion.
Let $\lambda\in\Irr(\sC_{00}^\rho)$.
Take $k\geq0$ and an isometry $S\in\sC_{00}^\rho(\lambda,(\orho\rho)^k)$.
Note that $\lambda\in\End(N)_0$ and $S\in N$.
Suppose that $\lambda$ is centrally trivial,
that is, $\lambda(x)=x$ for all $x\in N_\omega$.
Denote by $\sigma:=\rho(\orho\rho)^k\in\Mor(N,M)_0$.
Take an isometry $T\in \sC_{00}^\rho(\btr,(\orho\rho)^k)$.
Then for $x\in N_\omega$,
we have
\[
\rho(T S^*) \sigma(x)
=\rho(T\lambda(x)S^*)
=
\rho(TxS^*)
=
\sigma(x)\rho(TS^*).
\]
Namely, $\rho(TS^*)\in \sigma(N_\omega)'\cap M^\omega$.
Then by \cite[Proposition 3.2 (3)]{Popa-endo},
we have $\rho(TS^*)\in \sigma(N)\vee(\sigma(N)'\cap M)$.
Thus $1=\rho(T^*)\rho(TS^*)\rho(S)$
is contained in
$\rho(N)(\rho\lambda,\rho)$.

Take an orthonormal base $\{v_j\}_j$ of $(\rho\lambda,\rho)$
with respect to the inner product
$(v,w)1:=\phi_\rho(vw^*)$ for $v,w\in(\rho\lambda,\rho)$.
Then there uniquely exist $a_j\in N$ such that
$1=\sum_j \rho(a_j)v_j$.
For $x\in N$, we have
\[\sum_j\rho(a_jx)v_j
=
\sum_j\rho(a_j)v_j\rho\lambda(x)
=
\rho\lambda(x)
=
\sum_j\rho(\lambda(x)a_j)v_j.
\]
By orthogonality of $v_j$'s,
we have $a_j x=\lambda(x)a_j$ for all $x\in N$ and $j$.
Namely, $\lambda=\btr$,
and the action $\alpha^\rho$ of $\sC_{00}^\rho$ on $N$
is centrally free from \cite[Lemma 8.3]{MT-minimal}.

We will show the ``if" part.
Let $k\geq0$ and $\sigma:=\rho(\orho\rho)^k$.
We will check the condition \cite[Proposition 3.2 (3)]{Popa-endo}.
Since the action $\alpha^\rho$ of $\sC_{00}^\rho$ on $N$
is centrally free,
we have the following from Lemma \ref{lem:relcomm}:
\[
(\orho\rho)^{k+1}(N_\omega)'\cap \orho(M^\omega)
\subset
(\orho\rho)^{k+1}(N_\omega)'\cap N^\omega
=
((\orho\rho)^{k+1},(\orho\rho)^{k+1})
\vee
(\orho\rho)^{k+1}(N_\omega'\cap N^\omega).
\]
Applying the left inverse
$\phi_\orho(\cdot)=\ovl{R}_\rho^*\rho(\cdot)\ovl{R}_\rho$
to the both sides,
we obtain
\[
\sigma(N_\omega)'\cap M^\omega
\subset
(\sigma,\sigma)
\vee
\sigma(N_\omega'\cap N^\omega).
\]
Note the converse inclusion trivially holds,
and we actually have the equality.

For a simple $\lambda\prec \sigma$,
take an orthonormal base $\{T_j^\lambda\}_{j\in I_\lambda}$
in $(\lambda,\sigma)$.
Let $\varphi\in N_*$ be a faithful normal state
and $F$ the $\sigma(\varphi^\omega)$-preserving conditional expectation
from $M^\omega$ onto $\sigma(N_\omega)'\cap M^\omega$.
Let $x\in M$.
Then we have $x_{ij}^\lambda\in N_\omega'\cap N^\omega$
for $i,j\in I_\lambda$ and $\lambda\prec\sigma$
such that
\begin{equation}
\label{eq:FTT}
F(x)=\sum_{\lambda\prec\sigma}\sum_{i,j\in I_\lambda}
T_i^\lambda T_j^{\lambda*}\sigma(x_{ij}^\lambda).
\end{equation}
Let us take $y\in N_\omega'\cap N^\omega$,
a simple $\mu\prec\sigma$
and $m,n\in I_\mu$.
On the one hand,
by definition of $F$,
we obtain
\begin{align*}
\sigma(\varphi^\omega)(F(x)T_m^\mu T_n^{\mu*}\sigma(y))
&=
\sigma(\varphi^\omega)(xT_m^\mu T_n^{\mu*}\sigma(y))
=
\varphi^\omega(\phi_\sigma(xT_m^\mu T_n^{\mu*})y).
\end{align*}
On the other hand,
we have
\begin{align}
\sigma(\varphi^\omega)(F(x)T_m^\mu T_n^{\mu*}\sigma(y))
&=
\sum_{i\in I_\mu}
\sigma(\varphi^\omega)(T_i^\mu T_n^{\mu*}\sigma(x_{im}^\mu y))
\quad
\mbox{by (\ref{eq:FTT})}
\notag
\\
&=
\sum_{i\in I_\mu}
\varphi^\omega(\phi_\sigma(T_i^\mu T_n^{\mu*})x_{im}^\mu y).
\label{eq:TTxy}
\end{align}
Since $\phi_\sigma(T_i^\mu T_n^{\mu*})\in \sC_{00}^\rho(\btr,\btr)=\C$,
we have
\[
\phi_\sigma(T_i^\mu T_n^{\mu*})
=
\sigma(\varphi^\omega)(T_i^\mu T_n^{\mu*})
=
\frac{d(\mu)}{d(\sigma)}
\mu(\varphi^\omega)(T_n^{\mu*}T_i^\mu)
=
\delta_{i,n}\frac{d(\mu)}{d(\sigma)}.
\]
This implies that (\ref{eq:TTxy}) equals
$d(\mu)d(\sigma)^{-1}\varphi^\omega(x_{nm}^\mu y)$,
and we get $\phi_\sigma(xT_m^\mu T_n^{\mu*})=d(\mu)d(\sigma)^{-1}x_{nm}^\mu$,
which shows $x_{nm}^\mu\in N$.
Therefore we have
$F(M)=(\sigma,\sigma)\vee \sigma(N)$.

Similarly, we can show
$G(N)=((\orho\rho)^k,(\orho\rho)^k)\vee (\orho\rho)^k(N)$ for $k\geq0$,
where $G$ denotes the $(\orho\rho)^k(\varphi^\omega)$-preserving
conditional expectation from $N^\omega$
onto $(\orho\rho)^k(N_\omega)'\cap N^\omega$.
Hence the inclusion $\rho(N)\subset M$ is centrally free.

(2).
Fix $n\in\N$.
We set $a_{\orho\rho}:=\orho(h_\rho)k_{\orho}$
and
$a_{\orho\rho}^{n}
:=a_{\orho\rho}\orho\rho(a_{\orho\rho})
\cdots (\orho\rho)^{n-1}(a_{\orho\rho})$
that is a positive invertible operator in $((\orho\rho)^n,(\orho\rho)^n)$.
Then we have
$\phi_{(\orho\rho)^n}^S(\cdot)
=\phi_{(\orho\rho)^n}(a_{\orho\rho}^{n}\cdot a_{\orho\rho}^{n})$,
and from Lemma \ref{lem:linv} we have
\[
\phi_{(\orho\rho)^n}^S(x)
=
\sum_{\lambda}
\sum_{T}
\frac{d(\lambda)}{d(\orho\rho)^n}
\phi_{\lambda}(T^* a_{\orho\rho}^n x a_{\orho\rho}^n T)
\]
where the summation is taken for $\lambda \in\Irr(\sC_{00}^\rho)$
and $T\in \ONB(\lambda,(\orho\rho)^n)$.

The operator $a_{\orho\rho}^n\in (\orho\rho,\orho\rho)$
is acting on the Hilbert space $(\lambda,\orho\rho)$
by multiplication
$T\mapsto a_{\orho\rho}^nT$ as a positive invertible operator,
we can take an orthonormal base $\{T_j^\lambda\}_{j\in I_\lambda}$
of $(\lambda,(\orho\rho)^n)$
and positive scalars
$\{\mu_j^\lambda\}_{j\in I_\lambda}$
such that
$a_{\orho\rho}^nT_j^\lambda=\sqrt{\mu_j^\lambda} T_j^\lambda$
for $j\in I_\lambda$.
Then we have
\begin{equation}
\label{eq:rhoorhoS}
\phi_{(\orho\rho)^n}^S(x)
=
\sum_{\lambda,j}
\frac{d(\lambda)\mu_j^\lambda}{d(\orho\rho)^n}
\phi_\lambda(T_j^{\lambda*}x T_j^\lambda)
\quad
\mbox{for }
x\in N.
\end{equation}
Set the isometry
$S_{\rho}^n\in (\btr,(\orho\rho)^n)$
as follows:
$S_{\rho}^1:=S_\rho$
and
for $k\geq1$
\begin{align*}
S_\rho^{2k+1}
&:=
(\orho\rho)^{k}(S_\rho)
\cdot
(\orho\rho)^{k-1}(\orho(\ovl{S}_\rho)S_\rho)
\cdots
\orho\rho(\orho(\ovl{S}_\rho)S_\rho)
\cdot
\orho(\ovl{S}_\rho)S_\rho,
\\
S_\rho^{2k}
&:=
(\orho\rho)^{k-1}(\orho(\ovl{S}_\rho)S_\rho)
\cdots
\orho\rho(\orho(\ovl{S}_\rho)S_\rho)
\cdot
\orho(\ovl{S}_\rho)S_\rho.
\end{align*}
Readers should not confuse $S_\rho^n$
with the $n$-times product $(S_\rho)^n$.
Then we have
$a_{\orho\rho}^n S_{\rho}^n=S_{\rho}^n$
and
$\phi_{(\orho\rho)^n}^S(x)=S_\rho^{2n*}(\orho\rho)^n(x)S_\rho^{2n}$
for $x\in N$.
Also note that
$\varphi^S(T_j^\lambda T_j^{\lambda*})
=d(\lambda)\mu_j^\lambda/d(\rho)^{2n}$.

We set $E_{(\orho\rho)^n}^S:=(\orho\rho)^n\circ \phi_{(\orho\rho)^n}^S$
that is a faithful normal conditional expectation
from $N$ onto $(\orho\rho)^n(N)$.
It is not difficult to show that
$E_{(\orho\rho)^n}^S$ is the composition
of conditional expectations associated with
the following tunnel starting from the conditional expectation
$\orho\circ\phi_{\orho}^S\colon N\to\orho(M)$:
\[
(\orho\rho)^n(N)\subset \cdots
\subset \orho\rho(N)\subset \orho(M)
\subset N.
\]
Thanks to \cite[Proposition 2.7 (ii)]{Popa-endo},
we see the inclusion $(\orho\rho)^n(N)\stackrel{E_{(\orho\rho)^n}^S}\subset N$
is approximately inner.
Thus we have a sequence of Pimsner--Popa bases $\{m_i^k\}_{i\in I}$
for $k\in\N$
with respect to $E_{(\orho\rho)^n}^S$
such that $\|[m_i^k,\psi\circ\phi_{(\orho\rho)^n}^S]\|_{N_*}\to0$
for all $\psi\in N_*$
as $k\to\infty$ (see \cite[Definition 2.1]{Popa-endo}).

We put $x_{ij}^k:= T_j^{\lambda*} m_i^{k}S_\rho^n$
for $i\in I$, $j\in I_\lambda$ and $k\in\N$.
Then for $\psi\in N_*$,
we have
\begin{align*}
x_{ij}^k\psi
&=
d(\orho\rho)^n
T_j^{\lambda*} m_i^{k} (\psi\circ\phi_{(\orho\rho)^n}^S)
S_\rho^n
\quad
\mbox{by (\ref{eq:rhoorhoS})}
\\
&\approx
d(\orho\rho)^n
T_j^{\lambda*}(\psi\circ\phi_{(\orho\rho)^n}^S)
m_i^{k}
S_\rho
\\
&=
d(\lambda)\mu_j^\lambda
(\psi\circ\phi_\lambda)x_{ij}^k
\quad
\mbox{by (\ref{eq:rhoorhoS})},
\end{align*}
where $\approx$ means the norm of the difference converges to 0
as $k\to\infty$.
Thus
\[
\lim_{k\to\infty}
\|
x_{ij}^k\psi-d(\lambda)\mu_j^\lambda\lambda(\psi)x_{ij}^k
\|=0
\quad
\mbox{for all }
i\in I,\ j\in I_\lambda.
\]
Now for each $j\in I_\lambda$ and $k\in\N$,
we have
\begin{align*}
\sum_{i\in I}
x_{ij}^k x_{ij}^{k*}
&=
\sum_{i\in I}
T_j^{\lambda*} m_i^{k}S_\rho^n S_\rho^{n*} m_i^{k*}T_j^{\lambda*}
=
T_j^{\lambda*} E_{(\orho\rho)^n}^{-1}(S_\rho^n S_\rho^{n*})T_j^{\lambda}
\\
&=
T_j^{\lambda*} \cdot d^S(\orho\rho)^n S_\rho^{2n*}
\cdot S_\rho^n S_\rho^{n*}\cdot d^S(\orho\rho)^n
S_\rho^{2n}T_j^{\lambda}
\\
&=
d^S(\orho\rho)^{n},
\end{align*}
where $E_{(\orho\rho)^n}^{-1}$ denotes the dual operator valued weight
from $(\orho\rho)^n(N)'$ onto $N'$
(see \cite[Remark in p.62]{Wata} for example).
We have also used the fact
that $(\orho\rho)^n(N)\stackrel{E_{(\orho\rho)^n}^S}{\subset}N$
has an orthonormal base $d^S(\orho\rho)^n S_\rho^{2n*}$.
Therefore by \cite[Proposition 2.8]{MT-app},
we see $\lambda$ is an approximately inner endomorphism of rank
$d(\lambda)\mu_j^\lambda=d(\rho)^{2n}\varphi^S(T_j^\lambda T_j^{\lambda*})$.
\end{proof}

\begin{proof}[Proof of Theorem \ref{thm:subfactor2}]
We may and do assume that $N=Q$.
Let us first consider the case of $M$, $N$ and $P$ being infinite factors.
Let $\rho\colon N\to M$ and $\sigma\colon N\to P$ be the inclusion maps.
Let $\sC^\rho$ and $\sC^\sigma$ be the rigid C$^*$-2-category
associated with $\rho$ and $\sigma$, respectively.
Their natural actions on $(N,M)$ and $(N,P)$
are denoted by $\alpha^\rho$ and $\alpha^\sigma$.

Let $(S_\rho,\ovl{S}_\rho)$ and $(S_\sigma,\ovl{S}_\sigma)$
be the solutions of the conjugate equations of $\rho$ and $\sigma$,
respectively,
such that $E(x)=\rho(S_\rho^* \orho(x) S_\rho)$
and $F(y)=\sigma(S_\sigma^*\ovl{\sigma}(y)S_\sigma)$
for $x\in M$ and $y\in P$ as before.
Let $\varphi_\rho^S:=\lim_n \phi_{(\orho\rho)^n}^S$
and $\varphi_\sigma^S:=\lim_n \phi_{(\osi\sigma)^n}^S$
be states on $A_{\orho\rho}$ and $A_{\osi\sigma}$, respectively.

By assumption of the isomorphic standard invariants,
we have a unitary tensor equivalence
$(G,c)\colon \sC^\rho\to\sC^\sigma$
such that $G(\rho)=\sigma$,
$(c_{\orho,\rho}G(R_\rho),c_{\rho,\orho}G(\ovl{R}_\rho))
=(R_\sigma,\ovl{R}_\sigma)$
and $(c_{\orho,\rho}G(S_\rho),c_{\rho,\orho}G(\ovl{S}_\rho))
=(S_\sigma,\ovl{S}_\sigma)$.
Let $S_\rho^n$ and $S_\sigma^n$ be as defined
in the proof of Lemma \ref{prop:centappsubfactor} (2).
Let $h_\rho,k_\orho$
and $h_\sigma,k_\osi$
be the positive invertible operators
introduced after Remark \ref{rem:ergodicity}.
Then $G(h_\rho)=h_\sigma$ and $G(k_{\orho})=k_{\osi}$,
and $c_{\orho,\rho}G(a_{\orho\rho})c_{\orho,\rho}^*
=\osi(h_\sigma)k_\osi=a_{\osi\sigma}$.
This implies
$c_{w_{00}^n}G(a_{\orho\rho}^n)c_{w_{00}^n}^*
=a_{\osi\sigma}^n$,
where $w_{00}^n$ is the word expressing
$(\orho\rho)^n$, that is,
$w_{00}^n=(\orho,\rho,\dots,\orho,\rho)$
of length $2n$.

It follows from Lemma \ref{prop:centappsubfactor} (1)
that the actions $\alpha^\rho$ of $\sC_{00}^\rho$ and
$\alpha^\sigma$ of $\sC_{00}^\sigma$ on $N$
are centrally free.
We will show $\alpha_\lambda^\rho=\lambda$
and $\alpha_{G(\lambda)}^\sigma=G(\lambda)$
are approximately unitarily equivalent for all $\lambda\in\sC_{00}^\rho$.
It suffices to consider $\lambda$ being a simple object.
Take $n\geq1$ so that $\lambda\prec(\orho\rho)^n$.
Then from the proof of Lemma \ref{prop:centappsubfactor} (2),
$\lambda\in\End(N)_0$ is an approximately inner endomorphism
of rank $d(\rho)^{2n}\varphi_\rho^S(TT^*)$ for an isometry
$T\in\sC_{00}^\rho(\lambda,(\orho\rho)^n)$
that is an eigenvector of $a_{\orho\rho}^n$.

Then we have
$a_{\osi\sigma}^n c_{w_{00}^n}G(T)=c_{w_{00}^n}G(a_{\orho\rho}^n T)$,
and $c_{w_{00}^n}G(T)\in\sC_{00}^\sigma(G(\lambda),(\osi\sigma)^n)$
is an eigenvector of $a_{\osi\sigma}^n$.
Hence
$G(\lambda)\in\End(N)_0$
is an approximately inner endomorphism of rank
$d(\sigma)^{2n}\varphi_\sigma^{S}(c_{w_{00}^n}G(T)G(T)^*c_{w_{00}^n}^*)$,
which equals $d(\rho)^{2n}\varphi_\rho^{S}(TT^*)$.
Indeed, using $c_{w_{00}^n}G(S_\rho^n)=S_\sigma^n$,
we have
\begin{align*}
\varphi_\sigma^{S}(c_{w_{00}^n}G(T)G(T)^*c_{w_{00}^n}^*)
&=
\phi_{(\ovl{\sigma}\sigma)^n}^{S}(c_{w_{00}^n}G(T)G(T)^*c_{w_{00}^n}^*)
\\
&=
S_\sigma^{2n*}
(\ovl{\sigma}\sigma)^n
(c_{w_{00}^n}G(T)G(T)^*c_{w_{00}^n}^*)
S_\sigma^{2n}
\\
&=
G(S_\rho^{2n})^*
c_{w_{00}^{2n}}^*
(\ovl{\sigma}\sigma)^n
(c_{w_{00}^n}G(T)G(T)^*c_{w_{00}^n}^*)
c_{w_{00}^{2n}}
G(S_\rho^{2n})
\\
&=
G(S_\rho^{2n})^*
c_{(\orho\rho)^n,(\orho\rho)^n}^*
G((\orho\rho)^n)
(G(TT^* ))
c_{(\orho\rho)^n,(\orho\rho)^n}
G(S_\rho^{2n})
\\
&=
G(S_\rho^{2n})^*
G((\orho\rho)^n(TT^*))
G(S_\rho^{2n})
\\
&=
G(\phi_{(\orho\rho)^n}^S(TT^*))
=\varphi_\rho^S(TT^*),
\end{align*}
where
in the fourth equality,
we have used the following equalities:
\[
c_{w_{00}^{2n}}
=(\osi\sigma)^n(c_{w_{00}^n})c_{w_{00}^n}
c_{(\orho\rho)^n,(\orho\rho)^n},
\quad
\Ad c_{w_{00}^n}^*\circ(\osi\sigma)^n
=
G((\orho\rho)^n).
\]
By Proposition \ref{prop:appuer},
we see $\lambda$ and $G(\lambda)$ are approximately unitarily
equivalent.

Then it turns out from Corollary \ref{cor:CDclass}
that the action $\alpha^\rho$ and the cocycle action
$(\alpha_{G(\cdot)}^\sigma,c)$ of $\sC_{00}^\rho$
are strongly cocycle conjugate
via an approximate inner conjugacy $\pi_0\colon N\to N$.
Applying Theorem \ref{thm:M1P1isom} to $Z:=\rho$,
we have an isomorphism
$\pi_1\colon M\to P$
such that
$\pi_1\circ\alpha_Z^\rho=\alpha_{G(Z)}^\sigma\circ\pi_0$
and $\pi_1(T)=G(T)$ for all $T\in(\rho,\rho)$.
Hence we obtain $\pi_1\rho=\sigma\pi_0$,
which shows the inclusions $\rho(N)\subset M$
and $\sigma(Q)\subset P$ are isomorphic,
but we must further consider the conditional expectation
$\pi_1(E):=\pi_1\circ E\circ\pi_1^{-1}$
from $P$ onto $\sigma(Q)$.
By uniqueness of minimal expectation,
we have $\pi_1(E_\rho)=E_\sigma$,
that is, $\pi_1\circ E_\rho=E_\sigma\circ\pi_1$.
Since $\pi_1(h_\rho)=G(h_\rho)=h_\sigma$,
we obtain
\begin{align*}
\pi_1(E)(x)
&=
\pi_1(E(\pi_1^{-1}(x)))
=
\pi_1(E_\rho(h_\rho \pi_1^{-1}(x) h_\rho))
\\
&=
E_\sigma(\pi_1(h_\rho) x \pi_1(h_\rho))
=
E_\sigma(h_\sigma x h_\sigma)
=
F(x).
\end{align*}
Hence $\pi_1$ is an isomorphism
from $\rho(N)\stackrel{E}{\subset}M$
to $\sigma(Q)\stackrel{F}{\subset}P$.

Next we consider the case of $M$, $N$ and $P$ being finite factors.
Let $B$ be the infinite dimensional type I factor with separable predual.
Denote by $\rho$ and $\sigma$ be the inclusion maps
of $N$ into $M$ and $P$, respectively.
We consider the inclusions
$N\otimes B\stackrel{E\otimes\id}{\subset} M\otimes B$
and
$N\otimes B\stackrel{F\otimes\id}{\subset} P\otimes B$.
Then by the discussion above,
we obtain isomorphisms $\pi_0\colon N\otimes B\to N\otimes B$
and $\pi_1\colon M\otimes B\to P\otimes B$
such that
$\pi_1\circ(\rho\otimes\id)\circ\pi_0^{-1}=(\sigma\otimes\id)$
and
$\pi_0\circ (E\otimes\id)\circ\pi_1^{-1}=F\otimes\id$.

By $\tau_M$ and $\tau_P$ we denote the unique tracial states
on $M$ and $P$, respectively.
The canonical tracial weight on $B$ is denoted by $\Tr$.
Since we can take $\pi_0$ being
an approximately inner automorphism on $N\otimes B$,
and we have $\pi_0(\tau_N\otimes\Tr)=\tau_N\otimes\Tr$.
Note the approximate innerness of $E$ and $F$ implies
$\tau_N\circ E=\tau_M$ and $\tau_N\circ F=\tau_P$
from \cite[Proposition 2.6]{Popa-endo}.
Then we have
\begin{align*}
\pi_1(\tau_M\otimes\Tr)
&=
(\tau_M\otimes\Tr)\circ\pi_1^{-1}
=
(\tau_N\otimes\Tr)\circ(E\otimes\id)\circ\pi_1^{-1}
\\
&=\pi_0(\tau_N\otimes\Tr)\circ(F\otimes\id)
=\tau_P\otimes\Tr.
\end{align*}

Hence there exist unitaries $v\in N\otimes B$
and $w\in M\otimes B$
such that
$\Ad v\circ\pi_0(1\otimes x)=1\otimes x$
and
$\Ad w\circ\pi_1(1\otimes x)=1\otimes x$
for $x\in B$.
Thus we can take isomorphisms $\alpha\colon N\to N$
and $\beta\colon M\to P$
such that
$\Ad v\circ\pi_0(x\otimes1)=\alpha(x)\otimes1$
and
$\Ad w\circ\pi_1(y\otimes 1)=\beta(y)\otimes 1$
for $x\in N$ and $y\in M$.
Then we have
$\beta\rho\alpha^{-1}(x)\otimes1=\Ad (w(\sigma\otimes\id)(v^*))(\sigma(x)\otimes1)$
for all $x\in N$.
By standard argument,
we can take a unitary $u\in P$ so that
$\beta\rho\alpha^{-1}=\Ad u\circ \sigma$,
and we are done.
\end{proof}

\end{document}